\documentclass[11pt,centertags]{amsart}
\usepackage{amssymb}            
\usepackage[shortlabels]{enumitem}           
\usepackage{url}                
\usepackage{gloss}              
\usepackage[english]{babel}     
\usepackage[pdftex]{graphicx}
    \DeclareGraphicsExtensions{.eps,.pdf,.png,.jpg}
    \graphicspath{{}}
\usepackage{mathtools}          
\usepackage[marginratio=1:1,height=8.5in,width=6.5in,tmargin=1.25in]{geometry}
\usepackage{comment}            
\usepackage[normalem]{ulem}     
\usepackage[usenames,dvipsnames]{color}

\usepackage{todonotes}          
\usepackage{hyperref}
\hypersetup{
    colorlinks,
    citecolor=red,
    filecolor=orange,
    linkcolor=blue,
    urlcolor=purple
  }

\usepackage[T1]{fontenc}         
\usepackage[utf8]{inputenc}      
\usepackage{textcomp}            
\usepackage{bbm}                
\usepackage{mathpazo}           
\usepackage{mathrsfs}           

\allowdisplaybreaks

\newtheorem*{main*}{Main Theorem}

\newtheorem*{theorem*}{Theorem}

\newtheorem{question}{Question}
\newtheorem*{question*}{Question}

\newtheorem*{conjecture*}{Conjecture}

\theoremstyle{definition}

\newtheorem*{definition*}{Definition}

\theoremstyle{remark}

\numberwithin{equation}{section}

\newcommand{\isom}{\cong} 

\newcommand{\Isom}{\mathrm{Isom}}

\newcommand{\CC}{\mathbb{C}}

\newcommand{\RR}{\mathbb{R}}

\newcommand{\TT}{\mathbb{T}}

\newcommand{\ZZ}{\mathbb{Z}}

\newcommand{\bV}{\mathbf{V}}

\newcommand{\cA}{\mathcal{A}}

\newcommand{\cC}{\mathcal{C}}
\newcommand{\cD}{\mathcal{D}}
\newcommand{\cE}{\mathcal{E}}
\newcommand{\cF}{\mathcal{F}}

\newcommand{\cI}{\mathcal{I}}

\newcommand{\cL}{\mathcal{L}}
\newcommand{\cM}{\mathcal{M}}
\newcommand{\cN}{\mathcal{N}}

\newcommand{\cQ}{\mathcal{Q}}
\newcommand{\cR}{\mathcal{R}}
\newcommand{\cS}{\mathcal{S}}
\newcommand{\cT}{\mathcal{T}}
\newcommand{\cU}{\mathcal{U}}
\newcommand{\cV}{\mathcal{V}}
\newcommand{\cW}{\mathcal{W}}

\newcommand{\del}{\partial}

\newcommand{\id}{\mathrm{Id}}

\newcommand{\matii}[2]{\begin{pmatrix} #1\\#2\end{pmatrix}}

\newcommand{\mativ}[4]{\begin{pmatrix} #1\\#2\\#3\\#4\end{pmatrix}}
\newcommand{\matv}[5]{\begin{pmatrix} #1\\#2\\#3\\#4\\#5\end{pmatrix}}

\newcommand{\pr}{\mathrm{pr}}

\newcommand{\Sing}{\mathrm{Sing}}

\newcommand{\sn}{\mathbf{s}}
\newcommand{\cn}{\mathbf{c}}

\newcommand{\tY}{\widetilde{Y}}
\newcommand{\tc}{\widetilde{c}}
\newcommand{\tg}{\widetilde{g}}

\theoremstyle{definition}
\newtheorem{thm}{Theorem}[section]

\newtheorem{prop}[thm]{Proposition}
\newtheorem{dfn}[thm]{Definition}
\newtheorem{subdfn}[thm]{Definition}
\theoremstyle{definition}
\newtheorem{lem}[thm]{Lemma}
\newtheorem{sublem}[thm]{Lemma}

\newtheorem*{rmk}{Remark}
\newtheorem{cor}[thm]{Corollary}
\newtheorem{fact}[thm]{Fact}

\begin{document}

\author[Keith Burns]{Keith Burns}

\author[Tomas Gedeon]{Tomas Gedeon$^\dagger$}
\thanks{$\dagger$ This project was sponsored in part by the Air Force Research Laboratory under Agreement Number r FA8750-25-C-B027 and by grant NSF-DMS 2601860.}

\author[Yuping Ruan]{Yuping Ruan}


\title{A smooth Riemannian metric on an arbitrary product of spheres with Bernoulli geodesic flow}

\address{Northwestern University} 
\email{burns@math.northwestern.edu} 

\address{Montana State University}
\email{tgedeon@montana.edu}

\address{Yale University}
\email{ruan.yuping@yale.edu}

\begin{abstract}
We construct a smooth Riemannian metric on $S^{k_1}\times\dots\times S^{k_m}$ for any $m,k_1,\dots, k_m\in\ZZ_+$ whose geodesic flow is Bernoulli with respect to the Liouville measure.
\end{abstract}
\maketitle

\tableofcontents


%
%

\section{Introduction}
Let $(M,g)$ be a closed Riemannian manifold. The \emph{geodesic flow} of $(M,g)$ is a smooth flow $\{\varphi_t\}_{t\in\RR}$ defined on $TM$ such that for any vector $v\in TM$ and for any $t\in\RR$, we have $\varphi_t(v)=c'_v(t)$, where $c_v:\RR\to M$ is the unique geodesic in $(M,g)$ satisfying $c'_v(0)=v$. Since the geodesic flow preserves the length of vectors, unless otherwise mentioned, we focus on the geodesic flow restricted to the unit tangent bundle $T^1(M,g)$. 

E. Hopf \cite{Hop39} proved in the late 1930s the \emph{ergodicity} of the geodesic flow for compact surfaces with variable negative curvature. (See also \cite{Hed34, Hed35,Hop36} for hyperbolic surfaces. Throughout the rest of this paper, ergodicity of the geodesic flow always means ergodicity with respect to the Liouville measure.) The only orientable compact surfaces which do not carry negatively curved Riemannian metrics are the $2$-sphere and the $2$-torus. Later, examples of Riemannian metrics with ergodic geodesic flow on the $2$-sphere were constructed by Burago \cite{Bur88}, Donnay \cite{Don88a} and an unpublished result of Osserman \cite[A.3, Page 552]{Don88a}. (Examples on the $2$-torus can be constructed in a similar way.) Hence, any orientable compact surface admits a Riemannian metric with ergodic geodesic flow. It is then natural to ask the following question:
\begin{question}[{\cite[Question 0.1]{BG94}}]\label{Q:ergodicity}
Does every closed Riemannian manifold admit a Riemannian metric whose geodesic flow is ergodic, or satisfies stronger stochastic properties?
\end{question}
Similar to the surface case, Anosov \cite{Ano67} proved that the geodesic flow of any closed Riemannian manifold of negative curvature satisfies the $K$-property, which is stronger than ergodicity. If we remove the curvature assumptions, some partial results on Question \ref{Q:ergodicity} have been obtained. Katok-Burns \cite{KB94} constructed a smooth Riemannian metric with ergodic geodesic flow on any closed $3$-manifold. A notion of \emph{tubes} is introduced in \cite{KB94} as a higher dimensional generalization of the smooth spherical caps in \cite{Don88a}. Following this line, Burns and Gerber \cite{BG94} introduced \emph{nearly nonpositively curved Riemannian manifolds}, which are obtained by gluing nonpositively curved pieces and a generalized version of tubes in \cite{KB94}. The same paper \cite{BG94} constructed a smooth Riemannian metric with ergodic geodesic flow on any product manifold whose factors have dimension at most $3$. However, Question \ref{Q:ergodicity} remains largely open despite some reports of its solution \cite[page 87]{Ber00} and \cite[\S 10.9]{Ber03}. For example, Question \ref{Q:ergodicity} is still open for $\CC\mathbf{P}^2$.

When it comes to stronger stochastic properties of geodesic flows, Ornstein and Weiss \cite{OW73} showed that the geodesic flows on closed surfaces with variable negative curvature are \emph{Bernoulli}. For closed negatively curved manifolds of dimension greater than two, the Bernoulli property of geodesic flows was verified by Bunimovich \cite{Bun73, Bun74} and Ratner \cite{Rat74}, strengthening the above result by Anosov. Shortly after Donnay's work, analytic metrics with Bernoulli geodesic flow on $2$-spheres were constructed in \cite{BG89a}. The Bernoulli property of the examples in \cite{KB94,BG94} is established in \cite{KB94}.

Here and below, for any $n\in\ZZ_+$, we let $S^n$ be the $n$-dimensional sphere. Also, in this paper, spheres are always equipped with the standard differential structure.

The present paper constructs a smooth Riemannian metric with Bernoulli geodesic flow on any product of spheres. Namely, we have the following theorem.
\begin{thm}\label{thm:main}
Let $k_1,\dots,k_q\in\ZZ_+$ with $k_1+\dots+k_q\geq 2$. Then the product manifold $S^{k_1}\times \dots \times S^{k_q}$ admits a smooth Riemannian metric whose geodesic flow is Bernoulli. 
\end{thm}

\subsection*{Outline of the proof and plan of the paper} We first consider the special case of the $2$-dimensional torus $\TT^2=S^1\times S^1$. Any Riemannian metric $g$ satisfying the following properties has Bernoulli geodesic flow with respect to the Liouville measure:
\begin{itemize}
\item[\hypertarget{G1}{(G1)}]  $g$ is the standard flat torus metric except on a collection of finitely many disjoint discs $D_1,\dots, D_k$. 
\item[\hypertarget{G2}{(G2)}] For any $j\in\{1,\dots,k\}$, $(D_j,g|_{D_j})$ is an \emph{extended cap}. In particular, positive curvature only appears in extended caps. (See Definition \ref{dfn:caps} or \cite[Section 5]{BG94}).
\end{itemize}
Such Riemannian metrics are \emph{nearly nonpositively curved} in the sense of \cite{BG94}. (See Definition \ref{dfn:near nonpos curv} or \cite[Definition 5.1-Corollary 5.4]{BG94}.) In particular, by certain properties of extended caps, the geodesic flows of such manifolds behave in a similar way to the geodesic flows of nonpositively curved Riemannian manifolds. (See \cite[Fundamental Lemma 5.12]{BG94}, or the discussion between Definition \ref{dfn:near nonpos curv} and Definition \ref{dfn:geo rank}.) One can also generalize the definition of rank-one vectors from the nonpositively curved setting to the nearly nonpositively curved setting. (See Definition \ref{dfn:geo rank}.)

We say that an extended cap is \emph{inserted} to a surface if we remove a flat disc from the surface and replace it with an extended cap. Then for any Riemannian metric $g$ satisfying \hyperlink{G1}{(G1)} and \hyperlink{G2}{(G2)}, $(\TT^2,g)$ is obtained by inserting finitely many extended caps into the standard flat torus $\TT^2$. The insertion of extended caps in \hyperlink{G2}{(G2)} destroys the original product structure of the standard flat torus. A geodesic in $(\TT^2,g)$ does not have rank one if and only if one of the following holds:
\begin{itemize}
\item This geodesic does not intersect any extended cap. (Hence it is a geodesic for the flat torus which avoids a fixed open subset.)
\item This geodesic is the unique periodic geodesic inside an extended cap.
\end{itemize} 
As a corollary of the above, the subset of rank one vectors in $T^1(\TT^2,g)$ is connected and has full Lebesgue measure. This proves the Bernoulli property of the geodesic flow using \cite[Theorem 5.17]{BG94}. (See also Theorem \ref{thm:main criterion}.) 

The above construction can be generalized to any higher-dimensional torus $\TT^n$ with $n\geq 3$. Instead of inserting extended caps, we insert \emph{special tubes} to $\TT^n$. Here, special tubes are defined as the product of an extended caps and a codimension two torus with an arbitrary Riemannian metric. (See Definition \ref{dfn:special tubes} when $k=1$.) However, since we need to avoid product metrics when seeking Bernoulli geodesic flows, different types of tubes are needed in the construction of the Riemannian metric to make the corresponding geodesic flow Bernoulli. For example, if all inserted tubes in $\TT^3$ have cap parts in the first two components, the resulting Riemannian manifold is a product manifold and hence can never have Bernoulli geodesic flow. (See Figure \ref{fig:intro1}.) To obtain a Riemannian metric with Bernoulli geodesic flow on $\TT^3$, one can first insert a tube whose cap part is in the first two components, and then insert a tube whose cap part is in the last two components. (See Figure \ref{fig:intro2}.)

\begin{figure}[h]
	\centering
	\includegraphics[width=3in]{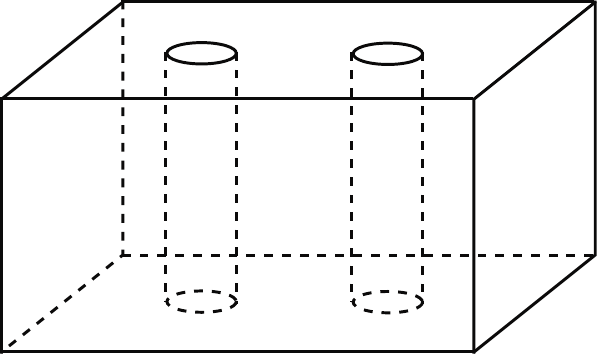}
	\caption{}\label{fig:intro1}
\end{figure}

\begin{figure}[h]
	\centering
	\includegraphics[width=3in]{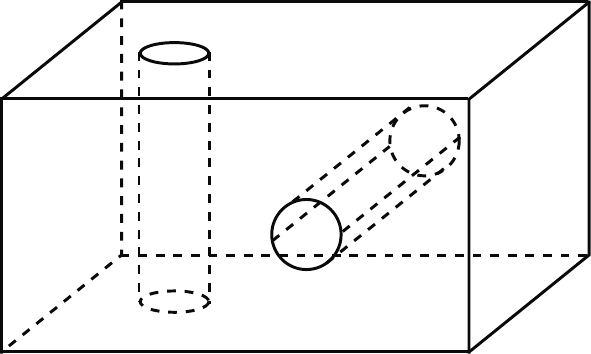}
	\caption{}\label{fig:intro2}
\end{figure}

The proof of Theorem \ref{thm:main} heavily relies on a special branched covering of $S^n$ by the $n$-dimensional torus. To illustrate, we consider the special case when $n=2$. Identifying $\TT^2$ with $\RR^2/\ZZ^2$, the quotient of $\TT^2$ under the relation $(x,y)\sim(-x,-y)$ is homeomorphic to $S^2$. (See Figure \ref{fig:intro3}.) In other words, one can find a discrete, abelian subgroup $\Gamma$ of the isometry group of $\TT^2$ (with respect to the standard flat torus metric) such that the following holds:
\begin{itemize}
\item $\Gamma\isom \ZZ/2\ZZ$.
\item $\Gamma\backslash\TT^2$ is homeomorphic to $S^2$.
\end{itemize}
Moreover, the natural projection map $\Pi:\TT^2\to\Gamma\backslash\TT^2$ has four fixed points. Under a suitable coordinate chart near any of these fixed points, the projection map restricted to this coordinate chart is given by $z\to z^2$ defined on a neighborhood of $0\in\CC$. To prove Theorem \ref{thm:main} for $S^2$ identified with $\Gamma\backslash\TT^2$, we construct our desired Riemannian metric $g$ on $S^2$ such that the following holds:
\begin{itemize}
\item[(G1')] $\Pi^*g$ is the standard flat torus metric except on four disjoint discs $D_1,\dots, D_4$, each of which is centered at a fixed point of $\Pi$. (Compare this with  \hyperlink{G1}{(G1)}.)
\item[(G2')] For any $j\in\{1,\dots,4\}$, $(D_j,g|_{D_j})$ is an ``extended cap with a rotational singularity at its center''. (See Definition \ref{dfn:special tubes} for details. Compare this with  \hyperlink{G2}{(G2)}.)
\end{itemize}
\begin{figure}[h]
	\centering
	\includegraphics[width=3in]{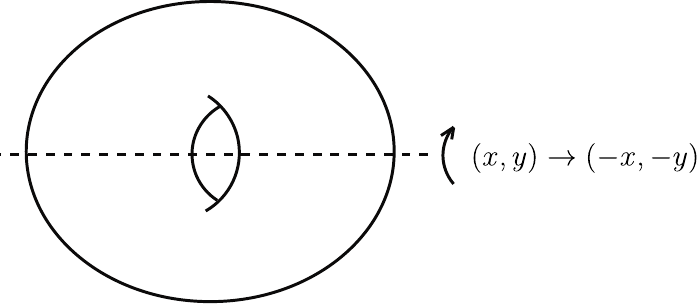}
	\caption{}\label{fig:intro3}
\end{figure}
The general construction in the proof of Theorem \ref{thm:main} is analogous to the above special case of $S^2$. Roughly speaking, we first observe that any product of spheres with total dimension $n$ is homeomorphic to $\Gamma\backslash \TT^n$ for some discrete, abelian subgroup $\Gamma$ of the isometry group of the standard flat torus $\TT^n$. Then, based on the ramification locus of the natural projection map $\Pi:\TT^n\to \Gamma\backslash \TT^n$, we construct our desired Riemannian metric $g$ on $\Gamma\backslash \TT^n$ so that $(\TT^n,\Pi^*g)$ is obtained by inserting tubes and ``tubes with rotational singularities at their centers'' into the standard $n$-dimensional flat torus.

The rest of this paper is organized as follows: Let $M=S^{k_1}\times \dots\times S^{k_q}$ with $k_1,\dots,k_q\geq 1$. We let $k_1+\dots+k_q=n\geq 2$.
\begin{enumerate}
\item In Section \ref{sect:branched.covering}, we first introduce the discrete, abelian subgroup $\Gamma$ of $\TT^n$ so that $\Gamma\backslash\TT^n$ is homeomorphic to $M$. (See Subsection \ref{subsect:good.subgp}.) Then we verify that the standard projection map $\TT^n\to\Gamma\backslash\TT^n$ is a smooth branched covering map, where the differential structure of $\Gamma\backslash\TT^n$ is induced by the standard differential structure on $M$. Moreover, for any $p\in\TT^n$, under a suitable coordinate chart and identifying $p$ with $0\in \CC^k\times \RR^{n-2k}$, the projection map restricted to this coordinate chart is given by the map $(z_1,\dots,z_k,t_{2k+1},\dots, t_n)\to(z_1^2,\dots, z_k^2,t_{2k+1},\dots, t_n)$ defined on a neighborhood of $0$ in $\CC^k\times \RR^{n-2k}$. (See Subsection \ref{subsect:proj=branched.cover}.) Some technical details regarding the ramification locus and branching locus can be found in Subsection \ref{subsect:locus}. 

The main result in Section \ref{sect:branched.covering} is Proposition \ref{prop:quotient.is.sphere}, whose proof is related to the classical ellipsoidal coordinates. We leave the main part of its proof to Appendix \ref{appA} due to its technical nature.  
\item Section \ref{sect:review} is a brief review of \cite[Sections 2-5]{BG94}. This includes the notion of nearly nonpositively curved Riemannian metrics, and the main criterion to prove the Bernoulli property of geodesic flows for nearly nonpositively curved Riemannian manifolds. 
\item In Section \ref{sect:construction}, we first introduce special caps and special tubes, possibly with rotational singularities at their centers. (See Subsection \ref{subsect:special tubes}. The existence of special caps is proved in Appendix \ref{sect:special.caps.exist}.) Then we introduce the language of inserting special tubes and describe the construction of our desired Riemannian metric for Theorem \ref{thm:main}. (See Definition \ref{dfn:insert/forget tubes}, Subsection \ref{subsect:constructions-1} and Subsection \ref{subsect:constructions-2}.) We also prove some necessary properties of this metric and its pullback to the covering torus in Lemma \ref{lem:properties.of.tg} and Lemma \ref{lem:properties.of.g.Gamma}. (For example, this metric is nearly nonpositively curved.)
\item Section \ref{sect:proof.of.thm} verifies the Bernoulli property of the geodesic flow for the Riemannian metric constructed in Subsection \ref{subsect:constructions-1}. The goal is to show that the set of all rank one vectors under this metric contains a connected subset of full Lebesgue measure, which requires the following two preparatory steps: (1) To study rank one unit vectors, we introduce the notion of \emph{nomadic} geodesic. Roughly speaking, a geodesic is nomadic if its lift in the covering torus visits all types of special tubes. Subsection \ref{subsect:nomadic} verifies that a unit vector has rank one if and only if the bi-infinite geodesic tangent to this vector is nomadic. (2) The desired subset of rank one vectors is defined in Subsection \ref{subsect:normal vectors}, which excludes some vectors with ``special angles''. (More precisely, this subset is defined in the statement of Theorem \ref{thm:main.true}, whose formulation heavily relies on Subsection \ref{subsect:normal vectors}.) Theorem \ref{thm:main} is proved in Subsection \ref{subsect:proof}. 
\item Appendix \ref{sect:garbage} contains supplementary details concerning some claims stated without proof in the main text.
\end{enumerate}

\section{Branched covering of $S^n$ by products of lower dimensional spheres}\label{sect:branched.covering}
Throughout this paper, unless otherwise mentioned, we denote by $\TT^1=\RR/2\pi\ZZ$. In other words, $\TT^1$ is the circle with arclength $2\pi$. For any integer $i\in\ZZ_+$, we associate with $i$ a copy of $\TT^1$, which is denoted by $\TT^1_i$. Unless otherwise mentioned, an arbitrary element in $\TT^1_i$ is always denoted by $\theta_i$. For any nonempty finite subset $A=\{i_1,\dots, i_k\}\subset \ZZ_+$ with $i_1,\dots,i_k$ distinct, we define 
\begin{align}\label{eqn:tori}
\TT_A=\TT_{i_1,\dots,i_k}:=\TT^1_{i_1}\times\cdots \times \TT^1_{i_k}.
\end{align} 
It is clear from the definition that $\dim \TT_A=|A|$. For simplicity, we let $\bV_n:=\{1,\dots,n\}$ and $\TT^n:=\TT_{\bV_n}$.

Let $\TT_{\emptyset}$ be a point. For any finite subset $A\subset\ZZ_+$, we identify $\TT_A$ with $\TT_A\times \TT_{\emptyset}$ naturally. Then for any finite subsets $A\subset A'\subset \ZZ_+$, we have 
\begin{align}\label{eqn:tori.natural.prod}
\TT_A\times \TT_{A'\setminus A}=\TT_{A'}.
\end{align}
We denote by 
\begin{align}\label{eqn:tori.proj}
\pr^{A'}_A:\TT_{A'}\to\TT_A
\end{align}
the natural projection. If $A=\{i_1,\dots, i_k\}$ and $A'=\{i_1,\dots,i_m\}$ for some $m\geq k>0$, we may omit the curly brackets for simplicity and use the notation $\pr^{i_1,\dots,i_m}_{i_1,\dots,i_k}$ instead of $\pr^{A'}_A$.

For any finite subset $A=\{i_1,\dots,i_k\}\subset \ZZ_+$ with $i_1,\dots,i_k$ distinct, for any $p\in\TT_A$ and for any $r\in\RR_+$, we let 
\begin{align}\label{eqn:ball.convention}
B_A(p,r)=B_{i_1,\dots,i_k}(p,r):=\{q\in\TT_A:d_{\TT_A}(p,q)<r\},
\end{align}
where $d_{\TT_A}(\cdot,\cdot)$ is the distance with respect to the standard product metric on $\TT_A$.

\subsection{A $(\ZZ/2\ZZ)^{n-1}$ action on $\TT^n$}\label{subsect:good.subgp}
We define $\rho:\TT^1\to \TT^1$ and $\kappa_y:\TT^1\to \TT^1$ such that $\rho(x)=x+\pi$ and $\kappa_y(x)=y-x$ for any $y\in \TT^1$. Then we have 
\begin{align}\label{eqn:commute}
\begin{cases}
\displaystyle \rho\circ\rho=\kappa_y\circ\kappa_y=\mathrm{Id},\forall y\in \TT^1,\\
\displaystyle \rho(\kappa_\pi(x))=\kappa_\pi(\rho(x))=\kappa_0(x), \\
\displaystyle \rho(\kappa_0(x))=\kappa_0(\rho(x))=\kappa_\pi(x),\\
\displaystyle \kappa_0(\kappa_\pi(x))=\kappa_\pi(\kappa_0(x))=\rho(x).
\end{cases}
\end{align}
For any $i\in\ZZ_+$, we define 
\begin{align}\label{eqn:generator}
\gamma_i:\TT_{i,i+1}\to\TT_{i,i+1},~\gamma_i(\theta_i,\theta_{i+1})=(\kappa_\pi(\theta_i),\kappa_0(\theta_{i+1})).
\end{align}
One can easily see that $\gamma_i$ is an isometry. By \eqref{eqn:tori.natural.prod}, for any finite subset $A\subset \ZZ_+$ such that $\{i,i+1\}\in A$, $\gamma_i$ naturally extends to an isometry $\gamma_i\times \id_{\TT_{A\setminus\{i,i+1\}}}$ on $\TT_A$. To avoid cumbersome notations, we use $\gamma_i$ instead of $\gamma_i\times  \id_{\TT_{A\setminus\{i,i+1\}}}$ for any finite subset $A\subset \ZZ_+$ such that $\{i,i+1\}\subset A$.

For any finite subset $A\subset \ZZ_+$, we let $\Gamma_A$ be the subgroup of $\Isom(\TT_A)$ defined by the following:
\begin{align}\label{eqn:Gamma.A}
\Gamma_A=\langle\cS_A\rangle\text{, where }\cS_A:=\{\gamma_i: \{i,i+1\}\subset A\}\subset \Isom(\TT_A).
\end{align} 
If $\cS_A=\emptyset$, we simply let $\Gamma_A=\{\id_{\TT_A}\}$, the trivial subgroup of $\Isom(\TT_A)$.) For simplicity, we write $\cS_n:=\cS_{\bV_n}$ and $\Gamma_n:=\Gamma_{\bV_n}$ for any $n\in\ZZ_+$. Following \eqref{eqn:Gamma.A} and the convention on $\gamma_i$ after \eqref{eqn:generator}, for any finite subsets $A\subset A'$ of $\ZZ_+$, when acting on $\TT_{A'}$, we identify $\Gamma_A$ as a subgroup of $\Gamma_{A'}$.

\begin{dfn}[Consecutive subsets]\label{dfn:consecutive}
A finite subset $A$ of $\ZZ_+$ is \emph{consecutive} if $A=\{m+1,\dots, m+k\}$ for some $m\in\ZZ_{\geq 0}$ and $k\in\ZZ_+$. 

A partition $\{A_1,\dots, A_l\}$ of $A$ is a \emph{consecutive partition} if $A_1,\dots, A_l$ are all consecutive.  $A_1,\dots,A_l$ are called \emph{atoms} of this partition.

If $A=\{m+1,\dots, m+k\}$ and $\{A_1,\dots, A_l\}$ is a consecutive partition of $A$, we let
\begin{align}\label{eqn:cutpoint}
\cE(A):=A\setminus\{m+k\}\text{ and } E_A(A_1,\dots, A_l):=\{i\in\cE(A):\{i,i+1\}\not\subset A_j,\forall~j\in\{1,\dots,l\}\}.
\end{align}
\end{dfn}
\begin{figure}[h]
	\centering
	\includegraphics[width=4in]{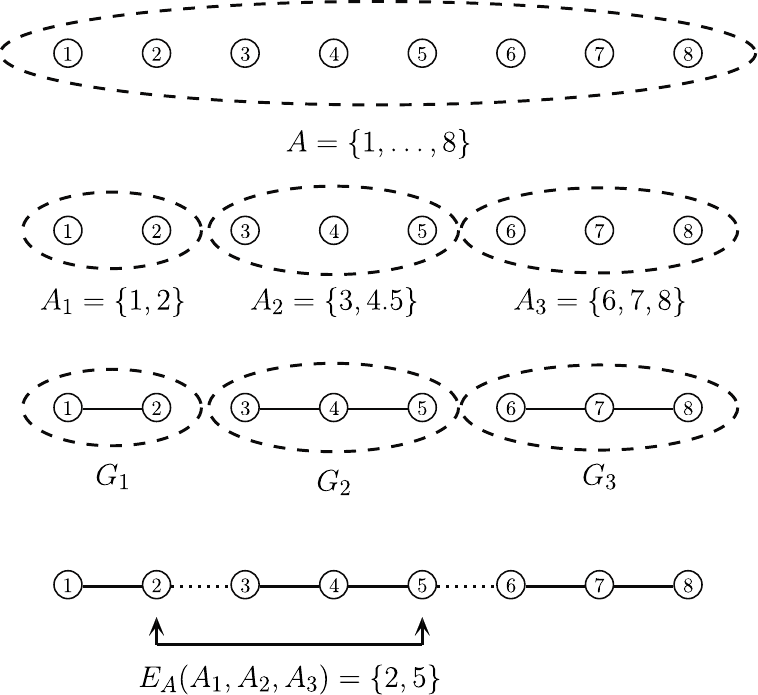}
	\caption{}\label{fig:consecutive partition}
\end{figure}
\begin{rmk}
It is straightforward to see that $\cE(A)=\{i\in \ZZ_+:\{i,i+1\}\in A\}$. Hence, by \eqref{eqn:Gamma.A}, we have $\cS_A=\{\gamma_i:i\in\cE(A)\}$.

One can easily check that the map $\{A_1,\dots, A_l\}\to E_A(A_1,\dots, A_l)$ gives a bijection between the collection of consecutive partitions of $A$ and the collection of subsets of $\cE(A)$. For this reason, for any $E\subset\cE(A)$, a consecutive partition $\{A_1,\dots, A_l\}$ is \emph{induced} by $E$ if $E=E_A(A_1,\dots, A_l)$.

If we consider a graph $G(A)$ with vertices $A$ and edges $\{\{i,i+1\}:i\in\cE(A)\}$, then any consecutive partition $\{A_1,\dots, A_l\}$ of $A$ corresponds to a set of connected subgraphs $\{G_1,\dots, G_l\}$ of $G(A)$ such that any vertex in $G(A)$ is contained in exactly one of $G_1,\dots, G_l$. (Under this correspondence, the vertex set of $G_j$ is $A_j$ for any $j\in\{1,\dots,l\}$.) On the other hand, any such $\{G_1,\dots, G_l\}$ is obtained by removing all edges in $\{\{i,i+1\}:i\in E_A(A_1\dots, A_l)\}$ from $G(A)$. See Figure \ref{fig:consecutive partition}.
\end{rmk}

Throughout this paper, we are mostly interested in $\Gamma_A$ when $A$ is a consecutive subset. For any $E\subset\cE(A)$, we let 
\begin{align}\label{eqn:good.subgroup.struct}
\cS_A(E):=\{\gamma_i\in\cS_A:i\not\in E\}=\{\gamma_i:i\in\cE(A)\setminus E\}\text{ and }\Gamma_A(E):=\langle\cS_A(E)\rangle.
\end{align}
\begin{dfn}[Good subgroups]\label{dfn:good subgp}
Let $A$ be a consecutive subset of $\ZZ_+$. A subgroup $\Gamma$ of $\Gamma_A$ is \emph{good} if $\Gamma=\Gamma_A(E)$ for some $E\subset\cE(A)$. (Equivalently, $\Gamma$ is a good subgroup if it is generated by $\Gamma\cap\cS_A$.)
\end{dfn}
It is not hard to check that many (products of) good subgroups are canonically isomorphic, and that the actions of many (products of) good subgroups on the corresponding tori are canonically equivalent. (See Lemma \ref{lem:properties.good.subgp} for more details.) Here, the equivalence of group actions is defined as follows: (See \cite[page 35]{Bre72} for the equivalence of $G$-spaces, which is of the same spirit.) 
\begin{dfn}\label{dfn:iso.of.G.space}
Let $M_1,M_2$ be closed manifolds and $G_i$ be a subgroup of the diffeomorphism group $\mathrm{Diff}(M_i)$ of $M_i$ for any $i\in\{1,2\}$. We say that the $G_1$ action on $M_1$ is \emph{equivalent} to the $G_2$ action on $M_2$ if there exist a diffeomorphism $F:M_1\to M_2$ and a group isomorphism $\rho:G_1\to G_2$ such that for any $g\in G_1$, we have $F\circ g=\rho(g)\circ F$.
\end{dfn}
We summarize some elementary properties of $\Gamma_A$ and its good subgroups as follows:
\begin{lem}\label{lem:properties.good.subgp}
Let $n\in\ZZ_+$ and $A$ be a consecutive subset of $\ZZ_+$. Then the following holds:
\begin{enumerate}[ref=\ref{lem:properties.good.subgp}(\arabic*)]
\item $\Gamma_n$ is commutative. Moreover, the map $(i_1,\dots, i_{n-1})\to \gamma^{i_1}_1\circ\dots\circ\gamma^{i_{n-1}}_{n-1}$ gives an isomorphism from $(\ZZ/2\ZZ)^{n-1}$ to $\Gamma_n$. 
\label{lem:properties.good.subgp-(1)}
\item If $A=\{m+1,\dots, m+k\}$, the bijective map $\gamma_{m+j}\to\gamma_j$ from $\cS_A\to\cS_k$ naturally extends to an isomorphism from $\Gamma_A$ to $\Gamma_k$. Moreover, under this isomorphism and the isometry $(\theta_{m+1},\dots,\theta_{m+k})\to (\eta_{1},\dots,\eta_{k})$ from $\TT_A\to \TT^k$ with $\eta_i:=\theta_{m+i}$ for any $i\in\{1,\dots,k\}$, the $\Gamma_A$-action on $\TT_A$ is equivalent to the $\Gamma_k$ action on $\TT^k$. (See Definition \ref{dfn:iso.of.G.space}.) 
\label{lem:properties.good.subgp-(2)}
\item For any subset $E\subset\cE(A)$ and for any consecutive subset $A'\subset A$, we have $\Gamma_A(E)\cap \Gamma_{A'}=\Gamma_{A'}(E\cap\cE(A'))=\Gamma_A(E\cup(\cE(A)\setminus\cE(A')))$. (In particular, if $\Gamma$ is a good subgroup of $\Gamma_A$, $\Gamma\cap \Gamma_{A'}$ is a good subgroup of $\Gamma_{A'}$.) 
\label{lem:properties.good.subgp-(3)}
\item For any subsets $E_1,E_2\subset \cE(A)$, we have 
$$\Gamma_A(E_1)\cap\Gamma_A(E_2)=\Gamma_A(E_1\cup E_2)\text{ and }\langle\Gamma_A(E_1),\Gamma_A(E_2)\rangle=\Gamma_A(E_1\cap E_2).$$ 
(In particular, this shows that the intersection of good subgroups of $\Gamma_A$ is a good subgroup of $\Gamma_A$). Moreover, when $E_1\cup E_2=\cE(A)$, the map $(h_1,h_2)\to h_1h_2$ gives a natural isomorphism from $\Gamma_A(E_1)\times \Gamma_A(E_2)\to \Gamma_A(E_1\cap E_2)$. 
\label{lem:properties.good.subgp-(4)}
\item Let $\Gamma=\Gamma_A(E)$ be a good subgroup of $\Gamma_A$, where $E\subset\cE(A)$. Let $\{A_1,\dots, A_l\}$ be the consecutive partition of $A$ induced by $E$. (See the remark after Definition \ref{dfn:consecutive}.) Then for any good subgroup $\Gamma'$ of $\Gamma_A$, the map $(h_1,\dots,h_l)\to h_1h_2\dots h_l$ gives a natural isomorphism from $(\Gamma_{A_1}\cap\Gamma')\times \dots\times (\Gamma_{A_l}\cap\Gamma')\to \Gamma\cap\Gamma'$. Under this isomorphism and the identity map on $\TT_A$, the $\Gamma\cap\Gamma'$-action on $\TT_A$ is equivalent to the product of the $\Gamma_{A_j}\cap\Gamma'$-action on $\TT_{A_j}$, where $1\leq j\leq l$. (See Definition \ref{dfn:iso.of.G.space}.) 
\label{lem:properties.good.subgp-(5)}
\end{enumerate}
\end{lem}
\begin{proof}
These assertions follow directly from \eqref{eqn:commute}-\eqref{eqn:good.subgroup.struct}.
\end{proof}
\begin{rmk}
To avoid repetitively referring to Lemma \ref{lem:properties.good.subgp-(2)} and changing subscripts, (especially when proving results using induction on the dimension of the torus), 
we will state most lemmas and propositions for $\TT_A$ when $A$ is consecutive. 
\end{rmk}
We end this subsection by studying the fixed points of elements in $\Gamma_A$ acting on $\TT_A$.
\begin{lem}\label{lem:fixed points of gamma}
Let $A$ be a consecutive subset of $\ZZ_+$. If $\gamma\in\Gamma_A$ has fixed points on $\TT_A$, there exist $i_1,\dots,i_q\in\cE(A)$ with $ i_1<i_2<\cdots <i_q$ such that the following holds
\begin{itemize}
\item $i_j<i_{j+1}-1$ for any $j\in\{1,\cdots ,q-1\}$.
\item $\gamma=\gamma_{i_1}\circ\gamma_{i_2}\circ\cdots \circ\gamma_{i_q}$.
\end{itemize}
Moreover, for any $i\in\cE(A)$, if we define 
\begin{align}\label{jth singularity}
\Omega(i)=\{(\theta_i,\theta_{i+1})\in\TT_{i,i+1}|\cos(\theta_{i})=\sin(\theta_{i+1})=0\} \text{ and }\Omega_A(i):=\Omega(i)\times\TT_{A\setminus\{i,i+1\}},
\end{align}
then the set of fixed points of $\gamma$ acting on $\TT_A$ is exactly $\bigcap_{j=1}^q\Omega_{A}(i_j)$.
\end{lem}
\begin{proof}
By Lemma \ref{lem:properties.good.subgp-(2)}, we assume WLOG that $A=\bV_n$. By \eqref{eqn:Gamma.A} and Lemma \ref{lem:properties.good.subgp-(1)}, we can write $\gamma=\gamma_{i_1}\circ\gamma_{i_2}\circ\cdots \circ\gamma_{i_q}$ for some choice of $1\leq i_1<i_2<\cdots <i_q\leq n-1$. If there exists some $j\in\{1,\cdots ,q-1\}$ such that $i_j+1=i_{j+1}$, then for any $(\theta_1,\cdots ,\theta_n)\in\TT^n$, the $i_{j+1}$-th component of $\gamma(\theta_1,\cdots ,\theta_n)$ equals $\kappa_0(\kappa_\pi(\theta_{i_{j+1}}))=\rho(\theta_{i_{j+1}})=\theta_{i_{j+1}}+\pi$ due to  \eqref{eqn:commute} and \eqref{eqn:generator}. This contradicts the assumption that $\gamma$ has fixed points. The rest of the proof follows directly from \eqref{eqn:commute} and \eqref{eqn:generator}.
\end{proof}

\subsection{Branched coverings of products of spheres by products of lower dimensional spheres}\label{subsect:proj=branched.cover}
Let $A$ be a consecutive subset of $\ZZ_+$ and $\Gamma$ be a good subgroup of $\Gamma_A$. Consider the quotient space $\Gamma\backslash\TT_A$ equipped with the quotient topology. We want to show that $\Gamma\backslash\TT_A$ is a closed smooth manifold which is diffeomorphic to products of spheres. By Lemma \ref{lem:properties.good.subgp-(5)}, it suffices to show that $M:=\Gamma_A\backslash\TT_A$ is a closed smooth manifold which is diffeomorphic $S^{|A|}$. We also characterize Riemannian metrics on $\Gamma\backslash\TT_A$ using certain positive semidefinite symmetric $2$-tensor field on $\TT_A$.

For any consecutive subset $A\subset\ZZ_+$ and for any good subgroups $\Gamma, \Gamma'$ of $\Gamma_A$ such that $\Gamma\subset\Gamma'$, We denote by 
\begin{align}\label{eqn:quotient.maps}
\Pi^{\Gamma}_{\Gamma'}:\Gamma\backslash\TT_A\to\Gamma'\backslash\TT_A\text{ and }\Pi_{\Gamma}:\TT_A\to\Gamma\backslash\TT_A.
\end{align}
the natural projection maps. 

Using the notations in \eqref{jth singularity}, for any consecutive subset $A\subset\ZZ_+$ and for any good subgroup $\Gamma=\Gamma_A(E)$ of $\Gamma_A$ with $E\subset\cE(A)$, we let
\begin{align}\label{total singularity subgp}
\Omega_{\Gamma;A}=\bigcup_{\gamma_i\in\cS_A\cap\Gamma}\Omega_{A}(i)=\bigcup_{i\in\cE(A)\setminus E}\Omega_A(i).
\end{align}
For simplicity, when $A=\bV_n$, we write
\begin{align}\label{total singularity}
\Omega_n(\cdot):=\Omega_{\bV_n}(\cdot),~\Omega_{\Gamma;n}(\cdot):=\Omega_{\Gamma;\bV_n}(\cdot)\text{ and }\Omega_n:=\Omega_{\Gamma_n;n}:=\bigcup_{i=1}^{n-1}\Omega_{n}(i).
\end{align}

By \eqref{jth singularity}, for any consecutive subset $A\subset\ZZ_+$ and for any  choice of $l_1,\dots,l_t\in\cE(A)$ with $l_1<\dots<l_t$, $\bigcap_{s=1}^t\Omega_A(l_s)=\emptyset$ if and only if there exist some $s\in\{1,\dots ,t-1\}$ such that $l_{s+1}-l_s=1$. 

We introduce the following singular index of a point. 
\begin{dfn}[$\Gamma$-singular index]\label{sing index}
Let $A$ be a consecutive subset of $\ZZ_+$ and $\Gamma$ be a good subgroup of $\Gamma_A$. For any point $p\in\TT_A$, we denote by $\text{Sing}_{\Gamma;A}(p):=\{i\in\cE(A)|p\in\Omega_A(i)\mathrm{~and~}\gamma_{i}\in\Gamma\}$ the \emph{$\Gamma$-singular index} of $p$ and $\text{Sing}^+_{\Gamma;A}(p)=\text{Sing}_{\Gamma;A}(p)\sqcup(\text{Sing}_{\Gamma;A}(p)+1)$ the \emph{extended $\Gamma$-singular index} of $p$. 


For simplicity, we write $\text{Sing}_A:=\text{Sing}_{\Gamma_A;A}$ and $\text{Sing}_A^+:=\text{Sing}_{\Gamma_A;A}^+$.
\end{dfn} 
\begin{rmk}
The singular index of $p$ describes the ``singularity type'' of the natural projection maps in \eqref{eqn:quotient.maps} at $p$. In particular, when $p\in\TT_A\setminus\Omega_{\Gamma;A}$, we have $\text{Sing}_{\Gamma;A}(p)=\text{Sing}^+_{\Gamma;A}(p)=\emptyset$ due to Lemma \ref{lem:fixed points of gamma}.

By Lemma \ref{lem:fixed points of gamma}, Lemma \ref{lem:properties.good.subgp-(1)} and Definition \ref{dfn:good subgp}, for any consecutive subset $A\subset\ZZ_+$ and for any $p\in\TT_A$, the stabilizer $\mathrm{Stab}_{\Gamma}(p)$ satisfies
$$\mathrm{Stab}_{\Gamma}(p)=\begin{cases}
\displaystyle \{1\},~&\mathrm{if}~p\not\in\Omega_{\Gamma;A},\\
\displaystyle \langle\gamma_{i_1},\cdots ,\gamma_{i_q}\rangle,~&\mathrm{if}~p\in\Omega_{\Gamma;A}\mathrm{~and~}\text{Sing}_{\Gamma;A}(p)=\{i_1,\cdots ,i_q\}.
\end{cases}$$
Also, if $\Gamma=\Gamma_A(E)$ for some $E\subset\cE(A)$, then we have $\text{Sing}_{\Gamma;A}(p)=\text{Sing}_{A}(p)\setminus E$ due to Definition \ref{dfn:good subgp}.
\end{rmk}

To study the local structures of the natural projection maps in \eqref{eqn:quotient.maps}, we require the following definition of a good neighborhood of an arbitrary point in a torus.


\begin{dfn}[Good neighborhood]\label{dfn:good.nbhd}
Let $A$ be a consecutive subset of $\ZZ_+$. For any $p\in\TT_A$, an open neighborhood $\cU_p$ of $p$ is \emph{good} if the following holds:
\begin{enumerate}
\item $\text{Sing}_{A}(p)\supset\Sing_{A}(p')$ for any $p'\in\overline{\cU_p}$.
\item For any consecutive subset $A'\subset A$ and for any $\gamma\in\Gamma_{A'}$, $\gamma\cU_p=\cU_p$ if and only if $\pr^A_{A'}(\gamma p)= \pr^A_{A'}(p)$; $\pr^A_{A'}(\overline{\gamma\cU_p})\cap \pr^A_{A'}(\overline{\cU_p})=\emptyset$ if and only if $\pr^A_{A'}(\gamma p)\neq \pr^A_{A'}(p)$.
\item For any $q\in\cU_p$ and for any $i\in A$, $|\pr^A_i(q)-\pr^A_i(p)|<\pi/8$.
\end{enumerate}
\end{dfn}
\begin{rmk}
Let $A=\{m+1,\dots,m+k\}$. Any point $p=(\theta_{m+1},\dots,\theta_{m+k})\in\TT_A$ admits a good neighborhood. Indeed, by \eqref{jth singularity} and Defintion \ref{sing index}, we have 
\begin{align}\label{eqn:good.nbhd.1}
\cos^2(\theta_i)+\sin^2(\theta_{i+1})=0\text{ if and only if } i\in \Sing_A(p).
\end{align}
Let 
\begin{align}\label{eqn:good.nbhd}
\cU_p=\left(\prod_{i\in A\setminus\text{Sing}_{A}^+(p)}B_i(\theta_i,\epsilon'))\right)\times\left(\prod_{i\in\text{Sing}_{A}(p)}B_{i,i+1}((\theta_i,\theta_{i+1}),\epsilon)\right).
\end{align}
Then $\cU_p$ is a good neighborhood for any $\epsilon\in(0,\pi/8)$ (independent of $p$) and for any $\epsilon'$ sufficiently small (depending on $p$).

Also, following Definition \ref{dfn:good subgp} and the remark after Definition \ref{sing index}, we have $\text{Sing}_{\Gamma;A}(p)\supset\Sing_{\Gamma;A}(p')$ for any $p'$ in a good neighborhood of $p$.
\end{rmk}

For any consecutive subset $A=\{m+1,\dots,m+k\}\subset\ZZ_+$ and for any $p=(\theta_{m+1},\dots,\theta_{m+k})\in\TT_A$, let $\cU_p$ be a good neighborhood of $p$. For any good subgroup $\Gamma$ of $\Gamma_A$, we define $\Xi_{\Gamma;A;p}:\cU_p\to \RR^{|A|}$ such that for any $j\in A$ and for any $(\theta_{m+1}+\delta_{m+1},\dots,\theta_{m+k}+\delta_{m+k})\in\cU_p$ with $\delta_{m+1},\dots,\delta_{m+k}\in(-\pi/8,\pi/8)$, we have
\begin{align}\label{eqn:z.to.z2}
(\Xi_{\Gamma;A;p}(\theta_{m+1}+\delta_{m+1},\dots,\theta_{m+k}+\delta_{m+k}))_{i-m}=
\begin{cases}
\delta_{i},~&\text{if }i\in A\setminus\text{Sing}^+_{\Gamma;A}(p), \\
\delta^2_{i}-\delta^2_{i+1},~&\text{if }i\in\text{Sing}_{\Gamma;A}(p),\\
2\delta_{i-1}\delta_{i},~&\text{if }i-1\in\text{Sing}_{\Gamma;A}(p).
\end{cases}
\end{align}
In particular, for any $i\in\text{Sing}_{\Gamma;A}(p)$, if we let $z=\delta_i+\mathbf{i}\delta_{i+1}\in\CC$, the $(i-m,i-m+1)$ component of $\Xi_{\Gamma;A;p}(\theta_{m+1}+\delta_{m+1},\dots,\theta_{m+k}+\delta_{m+k})$ is the same as $(\text{Re}(z^2),\text{Im}(z^2))$.

\begin{lem}\label{lem:local.decomp.proj}
Let $A$ be a consecutive subset of $\ZZ_+$ and $\Gamma$ be a good subgroup of $\Gamma_A$. For any $p\in\TT_A$ and for any good neighborhood $\cU_p$ of $p$, there is a map $\Theta_{\Gamma;A;p}:\Xi_{\Gamma;A;p}(\cU_p)\to \Gamma\backslash\TT_A$ such that $\Pi_\Gamma|_{\cU_p}=\Theta_{\Gamma;A;p}\circ\Xi_{\Gamma;A;p}$. Moreover, $\Theta_{\Gamma;A;p}$ is a homeomorphism onto its image.
\end{lem}
\begin{proof}
For any $q\in\cU_p$, by \eqref{eqn:z.to.z2}, Definition \ref{sing index} and Definition \ref{dfn:good.nbhd}, we have 
\begin{align}\label{eqn:local.decomp.proj-1}
\{x\in\cU_p:\Xi_{\Gamma;A;p}(x)=\Xi_{\Gamma;A;p}(q)\}=\{\gamma q:\gamma\in\Gamma\text{ and }\gamma p = p\}.
\end{align}
Hence, the map $\Theta_{\Gamma;A;p}$ satisfying $\Pi_\Gamma|_{\cU_p}=\Theta_{\Gamma;A;p}\circ\Xi_{\Gamma;A;p}$ is well-defined. It remains to show that $\Theta_{\Gamma;A;p}$ is a homeomorphism onto its image.

Notice that both $\Pi_\Gamma$ and $\Xi_{\Gamma;A;p}$ are open and continuous, the equation $\Pi_\Gamma|_{\cU_p}=\Theta_{\Gamma;A;p}\circ\Xi_{\Gamma;A;p}$ implies that $\Theta_{\Gamma;A;p}$ is open and continuous.
Injectivity of $\Theta_{\Gamma;A;p}$ follows from Definition \ref{dfn:good.nbhd} and \eqref{eqn:local.decomp.proj-1}. Therefore, $\Theta_{\Gamma;A;p}$ is a homeomorphism onto its image.
\end{proof}
\begin{prop}\label{prop:quotient.is.sphere}
For any consecutive subset $A\subset\ZZ_+$ and for any $k_1,\dots, k_q\in\ZZ_+$ such that $k_1+\dots+k_q=|A|$, there is a good subgroup $\Gamma\subset \Gamma_A$ and a smooth map $\Phi_{\Gamma;A}:\TT_A\to S^{k_1}\times \dots\times S^{k_q}$ such that the following holds:
\begin{itemize}
\item $\Phi_{\Gamma;A}$ is $\Gamma$-invariant. Moreover, $\Phi_{\Gamma;A}$ induces a homeomorphism $\overline{\Phi}_{\Gamma;A}:\Gamma\backslash\TT^n\to S^{k_1}\times \dots\times S^{k_q}$ such that $\Phi_{\Gamma;A}=\overline{\Phi}_{\Gamma;A}\circ\Pi_{\Gamma}$. (See \eqref{eqn:quotient.maps}.)
\item For any $p\in\TT_A$ and for any good neighborhood $\cU_p$ of $p$, there is a smooth map $\Psi_{\Gamma;A;p}:\Xi_{\Gamma;A;p}(\cU_p)\to S^{k_1}\times \dots\times S^{k_q}$ such that $\Phi_{\Gamma;A}|_{U_p}=\Psi_{\Gamma;A;p}\circ \Xi_{\Gamma;A;p}$. Moreover, $\Psi_{\Gamma;A;p}$ is a diffeomorphism onto its image.
\item For any $\{A_1,\dots, A_q\}$ of $A$ such that $|A_j|=k_j$ for any $j\in\{1,\dots,q\}$, one can choose $\Gamma=\Gamma_A(E_A(A_1,\dots, A_q))$. (See Definition \ref{dfn:good subgp} and the remark after Definition \ref{dfn:consecutive}.) Moreover, the map $\Phi_{\Gamma;A}$ is the product of the maps $\Phi_{\Gamma_{A_j};A_j}$ with $j\in\{1,\dots,q\}$.
\end{itemize}
\end{prop}
\begin{proof}
Assume that for any $A$, Proposition \ref{prop:quotient.is.sphere} holds for the special case when $q=1$ (and hence $\Gamma=\Gamma_A(A)=\Gamma_A$). Let $q\geq 2$, $\{A_1,\dots,A_q\}$ be an arbitrary consecutive partition of $A$ with $|A_j|=k_j$ for any $j\in\{1,\dots, q\}$, and $\Gamma=\Gamma_A(E_A(A_1,\dots, A_q))$. For simplicity, we write $E=E_A(A_1,\dots, A_q)$. By Lemma \ref{lem:properties.good.subgp-(3)} and the remark after Definition \ref{dfn:consecutive}, we have 
\begin{align}\label{eqn:quotient.is.sphere-1}
\Gamma\cap\Gamma_{A_j}=\Gamma_{A_j}(E\cap\cE(A_j))=\Gamma_{A_j}(\emptyset)=\Gamma_{A_j}, \text{for any }j\in\{1,\dots,q\},
\end{align}
where we also used \eqref{eqn:Gamma.A} and \eqref{eqn:cutpoint} in the last equality sign. Also, for any $(p_1,\dots,p_q)\in\TT_{A_1}\times\cdots\times\TT_{A_q}=\TT_A$ and for any $j\in\{1,\dots,q\}$, by \eqref{jth singularity}, the remark after Definition \ref{dfn:consecutive}, Definition \ref{sing index} and its remark, we have,
\begin{align}\label{eqn:quotient.is.sphere-2}
\begin{split}
\Sing_{A_j}(p_j)=&\{i\in\cE(A_j)|p_j\in\Omega_{A_j}(i)\}\\
=&\{i\in\cE(A)|p\in\Omega_{A}(i)\}\cap\cE(A_j)\\
=&\Sing_A(p)\cap\cE(A_j)=(\Sing_A(p)\setminus E)\cap\cE(A_j)=\Sing_{\Gamma;A}(p)\cap\cE(A_j).
\end{split}
\end{align}
Notice that for any good neighborhood $\cU_p$ and for any $j\in\{1,\dots ,q\}$, $\cU_{p_j}:=\pr^A_{A_j}(\cU_p)$ is a good neighborhood of $p_j$ in the sense of Definition \ref{dfn:good.nbhd}. Hence $\Xi_{\Gamma_{A_j};A_j;p_j}$ can be defined on $\cU_{p_j}$. Also, \eqref{eqn:z.to.z2}, \eqref{eqn:quotient.is.sphere-2} and the remark after Definition \ref{sing index} imply that $\Xi_{\Gamma;A;p}=\Xi_{\Gamma_{A_1};A_1;p_1}\times\cdots\times \Xi_{\Gamma_{A_q};A_q;p_q}$. Therefore, we choose $\Phi_{\Gamma;A}=\Phi_{\Gamma_{A_1};A_1}\times\cdots\times \Phi_{\Gamma_{A_q};A_q}$ and $\Psi_{\Gamma;A;p}=\Psi_{\Gamma_{A_1};A_1;p_1}\times\cdots\times \Psi_{\Gamma_{A_q};A_q;p_q}$. By Lemma \ref{lem:properties.good.subgp-(5)}, \eqref{eqn:quotient.is.sphere-1} and the above discussion, Proposition \ref{prop:quotient.is.sphere} for $A$ and $\Gamma$ follows directly from the above choices of $\Phi_{\Gamma;A},\Psi_{\Gamma;A;p}$ and the special case when $q=1$. Therefore, we only need to prove Proposition \ref{prop:quotient.is.sphere} when $q=1$.

By Lemma \ref{lem:properties.good.subgp-(2)}, it suffices to prove the special case when $A=\bV_n$ and $q=1$. In particular, $\Gamma=\Gamma_n$. Due to the technicality of the proof, we leave it to the appendix. (See Subsection \ref{subsect:quotient.is.sphere}.)
\end{proof}
\begin{rmk}
From now on, for any consecutive subset $A\subset\ZZ_+$ and for any good subgoup $\Gamma\subset\Gamma_A$, we equip $\Gamma\backslash\TT_A$ with the differential structure induced by $\overline{\Phi}_{\Gamma;A}$. Then one can easily check that the natural projection maps in \eqref{eqn:quotient.maps} are smooth.
\end{rmk}
We end this section by introducing a direct corollary of Proposition \ref{prop:quotient.is.sphere} without proofs.
\begin{cor}\label{cor:quotient is mfld}
Let $A$ be a consecutive subset of $\ZZ_+$ and $\Gamma$ be a good subgroup of $\Gamma_A$. We denote by $\cM_{\Gamma}(\TT_A)$ the set of positive semidefinite symmetric $2$-tensor fields on $\TT_A$ satisfying the following conditions:
\begin{enumerate}
\item[(1).] ($\Gamma$-symmetry) For any point $p\in\TT^n$, for any $g\in \cM_{\Gamma}(\TT_A)$, for any $X,Y\in T_p\TT^n$ and for any $\gamma\in\Gamma$, we have
$$g|_p(X,Y)=g|_{\gamma(p)}(D\gamma(X),D\gamma(Y)).$$
\item[(2).] (Compatibility with $\Omega_{\Gamma;A}$) For any $p\in\TT_A$ and for any $g\in\cM_\Gamma(\TT_A)$, there exist a good neighborhood $\cU_p$ of $p$ in $\TT_A$ and a Riemannian metric $g'$ on $\Xi_{\Gamma;A;p}(\cU_p)$ such that $g|_{\cU_p}=\Xi_{\Gamma;A;p}^*(g')$.
\end{enumerate}
Then, the map $g\to\Pi_{\Gamma}^*g$ defines a bijection from the set of Riemannian metrics on $\Gamma\backslash\TT_A$ to $\cM_{\Gamma}(\TT_A)$.
\end{cor}
\subsection{Ramification and branching locus}\label{subsect:locus}
The notion of singular index  (see Definition \ref{sing index}) and \eqref{jth singularity}-\eqref{total singularity} are closely related to where the natural projection maps in \eqref{eqn:quotient.maps} fail to be local diffeomorphisms. 
\begin{dfn}[Ramification and branching loci]\label{dfn:locus}
Let $f:X\to Y$ be a surjective smooth map between smooth manifolds of the same dimension. For any $x\in X$, we say that $x$ is a $\emph{ramification point}$ of $f$ if $f$ is not a local diffeomorphism near $x$. The \emph{ramification locus} of $f$ is the collection of all ramification points. The \emph{branching locus} of $f$ is the image of the ramification locus under $f$.
\end{dfn}
\begin{lem}\label{lem:branching.locus}
Let $A$ be a consecutive subset of $\ZZ_+$. For any good subgroups $\Gamma, \Gamma'$ of $\Gamma_A$ such that $\Gamma\subset\Gamma'$, the following holds:
\begin{enumerate}[ref=\ref{lem:branching.locus}(\arabic*)]
\item For any $p\in\TT_A$, $p$ is a ramification point of $\Pi_\Gamma$ if and only if $\Sing_{\Gamma;A}(p)\neq \emptyset$. In particular, the ramification locus of $\Pi_\Gamma:\TT_A\to \Gamma\backslash\TT_A$ is $\Omega_{\Gamma;A}$, which is $\Gamma_A$-invariant. \label{lem:branching.locus-1}
\item For any $p\in\TT_A$, $\Pi_\Gamma(p)$ is a ramification point of $\Pi^\Gamma_{\Gamma'}$ if and only if $\Sing_{\Gamma;A}(p)\subsetneq\Sing_{\Gamma';A}(p)$. In particular, the ramification locus of $\Pi^{\Gamma}_{\Gamma'}:\Gamma\backslash\TT_A\to \Gamma'\backslash\TT_A$ is $\Pi_\Gamma(\overline{\Omega_{\Gamma';A}\setminus\Omega_{\Gamma;A}})$.\label{lem:branching.locus-2}
\end{enumerate}
\end{lem}
\begin{proof}
The first half of both assertions follow directly from \eqref{eqn:z.to.z2}, Definition \ref{sing index} and Proposition \ref{prop:quotient.is.sphere}. The second half of both assertions then follows from \eqref{jth singularity}-\eqref{total singularity}.
\end{proof}


\section{Geometry and dynamics of nearly non-positive curvature}\label{sect:review}
This section is a brief review of some concepts and results in \cite{BG94}.
\subsection{Caps and nearly non-positive curvature}
\begin{dfn}[Caps and extended caps, {\cite[Section 5, page 14]{BG94}}]\label{dfn:caps}
A \emph{(monotone curvature) cap} is a two-dimensional open disc equipped with a radially symmetric Riemannian metric satisfying the following properties:
\begin{itemize}
\item The boundary of the disc is a closed geodesic.
\item The curvature is positive in the interior and vanishes on the boundary.
\item The curvature is a nonincreasing function of the distance from the center.
\end{itemize}
An \emph{extended cap} $C$ is a two-dimensional open disc equipped with a radially symmetric metric $g$ such that there exist a smaller open disc $C'\subset C$ with the same center satisfying the following properties:
\begin{itemize}
\item $(C',g|_{C'})$ is a cap;
\item $g$ is negatively curved on $C\setminus \overline{C'}$.
\end{itemize}
\end{dfn}
\begin{rmk}
The definition of an extended cap in {\cite[Section 5, page 14]{BG94}} only requires that $g$ is nonpositively curved on $C\setminus \overline{C'}$. The more restrictive definition here is used to simplify the discussion in Subsection \ref{subsect:nomadic}. (In fact, the constructions in \cite[Section 7, page 26]{BG94} already require that $g$ is negatively curved on $C\setminus \overline{C'}$, which is the same as in the above definition.) 

The existence of extended caps is proved in \cite[Proposition 2.1]{BG89a}. Moreover, by \cite[Lemma 2.3 and Proposition 2.4]{BG89a}, the following holds:
\begin{itemize}
\item Any nonconstant geodesic in $(C,g)$ defined on an infinite time interval must be everywhere tangent to $\del C'$. In particular, the curvature at any point on this geodesic is $0$. Hence, any parallel vector field along this geodesic is a Jacobi field.
\item There exists some $\delta>0$ such that for any $p,q\in C$ with $d_g(p,q)<\delta$, $p,q$ are joined by a geodesic with respect to $g$ in $C$.
\end{itemize}

%
\end{rmk}
\begin{dfn}[Nearly nonpositive curvature, regular neighborhoods and canonical decomposition {\cite[pages 14-15]{BG94}}]\label{dfn:near nonpos curv}

~A Riemannian manifold has \emph{nearly nonpositive curvature} if every point has an open neighborhood that is a product in which each factor has nonpositive sectional curvature or is an extended cap.

A \emph{regular neighborhood} of a point $p$ on a nearly nonpositively curved manifold is an open neighborhood $U$ of $p$ satisfying the following properties:
\begin{itemize}
\item $U$ is isometric to $U_0\times\cdots \times U_k$, where $U_0,\dots,U_k$ are connected, totally geodesic submanifolds containing $p$.
\item $U_0$ has nonpositive curvature.
\item Let $p=(p_0,\cdots,p_k)$ after identifying $U$ with $U_0\times\cdots \times U_k$. For any $i\geq 1$, $U_i$ is a surface with nearly nonpositive curvature which contains a cap $C$ containing $p_i$. $U_1,\cdots ,U_k$ are called \emph{cap factors}.
\end{itemize}
$U_0\times\cdots \times U_k$ is called the \emph{canonical decomposition} of $U$ at $p$.
\end{dfn}
\begin{rmk}
Let $U$ and $V$ be regular neighborhoods of the same point $p$ in a nearly nonpositively curved manifold $M$. Suppose their canonical decomposition is given by $U=U_0\times\cdots \times U_k$ and $V=V_0\times\cdots \times V_l$, then $T_pU_0=T_pV_0$ and $k=l$. Moreover, the decompositions of $T_pM$ given by
$$T_pM=T_pU_0\oplus\cdots \oplus T_pU_k\mathrm{~and~}T_pM=T_pV_0\oplus\cdots \oplus T_pV_l$$
are identical. See \cite[Lemma 5.2-Corollary 5.3]{BG94}.
\end{rmk}

When $M$ is nonpositively curved, for any geodesic $c$ and any Jacobi field $Y(t)$ along $c$ at $c(t)$, the function 
\begin{align}\label{eqn:noninc baby Lyap fcn2}
t\to\langle Y(t),Y'(t)\rangle
\end{align} 
is nondecreasing. Nearly nonpositively curved manifolds satisfy a similar property. To be specific, we first consider the case when $M$ is either nonpositively curved or is a nearly nonpositively curved surface. For any $p\in M$, any $v\in T_pM$ and any $(w, w')\in T_pM\times T_pM$, let $c$ be the (not necessarily unit speed) geodesic with initial vector $v$ and $Y$ be the Jacobi field aloing $c_v$ such that 
$$Y(0)=w~\mathrm{and}~ Y'(0)=w'.$$
Define $\tau_-\leq 0\leq \tau_+$ such that 
$$\tau_+=\min\{\tau\geq 0|R( c'(\tau),\cdot,c'(\tau),\cdot)\leq 0\}~\mathrm{and}~\tau_-=\max\{\tau\leq 0|R(c'(\tau),\cdot,c'(\tau),\cdot)\leq 0\}.$$
In particular, $\tau_\pm$ only depends on $M$ and $v$. We define the following bilinear map $Q^M_v:T_pM\oplus T_pM\to\RR$ as 
$$Q^M_v(w,w'):=\begin{cases}
\displaystyle \frac{\tau_+}{\tau_+-\tau_-}\langle Y(\tau_-),Y'(\tau_-)\rangle- \frac{\tau_-}{\tau_+-\tau_-}\langle Y(\tau_+),Y'(\tau_+)\rangle,~&\mathrm{if}~\tau_+\neq\tau_-,\\
\displaystyle \langle w,w'\rangle=\langle Y(0),Y'(0)\rangle,~&\mathrm{if}~\tau_+=\tau_-.
\end{cases}$$
In the general case, for any nearly nonpositively curved manifold $M$ and any $p\in M$, we choose a regular neighborhood $U$ of $p$ with canonical decomposition $U=U_0\times\cdots \times U_k$. Then for any $v, w,w'\in T_pM$ such that 
$$v=\sum_{i=0}^kv_i,~w=\sum_{i=0}^kw_i~\mathrm{and}~w'=\sum_{i=0}^kw_i'\mathrm{~for~some~}v_i,w_i,w_i'\in T_pU_i,~i\in\{0,\cdots ,k\},$$
we define
\begin{align}\label{eqn:infi Lyap2}
Q_v^M(w,w'):=\sum_{i=0}^kQ^{U_i}_{v_i}(w_i,w_i').
\end{align}
By the remark under Definition \ref{dfn:near nonpos curv}, the right-hand-side in \eqref{eqn:infi Lyap2} does not depend on the choice of $U$. Hence $Q_v^M$ is well-defined. Moreover, similar to \eqref{eqn:noninc baby Lyap fcn2}, for any Jacobi field $Y(t)$ along a geodesic $c(t)$ defined on a nearly nonpositively curved manifold $M$, the function
\begin{align}\label{eqn:noninc Lyap fcn2}
f_M(Y,t):=Q^M_{\dot\gamma(t)}(Y(t), Y'(t))
\end{align}
is nondecreasing. (See \cite[Fundamental Lemma 5.12]{BG94}.) In particular, if $M$ is nonpositively curved, \eqref{eqn:noninc Lyap fcn2} coincide with \eqref{eqn:noninc baby Lyap fcn2}.

When $M$ is nonpositively curved, for any geodesic $c(t)$ and any Jacobi field $Y(t)$, $f_M(Y,t)\equiv 0$ if and only if $Y(t)$ is a parallel Jacobi field. Hence the \emph{rank} of a geodesic $c(t)$ is the dimension of the space of Jacobi fields $Y(t)$ such that $f_M(Y,t)\equiv 0$. We borrow the same definition for nearly nonpositively curved manifolds.
\begin{dfn}[Rank of a geodesic, {\cite[Definition 5.15]{BG94}}]\label{dfn:geo rank}
The \emph{rank} of a geodesic (or geodesic segment) $c(t)$ in a nearly nonpositively curved manifold $M$ is the dimension of the space of Jacobi fields $Y(t)$ such that $f_M(Y,t)\equiv 0$. The rank of a non-zero vector $v$ is defined to be the rank of the bi-infinite geodesic with initial vector $v$.
\end{dfn}
\begin{rmk}
The original definition of the rank of a geodesic (or geodesic segment) in \cite[Definition 5.15]{BG94} differs slightly from the above definition. More precisely, we replace the condition ``$f_M(Y,t)\equiv\text{constant}$'' by ``$f_M(Y,t)\equiv 0$''. These two definitions are in fact equivalent. When $\dim(M)=1$ or $M$ is nonpositively curved, this equivalence follows directly from \cite[Lemma 5.8]{BG94}. The case when $\dim(M)\geq 3$ follows from \cite[Lemma 5.19]{BG94} and the $\dim(M)=2$ case. When $\dim(M)=2$, we sketch the proof of this equivalence: We observe that the Wronskian associated to any two solutions to \cite[Equation (4), page 33]{BG89a} is constant, which is also mentioned in the proof of \cite[Lemma 5.9]{BG94}. Then, following this observation and using the $C(\cdot), S(\cdot)$ functions in \cite[page 33]{BG89a}, one can prove that $b<0$ in the proof of \cite[Lemma 5.9]{BG94} (originally $b\leq 0$). Therefore, if the equality case of \cite[Lemma 5.9]{BG94} is achieved, then $\langle Y,Y'\rangle(t_1)=\langle Y,Y'\rangle(t_2)=0$. The rest of the proof then follows from \cite[Lemma 5.8]{BG94} and the above.

We emphasize that the equivalence above is not needed for any of our arguments. Alternatively, one may use the definition of \cite[Definition 5.15]{BG94} throughout the rest of this paper. All the results below admit corresponding formulations and proofs with essentially no change.

Since $c'(t)$ is a parallel Jacobi field along any geodesic $c(t)$, the rank of a geodesic (or a non-zero vector) is at least one. We say a Jacobi field is an \emph{orthogonal parallel Jacobi field} along $c(t)$ if it is parallel along $c(t)$ and is orthogonal to $c'(\cdot)$ everywhere. If $c(t)$ admits an orthogonal parallel Jacobi field, then $c(t)$ has rank at least two.
\end{rmk}

In the Subsection \ref{subsect:B.and SBAP}, we will introduce a criterion for the Bernoulli property of the geodesic flow on a nearly nonpositively curved manifold. Roughly speaking, if there are ``sufficiently many'' rank one unit vectors, then the geodesic flow is Bernoulli. 

We conclude this subsection by introducing three useful lemmas from \cite{BG94} without proofs.
\begin{lem}\label{lem:rank-1 lem 1}
A vector $v\in TM$ on a nearly nonpositively curved manifold $M$ is \emph{non-normal} if one of the projections of $v$ to a cap factor either vanishes or is tangent to the boundary of a cap. The function $v\to Q_v^M$ is continuous unless $v$ is non-normal. Moreover, the set of non-normal vectors is a subset of the vectors with rank greater than or equal to $2$ and is invariant under the geodesic flow. (See Lemma 5.7, the remark before Lemma 5.10 and the first paragraph in the proof of Lemma 5.17 in \cite{BG94}.)
\end{lem}
\begin{lem}[{\cite[Lemma 5.18]{BG94}}]\label{lem:rank-1 lem 2}
Suppose $c$ is a nonconstant geodesic in a manifold with nearly nonpositive curvature and there exists some $t_0$ such that the sectional curvature of every plane containing $c'(t_0)$ is negative. Then $\mathrm{rank}(c)=1$.
\end{lem}
\begin{lem}[{An adapted version of \cite[Lemma 5.19]{BG94}, using the remark after Definition \ref{dfn:geo rank}}]\label{lem:rank-1 lem 3}
Suppose $c$ is a geodesic in the product of two nearly nonpositively curved manifolds $A$ and $B$, and $Y$ is a Jacobi field along $c$ such that $f_{A\times B}(Y,t)\equiv 0$. Then $f_{A}(Y_A,t)\equiv 0$ and $f_{B}(Y_B,t)\equiv 0$, where $Y_A$ and $Y_B$ are images of $Y$ under projections onto $A$ and $B$ respectively, and are Jacobi fields along $c_A$ and $c_B$ respectively.
\end{lem}

\subsection{Strong ball avoidance property (SBAP) and the Bernoulli property of geodesic flow}\label{subsect:B.and SBAP}
\begin{dfn}[Nonseparating sets, {\cite[Definition 3.1]{BG94}}]\label{dfn:nonsep sets}
Let $X$ be a manifold, possibly without boundary, and $U\subset X$ be open and nonempty. A subset $N$ of $X$ is called \emph{nonseparating in $U$} if the following holds.
\begin{itemize}
\item $N\cap U$ is closed in $U$ and has empty interior.
\item $O\setminus N$ is connected for any connected open subset $O\subset U$.
\end{itemize}
If we say $N$ is nonseparating without mentioning $U$, we mean that $N$ is nonseparating in $X$.
\end{dfn}
\begin{dfn}[Ball avoidance property (BAP), {\cite[Definition 3.3]{BG94}}]\label{dfn:BAP}
Let $\varphi^t$ be a flow on a manifold $X$ and $U\subset X$ be a nonempty open subset. We say that $\varphi^t$ has the \emph{ball avoidance property (BAP) on $U$} if for any nonempty open subset $B\subset X$, the set $\{x\in X|\varphi^tx\not\in B,~\forall t\in\RR\}$ is nonseparating in $U$.
\end{dfn}
\begin{dfn}[Strong ball avoidance property (SBAP), {\cite[Definition 3.4]{BG94}}]\label{dfn:SBAP}
Let $\varphi^t$ be a flow on a manifold $X$ and $U\subset X$ be a nonempty open subset. We say that $\varphi^t$ has the \emph{strong ball avoidance property (SBAP) on $U$} if the product flow $\varphi^t\times\psi^t$ has the BAP on $U\times V$ whenever $\psi^t$ is a flow on a manifold $Y$ and $V$ is an open subset of $Y$ satisfying the following properties:
\begin{itemize}
\item $\psi^t$ has the BAP on $V$.
\item $V$ is not contained in a nonperiodic orbit of $\psi^t$.
\end{itemize}
\end{dfn}
We provide some basic facts regarding nonseparating sets, BAP and SBAP without proofs.
\begin{fact}\label{fact:SBAP}
Let $X$ be a manifold, possibly with boundary, and $U\subset X$ be open and nonempty. Let $N$ be a subset of $X$. Let $\varphi^t$ be a flow on $X$. Then the following holds:
\begin{enumerate}[ref=\ref{fact:SBAP}(\arabic*)]
\item If $N$ is nonseparating in $U$, it is nonseparating in any nonempty open subset of $U$. (This follows directly from Definition \ref{dfn:nonsep sets}.)\label{fact:SBAP-1}
\item A finite union of sets that are nonseparating in $U$ is also nonseparating in $U$. (See \cite[Lemma 3.2(iii)]{BG94} or \cite[Property 2, p.321]{KSS89}.) \label{fact:SBAP-2}
\item Assume that $X=\cF\times X_0$, where $\cF$, $X_0$ are manifolds. For any $x\in X_0$, we define $U_x:=U\cap (\cF\times \{x\})$ and $N_x=N\cap (\cF\times \{x\})$. If $N\cap U$ is closed in $U$ and $N_x$ is nonseparating in $U_x$ for any $x\in X_0$, then $N$ is non-separating in $U$. (This is a slight generalization of \cite[Property 4, p.321]{KSS89}. In \cite[Property 4, p.321]{KSS89}, $U$ is assumed to be $X$. One can check that their proof also works for any open subset $U$.)\label{fact:SBAP-3}
\item $N$ is nonseparating in $U$ if and only if for any $x\in U$, there exists an open neighborhood $U_x\subset U$ of $x$ such that $N$ is nonseparating in $U_x$. (See \cite[Lemma 3.2(iv)]{BG94}.) \label{fact:SBAP-4}
\item If the flow $\varphi^t$ has the BAP (resp. SBAP) on $U$, then it has the BAP (resp. SBAP) on any nonempty open subsets of $U$. (This is a direct corollary of the first item.) \label{fact:SBAP-5}
\item Let $U'$ be another open subset of $X$ such that $U\subset U'$ and that $U'\setminus U$ is nonseparating in $U'$. Suppose that $U$ is invariant under the flow $\varphi^t$ and $\varphi^t$ has the BAP (resp. SBAP) on $U$, then the flow $\varphi^t$ has the BAP (resp. SBAP) on $U'$. (See \cite[Lemma 3.6]{BG94}.) \label{fact:SBAP-6}
\end{enumerate}
\end{fact}
%
%
%
To conclude this section, we introduce a slightly modified version of \cite[Theorem 5.17]{BG94}.
\begin{thm}[Corollary of {\cite[Theorem 5.17]{BG94}}]\label{thm:main criterion}
Let $M$ be a closed Riemannian manifold with nearly nonpositive curvature and dimension at least $2$. Suppose that there exists a connected subset $\cW$ of the set of rank one unit vectors such that the complement of $\cW$ in $T^1M$ has zero Lebesgue measure, then the geodesic flow $\varphi^t$ on $M$ is Bernoulli on $T^1M$ and has the SBAP on the set of rank one unit vectors.
\end{thm}
The only difference between Theorem \ref{thm:main criterion} and \cite[Theorem 5.17]{BG94} is that we do not assume that $\cW$ is the set of rank one unit vectors in Theorem \ref{thm:main criterion}. The proof of Theorem \ref{thm:main criterion} relies on the following elementary lemma.
\begin{lem}\label{lem:conn of full mrs set}
Let $M$ be an open Riemannian manifold and $\cW_1\subset M$. Suppose there exists a connected subset $\cW$ of $\cW_1$ whose complement in $M$ has Lebesgue measure zero. Then $\cW_1$ is also connected and $M\setminus \cW_1$ has Lebesgue measure zero.
\end{lem}
\begin{proof}[Proof of Lemma \ref{lem:conn of full mrs set}]
It is clear that $M\setminus\cW_1$ has Lebesgue measure zero. If $\cW_1$ is not connected, then there exist nonempty disjoint open subsets $A,B\subset \cW_1$ such that $\cW_1=A\sqcup B$. Since $\cW$ has full Lebesgue measure, $A\setminus \cW$ and $B\setminus \cW$ both have Lebesgue measure zero. Notice that $A,B$ are open subsets of $\cW_1$ and $M\setminus\cW_1$ has Lebesgue measure zero. Therefore, both $A$ and $B$ have positive Lebesgue measure. In particular, $A\cap \cW$ and $B\cap \cW$ are disjoint nonempty open subsets of $\cW$. This contradicts the assumption that $\cW$ is connected.
\end{proof}

Now we present the proof of Theorem \ref{thm:main criterion}.

\begin{proof}[Proof of Theorem \ref{thm:main criterion}]
Let $\cW_1$ be the set of rank one unit vectors. Since $\cW\subset \cW_1$ and $T^1M\setminus \cW$ has Lebesgue measure zero,  $T^1M\setminus \cW_1$ also has Lebesgue measure zero. Connectedness of $\cW_1$ follows from Lemma \ref{lem:conn of full mrs set} applied to $\cW_1$, $T^1M$ and $\cW$. The rest of the proof follows directly from \cite[Theorem 5.17]{BG94}.
\end{proof}

\section{Construction of a smooth metric on products of spheres}\label{sect:construction}
\subsection{Special caps and special tubes}\label{subsect:special tubes}
\begin{dfn}[Generalized extended caps]\label{dfn:gen caps}
For any integer $k\geq 1$, a \emph{generalized extended $k$-cap with (Euclidean) radius $R$} is a two-dimensional open disc $C_{R}:=\{z\in\CC:|z|<R\}$ equipped with a smooth positive semidefinite symmetric $2$-tensor field $g$ such that the following holds: There exist a smooth, strictly increasing function $h:[0,R]\to\RR$ and an extended cap $C:=\{z\in\CC:|z|<h(R)\}$ with metric $g_C$ satisfying the following properties:
\begin{itemize}
\item There exists some $R_0\in (0,R)$ such that $h(r)=r^k$ for any $r\in [0,R_0]$.
\item Let $F:C_{R}\to C$ be the smooth map defined as
$$F(z)=h(|z|)\cdot|z|^{-k}\cdot z^k.$$
Then $g:=F^*g_C$.
\end{itemize}
\end{dfn}
\begin{rmk}
By Definition \ref{dfn:gen caps} and the radial symmetry of extended cap metrics (see Definition \ref{dfn:caps}), any generalized extended $k$-caps are invariant under reflections about the $x$-axis and the $y$-axis. Also, when $k=1$, the positive semidefinite symmetric $2$-tensor field of a generalized extended $k$-cap is in fact a Riemannian metric.
\end{rmk}

Let $A$ be a consecutive subset of $\ZZ_+$ with $|A|\geq 2$ and $\Gamma$  be a good subgroup of $\Gamma_A$. Although the standard product metric on $\TT_A$ is invariant under $\Gamma$ (see \eqref{eqn:generator}), it does not descend to a desired Riemannian metric on $\Gamma\backslash\TT_A$ via $\Pi_\Gamma$ for the following reasons:
\begin{itemize}
\item $\Pi_\Gamma$ may have ramification locus.
\item If $\Pi^*_\Gamma g$ is the standard torus metric on $\TT_A$ for some Riemannian metric $g$ on $\Gamma\backslash\TT_A$, the geodesic flow for $(\Gamma\backslash\TT_A,g)$ is not ergodic.
\end{itemize} 
In order to resolve these issues, we would like to modify the standard torus metric near $\Omega_{\Gamma_A;A}$. To be specific, we would like to replace the standard product metric by products of suitable generalized extended caps and lower dimensional tori near $\Omega_{\Gamma_A;A}$. This requires a generalized extended cap which can be smoothly glued to the standard Euclidean metric on its complement in $\CC\isom\RR^2$. For this reason, we introduce the following notions.

\begin{dfn}[Special $k$-caps]\label{dfn:special caps}
Let $C_{R}:=\{z\in\CC:|z|<R\}$. For any generalized extended $k$-cap $(C_{R},g)$ with Euclidean radius $R$, let $h$, $C:=\{z\in\CC:|z|<h(R)\}$ and $F:C_{R}\to C$ be the same as in Definition \ref{dfn:gen caps}. We say that $(C_{R},g)$ is \emph{special} if the positive semidefinite symmetric $2$-tensor field $g$ can be extended to a smooth positive semidefinite symmetric $2$-tensor field $\widetilde g$ on $\CC$ such that
$$\widetilde g|_z:=\begin{cases}
\displaystyle g|_z,~&\mathrm{if~}z\in C_{R},\\
\displaystyle dz\otimes d\overline{z}=dx\otimes dx+dy\otimes dy,~&\mathrm{if~}z\in\CC\setminus C_{R},
\end{cases}$$
where $z=x+\mathbf{i}y$.
\end{dfn}
\begin{rmk}
In the rest of this paper, for any $R>0$ and for any $k\in\ZZ_+$, we fix a choice of special $k$-cap with Euclidean radius $R$. Due to technicality of the argument, we prove the existence of special $k$-caps in Proposition \ref{prop:special.caps.exist}.
\end{rmk}
In Definition \ref{dfn:special tubes}, we would like to construct special $k$-tubes by first inserting special $k$-caps into the 2-dimensional torus $\TT_{i,j}$ with $i\neq j$. Therefore, it is convenient for us to introduce the following definition:
\begin{dfn}[Special $k$-caps on $2$-torus]\label{dfn:cap.on.torus} 
For any distinct integers $i,j\in\ZZ_+$, for any point $p_{i,j}:=(\theta_i,\theta_j)\in\TT_{i,j}$ and for any $R\in(0,\pi)$, let
\begin{align}\label{eqn:standard.local.disc}
\phi_{p_{i,j}}:B_{i,j}(p_{i,j},R)\to C_{R}:=\{z\in\CC:|z|<R\},~\phi_{p_{i,j}}(\theta_i+\delta_i,\theta_j+\delta_j):=\delta_i+\mathbf{i}\delta_j.
\end{align} 
For any $k\in\ZZ_+$, a \emph{$\{i,j\}$-special $k$-cap with Euclidean radius $R$ centered at $p_{i,j}$} is of the form $(B_{i,j}(p_{i,j},R),g_{\text{cap}})$, where $g_{\text{cap}}=\phi_{p_{i,j}}^*g$ for a special $k$-cap $(C_{R},g)$ with Euclidean radius $R$. 
\end{dfn}
\begin{rmk}
For any $y\in\TT^1$, let $\kappa_y:\TT^1\to\TT^1$ such that $\kappa_y(x)=y-x$. Then, for any $y_i\in \TT_i$ and for any special $k$-cap $(B_{i,j},g_{\text{cap}})$ with Euclidean radius $R$ centered at $(\theta_i,\theta_j)\in\TT_{i,j}$, $((\kappa_{y_i},\id)^{-1}(B_{i,j}),(\kappa_{y_i},\id)^*g_{\text{cap}})$ is a special $k$-cap centered at $(\kappa_{y_i}(\theta_i),\theta_j)$ with the same Euclidean radius. Similarly, for any $y_j\in\TT_j$, $((\id,\kappa_{y_j})^{-1}(B_{i,j}),(\id,\kappa_{y_j})^*g_{\text{cap}})$ is a special $k$-cap centered at $(\theta_i,\kappa_{y_j}(\theta_j))$ with the same Euclidean radius.
\end{rmk}
\begin{dfn}[Tubes]\label{dfn:tubes}
Let $A$ be a consecutive subset of $\ZZ_+$ with $|A|\geq 2$. For any distinct $i, j\in A$ and for any $p_{i,j}\in\TT_{i,j}$, A \emph{$\{i,j\}$-tube} is a set of the form $B_{i,j}(p_{i,j},R)\times \TT_{A\setminus\{i,j\}}$. Here, $B_{i,j}(p_{i,j},R)$ is called the \emph{base} of this tube.
\end{dfn}

\begin{dfn}[Standard tubes and special $k$-tubes]\label{dfn:special tubes}
Let $A$ be a consecutive subset of $\ZZ_+$ with $|A|\geq 2$. Let $\cT$ be a $\{i,j\}$-tube with base $B_{i,j}$.

We say that $(\cT,g)$ is a \emph{standard tube} with base $B_{i,j}$ for some positive semidefinite symmetric $2$-tensor field $g$ if $(\cT,g)=(B_{i,j},d\theta_i\otimes d\theta_i+d\theta_j\otimes d\theta_j)\times (\TT_{A\setminus\{i,j\}},g_h)$ for some positive semidefinite symmetric $2$-tensor field $g_h$;
$(\cT,g)$ is a \emph{special $k$-tube} if $(\cT,g)$ is the product of a $\{i,j\}$-special $k$-cap on $B_{i,j}$ and $ (\TT_{A\setminus\{i,j\}},g_h)$ for some positive semidefinite symmetric $2$-tensor field $g_h$.

In both definitions, $(\TT_{A\setminus\{i,j\}},g_h)$ is called the \emph{height} of the standard tube/special $k$-tube.
%
\end{dfn}
\begin{dfn}[Inserting and forgetting special tubes]\label{dfn:insert/forget tubes}
Let $A$ be a consecutive subset of $\ZZ_+$ and $g$ be a smooth positive semidefinite symmetric 2-tensor field on $\TT_A$ with $|A|\geq 2$. Let $i,j\in A$ be distinct and $\cT$ be a $\{i,j\}$-tube with base $B_{i,j}$. (See Definition \ref{dfn:tubes} and Figure \ref{fig:insert and forget}.)
\begin{enumerate}
\item (\textbf{Inserting special tube structures}) If $(\cT,g|_{\cT})$ is a $\{i,j\}$-standard tube, then we say that $(\TT_A,g')$ is obtained by \emph{inserting a $\{i,j\}$-special $k$-tube with base $B_{i,j}$ into $(\TT_A,g)$} if $g'=g$ on $\TT_A\setminus\cT$ and $(\cT,g'|_{\cT})$ is a $\{i,j\}$-special $k$-tube with the same height as $(\cT,g|_{\cT})$.
\item (\textbf{Forgetting special tube structures}) If $(\cT,g|_{\cT})$ is a $\{i,j\}$-special $k$-tube, then we say that $(\TT_A,g')$ is obtained by \emph{forgetting a $\{i,j\}$-special $k$-tube with base $B_{i,j}$ from $(\TT_A,g)$} if $g'=g$ on $\TT_A\setminus\cT$ and $(\cT,g'|_{\cT})$ is a $\{i,j\}$-standard tube with the same height as $(\cT,g|_{\cT})$.
\end{enumerate}
\end{dfn}
\begin{figure}[h]
	\centering
	\includegraphics[width=4in]{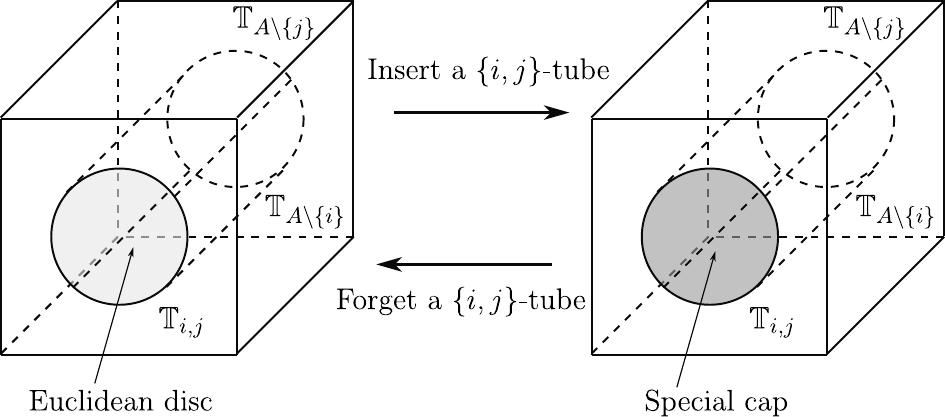}
	\caption{}\label{fig:insert and forget}
\end{figure}
\begin{rmk}
By Definition \ref{dfn:special tubes} and Definition \ref{dfn:insert/forget tubes}, for any smooth positive semidefinite symmetric $2$-tensor field $g$ on $\TT_A$, the tensor field obtained by inserting/forgetting a special tube is also a smooth positive semidefinite symmetric $2$-tensor field. One can also verify the following elementary properties:
\begin{itemize}
\item[\hypertarget{T1}{(T1)}] When inserting/forgetting multiple special tubes, the order does not matter.
\item[\hypertarget{T2}{(T2)}] Let $g$ be a smooth positive semidefinite symmetric $2$-tensor field on $\TT_A$. If $(\TT_A,g')$ is obtained by inserting a special tube into $(\TT_A,g)$, and $(\TT_A,g'')$ is obtained by forgetting the same special tube from $(\TT_A,g')$, then $g=g''$. Similarly, if $(\TT_A,g')$ is obtained by forgetting a special tube from $(\TT_A,g)$, and $(\TT_A,g'')$ is obtained by inserting the same special tube back into $(\TT_A,g')$, then $g=g''$. 
\item[\hypertarget{T3}{(T3)}] Let $A, A'$ be consecutive subsets of $\ZZ_+$ such that $A\subset A'$. Let $g$ be a smooth positive semidefinite symmetric $2$-tensor field on $\TT_A$ and $h$ be a smooth positive semidefinite symmetric $2$-tensor field on $\TT_{A'\setminus A}$. If $(\TT_A,g')$ is obtained by inserting a $\{i,j\}$-special $k$-tube into (or forgetting a $\{i,j\}$-special $k$-tube from) $(\TT_A,g)$ with distinct $i,j\in A$. Then $(\TT_A,g')\times (\TT_{A'\setminus A},h)$ can be alternatively obtained by inserting a $\{i,j\}$-special $k$-tube with the same base into (or forgetting a $\{i,j\}$-special $k$-tube with the same base from) $(\TT_A,g)\times (\TT_{ A'\setminus A},h)$.
\end{itemize}
\end{rmk}

\subsection{Tubular neighborhoods of $\Omega_A(i)$ and their properties}\label{subsect:tubes.that.we.consider}
Let $A$ be a consecutive subset of $\ZZ_+$ with $|A|\geq 2$. For any $i\in\cE(A)$ and for any $R\in(0,\pi/8)$, we define the following tubular neighborhood of $\Omega_A(i)$:
\begin{align}\label{eqn:j-tubes}
\cT_{i;A}(R):=\bigsqcup_{p_{i,i+1}\in\Omega(i)}B_{i,i+1}(p_{i,i+1},R)\times \TT_{ A\setminus\{i,i+1\}}\supset\Omega_{A}(i).
\end{align}
(See \eqref{eqn:tori}, \eqref{eqn:ball.convention} and \eqref{jth singularity} for the relevant notations. See also \eqref{jth singularity} for why the right-hand-side is a disjoint union.) One can easily check that the connected components of $\cT_{i;A}(R)$ are all $\{i,i+1\}$-tubes whose bases are centered at points in $\Omega(i)$. (See \eqref{jth singularity} and Figure \ref{fig:tubes in construction}.) In the rest of this paper, these tubes are the only $\{i,i+1\}$-tubes that we will consider.

\begin{figure}[h]
	\centering
	\includegraphics[width=4.5in]{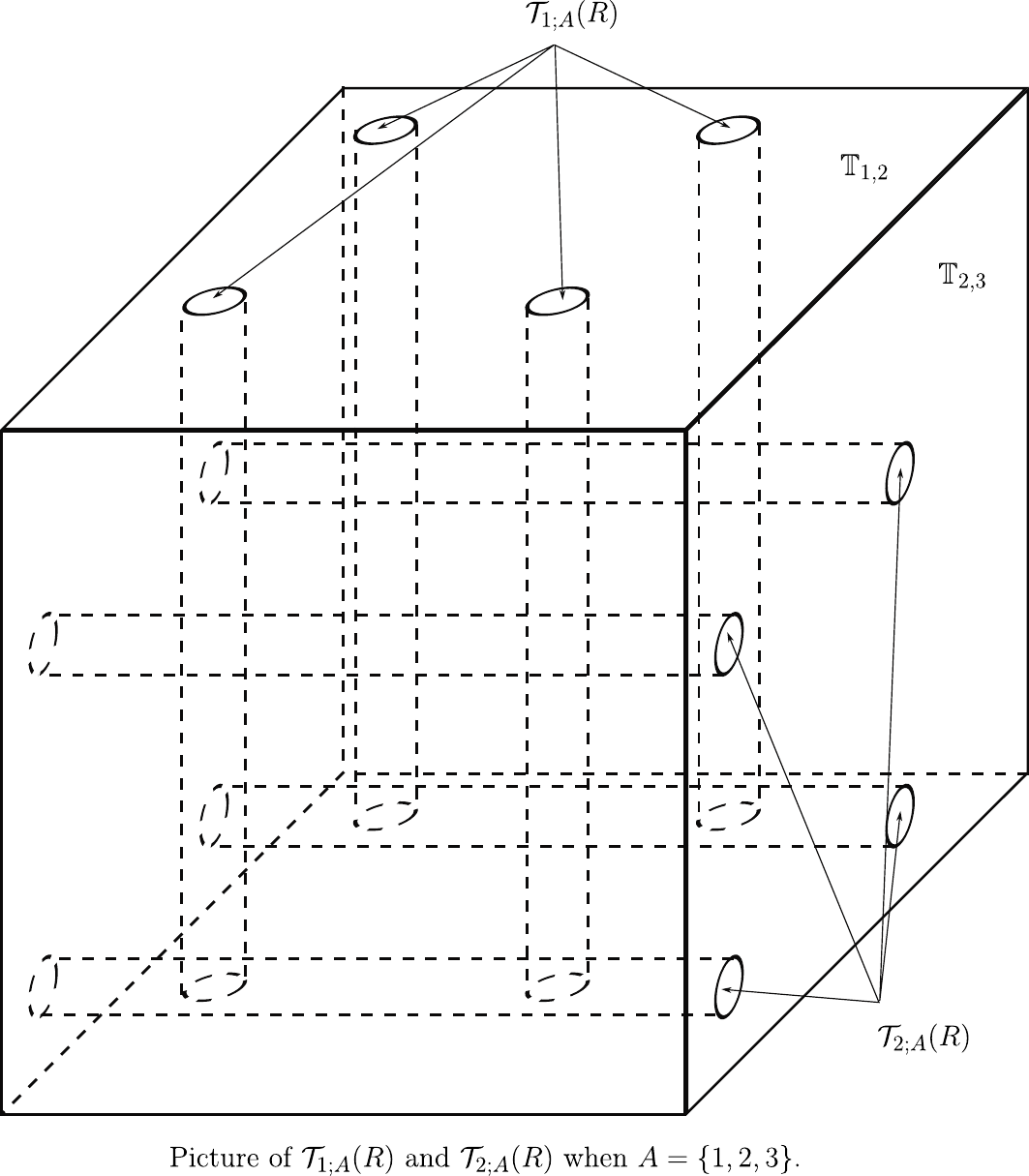}
	\caption{}\label{fig:tubes in construction}
\end{figure}

%

For any finite subset $A\subset\ZZ_+$, for any $i\in A$ and for any $y\in\TT^1$, we define an isometry $\kappa_{y;i}:\TT_A\to\TT_A$ such that for any $p\in\TT_A$, we have
\begin{align}\label{eqn:larger.gp}
\pr^A_j(\kappa_{y;i}(p))=\begin{cases}
\pr^A_j(p),~&\text{if }j\in A\setminus\{i\},\\
\kappa_y(\pr^A_j(p)),~&\text{if }j=i,
\end{cases}
\end{align}
where $\kappa_y(x)=y-x$. (See \eqref{eqn:tori.proj} and the discussion before \eqref{eqn:commute}.) In other words, $\kappa_{y;i}$ is an isometry on $\TT_A$ which only performs a reflection on the $i$-th component and keeps the other components unchanged.

We summarize some basic properties of $\cT_{i;A}(R)$.
\begin{lem}\label{lem:tube.index.properties-1}
Let $A$ be a consecutive subset of $\ZZ_+$ and $R\in(0,\pi/8)$. Then the following holds:
\begin{enumerate}[ref=\ref{lem:tube.index.properties-1}(\arabic*)]
\item For any $i\in\cE(A)$ and for any $j\in A$, $\cT_{i;A}(R)$ and $\TT_A\setminus\cT_{i;A}(R)$ are invariant under $\kappa_{0;j}$ and $\kappa_{\pi;j}$. In particular, $\cT_{i;A}(R)$ and $\TT_A\setminus\cT_{i;A}(R)$ are $\Gamma_A$- invariant subsets. \label{lem:tube.index.properties-1-1}
\item Let $i\in A$ and $p,q\in\TT_A$ such that $|\pr_i(p)-\pr_i(q)|\leq\pi/4$ and $|\pr_{i+1}(p)-\pr_{i+1}(q)|\leq\pi/4$. If in addition that $p\in\cT_{i;A}(R)$, then for any $j\in\{i-1,i+1\}\cap\cE(A)$ and for any $R'\in(0,\pi/8)$, we have $q\not\in\cT_{j;A}(R')$. \label{lem:tube.index.properties-1-2}
\item For any $p\in\TT_A$, there is a point $q\in\TT_A$ such that for any $\epsilon'>0$ sufficiently small and for any good neighborhood $\cU_q$ of the form \eqref{eqn:good.nbhd} with $\epsilon=R$ and the above choice of $\epsilon'$, we have $p\in\cU_q$ and that $\text{Sing}_{A}(q)=\{j\in\cE(A):p\in\cT_{j;A}(R)\}$. (See Definition \ref{sing index} for the definition of $\text{Sing}_A$.)\label{lem:tube.index.properties-1-3}
\item For any $E\subset\cE(A)$ and for any good subgroup $\Gamma\subset\Gamma_A$, we have $\Pi_\Gamma^{-1}(\Pi_\Gamma(\bigcup_{i\in E}\cT_{i,A}(R)))=\bigcup_{i\in E}\cT_{i;A}(R)$ and $\Pi_\Gamma^{-1}(\Pi_\Gamma(\TT_A\setminus \bigcup_{i\in E}\cT_{i,A}(R)))=\TT_A\setminus \bigcup_{i\in E}\cT_{i,A}(R)$. Similarly, we have $\Pi_\Gamma^{-1}(\Pi_\Gamma(\bigcup_{i\in E}\overline{\cT_{i,A}(R)}))=\bigcup_{i\in E}\overline{\cT_{i;A}(R)}$ and $\Pi_\Gamma^{-1}(\Pi_\Gamma(\TT_A\setminus \bigcup_{i\in E}\overline{\cT_{i,A}(R)}))=\TT_A\setminus \bigcup_{i\in E}\overline{\cT_{i,A}(R)}$.\label{lem:tube.index.properties-1-4}
\item For any $i,i+1\in\cE(A)$ and for any good subgroup $\Gamma\subset\Gamma_A$, we have $\overline{\Pi_\Gamma(\cT_{i;A}(R))}\cap\overline{\Pi_{\Gamma}(\cT_{i+1;A}(R))}=\emptyset$. \label{lem:tube.index.properties-1-5}
\end{enumerate}
\end{lem}

\begin{proof}
\begin{enumerate}
\item Invariance under $\kappa_{0;j}$ and $\kappa_{\pi;j}$ follows directly from \eqref{jth singularity}, \eqref{eqn:j-tubes} and \eqref{eqn:larger.gp}. The rest follows directly from \eqref{eqn:generator}, \eqref{eqn:Gamma.A} and \eqref{eqn:larger.gp}.
%
\item This follows directly from \eqref{jth singularity} and \eqref{eqn:j-tubes}.
\item Choose $q\in\TT_A$ satisfying the following:
\begin{itemize}
\item $\pr^A_i(q)=\pr^A_i(p)$ for any $i\in A$ such that $\{i,i-1\}\cap\{j\in\cE(A):p\in\cT_{j;A}(R)\}=\emptyset$.
\item $\pr^A_{i,i+1}(q)=p_{i,i+1}$ for any $i\in\{j\in\cE(A):p\in\cT_{j;A}(R)\}$, where $p_{i,i+1}\in\Omega(j)$ is the unique point such that $\pr^A_{i,i+1}(p)\in B_{i,i+1}(p_{i,i+1},R)$.
\end{itemize}
The above two bullet points uniquely determines $q$. Also, the second bullet point implies that $\{j\in\cE(A):p\in\cT_{j;A}(R)\}\subset\text{Sing}_{A}(q)$. On the other hand, if there exist $i\in\text{Sing}_{A}(q)\setminus \{j\in\cE(A):p\in\cT_{j;A}(R)\}$, by the fact that $\Sing_A(q)\cap(1+\Sing_A(q))=\emptyset$, we have $i\pm1\not\in\Sing_A(q)$. (See Definition \ref{sing index}.) Recall that $\{j\in\cE(A):p\in\cT_{j;A}(R)\}\subset\text{Sing}_{A}(q)$, we conclude that $i,i\pm 1\not\in\{j\in\cE(A):p\in\cT_{j;A}(R)\}$. In particular, by the first bullet point, we have $\pr^A_i(q)=\pr^A_i(p)$ and $\pr^A_{i+1}(q)=\pr^A_{i+1}(p)$. The assumption that $i\in\text{Sing}_A(q)$ and the above imply that $i\in\text{Sing}_A(p)$. Notice that $\text{Sing}_A(p)\subset\{j\in\cE(A):p\in\cT_{j;A}(R)\}$, $i\in\text{Sing}_A(p)$ contradicts the assumption that $i\in\text{Sing}_{A}(q)\setminus \{j\in\cE(A):p\in\cT_{j;A}(R)\}$. Therefore, we have $\{j\in\cE(A):p\in\cT_{j;A}(R)\}=\text{Sing}_{A}(q)$. It is then clear that $p\in \cU_q$ for any neighborhood $\cU_q$ of $q$ defined in \eqref{eqn:good.nbhd} with $\epsilon=R$ and $\epsilon'>0$. Moreover, by the remark after Definition \ref{dfn:good.nbhd}, $\cU_q$ is a good neighborhood of $q$ as long as $\epsilon'$ is sufficiently small.
\item This is a straightforward corollary of the first assertion. 
\item This is a straightforward corollary of the second and the fourth assertion. \qedhere
\end{enumerate}
\end{proof}

\subsection{Construction of a smooth positive semidefinite symmetric $2$-tensor field $\tg_{\Gamma;A}$ on $\TT_A$}\label{subsect:constructions-1}
For any consecutive subset $A\subset\ZZ_+$ and for any good subgroup $\Gamma\subset\Gamma_A$, 
we first construct a smooth positive semidefinite symmetric $2$-tensor field $\tg_{\Gamma;A}$ on $\TT_A$. This will be the pull-back of our desired Riemannian metric on $\Gamma\backslash\TT_A$ by the natural projection map $\Pi_\Gamma:\TT_A\to\Gamma\backslash\TT_A$. (See Subsection \ref{subsect:constructions-2}.)
\begin{dfn}[Construction of $\tg_{\Gamma;A}$]\label{dfn:construction}
Let $A$ be a consecutive subset of $\ZZ_+$ and $\Gamma$ be a good subgroup of $\Gamma_A$. Write $\Gamma=\Gamma_A(E)$ for some $E\subset\cE(A)$. For any $E'\subset\cE(A)$ we let $\tg_{\Gamma;A;E'}$ be a smooth positive semidefinite symmetric $2$-tensor field on $\TT_A$ such that $(\TT_A,\tg_{\Gamma;A;E'})$ is obtained by inserting the following special tubes into $\TT_A$ with the standard product metric: 
\begin{itemize}
\item The $\{i,i+1\}$-special $1$-tube with base $B_{i,i+1}(p_{i,i+1},\pi/10)$, for any $p_{i,i+1}\in\Omega(i)$ and for any $i\in  E$ such that $i\not\in E'$.
\item The $\{i,i+1\}$-special $2$-tube with base $B_{i,i+1}(p_{i,i+1},\pi/10)$, for any $p_{i,i+1}\in\Omega(i)$ and for any $i\in\cE(A)\setminus E$ such that $i\not\in E'$.
\end{itemize}
For simplicity, we let $\tg_{\Gamma;A;\emptyset}=\tg_{\Gamma;A}$.
\end{dfn}

\begin{rmk}
The constructions in Definition \ref{dfn:construction} are well-defined due to Lemma \ref{lem:tube.index.properties-1-2}, \hyperlink{T1}{(T1)} and \hyperlink{T2}{(T2)}. 

Let $\Gamma'=\Gamma_A(E')$. For any $q\in\TT_A$, if there exists a good neighborhood $\cU_q$ of $q$ of the form \eqref{eqn:good.nbhd} with $\epsilon\in(\pi/10,\pi/8)$, such that 
\begin{align}\label{eqn:construction-1}
\cU_q\subset
\begin{cases}
\displaystyle \TT_A\setminus \left(\bigcup_{i\in\cE(A)}\cT_{i;A}(\pi/10)\right),~&\text{if }\Sing_A(q)=\emptyset,\\
\displaystyle \left(\bigcap_{i\in\Sing_A(q)}\cT_{i;A}(\epsilon)\right)\setminus\left(\bigcup_{i\in\cE(A)\setminus \Sing_A(q)}\cT_{i;A}(\pi/10)\right),~&\text{otherwise,}
\end{cases}
\end{align}
then $(\cU_q,\tg_{\Gamma;A;E'}|_{\cU_q})$ is the product of the following manifolds equipped with positive semidefinite symmetric 2-tensor fields:
\begin{itemize}
\item The union of $(B_{i,i+1}(\pr^A_{i,i+1}(q),\epsilon)\setminus B_{i,i+1}(\pr^A_{i,i+1}(q),\pi/10), d\theta_i\otimes d\theta_i+d\theta_{i+1}\otimes d\theta_{i+1})$ and a $\{i,i+1\}$-special $1$-caps centered at $\pr^A_{i,i+1}(q)$ with Euclidean radius $\pi/10$, where $i\in \{j\in E\setminus E':q\in\Omega_A(j)\}=\Sing_{\Gamma';A}(q)\setminus\Sing_{\Gamma'\cap\Gamma;A}(q)$. (See Definition \ref{sing index} and its remark.)
\item The union of $(B_{i,i+1}(\pr^A_{i,i+1}(q),\epsilon)\setminus B_{i,i+1}(\pr^A_{i,i+1}(q),\pi/10), d\theta_i\otimes d\theta_i+d\theta_{i+1}\otimes d\theta_{i+1})$ and a$\{i,i+1\}$-special $2$-caps centered at $\pr^A_{i,i+1}(q)$ with Euclidean radius $\pi/10$, where $i\in \{j\in\cE(A)\setminus\{E\cup E'\}:q\in\Omega_A(j)\}=\Sing_{\Gamma'\cap\Gamma;A}(q)$. (See Definition \ref{sing index} and its remark.)
\item $\pr^A_{A\setminus\Sing^+_{\Gamma';A}(q)}(\cU_q)$ equipped with the standard Euclidean metric $\displaystyle \sum_{i\in A\setminus\Sing^+_{\Gamma';A}(q)}d\theta_i\otimes d\theta_i$. (See Definition \ref{sing index}.)
\end{itemize}
Moreover, open subsets of the above type cover $\TT_A$. Indeed, for any $p\in\TT_A$ and for any $\epsilon\in(\pi/10,\pi/8)$, by Lemma \ref{lem:tube.index.properties-1-3}, there exist some $q\in\TT_A$ and a good neighborhood $\cU_q$ with the above choice of $\epsilon$ and a sufficiently small $\epsilon'\in(0,(\epsilon-(\pi/10))/2)$, such that $\Sing_A(q)=\{j\in\cE(A):p\in\cT_{j;A}(\epsilon)\}$. Using \eqref{eqn:good.nbhd}, \eqref{eqn:j-tubes} and Lemma \ref{lem:tube.index.properties-1-5}, it is not hard to verify that $\cU_q$ satisfies \eqref{eqn:construction-1}. 
\end{rmk}


We summarize some basic properties of $\tg_{\Gamma;A;E'}$ as follows:

\begin{lem}[Properties of $\tg_{\Gamma;A;E'}$]\label{lem:properties.of.tg}
Let $A$ be a consecutive subset of $\ZZ_+$ and $E'\subset\cE(A)$. For simplicity, we write $\Gamma'=\Gamma_A(E')$. Let $\Gamma=\Gamma_A(E)$ be an arbitrary good subgroup of $\Gamma_A$, where $E\subset\cE(A)$. The following holds:
\begin{enumerate}[ref=\ref{lem:properties.of.tg}(\arabic*)]
\item $\tg_{\Gamma;A;E'}$ is $\Gamma_A$-invariant. \label{lem:properties.of.tg-1}
\item $\tg_{\Gamma;A;E'}\in\cM_{\Gamma\cap\Gamma'}(\TT_A)$. (See Corollary \ref{cor:quotient is mfld} for the definition of $\cM_{\Gamma\cap\Gamma'}(\TT_A)$.)
\label{lem:properties.of.tg-2}
\item Let $A_1,\dots,A_q$ be the consecutive partition of $A$ induced by $E'$. Then the identity map on $\TT_A$ gives the following isometry between manifolds with positive semidefinite symmetric $2$-tensor fields:
\begin{align}\label{eqn:prod.struct.after.forget}
(\TT_A,\tg_{\Gamma;A;E'})\isom\prod_{j=1}^q(\TT_{A_j},\tg_{\Gamma\cap\Gamma_{A_j};A_j}).
\end{align}\label{lem:properties.of.tg-3}
\item When restricted to $\TT_A\setminus\bigcup_{j\in E'}\cT_{j;A}(\pi/10)$, $\tg_{\Gamma;A;E'}$ and $\tg_{\Gamma;A}$ are the same. \label{lem:properties.of.tg-4}
\end{enumerate}
\end{lem}
\begin{proof}
\begin{enumerate}
\item This follows directly from Definition \ref{dfn:special tubes}, Definition \ref{dfn:insert/forget tubes}, Definition \ref{dfn:construction}, Lemma \ref{lem:tube.index.properties-1}, the remark after Definition \ref{subsect:tubes.that.we.consider} and the remark after Definition \ref{dfn:cap.on.torus}.
\item This follows directly from \eqref{eqn:good.nbhd}, \eqref{eqn:z.to.z2}, Proposition \ref{prop:quotient.is.sphere}, Lemma \ref{lem:properties.of.tg-1}, Definition \ref{dfn:cap.on.torus}, and the remark after Definition \ref{dfn:construction}.
\item This follows directly from Lemma \ref{lem:properties.good.subgp-(3)},  \hyperlink{T3}{(T3)}, and Definition \ref{dfn:construction}.
\item This follows directly from \hyperlink{T1}{(T1)}, Definition \ref{dfn:insert/forget tubes} and Definition \ref{dfn:construction}.\qedhere
\end{enumerate}
\end{proof}

\subsection{Construction of a smooth Riemannian metric on $g_{\Gamma;A}$ on $\Gamma\backslash \TT_A$}\label{subsect:constructions-2}
Thanks to Corollary \ref{cor:quotient is mfld} and Lemma \ref{lem:properties.of.tg-2}, we can define the following Riemannian metrics:
\begin{dfn}\label{dfn:construction.of.g}
Let $A$ be a consecutive subset of $\ZZ_+$ and $\Gamma$ be a good subgroup of $\Gamma_A$. For any $E'\subset\cE(A)$, we let $g_{\Gamma;A;E'}$ be the unique smooth Riemannian metric on $(\Gamma\cap\Gamma_A(E'))\backslash\TT_A$ such that $\Pi_{\Gamma\cap\Gamma_A(E')}^*g_{\Gamma;A;E'}=\tg_{\Gamma;A;E'}$.

Similar to Definition \ref{dfn:construction}, we let $g_{\Gamma;A;\emptyset}=g_{\Gamma;A}$.
\end{dfn}

We summarize some basic properties of $g_{\Gamma;A;E'}$ as follows:
\begin{lem}[Properties of $g_{\Gamma;A;E'}$]\label{lem:properties.of.g.Gamma}
Let $A$ be a consecutive subset of $\ZZ_+$ and $E'\subset\cE(A)$. For simplicity, we write $\Gamma'=\Gamma_A(E')$. Let $\Gamma=\Gamma_A(E)$ be an arbitrary good subgroup of $\Gamma_A$, where $E\subset\cE(A)$. The following holds:
\begin{enumerate}[ref=\ref{lem:properties.of.g.Gamma}(\arabic*)]
\item $((\Gamma\cap\Gamma')\backslash\TT_A, g_{\Gamma;A;E'})$ is nearly nonpositively curved in the sense of Definition \ref{dfn:near nonpos curv}. Moreover, for any point $q\in\TT_A$ and for any good neighborhood $\cU_q$ of the form \eqref{eqn:good.nbhd} with $\epsilon\in(\pi/10,\pi/8)$ satisfying \eqref{eqn:construction-1}, $(\Pi_{\Gamma\cap\Gamma'}(\cU_q),g_{\Gamma;A;E'}|_{\Pi_{\Gamma\cap\Gamma'}(\cU_q)})$ is nearly nonpositively curved in the sense of Definition \ref{dfn:near nonpos curv}.\label{lem:properties.of.g.Gamma-1}
\item  Let $A_1,\dots,A_q$ be the consecutive partition of $A$ induced by $E'$. Then, as Riemannian manifolds, $((\Gamma\cap\Gamma')\backslash\TT_A, g_{\Gamma;A;E'})$ is naturally isometric to the product of $((\Gamma\cap\Gamma_{A_i})\backslash\TT_{A_i},g_{\Gamma\cap\Gamma_{A_i};A_i})$ for any $i\in\{1,\dots,q\}$. (See Lemma \ref{lem:properties.good.subgp-(5)}, Lemma \ref{lem:properties.of.tg-3} and the commutative diagram in Figure \ref{fig:comm.diag}.)\label{lem:properties.of.g.Gamma-2}
\item When restricted to $\Pi_{\Gamma\cap\Gamma'}(\TT_A\setminus\bigcup_{i\in E'}\cT_{i,A}(\pi/10))$, we have $g_{\Gamma;A;E'}=(\Pi^{\Gamma\cap\Gamma'}_\Gamma)^*g_{\Gamma;A}$.\label{lem:properties.of.g.Gamma-3}

\begin{figure}[h]
	\centering
	\includegraphics[width=6in]{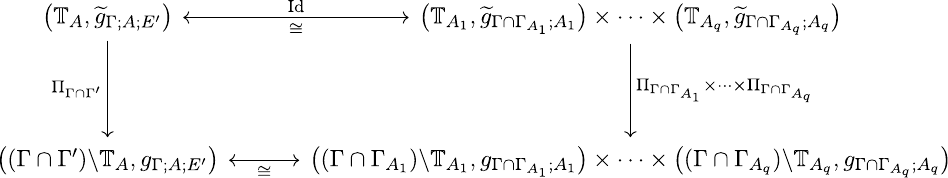}
	\caption{}\label{fig:comm.diag}
\end{figure}

\item Let $I$ be an interval and $c:I\to (\Gamma\backslash \TT_A,g_{\Gamma;A})$ be a geodesic. Assume that $c(I)\subset \Pi_{\Gamma}(\TT_A\setminus\bigcup_{i\in E'}\cT_{i;A}(\pi/10))$, then the following holds:
\begin{enumerate}
\item There exists a geodesic $\tc_{\Gamma\cap\Gamma'}:I\to (\Gamma\cap\Gamma'\backslash\TT_A, g_{\Gamma;A;E'})$ such that $\Pi_\Gamma^{\Gamma\cap\Gamma'}\circ\tc_{\Gamma\cap\Gamma'}=c$ and $\tc_{\Gamma\cap\Gamma'}(I)\in\Pi_{\Gamma\cap\Gamma'}(\TT_A\setminus\bigcup_{i\in E'}\cT_{i;A}(\pi/10))$. 
\item For any continuous map $\tc:I\to (\Gamma\cap\Gamma'\backslash\TT_A, g_{\Gamma;A;E'})$ such that $\Pi_\Gamma^{\Gamma\cap\Gamma'}\circ\tc_{\Gamma\cap\Gamma'}=c$, $\tc$ is a geodesic in $ (\Gamma\cap\Gamma'\backslash\TT_A, g_{\Gamma;A;E'})$.
Moreover, the map $Y\to (\Pi_{\Gamma}^{\Gamma\cap\Gamma'})_*Y$ gives a bijection from the collection of (orthogonal) Jacobi fields along $\tc_{\Gamma\cap\Gamma'}$, to the collection of (orthogonal) Jacobi fields along $c$.
\end{enumerate}\label{lem:properties.of.g.Gamma-4}
\item For any $E''\subset\cE(A)$ and for any $p\in \TT_A\setminus\left(\bigcup_{j\in E''}\overline{\cT_{j;A}(\pi/10)}\right)$, there exist $q\in\TT_A$ and a good neighborhood $\cU_q$ of the form \eqref{eqn:good.nbhd} containing $p$, such that $\cU_q$ satisfies \eqref{eqn:construction-1} and that $\cU_q\subset \TT_A\setminus\left(\bigcup_{j\in E''}\overline{\cT_{j;A}(\pi/10)}\right)$. In particular, the Riemannian manifold
$$\left(\Pi_{\Gamma\cap\Gamma'}\left(\TT_A\setminus\left(\bigcup_{j\in E''}\overline{\cT_{j;A}(\pi/10)}\right)\right), g_{\Gamma;A;E'}|_{\Pi_{\Gamma\cap\Gamma'}\left(\TT_A\setminus\left(\bigcup_{j\in E''}\overline{\cT_{j;A}(\pi/10)}\right)\right)}\right)$$ 
is nearly nonpositively curved in the sense of Definition \ref{dfn:near nonpos curv}. \label{lem:properties.of.g.Gamma-5}
\end{enumerate}
\end{lem}
\begin{proof}
\begin{enumerate}
\item This follows directly from Definition \ref{dfn:near nonpos curv}, Definition \ref{dfn:gen caps}, Proposition \ref{prop:quotient.is.sphere}, Definition \ref{dfn:construction.of.g} and the remark after Definition \ref{dfn:construction}.
\item This follows directly from Lemma \ref{lem:properties.good.subgp-(5)}, Lemma \ref{lem:properties.of.tg-3} and Definition \ref{dfn:construction.of.g}. 
\item This follows directly from \eqref{eqn:quotient.maps}, Lemma \ref{lem:tube.index.properties-1-4}, Lemma \ref{lem:properties.of.tg-4}. and Definition \ref{dfn:construction.of.g}. 
\item  By Lemma \ref{lem:properties.good.subgp-(4)}, \eqref{eqn:quotient.maps} and Lemma \ref{lem:tube.index.properties-1-4}, we have 
$$(\Pi_{\Gamma}^{\Gamma\cap\Gamma'})^{-1}\left(\Pi_{\Gamma}\left(\TT_A\setminus\bigcup_{i\in E'}\cT_{i;A}(\pi/10)\right)\right)=\Pi_{\Gamma\cap\Gamma'}\left(\TT_A\setminus\bigcup_{i\in E'}\cT_{i;A}(\pi/10)\right).$$
Hence by Lemma \ref{lem:properties.of.g.Gamma-3} and Definition \ref{dfn:construction.of.g}, it suffices to show that $\Pi_\Gamma^{\Gamma\cap\Gamma'}$ is a local diffeomorphism when restricted to a neighborhood of $\Pi_{\Gamma\cap\Gamma'}(\TT_A\setminus\bigcup_{i\in E'}\cT_{i;A}(\pi/10))$. Notice that for any $p\in\TT_A\setminus\bigcup_{i\in E'}\cT_{i;A}(\pi/10)$, by \eqref{eqn:j-tubes} and Definition \ref{sing index}, we have $\Sing_A(p)\cap E'=\emptyset$. Following Lemma \ref{lem:properties.good.subgp-(4)} and the remark after Definition \ref{sing index}, we have $\Sing_{\Gamma;A}(p)=\Sing_A(p)\setminus E=\Sing_A(p)\setminus (E\cup E')=\Sing_{\Gamma\cap\Gamma';A}(p)$. Lemma \ref{lem:branching.locus-2} and the above imply that $\Pi_\Gamma^{\Gamma\cap\Gamma'}$ is a local diffeomorphism when restricted to a neighborhood of $\Pi_{\Gamma\cap\Gamma'}(\TT_A\setminus\bigcup_{i\in E'}\cT_{i;A}(\pi/10))$. This finishes the proof.
\item This follows directly from \eqref{jth singularity}, \eqref{eqn:good.nbhd}, \eqref{eqn:j-tubes}, Lemma \ref{lem:properties.of.g.Gamma-1} and the remark after Definition \ref{dfn:construction}.\qedhere
\end{enumerate}
\end{proof}
\section{The Bernoulli property of geodesic flow on $(\Gamma\backslash\TT_A,g_{\Gamma;A})$}\label{sect:proof.of.thm}

\subsection{Characterizing rank one geodesics: nomadic geodesics}\label{subsect:nomadic}
Throughout this subsection, unless otherwise mentioned, we let $A$ be a consecutive subset of $\ZZ_+$ with $|A|\geq 2$ and $\Gamma$ be an arbitrary good subgroup of $\Gamma_A$. 
\begin{dfn}[Lifts]\label{dfn:geo lift}
Let $I$ be an interval and $c:I\to (\Gamma\backslash\TT_A, g_{\Gamma;A})$ be a geodesic. A \emph{lift} $\tc$ of $c$ is a continuous map $\tc: I\to \TT_A$ such that $\Pi_\Gamma\circ\tc=c$.
\end{dfn}
\begin{rmk}
It is not hard to verify that if $c^{-1}(\Pi_\Gamma(\Omega_{\Gamma;A}))$ does not have accumulation points, then $c$ admits a lift. In other words, one can always lift a geodesic which does not accumulate on the branching locus of $\Pi_\Gamma$.

Later, we will frequently use the fact that for any $\Gamma$-invariant subset $U$ of $\TT_A$, we have $c^{-1}(\Pi_\Gamma(U))=\tc^{-1}(U)$. (For example, one can choose $U$ as $\cT_{i;A}(R)$ or $\TT_A\setminus \cT_{i;A}(R)$ for any $i\in\cE(A)$ and for any $R\in(0,\pi/8)$. See Lemma \ref{lem:tube.index.properties-1-1}.)
\end{rmk}
\begin{dfn}[Nomadic geodesic]\label{dfn:nomadic}
Let $I$ be an interval and $c:I\to (\Gamma\backslash\TT_A, g_{\Gamma;A})$ be a geodesic. For any $i\in\cE(A)$, we say that $c$ \emph{passes through} a $\{i,i+1\}$-special tube if there exists $t_-, t,t_+\in I$ with $t_-<t<t_+$ such that $c(t)\in\Pi_\Gamma(\cT_{i;A}(\pi/10))$ and $c(t_\pm)\not\in \overline{\Pi_\Gamma(\cT_{i;A}(\pi/10))}$.

We say that $c$ is \emph{nomadic} if for any $i\in\cE(A)$, it passes through a $\{i,i+1\}$-special tube.
\end{dfn}
It would be very inconvenient for us if there are many geodesics which pass through a special tube infinitely many times on a finite interval. The following lemma essentially verifies that ``most'' geodesics do not have such crazy behavior. This is used in the proof of Proposition \ref{prop:nomadic.is.rank-1}.
\begin{lem}\label{lem:pass.thru.not.crazy}
Let $I$ be a closed interval and $c:I\to (\Gamma\backslash\TT_A,g_{\Gamma;A})$ be a geodesic with constant speed no larger than one. If $c^{-1}(\Pi_\Gamma(\Omega_{\Gamma;A}))$ does not have accumulation points, then there exists some $\delta>0$ depending only on $(\Gamma\backslash\TT_A,g_{\Gamma;A})$ such that for any $i\in\cE(A)$, the distance between any two distinct connected components in $c^{-1}(\Pi_\Gamma(\cT_{i;A}(\pi/10)))$ is at least $\delta$.
\end{lem}
\begin{proof}
By the remark after Definition \ref{dfn:geo lift}, there is a lift $\tc:I\to\TT_A$ of $c$. Moreover, by Definition \ref{dfn:construction.of.g}, after removing $I\cap\tc^{-1}(\Omega_{\Gamma;A})$ from $I$, $\tc$ is a piecewise geodesic in $(\TT_A\setminus \Omega_{\Gamma;A}, \tg_{\Gamma;A})$ with constant speed no larger than one.  Observe the following:
\begin{itemize}
\item If there exist $t_1,t_2\in I$ such that $\tc(t_1)\in\cT_{i;A}(\pi/10)$ and $\tc(t_2)\in\TT_A\setminus\cT_{i;A}(\pi/9)$, then $|t_1-t_2|\geq (\pi/9-\pi/10)=\pi/90$.  In particular, if there exist $t_1',t_2'\in I$ such that $\tc(t_1')$ and $\tc(t_2')$ are in different connected components of $\cT_{i;A}(\pi/10)$, then $|t_1'-t_2'|\geq \pi/90$. (This follows from Lemma \ref{lem:branching.locus-1}, Definition \ref{dfn:construction}, Definition \ref{dfn:construction.of.g} and the fact that $c$ has constant speed no larger than one.)
\item There exists some $\delta\in(0,\min\{\text{injrad}(\Gamma\backslash\TT_A,g_{\Gamma;A})/2,\pi/180\})$ such that for any $i\in\cE(A)$, if $\tc(t_1)$ and $\tc(t_2)$ are in the same connected component of $\cT_{i;A}(\pi/10)$ and $0<t_2-t_1<2\delta$, then $c([t_1,t_2])\subset \Pi_\Gamma(\cT_{i;A}(\pi/10))$. (When $|A|=2$, $\Pi_\Gamma(\cT)$ equipped with the restriction of $g_{\Gamma;A}$ is an extended cap. Since $(\Gamma\backslash\TT_A,g_{\Gamma;A})$ is a nearly nonpositively curved surface with finitely many extended caps (due to Definition \ref{dfn:construction} and Definition \ref{dfn:construction.of.g}), the $|A|=2$ case follows directly from Definition \ref{dfn:caps} and its remark. The case when $|A|\geq 3$ can be reduced to the case when $|A|=2$ using Lemma \ref{lem:tube.index.properties-1-4}-\ref{lem:tube.index.properties-1-5} and Lemma \ref{lem:properties.of.g.Gamma-2}-\ref{lem:properties.of.g.Gamma-4}, where $E'$ in Lemma \ref{lem:properties.of.g.Gamma} is chosen to be $\{i-1,i+1\}\cap\cE(A)$.)
\end{itemize}
By the above two bullet points, for any connected components $I_1,I_2$ of $c^{-1}(\Pi_\Gamma(\cT_{i;A}(\pi/10)))$, $d(I_1,I_2)\leq \delta$ is equivalent to $I_1=I_2$. This finishes the proof.
\end{proof}

We would like to show that for bi-infinite geodesics, being nomadic is equivalent to having rank one. (See Corollary \ref{cor:nomadic=rank-1}.) This contains the following two parts:
\begin{enumerate}
\item If a bi-infinite geodesic fails to be nomadic, there are two possibilities:
\begin{itemize}
\item This geodesic does not visit $\Pi_{\Gamma}(\cT_{i;A}(\pi/10))$ for some $i\in\cE(A)$.
\item This geodesic enters $\Pi_{\Gamma}(\cT_{i;A}(\pi/10))$ for some $i\in\cE(A)$ and never leaves there. (Recall that in Definition \ref{dfn:nomadic}, being nomadic requires that the geodesic leaves after it enters $\overline{\Pi_{\Gamma}(\cT_{i;A}(\pi/10))}$ for any $i\in\cE(A)$.)
\end{itemize}
Lemma \ref{lem:missing.tube=higher.rank} proves that geodesics in the first case must have rank at least two. The second assertion in Lemma \ref{lem:rank-1.not.in.sing} proves that geodesics in the second case must have rank at least two.
\item If a bi-infinite geodesic is nomadic, we would like to show that it has rank one. The main proof of this part is in Proposition \ref{prop:nomadic.is.rank-1}. By Definition \ref{dfn:geo rank}, we need to prove that any Jacobi field $Y$ with $f_{(\Gamma\backslash\TT_A,g_{\Gamma;A})}(Y;t)\equiv 0$ must be a scalar multiple of the tangent vector field of this geodesic. 

If a geodesic accumulates on the branching locus of $\Pi_\Gamma$, it would be quite inconvenient for us to study the Jacobi fields along it. Fortunately, such geodesics must have rank at least two and cannot be nomadic due to the first assertion in Lemma \ref{lem:rank-1.not.in.sing}. 
\end{enumerate}
\begin{lem}\label{lem:missing.tube=higher.rank}
Let $I$ be an interval and $c:I\to (\Gamma\backslash\TT_A, g_{\Gamma;A})$ be a non-constant geodesic. Let $i\in\cE(A)$. If $c(I)\cap\Pi_\Gamma(\cT_{i;A}(\pi/10))=\emptyset$, then $c$ has rank at least two.
\end{lem}
\begin{proof}
Let $E'=\{i\}\subset\cE(A)$ and $\Gamma':=\Gamma_A(E')$. In particular, $\Gamma'$ and $\Gamma\cap\Gamma'$ are good subgroups of $\Gamma_A$. (See Definition \ref{dfn:good subgp} and Lemma \ref{lem:properties.good.subgp-(4)}.) 
By Lemma \ref{lem:properties.good.subgp-(4)}, Lemma \ref{lem:tube.index.properties-1-4} and Lemma \ref{lem:properties.of.g.Gamma-4}, there is a smooth geodesic $\tc_{\Gamma\cap\Gamma'}:I\to\Pi_{\Gamma\cap\Gamma'}(\TT_A\setminus\cT_{i;A}(\pi/10))$ in $((\Gamma\cap\Gamma')\backslash\TT_A,g_{\Gamma;A;E'})$ such that $\Pi_{\Gamma}^{\Gamma\cap\Gamma'}\circ\tc_{\Gamma\cap\Gamma'}=c$. By Lemma \ref{lem:properties.of.g.Gamma-2}, $\tc_{\Gamma\cap\Gamma'}$ admits an orthogonal parallel Jacobi field $Y(\cdot)$  with respect to $g_{\Gamma;A;E'}$. Hence by Lemma \ref{lem:properties.of.g.Gamma-3} and Lemma \ref{lem:properties.of.g.Gamma-4}, $(\Pi_{\Gamma}^{\Gamma\cap\Gamma'})_*Y(\cdot)$ is an orthogonal parallel Jacobi field along $c$ with respect to $g_{\Gamma;A}$. This proves that $c$ has rank at least two.
\end{proof}
\begin{lem}\label{lem:rank-1.not.in.sing}
Let $I$ be an interval and $c:I\to (\Gamma\backslash\TT_A, g_{\Gamma;A})$ be a geodesic. Let $i\in\cE(A)$. 
\begin{enumerate}[ref=\ref{lem:rank-1.not.in.sing}(\arabic*)]
\item If there exists some $p\in\overline{\cT_{i;A}(\pi/10)}$ such that $c^{-1}(\Pi_\Gamma(\pr^A_{i,i+1}(\{p\})\times \TT_{A\setminus\{i,i+1\}}))$ has accumulation points, then $c(I)\subset \Pi_\Gamma(\pr^A_{i,i+1}(\{p\})\times \TT_{A\setminus\{i,i+1\}})$. In particular, $c$ is constant or has rank at least two. \label{lem:rank-1.not.in.sing-1}
\item If there is an infinite interval $I'\subset I$ such that $c(I')\subset \Pi_\Gamma(\overline{\cT_{i;A}(\pi/10)})$, then the bi-infinite geodesic containing $c$ is contained in $\Pi_\Gamma(\overline{\cT_{i;A}(\pi/10)})$. In particular, $c$ is either constant, or has rank at least two. \label{lem:rank-1.not.in.sing-2}
\end{enumerate}
\end{lem}
\begin{proof}
We prove both assertions via induction on $|A|$. When $|A|=2$, the first assertion obviously holds. (If $c$ is a geodesic such that $c^{-1}(p)$ has accumulation points for some $p$, then $c\equiv p$.) The second assertion is a direct corollary of the remark after Definition \ref{dfn:caps}. 

Assume that the lemma holds whenever $2\leq |A|\leq k-1$ for some $k\geq 3$. Consider the case when $|A|=k$. Let $E'=\{i-1,i+1\}\cap\cE(A)$, which is not empty. The proof of the two assertions are very similar. Due to this reason, we prove the first assertion in complete details. We will omit some repetitive details in the proof of the second assertion.
\begin{enumerate}
\item  If $c(I)\subset \Pi_\Gamma(\pr^A_{i,i+1}(\{p\})\times \TT_{A\setminus\{i,i+1\}})$, then by Lemma \ref{lem:tube.index.properties-1-5}, $c(I)\cap\Pi_\Gamma(\cT_{j;A}(\pi/10))=\emptyset$ for any $j\in E'$. Hence, $c$ is either constant or has rank at least two due to Lemma \ref{lem:missing.tube=higher.rank}. Therefore, it remains to prove that $c(I)\subset \Pi_\Gamma(\pr^A_{i,i+1}(\{p\})\times \TT_{A\setminus\{i,i+1\}})$.

It suffices to show that for any accumulation point $t_0\in I$ of $c^{-1}(\Pi_\Gamma(\pr^A_{i,i+1}(\{p\})\times \TT_{A\setminus\{i,i+1\}}))$, there exist some $\epsilon>0$ and an open interval $I_0:=(t_0-\epsilon,t_0+\epsilon)\cap I$, such that
\begin{align}\label{eqn:rank-1.not.in.sing-1}
c(I_0)\subset \Pi_\Gamma(\pr^A_{i,i+1}(\{p\})\times \TT_{A\setminus\{i,i+1\}}).
\end{align} 
By the fact that $\Pi_\Gamma$ is an open map and the fact that $\overline{\cT_{i;A}(\pi/10)}\subset\cT_{i;A}(\pi/9)$, there exists $\epsilon>0$ sufficiently small such that $c(I_0)\subset\Pi_\Gamma(\cT_{i;A}(\pi/9))$. By Lemma \ref{lem:tube.index.properties-1-4} and the fact that $c(I)\cap\Pi_\Gamma(\cT_{j;A}(\pi/10))=\emptyset$ for any $j\in E'$, we have $c(I_0)\subset\Pi_\Gamma(\TT_A\setminus\bigcup_{j\in E'}\cT_{j;A}(\pi/10))$. Let $\Gamma'=\Gamma_A(E')$. Then by Lemma \ref{lem:properties.of.g.Gamma-4} and the above, there is a smooth geodesic $\tc_{\Gamma\cap\Gamma'}:I_0\to\Pi_{\Gamma\cap\Gamma'}(\TT_A\setminus\bigcup_{j\in E'}\cT_{j;A}(\pi/10))$ in $((\Gamma\cap\Gamma')\backslash\TT_A,g_{\Gamma;A;E'})$ such that $\Pi_{\Gamma}^{\Gamma\cap\Gamma'}\circ\tc_{\Gamma\cap\Gamma'}=c|_{I_0}$. Notice that by \eqref{eqn:generator}, \eqref{eqn:Gamma.A} and Definition \ref{dfn:good subgp}, for any $\gamma\in\Gamma$, we have
\begin{align}\label{eqn:rank-1.not.in.sing-2}
\Pi_\Gamma(\pr^A_{i,i+1}(\{p\})\times \TT_{A\setminus\{i,i+1\}})=\Pi_\Gamma(\pr^A_{i,i+1}(\{\gamma p\})\times \TT_{A\setminus\{i,i+1\}}).
\end{align}
Hence by \eqref{eqn:quotient.maps} and \eqref{eqn:rank-1.not.in.sing-2}, we have
$$(\Pi_{\Gamma}^{\Gamma\cap\Gamma'})^{-1}(\Pi_\Gamma(\pr^A_{i,i+1}(\{p\})\times \TT_{A\setminus\{i,i+1\}}))=\bigcup_{\gamma\in\Gamma}\Pi_{\Gamma\cap\Gamma'}(\pr^A_{i,i+1}(\{\gamma p\})\times \TT_{A\setminus\{i,i+1\}}).$$
In particular, by the assumption on $t_0$ and the fact that $\Pi_{\Gamma}^{\Gamma\cap\Gamma'}\circ\tc_{\Gamma\cap\Gamma'}=c|_{I_0}$, there exists some $\gamma\in\Gamma$ such that $t_0$ is an accumulation point for $(\tc_{\Gamma\cap\Gamma'})^{-1}(\Pi_{\Gamma\cap\Gamma'}(\pr^A_{i,i+1}(\{\gamma p\})\times \TT_{A\setminus\{i,i+1\}}))$. Since $E'=\{i-1,i+1\}\cap\cE(A)\neq\emptyset$, the consecutive partition of $A$ induced by $E'$ has at least two atoms with one of the atoms containing $i,i+1$. (See Definition \ref{dfn:consecutive}.) Hence, by Lemma \ref{lem:properties.of.g.Gamma-2} and the inductive step, we have $\tc_{\Gamma\cap\Gamma'}(I_0)\subset\Pi_{\Gamma\cap\Gamma'}(\pr^A_{i,i+1}(\{\gamma p\})\times \TT_{A\setminus\{i,i+1\}})$. Recall that $\gamma\in \Gamma$, \eqref{eqn:rank-1.not.in.sing-1} then follows from \eqref{eqn:rank-1.not.in.sing-2}, the fact that $\Pi_{\Gamma}^{\Gamma\cap\Gamma'}\circ\tc_{\Gamma\cap\Gamma'}=c|_{I_0}$ and the above discussion.

\item Similar to the first paragraph of the proof of the first assertion, if $c(\RR)$ is contained in $\Pi_\Gamma(\overline{\cT_{i;A}(\pi/10)})$, $c$ is either constant or has rank at least two due to Lemma \ref{lem:missing.tube=higher.rank}. Therefore it remains to prove that $c(\RR)$ is contained in $\Pi_\Gamma(\overline{\cT_{i;A}(\pi/10)})$. 

Let $\Gamma'=\Gamma_A(E')$. Similar to the proof of the first assertion, there is a smooth geodesic $\tc_{\Gamma\cap\Gamma'}:I'\to\Pi_{\Gamma\cap\Gamma'}(\TT_A\setminus\bigcup_{j\in E'}\cT_{j;A}(\pi/10))$ in $((\Gamma\cap\Gamma')\backslash\TT_A,g_{\Gamma;A;E'})$ such that $\Pi_{\Gamma}^{\Gamma\cap\Gamma'}\circ\tc_{\Gamma\cap\Gamma'}=c|_{I'}$. Since $c(I')\subset\Pi_{\Gamma}(\overline{\cT_{i;A}(\pi/10)})$, we have $\tc_{\Gamma\cap\Gamma'}(I')\subset \Pi_{\Gamma\cap\Gamma'}(\overline{\cT_{i;A}(\pi/10)})$ (due Lemma \ref{lem:tube.index.properties-1-4}). 

We extend $\tc_{\Gamma\cap\Gamma'}$ and $c$ to bi-infinite geodesics in $((\Gamma\cap\Gamma')\backslash\TT_A,g_{\Gamma;A})$ and $(\Gamma\backslash\TT_A,g_{\Gamma;A;E'})$ respectively. For simplicity, we use the same notations for the extended geodesics. Since $E'=\{i-1,i+1\}\cap\cE(A)\neq\emptyset$, the consecutive partition of $A$ induced by $E'$ has at least two atoms with one of the atoms containing $i,i+1$. (See Definition \ref{dfn:consecutive}.) By Lemma \ref{lem:properties.of.g.Gamma-2} and the inductive step, we have $\tc_{\Gamma\cap\Gamma'}(\RR)\subset\Pi_{\Gamma\cap\Gamma'}(\overline{\cT_{i;A}(\pi/10)})$. Notice that by Lemma \ref{lem:tube.index.properties-1-4}, Lemma \ref{lem:tube.index.properties-1-5} and the definition of $E'$, we have
$$\Pi_{\Gamma\cap\Gamma'}(\overline{\cT_{i;A}(\pi/10)})\subset\Pi_{\Gamma\cap\Gamma'}({\cT_{i;A}(\pi/9)})\subset \Pi_{\Gamma\cap\Gamma'}\left(\TT_A\setminus\bigcup_{j\in E'}\cT_{j;A}(\pi/10)\right).$$
By Lemma \ref{lem:properties.of.g.Gamma-3} and the fact that $\Pi_{\Gamma}^{\Gamma\cap\Gamma'}\circ\tc_{\Gamma\cap\Gamma'}|_{I'}=c|_{I'}$, we have $\Pi_{\Gamma}^{\Gamma\cap\Gamma'}\circ\tc_{\Gamma\cap\Gamma'}=c$ as bi-infinite geodesics. In particular, $c(\RR)\subset\Pi_{\Gamma}(\overline{\cT_{i;A}(\pi/10)})$. This completes the proof. \qedhere
\end{enumerate}
\end{proof}

\begin{prop}\label{prop:nomadic.is.rank-1}
Let $c:\RR\to (\Gamma\backslash\TT_A, g_{\Gamma;A})$ be a geodesic. If $c$ is nomadic, then $c$ has rank one.
\end{prop}
\begin{proof}

We first show that $c^{-1}(\Pi_\Gamma(\Omega_{\Gamma; A}))$ does not have accumulation points. If not, then by \eqref{jth singularity} and \eqref{total singularity subgp}, there exist some $i\in\cE(A)$ and some $p_{i,i+1}\in\Omega(i)$ such that $c^{-1}(\Pi_\Gamma(\{p_{i,i+1}\}\times \TT_{A\setminus\{i,i+1\}}))$ has accumulation points. Then by Lemma \ref{lem:rank-1.not.in.sing-1}, $c$ is not nomadic in the sense of Definition \ref{dfn:nomadic}. This contradicts the assumption on $c$. Hence $c^{-1}(\Pi_\Gamma(\Omega_{\Gamma; A}))$ does not have accumulation points. We write $\Sigma:=c^{-1}(\Pi_\Gamma(\Omega_{\Gamma; A}))$. 

When $|A|=2$, by Lemma \ref{lem:properties.of.g.Gamma-1}, $(\Gamma\backslash\TT_A,g_{\Gamma;A})$ is a nearly nonpositively curved surface. If $c$ is nomadic, then $c$ passes through an extended cap. (See Definition \ref{dfn:special caps}, Definition \ref{dfn:construction} and Definition \ref{dfn:construction.of.g}.) In particular, it visits a region with negative curvature. (See Definition \ref{dfn:caps}) Therefore, by Lemma \ref{lem:rank-1 lem 2}, $c$ has rank one.

From now on, we assume that $|A|\geq 3$. Let $Y(t)$ be an arbitrary Jacobi field along $c(t)$ such that $f_{(\Gamma\backslash\TT_A,g_{\Gamma;A})}(Y,t)\equiv 0$ whenever $t\in \RR$. (See \eqref{eqn:infi Lyap2} and \eqref{eqn:noninc Lyap fcn2} for the relevant definitions.) Recall that $c^{-1}(\Pi_\Gamma(\Omega_{\Gamma; A}))=\Sigma$ does not have accumulation points and that $\Omega_{\Gamma;A}$ is the ramification locus of $\Pi_\Gamma$. (See Lemma \ref{lem:branching.locus}.) Following the remark after Definition \ref{dfn:geo lift}, there is a lift $\tc:\RR\to\TT_A$ of $c$. Moreover, $\tc$ restricted to any connected component of $\RR\setminus\Sigma$ is a geodesic segment in $(\TT_A\setminus\Omega_{\Gamma;A},\tg_{\Gamma;A})$ due to Definition \ref{dfn:construction.of.g}. In particular, for any $t\in \RR\setminus \Sigma$, one can define $\tY(t)\in T_{\tc(t)}\TT_A$ such that $(\Pi_\Gamma)_*\tY(t)=Y(t)$. In addition, $\tY(t)$ is a Jacobi field on every geodesic subsegment of $\tc$. 

For any $i\in A$ and for any $t\in \RR\setminus \Sigma$, we write $\tY_i(t):=(\pr^A_i)_*\tY(t)$ and $\tc_i'(t):=(\pr^A_i)_*\tc'(t)$.
Following the remark after Definition \ref{dfn:geo rank}, it suffices to show that there exists some $\alpha\in\RR$ such that for any $i\in A$ and for any $t\in \RR\setminus\Sigma$, we have
\begin{align}\label{eqn:nomadic.is.rank.1.goal}
\tY_i(t)=\alpha\cdot\tc_i'(t).
\end{align}
For any $i\in\cE(A)$, let $\cI_i$ be the set of all connected components of $ c^{-1}(\Pi_\Gamma(\cT_{i;A}(\pi/10)))$. Following Lemma \ref{lem:rank-1.not.in.sing-2} and the assumption that $c$ is nomadic, we have 
\begin{align}\label{eqn:finite.and.nonempty}
\cI_i\neq\emptyset\text{ and length}(I)<\infty\text{ for any }I\in\cI_i. 
\end{align}
To prove \eqref{eqn:nomadic.is.rank.1.goal}, we need the following definition and a technical lemma.
\begin{subdfn}[Adjacent]\label{subdfn:adjacent}
For any $i_1,i_2\in\cE(A)$, for any $I_1\in\cI_{i_1}$, for any $I_2\in\cI_{i_2}$ and for any $i\in A$, we say that $I_1$ and $I_2$ are \emph{$i$-adjacent} if the following holds:
\begin{itemize}
\item $i\in\{i_1,i_1+1\}\cap\{i_2,i_2+1\}$.
\item $I_1\neq I_2$ (and hence $\overline{I_1}\cap \overline{I_2}=\emptyset$ due to  the definition of $\cI_i$, Lemma \ref{lem:tube.index.properties-1-5} and Lemma \ref{lem:pass.thru.not.crazy}). Moreover, if $I_1=(a,b)$ and $I_2=(a',b')$ for some $a<b<a'<b'$, then $[b,a']\cap c^{-1}(\Pi_\Gamma(\cT_{j;A}(\pi/10)))=\emptyset$ for any $j\in\{i,i-1\}\cap\cE(A)$.
\end{itemize}
In addition, we say that $I_1,I_2$ are \emph{completely adjacent} if the following holds:
\begin{itemize}
\item $\{i_1,i_1+1\}\cap\{i_2,i_2+1\}\neq\emptyset$.
\item $I_1$ and $I_2$ are $i$-adjacent for any $i\in\{i_1,i_1+1\}\cap\{i_2,i_2+1\}$.
\end{itemize}
\end{subdfn}
\begin{sublem}\label{sublem:nomadic=rank-1}
The following properties hold:
\begin{enumerate}[ref=\ref{sublem:nomadic=rank-1}(\arabic*)]
\item For any $i\in\cE(A)$ and for any $I\in\cI_i$, we have $(\tc_i'(t),\tc_{i+1}'(t))\neq (0,0)$ for any $t\in I\setminus\Sigma$. Moreover, there exists an open neighborhood $I'$ of $\overline{I}$ and a unique constant $\alpha_{i,i+1}\in\RR$ such that $\tY_i(t)=\alpha_{i,i+1}\cdot\tc_i'(t)$ and $\tY_{i+1}(t)=\alpha_{i,i+1}\cdot\tc_{i+1}'(t)$ for any $t\in I'\setminus\Sigma$.
\label{sublem:nomadic=rank-1-1}
\item Let $i\in A$. For any interval $I$ such that $I\cap \bigcup_{j\in\{i,i-1\}\cap\cE( A)}c^{-1}(\Pi_\Gamma(\cT_{j;A}(\pi/10)))=\emptyset$, we have $\tc_i'(t)$ and $\tY_i(t)$ are constant functions on $I\setminus \Sigma$.
\label{sublem:nomadic=rank-1-2}
\item Let $i\in A$. For any $i_1,i_2\in\cE(A)$, for any $I_1\in\cI_{i_1}$ and $I_2\in\cI_{i_2}$ that are $i$-adjacent, there exist an open interval $I$ and a unique constant $\alpha\in\RR$  such that the following holds:
\begin{itemize}
\item $I\supset\overline{I_1\sqcup I_2}$.
\item $\tY_i(t)=\alpha\cdot\tc_i'(t)$ for any $t\in I\setminus\Sigma$.
\item There exists an open neighborhood $I_1'$ of $\overline{I_1}$ such that $\tY_{i_1}(t)=\alpha\cdot\tc_{i_1}'(t)$ and $\tY_{i_1+1}(t)=\alpha\cdot\tc_{i_1+1}'(t)$ for any $t\in I_1'\setminus\Sigma$.
\item There exists an open neighborhood $I_2'$ of $\overline{I_2}$ such that $\tY_{i_2}(t)=\alpha\cdot\tc_{i_2}'(t)$ and $\tY_{i_2+1}(t)=\alpha\cdot\tc_{i_2+1}'(t)$ for any $t\in I_2'\setminus\Sigma$.
\end{itemize}
\label{sublem:nomadic=rank-1-3}
\item Let $i\in A$. If there exists an open interval $I\subset \RR$ and a constant $\alpha\in\RR$ such that the following holds:
\begin{itemize}
\item $\tY_i(t)=\alpha \cdot\tc_i'(t)$ for any $t\in I\setminus\Sigma$.
\item $\overline{\bigcup_{j\in \cE(A)\cap\{i,i-1\}}\bigcup_{J\in\cI_{j}}J}\subset I$.
\end{itemize}
Then $\tY_i(t)=\alpha \cdot\tc_i'(t)$ for any $t\in\RR\setminus\Sigma$.
\label{sublem:nomadic=rank-1-4}
\end{enumerate}
\end{sublem}

\begin{proof}[Proof of Lemma \ref{sublem:nomadic=rank-1}]
\begin{enumerate}
\item Let $E'=\{i-1,i+1\}\cap\cE(A)$ and $I'$ be the unique connected component in $c^{-1}(\Pi_\Gamma(\cT_{i;A}(\pi/9)))$ containing $I$. For simplicity, we write $\Gamma'=\Gamma_A(E')$.
By Lemma \ref{lem:tube.index.properties-1-4} and Lemma \ref{lem:tube.index.properties-1-5}, 
\begin{align}\label{eqn:non.ints}
c(I')\subset \Pi_\Gamma\left(\TT_A\setminus\left(\bigcup_{j\in E'}\overline{\cT_{j;A}(\pi/10)}\right)\right).
\end{align} 
Hence, by Definition \ref{dfn:geo lift} and Lemma \ref{lem:properties.of.g.Gamma-4}, the following holds:
\begin{itemize}
\item $\Pi_{\Gamma\cap\Gamma'}\circ\tc|_{I'}$ is a geodesic in $(\Gamma\cap\Gamma'\backslash\TT_A,g_{\Gamma;A;E'})$. 
\item $(\Pi_{\Gamma\cap\Gamma'})_*\tY|_{I'\setminus\Sigma}$ extends to a smooth Jacobi field along $\Pi_{\Gamma\cap\Gamma'}\circ\tc|_{I'}$ in $(\Gamma\cap\Gamma'\backslash\TT_A,g_{\Gamma;A;E'})$.  (We will use the same notation for the extended Jacobi field.)
\end{itemize}
We first claim that 
\begin{align}\label{eqn:middle.cover.lyap.const}
f_{(\Gamma\backslash\TT_A,g_{\Gamma;A})}(Y,t)=f_{(\Gamma\cap\Gamma')\backslash\TT_A, g_{\Gamma;A;E'})}((\Pi_{\Gamma\cap\Gamma'})_*\tY,t)\equiv 0,~\forall t\in I'.
\end{align}
For any $t\in I'$, by Lemma \ref{lem:tube.index.properties-1-4}, Definition \ref{dfn:geo lift} and \eqref{eqn:non.ints}, we have $\tc(t)\in\TT_A\setminus\left(\bigcup_{j\in E'}\overline{\cT_{j;A}(\pi/10)}\right).$
By Lemma \ref{lem:properties.of.g.Gamma-1} and Lemma \ref{lem:properties.of.g.Gamma-5}, there exists a point $q\in\TT_A$ and a good neighborhood $\cU_q$ of $q$ such that $\tc(t)\in\cU_q\subset \TT_A\setminus\left(\bigcup_{j\in E'}\overline{\cT_{j;A}(\pi/10)}\right)$ and that $(\Pi_\Gamma(\cU_q),g_{\Gamma;A}|_{\Pi_\Gamma(\cU_q)})$ is nearly nonpositively curved.
Observe the following:
\begin{itemize}
\item $g_{\Gamma;A;E'}=(\Pi^{\Gamma\cap\Gamma'}_\Gamma)^*g_{\Gamma;A}$ when restricted to $\Pi_{\Gamma\cap\Gamma'}(\TT_A\setminus\bigcup_{j\in E'}\overline{\cT_{j,A}(\pi/10)})$. (This is due to Lemma \ref{lem:properties.of.g.Gamma-3}.)
\item For any nearly nonpositively curved manifolds $(M,g)$, for any open subset $U$ of $M$, for any geodesic $c_M:I_M\to (U,g|_U)$, for any Jacobi field $Y_M$ along $c_M$ and for any $t\in I_M\subset c_M^{-1}(U)$, if $(U,g|_U)$ is nearly nonpositively curved, then $f_{(M,g)}(Y_M,t)=f_{(U,g|_U)}(Y_M,t)$. (This is due to \eqref{eqn:infi Lyap2}, \eqref{eqn:noninc Lyap fcn2} and \cite[Corollary 5.3]{BG94}. See also \cite[page 18, lines 14-16]{BG94}.)
\item For any $p\in\TT_A$ and any good neighborhood $\cU_p$ of $p$. If $\cU_p\subset \TT_A\setminus \left(\bigcup_{j\in E'}\cT_{j;A}(\pi/10)\right)$, then $\Pi^{\Gamma\cap\Gamma'}_{\Gamma}|_{\Pi_{\Gamma\cap\Gamma'}(\cU_p)}$ is a diffeomorphism onto its image. (To see this, for any $q_1,q_2\in\cU_p$, if $\Pi_\Gamma(q_1)=\Pi_\Gamma(q_2)$, by Definition \ref{dfn:good.nbhd}, there exists $\gamma\in\text{Stab}_\Gamma(p)$ such that $\gamma(q_1)=q_2$. Observe that $E'\cap\Sing_A(p)=\emptyset$ due to Definition \ref{sing index} and \eqref{eqn:j-tubes}. Write $\Gamma=\Gamma_A(E)$. Recall that $\Gamma'=\Gamma_A(E')$. By Lemma \ref{lem:properties.good.subgp-(4)} and the remark after Definition \ref{sing index}, we have 
$$\Sing_{\Gamma\cap\Gamma';A}(p)=\Sing_A(p)\setminus(E\cup E')=\Sing_A(p)\setminus E=\Sing_{\Gamma;A}(p).$$ 
Hence, by the remark after Definition \ref{sing index} again, we have $\text{Stab}_\Gamma(p)=\text{Stab}_{\Gamma\cap\Gamma'}(p)$. This implies that $\Pi_{\Gamma\cap\Gamma'}(q_1)=\Pi_{\Gamma\cap\Gamma'}(q_2)$. Notice that by Lemma \ref{lem:properties.of.g.Gamma-3}, $\Pi^{\Gamma\cap\Gamma'}_{\Gamma}$ is everywhere nondegenerate when restricted to $\Pi_{\Gamma\cap\Gamma'}(\TT_A\setminus\bigcup_{j\in E'}\cT_{j,A}(\pi/10))$. Therefore, $\Pi^{\Gamma\cap\Gamma'}_{\Gamma}|_{\Pi_{\Gamma\cap\Gamma'}(\cU_p)}$ is a diffeomorphism onto its image.)
\end{itemize}
Hence, by Lemma \ref{lem:properties.of.g.Gamma-3} and the three bullet points in the above, we have 
\begin{align*}
f_{(\Gamma\backslash\TT_A,g_{\Gamma;A})}(Y,t)=&f_{(\Pi_\Gamma(\cU_q),g_{\Gamma;A}|_{\Pi_\Gamma(\cU_q)})}(Y,t)\\
=&f_{(\Pi_{\Gamma\cap\Gamma'}(\cU_q), g_{\Gamma;A;E'}|_{\Pi_{\Gamma\cap\Gamma'}(\cU_q)})}((\Pi_{\Gamma\cap\Gamma'})_*\tY,t)=f_{((\Gamma\cap\Gamma')\backslash\TT_A, g_{\Gamma;A;E'})}((\Pi_{\Gamma\cap\Gamma'})_*\tY,t)
\end{align*}
\eqref{eqn:middle.cover.lyap.const} then follows from the above and the assumption that $f_{(\Gamma\backslash\TT_A,g_{\Gamma;A})}(Y,t)\equiv 0$.

Notice that $\{i,i+1\}$ is an atom in the consecutive partition induced by $E'$. (See the remark after Definition \ref{dfn:consecutive}.) Therefore, by Lemma \ref{lem:properties.of.g.Gamma-2}, the following holds: 
\begin{itemize}
\item $\Pi_{\Gamma\cap\Gamma_{i,i+1}}\circ \pr^A_{i,i+1}\circ\tc|_{I'}$ is a nonconstant geodesic in $(\Gamma\cap\Gamma_{i,i+1}\backslash\TT_{i,i+1},g_{\Gamma\cap\Gamma_{i,i+1};\{i,i+1\}})$. (Recall that $I\subset I'$. The fact that this geodesic is nonconstant follows from the definition of $I$, Lemma \ref{lem:rank-1.not.in.sing-1} and the assumption that $c$ is nomadic in the sense of Definition \ref{dfn:nomadic}.)
\item $(\Pi_{\Gamma\cap\Gamma_{i,i+1}}\circ \pr^A_{i,i+1})_*\tY|_{I'\setminus\Sigma}$ extends to a smooth Jacobi field along $\Pi_{\Gamma\cap\Gamma_{i,i+1}}\circ\pr^A_{i,i+1}\circ\tc|_{I'}$ in $(\Gamma\cap\Gamma_{i,i+1}\backslash\TT_{i,i+1},g_{\Gamma\cap\Gamma_{i,i+1};\{i,i+1\}})$. (We use the same notation for the extended Jacobi field.)
\end{itemize}
\begin{figure}[h]
	\centering
	\includegraphics[width=4.5in]{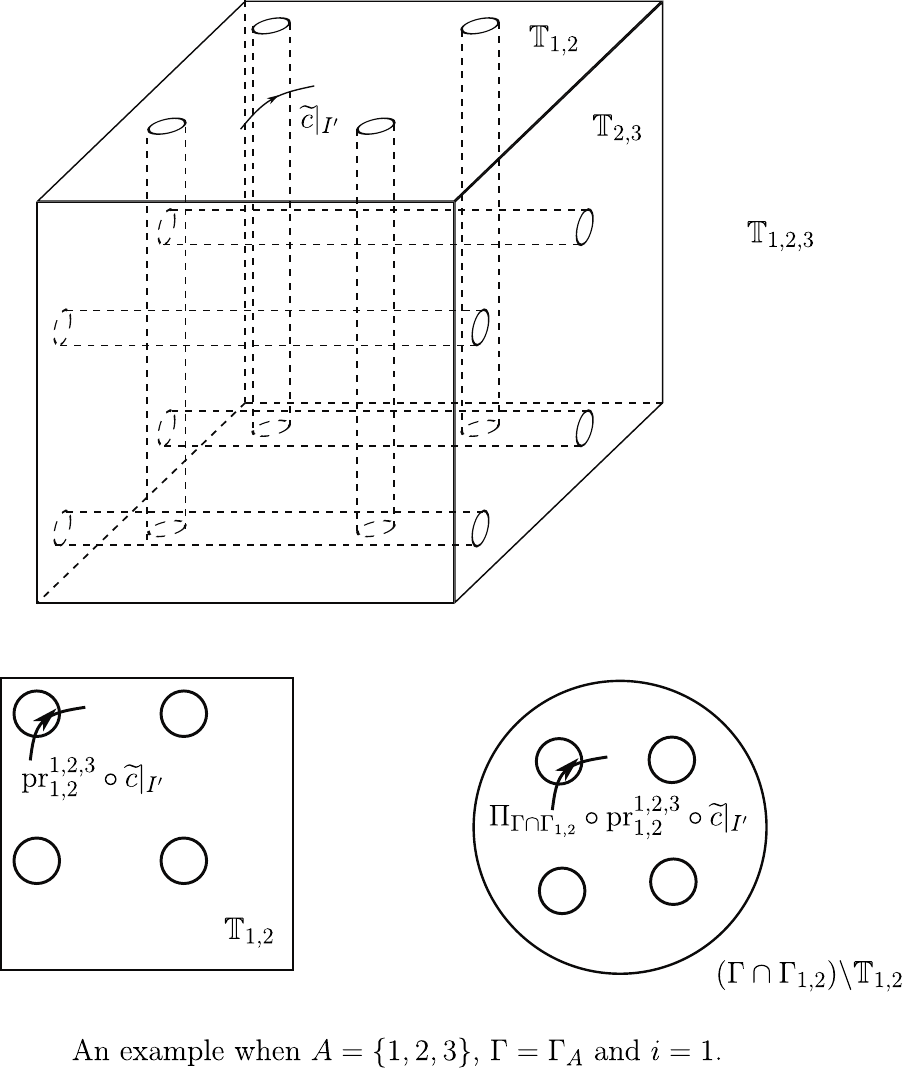}
	\caption{}\label{fig:tube geo}
\end{figure}
Moreover, by Lemma \ref{lem:rank-1 lem 3}, Lemma \ref{lem:properties.of.g.Gamma-2} and \eqref{eqn:middle.cover.lyap.const}, we have
\begin{align}\label{eqn:factor.middle.cover.lyap.const}
f_{(\Gamma\cap\Gamma_{\{i,i+1\}}\backslash\TT_{i,i+1},g_{\Gamma\cap\Gamma_{\{i,i+1\}};\{i,i+1\}})}((\Pi_{\Gamma\cap\Gamma_{\{i,i+1\}}}\circ\pr^A_{i,i+1})_*\tY,t)\equiv 0,~\forall t\in I'.
\end{align}
Recall the following facts:
\begin{itemize}
\item $(\Gamma\cap\Gamma_{i,i+1}\backslash\TT_{i,i+1},g_{\Gamma\cap\Gamma_{i,i+1};\{i,i+1\}})$ is a nearly nonpositively curved surface. (This is due to Lemma \ref{lem:properties.of.g.Gamma-1}.)
\item By Definition \ref{dfn:construction}, Definition \ref{dfn:construction.of.g}, and the definition of $I'$ and $I$, the geodesic segment $\Pi_{\Gamma\cap\Gamma_{i,i+1}}\circ \pr^A_{i,i+1}\circ\tc|_{I'}$ passes through a point with negative curvature. (See also Definition \ref{dfn:caps} and Definition \ref{dfn:near nonpos curv}. See Figure \ref{fig:tube geo} for a picture of $\Pi_{\Gamma\cap\Gamma_{i,i+1}}\circ \pr^A_{i,i+1}\circ\tc|_{I'}$.)
\end{itemize}
Therefore, by Definition \ref{dfn:geo rank}, Lemma \ref{lem:rank-1 lem 2}, \eqref{eqn:factor.middle.cover.lyap.const} and the fact that $\Pi_{\Gamma\cap\Gamma_{i,i+1}}\circ \pr^A_{i,i+1}\circ\tc|_{I'}$ is nonconstant, there exists a unique constant $\alpha_{i,i+1}\in\RR$ such that 
$$(\Pi_{\Gamma\cap\Gamma_{\{i,i+1\}}}\circ\pr^A_{i,i+1})_*\tY(t)=\alpha_{i,i+1}\cdot(\Pi_{\Gamma\cap\Gamma_{i,i+1}}\circ \pr^A_{i,i+1}\circ\tc)'(t)\text{ for any }t\in I'.$$
In particular, by Lemma \ref{lem:branching.locus-1} and the definition of $\Sigma$, we have $\tY_i(t)=\alpha_{i,i+1}\cdot\tc_i'(t)$ and $\tY_{i+1}(t)=\alpha_{i,i+1}\cdot\tc_{i+1}'(t)$ for any $t\in I'\setminus\Sigma$.
\item The proof of this assertion follows from the same argument as in the proof of Lemma \ref{sublem:nomadic=rank-1-1}. The only difference is that we use $I$ instead of $I'$ and use $E'=\{i-1,i\}\cap\cE(A)$ instead of $\{i-1,i+1\}\cap \cE(A)$. In particular, $\{i\}$ is an atom in the induced consecutive partition of $A$ by $E'$. (See the remark after Definition \ref{dfn:consecutive}.) This implies that $\tc_i|_I$ is a geodesic in the circle $\TT_i$, and that $\tY_i|_{I\setminus\Sigma}$ extends to a smooth Jacobi field along $\tc_i|_I$ with $f_{\TT_i}(\tY_i,t)\equiv 0$ for any $t\in I$. Hence both $\tc_i$ and $\tY_i$ are constant along $I$. This finishes the proof.
\item Without loss of generality, we let $I_1=(a_1,b_1)$ and $I_2=(a_2,b_2)$ with $a_1<b_1<a_2<b_2$. Let $I=(a_1-\epsilon,b_2+\epsilon)$ for some $\epsilon>0$ such that $(a_1-\epsilon,b_1+\epsilon)$ and $(a_2-\epsilon,b_2+\epsilon)$ are subintervals of the open neighborhoods of $\overline{I_1}$ and $\overline{I_2}$ described in Lemma \ref{sublem:nomadic=rank-1-1} respectively. By Lemma \ref{sublem:nomadic=rank-1-1} applied to $I_1,I_2$, Lemma \ref{sublem:nomadic=rank-1-2} applied to $(b_1,a_2)$ and the assumption that $I_1,I_2$ are $i$-adjacent, there exist unique constants $\alpha(I_1),\alpha(I_2)\in\RR$ such that 
\begin{align}\label{eqn:half.const.bridge-1}
\begin{cases}
\tY_i(t)=\alpha(I_1)\cdot\tc_i'(t)\text{ for any }t\in (a_1-\epsilon,a_2)\setminus\Sigma,\\
(\tY_{i_1}(t),\tY_{i_1+1}(t))=\alpha(I_1)\cdot(\tc_{i_1}'(t),\tc_{i_1+1}'(t))\text{ for any }t\in (a_1-\epsilon,b_1+\epsilon)\setminus \Sigma,\\
\end{cases}
\end{align}
and
\begin{align}\label{eqn:half.const.bridge-2}
\begin{cases}
\tY_i(t)=\alpha(I_2)\cdot\tc_i'(t)\text{ for any }t\in (b_1,b_2+\epsilon)\setminus\Sigma,\\
(\tY_{i_2}(t),\tY_{i_2+1}(t))=\alpha(I_2)\cdot(\tc_{i_2}'(t),\tc_{i_2+1}'(t))\text{ for any }t\in (a_2-\epsilon,b_2+\epsilon)\setminus \Sigma.\\
\end{cases}
\end{align}
It suffices to show that $\alpha(I_1)=\alpha(I_2)$. We split the proof into two separate cases:

\textbf{Case 1}: Assume that $I_1$ and $I_2$ are completely adjacent in the sense of Definition \ref{subdfn:adjacent}. Then \eqref{eqn:half.const.bridge-1} and \eqref{eqn:half.const.bridge-2} hold with $i$ replaced by any $j\in\{i_1,i_1+1\}\cap\{i_2,i_2+1\}$.   
Observe that there exist $t\in(b_1,a_2)\setminus \Sigma$ and $j\in\{i_1,i_1+1\}\cap\{i_2,i_2+1\}$ such that $\tc_j'(t)\not\equiv 0$. (When $i_1=i_2$, this is due to Lemma \ref{lem:rank-1.not.in.sing-1}, the assumption that $c$ is nomadic and the assumptions on $I_1,I_2$. When $i_1\neq i_2$, this is due to \eqref{jth singularity}, \eqref{eqn:j-tubes} and the assumptions on $I_1,I_2$.) Hence, $\alpha(I_1)=\alpha(I_2)$ follows directly from \eqref{eqn:half.const.bridge-1} (with $i$ replaced by $j$) and \eqref{eqn:half.const.bridge-2} (with $i$ replaced by $j$).

\textbf{Case 2}: If $I_1$ and $I_2$ are not completely adjacent in the sense of Definition \ref{subdfn:adjacent}, then $i_1=i_2$. Moreover, there exists a finite collection of pairwise disjoint open intervals $J_1:=I_1,J_2,\dots,J_{m-1},J_m:=I_2$ such that the following holds: 
\begin{itemize}
\item For any $j\in\{1,\dots,m\}$, $J_j\in\cI_{l_j}$ for some $l_j\in\cE(A)$. (See the definition of $\cI_l$ before Definition \ref{subdfn:adjacent}.)
\item For any $j\in\{1,\dots,m-1\}$, $(\inf J_{j+1})-(\sup J_j)$ is uniformly bounded from below by a positive number. (This follows from Lemma \ref{lem:pass.thru.not.crazy} and Lemma \ref{lem:tube.index.properties-1-5}.)
\item For any $j\in\{1,\dots,m-1\}$, $J_j$ and $J_{j+1}$ are completely adjacent in the sense of Definition \ref{subdfn:adjacent}.
\end{itemize}
Hence, $\alpha(I_1)=\alpha(I_2)$ follows directly from \textbf{Case 1}.

We note that the above collection of pairwise disjoint intervals $\{J_1,\dots,J_m\}$ can be constructed in the following recursive way: Start with $\{I_1,I_2\}$. For any two distinct intervals $I_1'=(a_1',b_1'),I_2'=(a_2',b_2')$ in this collection with $a_1'<b_1'<a_2'<b_2'$, if $I_1',I_2'$ are $j$-adjacent but not completely adjacent, there exists some $i\in\{j,j-1\}\cap\cE(A)$ such that $I_1',I_2'\in\cI_{i}$. Let $j'$ be the unique element in $\{i,i+1\}\setminus\{j\}$. Then $I_1',I_2'$ are not $j'$-adjacent. Extend the existing collection of intervals by adding all intervals in $\cI_{j'}$ which intersects $[b_1',a_2']$. It is easy to see that the first bullet point holds for this extended collection. Lemma \ref{lem:pass.thru.not.crazy} and Lemma \ref{lem:tube.index.properties-1-5} guarantee that the second bullet point holds for the extended collection. Repeat the above constructions, which will end in finitely many steps due to the second bullet point. The final collection of intervals must satisfy the third bullet point because otherwise there is room to further extend this collection using the above recursive constructions.
\item This follows directly from Lemma \ref{sublem:nomadic=rank-1-2}.\qedhere
\end{enumerate}
\end{proof}
Back to the proof of Proposition \ref{prop:nomadic.is.rank-1}. 
By Lemma \ref{lem:pass.thru.not.crazy}, there are at most finitely many intervals in $\bigcup_{j\in \cE(A)\cap\{i,i-1\}}\cI_j$ which intersects any given bounded interval. Moreover, by Definition \ref{subdfn:adjacent} and the definition of $\cI_i$ before \eqref{eqn:finite.and.nonempty}, for any distinct $I_1=(a,b),I_2=(a',b')\in \bigcup_{j\in \cE(A)\cap\{i,i-1\}}\cI_j$ with $a<b<a'<b'$, $I_1,I_2$ are $i$-adjacent if and only if $I\cap[b,a']=\emptyset$ for any $I\in \bigcup_{j\in \cE(A)\cap\{i,i-1\}}\cI_j$. Therefore, by \eqref{eqn:finite.and.nonempty} and Lemma \ref{sublem:nomadic=rank-1}, there exists a unique $\alpha_i\in\RR$ such that the following holds:
\begin{itemize}
\item $\tY_i(t)=\alpha_i\cdot\tc_i'(t)$ for any $t\in\RR\setminus\Sigma$.
\item For any $j\in \cE(A)\cap\{i,i-1\}$ and for any $I\in \cI_j$, there exists an open neighborhood $I'$ of $\overline{I}$ such that $\tY_j(t)=\alpha_{i}\cdot\tc_j'(t)$ and $\tY_{j+1}(t)=\alpha_{i}\cdot\tc_{j+1}'(t)$ for any $t\in I'\setminus\Sigma$.
\end{itemize}
Then, by \eqref{eqn:finite.and.nonempty} and Lemma \ref{sublem:nomadic=rank-1-1}, for any $i,j\in A$, we have $\alpha_i=\alpha_j$. This proves \eqref{eqn:nomadic.is.rank.1.goal} and hence finishes the proof of Proposition \ref{prop:nomadic.is.rank-1}.
\end{proof}

\begin{cor}\label{cor:nomadic=rank-1}
A bi-infinite geodesic in $(\Gamma\backslash\TT_A, g_{\Gamma;A})$ has rank one if and only if it is nomadic.
\end{cor}
\begin{proof}
By Proposition \ref{prop:nomadic.is.rank-1}, it suffices to show that all rank one, bi-infinite geodesics are nomadic. If not, then there exist $i\in\cE(A)$ and a rank one, bi-infinite geodesic $c:\RR\to (\Gamma\backslash\TT_A, g_{\Gamma;A})$ such that $c$ does not pass through a $\{i,i+1\}$-special tube. By Definition \ref{dfn:nomadic}, either of the following cases holds:
\begin{itemize}
\item $c(\RR)\cap\Pi_\Gamma(\cT_{i;A}(\pi/10))=\emptyset$.
\item There exists an infinite interval $I$ such that $c(I)\subset \Pi_\Gamma(\cT_{i;A}(\pi/10))$.
\end{itemize}
By Lemma \ref{lem:missing.tube=higher.rank} and Lemma \ref{lem:rank-1.not.in.sing-2}, either of the above cases contradicts the assumption that $c$ has rank one. This finishes the proof.
\end{proof}

\subsection{Construction of $\cN_{\Gamma;A}$}\label{subsect:normal vectors}


\begin{dfn}[$i$-th angle]\label{dfn:j.angle}
Let $A$ be a consecutive subset of $\ZZ_+$ with $|A|\geq 2$. For any $v\in T\TT_A$ and for any $i\in\cE(A)$, if $(\pr^A_{i,i+1})_*(v)\neq 0$, we define $\beta_{i}(v)\in[0,\pi]/\{0\sim \pi\}$ such that $(\pr^A_{i,i+1})_*(v)=(r\cos(\beta_{i}(v)),r\sin(\beta_{i}(v)))$ for some $r\in\RR\setminus\{0\}$.
\end{dfn}
\begin{dfn}[Exceptional angle]\label{dfn:exceptional value}
Let $A$ be a consecutive subset of $\ZZ_+$. For any $i\in\cE(A)$, we say that $\beta\in[0,\pi]/\{0\sim\pi\}$ is a \emph{$i$-exceptional angle} if under the standard torus metric, there is a vector $v\in T\TT_A$ with $(\pr^A_{i,i+1})_*(v)\neq 0$ and $\beta_i(v)=\beta$, such that the bi-infinite geodesic with initial vector $v$ never visits $\cT_{i;A}(\pi/10)$. (See \eqref{eqn:j-tubes}.)

We say that a vector $v\in T\TT_A$ has \emph{no exceptional angles} if for any $i\in\cE(A)$, $\beta_i(v)$ is well-defined and is not $i$-exceptional.

For any $p\in\TT_A$, we denote by $T^{\text{NE}}_p\TT_A$ the collection of vectors with footpoint $p$ which have no exceptional angles. For any subset $U$ of $\TT_A$, we let $T^{\text{NE}}U:=\bigcup_{p\in U}T^{\text{NE}}_p\TT_A$.
\end{dfn}
\begin{rmk}
One can check that for any $i\in\cE(A)$, $0\sim\pi$ and $\pi/2$ are $i$-exceptional angles. Also, there are only finitely many $i$-exceptional angles.
\end{rmk}
\begin{dfn}\label{dfn:normal.vectors}
Let $A$ be a consecutive subset of $\ZZ_+$. If $|A|=1$, we define $\widetilde{\cN}_A$ to be the collection of all non-zero vectors in $T\TT_A$.

If $|A|\geq 2$, we define $\widetilde{\cN}_A$ to be a subset of $T\TT_A$ such that for any $v\in T\TT_A$, $v\in\widetilde{\cN}_A$ if and only if the following holds:
\begin{enumerate}
\item The footpoint of $v$ is in $\TT_A\setminus\Omega_{\Gamma_A;A}$.
\item For any $i\in\cE(A)$, $(\pr^A_{i,i+1})_*(v)\neq 0$. Moreover, $\beta_i(v)$ is not a $i$-exceptional angle, or there exists $j\in\{i-1,i,i+1\}\cap\cE(A)$ such that the following holds:
\begin{itemize}
\item The footpoint of $v$ is in $\cT_{j;A}(\pi/10)$.
\item If $j\neq i$, then for any $l\in\{i,i+1\}\setminus\{j,j+1\}$, we have $(\pr^A_l)_*(v)\neq 0$.
\end{itemize}
\end{enumerate}
\end{dfn}
\begin{rmk}
One can easily verify that $\widetilde{\cN}_A$ satisfies the following properties:
\begin{enumerate}
\item[\hypertarget{C1}{(C1)}] For any $v\in T^{\text{NE}}(\TT_A\setminus\Omega_{\Gamma_A;A})$ and for any $w\in\widetilde{\cN}_A$, if $(\pr^A_i)_*(v)=(\pr^A_i)_*(w)$ for any $i\in A$, then $v$ and $w$ are connected by a path in $\widetilde{\cN}_A$. (This uses the fact that $T^{\text{NE}}(\TT_A\setminus\Omega_{\Gamma_A;A})$ is invariant under parallel transport and the fact that $\TT_A\setminus\Omega_{\Gamma_A;A}$ is path-connected. See Definition \ref{dfn:j.angle}-Definition \ref{dfn:normal.vectors}, \eqref{jth singularity} and \eqref{total singularity subgp}.)
\item[\hypertarget{C2}{(C2)}] For any $i\in\cE(A)$, for any $q\in\cT_{i;A}(\pi/10)\setminus\Omega_{\Gamma_A;A}$ and for any vector $v\in T_q\TT_A\cap\widetilde{\cN}_A$, if $(\pr^A_{j})_*(v)\neq 0$ for any $j\in\{i-1,i+2\}\cap A$, then we have
$$\{w\in T_q\TT_A:(\pr^A_{i,i+1})_*(w)\neq 0\text{ and }(\pr^A_{j})_*(w)=(\pr^A_{j})_*(v)~\forall j\in A\setminus\{i,i+1\}\}\subset\widetilde{\cN}_A.$$

\end{enumerate}
\end{rmk}
\begin{prop}\label{prop:normal.vectors.open.connected}
Let $A$ be a consecutive subset of $\ZZ_+$ with $|A|\geq 2$. Then $\widetilde{\cN}_A$ is a $\Gamma_A$-invariant, open and connected subset of $T\TT_A$. Moreover, $T\TT_A\setminus\widetilde{\cN}_A$ has Lebesgue measure zero.
\end{prop}
\begin{proof}
Openness of $\widetilde{\cN}_A$ follows directly from Definition \ref{dfn:normal.vectors} and the remark after Definition \ref{dfn:exceptional value}. Observe the following:
\begin{itemize}
\item For any $i\in A$ and for any $v\in T\TT_A$, if $(\pr^A_{i})_*(v)\neq 0$, then $(\pr^A_{i})_*(\gamma_*(v))\neq 0$ for any $\gamma\in\Gamma_A$. (This follows from \eqref{eqn:generator} and \eqref{eqn:Gamma.A}.)
\item For any $i\in\cE(A)$ and for any $v\in T\TT_A\setminus \Omega_{\Gamma_A; A}$ with $(\pr^A_{i,i+1})_*v\neq 0$, if $\beta_i(v)$ is not $i$-exceptional, $\beta_i(\gamma_*(v))$ is also not $i$-exceptional for any $\gamma\in\Gamma_A$. (This follows from Definition \ref{dfn:exceptional value} and Lemma \ref{lem:tube.index.properties-1-1}.)
\item For any $i\in\cE(A)$ and for any $v\in T\TT_A$, if the footpoint of $v$ is in $\cT_{i;A}(\pi/10)$, then the footpoint of $\gamma_*(v)$ is also in $\cT_{i;A}(\pi/10)$ for any $\gamma\in\Gamma_A$. (This follows from Lemma \ref{lem:tube.index.properties-1-1}.)
\end{itemize} 
Hence, by Definition \ref{dfn:normal.vectors} and the above, $\widetilde{\cN}_A$ is a $\Gamma_A$-invariant. 

Now, we prove that $T\TT_A\setminus\widetilde{\cN}_A$ has Lebesgue measure zero in $T\TT_A$. Notice that by Definition \ref{dfn:normal.vectors}, we have 
\begin{align*}
\widetilde{N}_A\supset T^{\text{NE}}(\TT_A\setminus\Omega_{\Gamma_A;A})= T^{\text{NE}}\TT_A\setminus \left(\bigcup_{p\in\Omega_{\Gamma_A;A}}T_p\TT_A\right) .
\end{align*} 
Therefore, it suffices to show that the following subset 
\begin{align}\label{eqn:comp.zero.msr}
T\TT_A\setminus(T^{\text{NE}}(\TT_A\setminus\Omega_{\Gamma_A;A}))=(T\TT_A\setminus T^{\text{NE}}\TT_A)\cup\bigcup_{p\in\Omega_{\Gamma_A;A}}T_p\TT_A.
\end{align}
has Lebesgue measure zero in $T\TT_A$. This is an immediate corollary of \eqref{jth singularity}, \eqref{total singularity subgp} and the remark after Definition \ref{dfn:exceptional value}.

It remains to show that $\widetilde{\cN}_A$ is connected. We first make the following claim:
\begin{enumerate}
\item[\hypertarget{C3}{(C3)}] For any $v,w\in T^{\text{NE}}(\TT_A\setminus\Omega_{\Gamma;A})$, $v$ and $w$ are connected by a path in $\widetilde{\cN}_A$. 
\end{enumerate}
Assume that \hyperlink{C3}{(C3)} holds, if $\widetilde{\cN}_A$ is not connected, there exist two disjoint open subsets $U_1,U_2\subset T\TT_A$ such that $U_1\sqcup U_2=\widetilde{\cN}_A$. (Here, we also used the fact that $\widetilde{\cN}_A$ is open, which is proved in the beginning.) By the fact that \eqref{eqn:comp.zero.msr} has zero Lebesgue measure in $T\TT_A$, there exist $v_1,v_2\in T^{\text{NE}}(\TT_A\setminus\Omega_{\Gamma_A;A})$ such that $v_i\in U_i$ for any $i\in\{1,2\}$. However, this contradicts with \hyperlink{C3}{(C3)}. Therefore, $\widetilde{\cN}_A$ is connected. Hence, it remains to prove \hyperlink{C3}{(C3)}.

\textbf{Proof of \hyperlink{C3}{(C3)}}: The proof of \hyperlink{C3}{(C3)} uses induction on the following quantity:
\begin{align}\label{eqn:vec.comp.difference}
D(v,w):=\left|\left\{j\in A:(\pr^A_j)_*(v)\neq(\pr^A_j)_*(w)\right\}\right|\in\{0,1,\dots,|A|\},\text{ for any }v,w\in T\TT_A.
\end{align}
When $D(v,w)=0$, \hyperlink{C3}{(C3)} follows from \hyperlink{C1}{(C1)}. Assume that $v,w$ can be connected by a path in $\widetilde{\cN}_A$ whenever $D(v,w)\leq k-1$ for some $k\in\{1,\dots,|A|\}$, consider the case when $D(v,w)=k$. By Definition \ref{dfn:exceptional value} and the assumption on $v,w$ in \hyperlink{C3}{(C3)}, we have 
\begin{align}\label{eqn:NE.implies.no.zero}
(\pr^A_j)_*(v),(\pr^A_j)_*(w)\in\RR\setminus\{0\},\text{ for any }j\in A.
\end{align}
Let $i$ be the smallest element in $A$ such that $(\pr^A_{i})_*(v)\neq (\pr^A_{i})_*(w)$. When $i\not\in\cE(A)$, $i$ is the largest element in $A$ and $D(v,w)=1$. Choose $p\in\cT_{i-1;A}(\pi/10)\setminus\Omega_{\Gamma_A;A}$. By \hyperlink{C1}{(C1)}, we assume without loss of generality that $v,w\in T_p\TT_A$. It follows from  \hyperlink{C2}{(C2)} and \eqref{eqn:NE.implies.no.zero} that $v$ and $w$ are connected by a path in $\widetilde{\cN}_A$.

When $i\in\cE(A)$, we choose $p\in\cT_{i;A}(\pi/10)\setminus\Omega_{\Gamma_A;A}$. By \hyperlink{C1}{(C1)}, we assume without loss of generality that $v,w\in T_p\TT_A$.
Choose $w'\in T_p\TT_A$ such that the following holds:
\begin{itemize}
\item $(\pr^A_{i})_*(w')=(\pr^A_{i})_*(v)$.
\item If $(\pr^A_{i+1})_*(w)=(\pr^A_{i+1})_*(v)$, then $(\pr^A_{i+1})_*(w')=(\pr^A_{i+1})_*(w)=(\pr^A_{i+1})_*(v)$.
\item $(\pr^A_{A\setminus\{i,i+1\}})_*(w')=(\pr^A_{A\setminus\{i,i+1\}})_*(w).$
\item If $(\pr^A_{i+1})_*(w)\neq (\pr^A_{i+1})_*(v)$, then $(\pr^A_{i+1})_*(w')$ is chosen so that $\beta_j(w')$ is not $j$-exceptional for any $j\in\{i,i+1\}\cap \cE(A)$. (Such a choice is guaranteed by the remark after Definition \ref{dfn:exceptional value} and \eqref{eqn:NE.implies.no.zero}.)
\end{itemize}
Then, by the assumption that $v,w\in T_p^{\text{NE}}\TT_A$, the assumption on $i$ and the above, we have $w'\in T_p^{\text{NE}}\TT_A$. Apply \hyperlink{C2}{(C2)} and \eqref{eqn:NE.implies.no.zero}, $w$ and $w'$ are connected by a path in $\widetilde{\cN}_A$.
Notice that $D(v,w')\leq D(v,w)-1$. By the inductive step, $w'$ and $v$ are connected by a path in $\widetilde{\cN}_A$. This proves \hyperlink{C3}{(C3)}.
\end{proof}
\begin{dfn}\label{dfn:projected.normal.vector}
Let $A$ be a consecutive subset of $\ZZ_+$. For any good subgroup $\Gamma\subset\Gamma_A$, we define $\cN_{\Gamma;A}:=(\Pi_\Gamma)_*\widetilde{\cN}_A\subset T(\Gamma\backslash\TT_A)$.
\end{dfn}
\begin{lem}\label{lem:basic.properties.normal.vector}
Let $A$ be a consecutive subset of $\ZZ_+$. For any good subgroup $\Gamma\subset\Gamma_A$, the following holds:
\begin{enumerate}[ref=\ref{lem:basic.properties.normal.vector}(\arabic*)]
\item $\cN_{\Gamma;A}$ is an open subset of $T(\Gamma\backslash\TT_A)$ whose complement has zero Lebesgue measure. Moreover, when $|A|\geq 2$, $\cN_{\Gamma;A}$ is connected. \label{lem:basic.properties.normal.vector-1}
\item For any $v\in\cN_{\Gamma;A}$ and for any $\widetilde{v}\in T\TT_A$ such that $(\Pi_\Gamma)_*(\widetilde{v})=v$, we have $v\neq 0$ and $((\Pi_\Gamma)_*)^{-1}(v)=((\Pi_\Gamma)_*)^{-1}(v)\cap\widetilde{\cN}_{\Gamma;A}=\Gamma\cdot\widetilde{v}$. \label{lem:basic.properties.normal.vector-2}
\item When $|A|\geq 2$, for any $i\in\cE(A)$, we define
\begin{align}\label{eqn:j.higher.rk.normal.vector}
\widetilde{\cN}_{\Gamma;A}(i):=\left\{\widetilde{v}\in \widetilde{\cN}_A:(\Pi_\Gamma)_*(\widetilde{v})\in\cQ_{\Gamma;A}(i)\right\},
\end{align}
where
\begin{align}\label{eqn:vectors.avoiding.j.tube}
\cQ_{\Gamma;A}(i):=\left\{v\in T(\Gamma\backslash\TT_A)\left|\begin{aligned}
&v\text{ is tangent to a bi-infinite geodesic } c\text{ in}\\
&(\Gamma\backslash\TT_A,g_{\Gamma;A})\text{ such that }c^{-1}(\Pi_\Gamma(\cT_{i;A}(\pi/10))=\emptyset
\end{aligned}\right.\right\}.
\end{align}
Then $\widetilde{\cN}_{\Gamma;A}(i)$ is $\Gamma$-invariant. Moreover, we have $\widetilde{\cN}_{\Gamma;A}(i)\subset \widetilde{\cN}_{\{l\in A:j\leq i\}}\times\widetilde{\cN}_{\{l\in A:j>i\}}$.\label{lem:basic.properties.normal.vector-3}
\end{enumerate}
\end{lem}
\begin{proof}
\begin{enumerate}
\item When $|A|=1$, $\Gamma=\{\id\}$ and $\Pi_\Gamma=\id|_{\TT_A}$. Hence, this assertion follows directly from Definition \ref{dfn:normal.vectors}.

When $|A|\geq 2$, observe that the footpoint of any vector in $\widetilde{\cN}_A$ is not in $\Omega_{\Gamma_A;A}$. (See Definition \ref{dfn:normal.vectors}.) Notice that $\Omega_{\Gamma;A}\subset\Omega_{\Gamma_A;A}$. (See \eqref{eqn:Gamma.A} and \eqref{total singularity subgp}.) Therefore, by Lemma \ref{lem:branching.locus-1}, $(\Pi_\Gamma)_*$ is a local diffeomorphism when restricted to $\widetilde{\cN}_A$. The openness of $\cN_{\Gamma;A}$ is then a corollary of the openness of $\widetilde{\cN}_A$ (proved in Proposition \ref{prop:normal.vectors.open.connected}) and the above.

By Lemma \ref{lem:branching.locus-1} again, $(\Pi_\Gamma)_*$ is a local diffeomorphism when restricted to $T(\TT_A\setminus\Omega_{\Gamma_A;A})$.  Notice that the complement of $(\Pi_{\Gamma})_*T(\TT_A\setminus\Omega_{\Gamma_A;A})$ in $T(\Gamma\backslash\TT_A)$ has zero Lebesgue measure. By Proposition \ref{prop:normal.vectors.open.connected} and the fact that $\widetilde{\cN}_A\subset T(\TT_A\setminus\Omega_{\Gamma_A;A})$, $\cN_{\Gamma;A}$ is connected and  $T(\Gamma\backslash\TT_A)\setminus\cN_{\Gamma;A}$ has zero Lebesgue measure in $T(\Gamma\backslash\TT_A)$.
\item The fact that $v\neq 0$ follows from Definition \ref{dfn:normal.vectors} and  Lemma \ref{lem:branching.locus-1}. The rest follows from Proposition \ref{prop:normal.vectors.open.connected}. 
\item $\Gamma$-invariance of $\widetilde{\cN}_{\Gamma;A}(i)$ follows directly from \eqref{eqn:j.higher.rk.normal.vector} and Proposition \ref{prop:normal.vectors.open.connected}. Recall that for any $A'\subset A$ and for any $j\in A'$, by \eqref{eqn:tori.natural.prod} and \eqref{eqn:j-tubes}, we have $\cT_{j,A'}(\pi/10)\times \TT_{A\setminus A'}=\cT_{j;A}(\pi/10)$. The fact that $\widetilde{\cN}_{\Gamma;A}(i)\subset \widetilde{\cN}_{\{j\in A:j\leq i\}}\times\widetilde{\cN}_{\{j\in A:j>i\}}$ follows directly from Lemma \ref{lem:tube.index.properties-1}, the remark after Definition \ref{dfn:exceptional value}, Definition \ref{dfn:normal.vectors}, \eqref{eqn:j.higher.rk.normal.vector}, \eqref{eqn:vectors.avoiding.j.tube} and the above. \qedhere
\end{enumerate}
\end{proof}

\subsection{Proof of Theorem \ref{thm:main}}\label{subsect:proof}
Thoughout this subsection, for any Riemannian manifold $M$ and for any $a>0$, we let $T^aM:=\{v\in TM:\|v\|=a\}$, where $\|\cdot\|$ is induced by the Riemannian metric on $M$. 

By Proposition \ref{prop:quotient.is.sphere}, it suffices to prove the following theorem instead:
\begin{thm}\label{thm:main.true}
Let $A$ be a consecutive subset of $\ZZ_+$ with $|A|\geq 2$ and $\Gamma$ be a good subgroup of $\Gamma_A$. We denote by $\cR_{\Gamma;A}$ the collection of rank one vectors in $(\Gamma\backslash\TT_A,g_{\Gamma;A})$. (In particular, $R_{\Gamma;A}$ is open due to Lemma \ref{lem:rank-1 lem 1} and Lemma \ref{lem:properties.of.g.Gamma-1}. This can be alternatively verified via Definition \ref{dfn:nomadic} and Corollary \ref{cor:nomadic=rank-1}.) We also denote by $\varphi_{\Gamma;A}$ the geodesic flow on $T(\Gamma\backslash \TT_A,g_{\Gamma;A})$. Then the following holds:
\begin{enumerate}
\item $\cN_{\Gamma;A}\setminus\cR_{\Gamma;A}$ is non-separating in $\cN_{\Gamma;A}$ and has Lebesgue measure zero in $T(\Gamma\backslash\TT_A)$. In particular, $\cN_{\Gamma;A}\cap\cR_{\Gamma;A}$ is open and connected.
\item For any $a>0$, the geodesic flow $\varphi_{\Gamma;A}$ on $T^a(\Gamma\backslash\TT_A,g_{\Gamma;A})$ is Bernoulli with respect to the Liouville measure. Moreover, $\varphi_{\Gamma;A}$ has the SBAP on the set $\cN_{\Gamma;A}\cap T^a(\Gamma\backslash\TT_A,g_{\Gamma;A})$ in $T^a(\Gamma\backslash\TT_A,g_{\Gamma;A})$.
\end{enumerate}
\end{thm}
\begin{proof}
The main part of the proof is an inductive proof of (1). Before that we first show that (1)$\implies$(2).

\textbf{Proof of }(1)$\implies$(2): Notice that for any $\tau>0$, the map $v\to av$ gives a smooth conjugation between the time-$\tau$ map of the geodesic flow on $T^1(\Gamma\backslash\TT_A,g_{\Gamma;A})$ and the time-$(\tau/a)$ map of the geodesic flow on $T^a(\Gamma\backslash\TT_A,g_{\Gamma;A})$. Also, we recall that $\cN_{\Gamma;A}$ is invariant under scalar multiplication by $\RR_+$. (See Definition \ref{dfn:normal.vectors} and Definition \ref{dfn:projected.normal.vector}.) Therefore, it suffices to prove (2) when $a=1$.

Observe that $\cR_{\Gamma;A}$ and $\cN_{\Gamma;A}$ are invariant under scalar multiplication by $\RR_+$. By (1), Definition \ref{dfn:nonsep sets} and Lemma \ref{lem:basic.properties.normal.vector-1}, $\cN_{\Gamma;A}\cap\cR_{\Gamma;A}\cap T^1(\Gamma\backslash\TT_A,g_{\Gamma;A})$ is a connected open subset and has full Lebesgue measure in $T^1(\Gamma\backslash\TT_A,g_{\Gamma;A})$. Hence, by Theorem \ref{thm:main criterion}, the geodesic flow $\varphi_{\Gamma;A}$ is Bernoulli on $T^1(\Gamma\backslash\TT_A,g_{\Gamma;A})$ and has the SBAP on $\cR_{\Gamma;A}\cap T^1(\Gamma\backslash\TT_A,g_{\Gamma;A})$.

By (1), Fact \ref{fact:SBAP-4} and the fact that $\cR_{\Gamma;A}$ and $\cN_{\Gamma;A}$ are open and invariant under scalar multiplication by $\RR_+$, $(\cN_{\Gamma;A}\setminus\cR_{\Gamma;A})\cap T^1(\Gamma\backslash\TT_A,g_{\Gamma;A})$ is nonseparating in $(\cN_{\Gamma;A}\cup\cR_{\Gamma;A})\cap T^1(\Gamma\backslash\TT_A,g_{\Gamma;A})$. Observe that $\cR_{\Gamma;A}\cap T^1(\Gamma\backslash\TT_A,g_{\Gamma;A})$ is invariant under the geodesic flow. By Fact \ref{fact:SBAP-6} and the above, the geodesic flow has the SBAP on $(\cN_{\Gamma;A}\cup\cR_{\Gamma;A})\cap T^1(\Gamma\backslash\TT_A,g_{\Gamma;A})$. By Fact \ref{fact:SBAP-5} and the fact that $\cN_{\Gamma;A}$ is open (due to Lemma \ref{lem:basic.properties.normal.vector-1}), the geodesic flow has the SBAP on $\cN_{\Gamma;A}\cap T^1(\Gamma\backslash\TT_A,g_{\Gamma;A})$  in $ T^1(\Gamma\backslash\TT_A,g_{\Gamma;A})$. This verifies (2).

\textbf{Proof of }(1): By Definition \ref{dfn:nonsep sets} and Lemma \ref{lem:basic.properties.normal.vector-1}, it remains to prove that $\cN_{\Gamma;A}\setminus\cR_{\Gamma;A}$ has zero Lebesgue measure in $T(\Gamma\backslash\TT_A)$ and is non-separating in $\cN_{\Gamma;A}$. We prove this via induction on $|A|$.

When $|A|=2$, let $i$ be the unique element in $\cE(A)$. By Lemma \ref{lem:properties.of.g.Gamma-1}, Corollary \ref{cor:nomadic=rank-1} and Lemma \ref{lem:rank-1.not.in.sing-2}, for any $v\in T^1(\Gamma\backslash\TT_A,g_{\Gamma;A})$, $v$ has rank at least two if and only if one of the following two cases holds:
\begin{itemize}
\item The bi-infinite geodesic tangent to $v$ never visits $\Pi_\Gamma(\cT_{i;A}(\pi/10))$.
\item There is an extended cap in $(\Gamma\backslash\TT_A,g_{\Gamma;A})$ such that $v$ is tangent the unique periodic geodesic in this extended cap. (See Definition \ref{dfn:construction}, Definition \ref{dfn:construction.of.g}, Definition \ref{dfn:caps} and its remark.)
\end{itemize}
Definition \ref{dfn:normal.vectors}, Proposition \ref{prop:normal.vectors.open.connected} and Definition \ref{dfn:projected.normal.vector} imply any vector satisfying the first case is not in $\cN_{\Gamma;A}$. Since there are only finitely many extended caps in $v\in T^1(\Gamma\backslash\TT_A,g_{\Gamma;A})$ (due to Definition \ref{dfn:construction} and Definition \ref{dfn:construction.of.g}), $\cN_{\Gamma;A}\setminus\cR_{\Gamma;A}$ is contained in the union of finitely many $2$-dimensional closed submanifolds of $T\TT_A$. Observe that $\dim(T\TT_A)=2|A|=4$ and that $\cN_{\Gamma;A}$ is open (due to Lemma \ref{lem:basic.properties.normal.vector-1}). Therefore, $\cN_{\Gamma;A}\setminus\cR_{\Gamma;A}$ has zero Lebesgue measure in $T(\Gamma\backslash\TT_A)$ and is non-separating in $\cN_{\Gamma;A}$. This verifies (1) when $|A|=2$.

Assume that (1) holds whenever $|A|\leq k-1$ (and hence (2) holds whenever $|A|\leq k-1$) for some $k\geq 3$. Consider the case when $|A|=k$. Notice that $|\cE(A)|=|A|-1\geq2$. Then by Lemma \ref{lem:tube.index.properties-1-5}, Corollary \ref{cor:nomadic=rank-1} and Lemma \ref{lem:rank-1.not.in.sing-2}, a bi-infinite geodesic has rank at least two if and only if it fails to visit $\Pi_\Gamma(\cT_{i;A}(\pi/10))$ for some $i\in\cE(A)$. Hence, by Lemma \ref{lem:missing.tube=higher.rank} and \eqref{eqn:j.higher.rk.normal.vector}, we have 
\begin{align}\label{eqn:higher.rk.normal.decomp}
\cN_{\Gamma;A}\setminus\cR_{\Gamma;A}=\bigcup_{i\in\cE(A)}\cN_{\Gamma;A}(i),
\end{align}
where $\cN_{\Gamma;A}(i):=(\Pi_\Gamma)_*\widetilde{\cN}_{\Gamma;A}(i)$ for any $i\in\cE(A)$. 

For any $i\in\cE(A)$, we let $A_{\leq i}:=\{j\in A:j\leq i\}$ and $A_{>i}:=\{j\in A:j>i\}$. For simplicity, we write
\begin{align}\label{eqn:j.subgroup.Gamma}
\Gamma(i):=\Gamma_A(\{i\})\cap\Gamma, ~\Gamma_-(i):=\Gamma\cap\Gamma_{A_{\leq i}}\text{ and }\Gamma_+(i):=\Gamma\cap\Gamma_{A_{> i}}. 
\end{align}
Notice that $\{A_{\leq i}, A_{>i}\}$ is the consecutive partition of $A$ induced by $\{i\}$. (See Definition \ref{dfn:consecutive} and its remark.) By Lemma \ref{lem:properties.good.subgp-(5)}, $\Gamma_-(i)\times \Gamma_+(i)$ is naturally isomorphic to $\Gamma(i)$ via the map $(h_1,h_2)\to h_1h_2$. Moreover, Lemma \ref{lem:properties.of.g.Gamma-2} implies that we have the following natural isometry induced by \eqref{eqn:prod.struct.after.forget}:
\begin{align}\label{eqn:j.natural.isom}
(\Gamma_-(i)\backslash\TT_{A_{\leq i}},g_{\Gamma_-(i);A_{\leq i}})\times (\Gamma_+(i)\backslash\TT_{A_{> i}},g_{\Gamma_+(i);A_{> i}})\isom(\Gamma(i)\backslash\TT_A,g_{\Gamma;A;\{i\}}).
\end{align}
To further simplify notations, for any $a,b>0$, we let 
\begin{align}\label{eqn:j.tangent.bundle.ab}
T^{a,b}(\Gamma(i)\backslash\TT_A,g_{\Gamma;A;\{i\}}):=T^a(\Gamma_-(i)\backslash\TT_{A_{\leq i}},g_{\Gamma_-(i);A_{\leq i}})\times T^b(\Gamma_+(i)\backslash\TT_{A_{> i}},g_{\Gamma_+(i);A_{> i}})
\end{align}
We define the \emph{Liouville measure} on $T^{a,b}(\Gamma(i)\backslash\TT_A,g_{\Gamma;A;\{i\}})$ as the product measure of the Liouville measure of each component on the right-hand-side of \eqref{eqn:j.tangent.bundle.ab}. We also define the \emph{geodesic flow} on $T^{a,b}(\Gamma(i)\backslash\TT_A,g_{\Gamma;A;\{i\}})$ as the product of the geodesic flow of each component on the right-hand-side. By \eqref{eqn:j.natural.isom}, this is the same as the geodesic flow on $T(\Gamma(i)\backslash\TT_A,g_{\Gamma;A;\{i\}})$ restricted to $T^{a,b}(\Gamma(i)\backslash\TT_A,g_{\Gamma;A;\{i\}})$.

Consider the following statements:
\begin{enumerate}
\item[(a)] For any $i\in\cE(A)$, $\cN_{\Gamma;A}(i)$ has Lebesgue measure zero in $T(\Gamma\backslash\TT_A)$ and is nonseparating in $\cN_{\Gamma;A}$.
\item[(b)] For any $i\in\cE(A)$, $(\Pi_{\Gamma(i)})_*\widetilde{\cN}_{\Gamma;A}(i)$ has Lebesgue measure zero in $T(\Gamma(i)\backslash\TT_A)$ and is nonseparating in $\cN_{\Gamma(i);A}$.
\item[(c)] For any $a,b>0$ and for any $i\in\cE(A)$, $((\Pi_{\Gamma(i)})_*\widetilde{\cN}_{\Gamma;A}(i))\cap T^{a,b}(\Gamma(i)\backslash\TT_A,g_{\Gamma;A;\{i\}})$ has Lebesgue measure zero in $T^{a,b}(\Gamma(i)\backslash\TT_A,g_{\Gamma;A;\{i\}})$ and is nonseparating in $\cN_{\Gamma(i);A}\cap T^{a,b}(\Gamma(i)\backslash\TT_A,g_{\Gamma;A;\{i\}})$.
\end{enumerate}
We would like to first show that (c)$\implies$(b)$\implies$(a)$\implies$(1). 

\textbf{Proof of }(a)$\implies$(1): This follows immediately from Fact \ref{fact:SBAP-2} and \eqref{eqn:higher.rk.normal.decomp}.

\textbf{Proof of } (b)$\implies$(a): 
It is clear from \eqref{eqn:j.higher.rk.normal.vector} that $\widetilde{\cN}_{\Gamma;A}(i)$ is a closed subset of $\widetilde{\cN}_{A}$. Moreover, by Proposition \ref{prop:normal.vectors.open.connected} and \eqref{eqn:j.higher.rk.normal.vector}, both $\widetilde{\cN}_{\Gamma;A}(i)$ and $\widetilde{\cN}_{A}$ are invariant under $\Gamma$. Hence, ${\cN}_{\Gamma;A}(i)$ is a closed subset of ${\cN}_{A}$. Recall that the footpoint of any vector in $\widetilde{\cN}_A$ is not in $\Omega_{\Gamma_A;A}$. (See Definition \ref{dfn:normal.vectors}.) Notice that $\Omega_{\Gamma;A}\subset\Omega_{\Gamma_A;A}$. (See \eqref{total singularity subgp}.) Therefore, by Lemma \ref{lem:branching.locus-1}, $(\Pi_\Gamma)_*$ is a local diffeomorphism when restricted to $\widetilde{\cN}_A$. In particular, $(\Pi^{\Gamma(i)}_\Gamma)_*$ is a local diffeomorphism when restricted to a neighborhood of $\cN_{\Gamma(i);A}$. (See \eqref{eqn:quotient.maps}, Definition \ref{dfn:projected.normal.vector} and \eqref{eqn:j.subgroup.Gamma}.) Hence, (a) is a direct corollary of (b), Fact \ref{fact:SBAP-1} and Fact \ref{fact:SBAP-4}.

\textbf{Proof of } (c)$\implies$(b): By Lemma \ref{lem:branching.locus-1}, Definition \ref{dfn:normal.vectors} and the discussion after \eqref{eqn:j.subgroup.Gamma}, $(\Pi_{\Gamma(i)})_*$ is a local diffeomorphism when restricted to $\widetilde{\cN}_{A_{\leq i}}\times \widetilde{\cN}_{A_{>i}}$. By Definition \ref{dfn:normal.vectors}, Lemma \ref{lem:basic.properties.normal.vector-3}, \eqref{eqn:j.tangent.bundle.ab} and the above, we have 
\begin{align}\label{eqn:key.inclusion.for.induction}
(\Pi_{\Gamma(i)})_*\widetilde{\cN}_{\Gamma;A}(i)\subset (\Pi_{\Gamma(i)})_*(\widetilde{\cN}_{A_{\leq i}}\times \widetilde{\cN}_{A_{>i}})\subset \bigcup_{a,b>0}T^{a,b}(\Gamma(i)\backslash\TT_A,g_{\Gamma;A;\{i\}}).
\end{align}
(b) is then a direct corollary of \eqref{eqn:key.inclusion.for.induction} and Fact \ref{fact:SBAP-3} applied to (c).

It remains for us to prove (c) when $|A|=k$.


\textbf{Proof of }(c): Let 
\begin{align}\label{eqn:middle.lift.vectors.avoid.j.tube}
\widetilde{\cQ}_{\Gamma;A}(i):=\{\widetilde{v}\in T(\Gamma(i)\backslash\TT_A): (\Pi^{\Gamma(i)}_\Gamma)_*(\widetilde{v})\in\cQ_{\Gamma;A}(i)\},
\end{align}
where $\cQ_{\Gamma;A}(i)$ is defined in \eqref{eqn:vectors.avoiding.j.tube}. Observe the following:
\begin{enumerate}
\item[\hypertarget{M1}{(M1)}] $v\in\widetilde{\cQ}_{\Gamma;A}(i)$ if and only if it is tangent to a bi-infinite geodesic in $(\Gamma(i)\backslash\TT_A, g_{\Gamma;A;\{i\}})$ which never intersects $\Pi_{\Gamma(i)}(\cT_{i;A}(\pi/10))$. In particular, $\widetilde{\cQ}_{\Gamma;A}(i)$ is closed and is invariant under the geodesic flow in $T(\Gamma(i)\backslash\TT_A, g_{\Gamma;A;\{i\}})$. (This follows from  \eqref{eqn:vectors.avoiding.j.tube}, \eqref{eqn:middle.lift.vectors.avoid.j.tube}, Lemma \ref{lem:branching.locus-2}, Lemma \ref{lem:tube.index.properties-1-3}, Lemma \ref{lem:tube.index.properties-1-4} and Lemma \ref{lem:properties.of.g.Gamma-4}.)
\item[\hypertarget{M2}{(M2)}] $(\Pi_{\Gamma(i)})_*\widetilde{\cN}_{\Gamma;A}(i)=\cN_{\Gamma(i);A}\cap \widetilde{\cQ}_{\Gamma;A}(i)$. (This follows from  \eqref{eqn:j.higher.rk.normal.vector}, \eqref{eqn:vectors.avoiding.j.tube}, \eqref{eqn:middle.lift.vectors.avoid.j.tube} and Definition \ref{dfn:projected.normal.vector}.)
\end{enumerate}
Since $k\geq 3$, one of $A_{\leq i}$ and $A_{>i}$ has cardinality at least two. We consider the following two cases:

\textbf{Case 1}: Assume that $|A_{\leq i}|\geq 2$ and $|A_{>i}|\geq 2$. Then the following holds:
\begin{itemize}
\item The geodesic flow on $T^a(\Gamma_-(i)\backslash\TT_{A_{\leq i}},g_{\Gamma_-(i);A_{\leq i}})$ is Bernoulli with respect to the Liouville measure. Moreover, it has the SBAP on the set  $\cN_{\Gamma_-(i);A_{\leq i}}\cap T^a(\Gamma_-(i)\backslash\TT_{A_{\leq i}},g_{\Gamma_-(i);A_{\leq i}})$. (This is due to (2) applied to the case when $A,\Gamma$ are replaced by $A_{\leq i}, \Gamma_-(i)$.)
\item The geodesic flow on $T^b(\Gamma_+(i)\backslash\TT_{A_{> i}},g_{\Gamma_+(i);A_{> i}})$ is Bernoulli with respect to the Liouville measure. Moreover, it has the SBAP on the set  $\cN_{\Gamma_+(i);A_{> i}}\cap T^b(\Gamma_+(i)\backslash\TT_{A_{> i}},g_{\Gamma_+(i);A_{> i}})$. (This is due to (2) applied to the case when $A,\Gamma$ are replaced by $A_{> i}, \Gamma_+(i)$.)
\end{itemize}
As a corollary, the following holds:
\begin{itemize}
\item The geodesic flow on $T^{a,b}(\Gamma(i)\backslash\TT_A,g_{\Gamma;A;\{i\}})$ is ergodic with respect to the Liouville measure. (This follows from \eqref{eqn:j.tangent.bundle.ab} and the fact that the product of mixing flows is mixing \cite[Proposition 3.4.17, pages 185]{FH19}.)
\item The geodesic flow on $T^{a,b}(\Gamma(i)\backslash\TT_A,g_{\Gamma;A;\{i\}})$ has the BAP on the set $(\cN_{\Gamma_-(i);A_{\leq i}}\times \cN_{\Gamma_+(i);A_{>i}})\cap T^{a,b}(\Gamma(i)\backslash\TT_A,g_{\Gamma;A;\{i\}})$. (This follows from Definition \ref{dfn:SBAP}.)
\end{itemize}
By \hyperlink{M1}{(M1)}, \hyperlink{M2}{(M2)} and the first bullet point in the above, $((\Pi_{\Gamma(i)})_*\widetilde{\cN}_{\Gamma;A}(i))\cap T^{a,b}(\Gamma(i)\backslash\TT_A,g_{\Gamma;A;\{i\}})$ has Lebesgue measure zero in $T^{a,b}(\Gamma(i)\backslash\TT_A,g_{\Gamma;A;\{i\}})$. By \hyperlink{M1}{(M1)}, the second bullet point in the above and Definition \ref{dfn:BAP}, $\widetilde{\cQ}_{\Gamma;A}(i)\cap T^{a,b}(\Gamma(i)\backslash\TT_A,g_{\Gamma;A;\{i\}})$ is nonseparating in $(\cN_{\Gamma_-(i);A_{\leq i}}\times \cN_{\Gamma_+(i);A_{>i}})\cap T^{a,b}(\Gamma(i)\backslash\TT_A,g_{\Gamma;A;\{i\}})$.  Observe the following:
\begin{itemize} 
\item For any $\widetilde{v}\in (\cN_{\Gamma(i);A}\setminus\widetilde{\cQ}_{\Gamma;A}(i))\cap T^{a,b}(\Gamma(i)\backslash\TT_A,g_{\Gamma;A;\{i\}})$, there exists an open neighborhood of $\widetilde{v}$ in $\cN_{\Gamma(i);A}\cap T^{a,b}(\Gamma(i)\backslash\TT_A,g_{\Gamma;A;\{i\}})$ which does not intersect $\widetilde{\cQ}_{\Gamma;A}(i)$. (This is due to \hyperlink{M1}{(M1)} and Lemma \ref{lem:basic.properties.normal.vector-1}.)
\item $\cN_{\Gamma(i);A}\cap \widetilde{\cQ}_{\Gamma;A}(i)\subset(\Pi_{\Gamma(i)})_*(\widetilde{\cN}_{A_{\leq i}}\times \widetilde{\cN}_{A_{>i}})=\cN_{\Gamma_-(i);A_{\leq i}}\times \cN_{\Gamma_+(i);A_{>i}}$. (Set inclusion is due to \eqref{eqn:key.inclusion.for.induction} and \hyperlink{M2}{(M2)}. Equality is due to Definition \ref{dfn:projected.normal.vector} and the discussion after \eqref{eqn:j.subgroup.Gamma}.)
\end{itemize}
Therefore, by Fact \ref{fact:SBAP-1}, Fact \ref{fact:SBAP-4} and the above, $\widetilde{\cQ}_{\Gamma;A}(i)\cap T^{a,b}(\Gamma(i)\backslash\TT_A,g_{\Gamma;A;\{i\}})$ is nonseparating in $\cN_{\Gamma(i);A}\cap T^{a,b}(\Gamma(i)\backslash\TT_A,g_{\Gamma;A;\{i\}})$. By Definition \ref{dfn:nonsep sets} and \hyperlink{M2}{(M2)}, this is equivalent to (c).
%
%
%

\textbf{Case 2}: Assume that $|A_{\leq i}|= 1$ or $|A_{>i}|=1$. WLOG we assume that $|A_{\leq i}|= 1$. (The proof for $|A_{>i}|=1$ is similar.) Then by Definition \ref{dfn:good subgp}, Definition \ref{dfn:normal.vectors} and Definition \ref{dfn:projected.normal.vector}, we have
\begin{align}\label{eqn:dim.1=circle}
(\Gamma_-(i)\backslash\TT_{A\leq i},g_{\Gamma_-(i);A_{\leq i}})=\TT_i\text{ and }\cN_{\Gamma_-(i);A_{\leq i}}\cap T^a(\Gamma_-(i)\backslash\TT_{A_{\leq i}}; g_{\Gamma_-(i);A_{\leq i}})=T^a\TT_i.
\end{align} 
Choose an orientation for the circle $\TT_i$. We let $T^a_{+}\TT_i$ ($T^a_-\TT_i$ resp.) be the collection of vectors in $T^a\TT_i$ whose geodesic is following the positive (negative resp.) direction. It is easy to see that the geodesic flow in $T^a\TT_i$ leaves $T^a_\pm\TT_i$ invariant. Observe the following:
\begin{itemize}
\item The geodesic flow restricted to $\TT^a_\pm\TT_i$ has a unique orbit which is periodic. In particular, it is ergodic. Moreover, it has the BAP on any open subsets of $T^a_\pm\TT_i$. (See Definition \ref{dfn:BAP}.)
\item The geodesic flow on $T^b(\Gamma_+(i)\backslash\TT_{A_{> i}},g_{\Gamma_+(i);A_{> i}})$ is Bernoulli with respect to the Liouville measure. Moreover, it has the SBAP on the set  $\cN_{\Gamma_+(i);A_{> i}}\cap T^b(\Gamma_+(i)\backslash\TT_{A_{> i}},g_{\Gamma_+(i);A_{> i}})$. (This is due to (2) applied to the case when $A,\Gamma$ are replaced by $A_{> i}, \Gamma_+(i)$.)
\end{itemize}
As a corollary of \eqref{eqn:dim.1=circle} and the above, the following holds:
\begin{itemize}
\item The geodesic flow restricted to each connected component of $T^{a,b}(\Gamma(i)\backslash\TT_A,g_{\Gamma;A;\{i\}})$ is ergodic with respect to the Liouville measure. (This follows from \eqref{eqn:j.tangent.bundle.ab} and the fact that the product of a mixing flow and an ergodic flow is ergodic \cite[Theorem 3.4.19, pages 186-187]{FH19}.)
\item The geodesic flow on $T^{a,b}(\Gamma(i)\backslash\TT_A,g_{\Gamma;A;\{i\}})$ has the BAP on the set $(\cN_{\Gamma_-(i);A_{\leq i}}\times \cN_{\Gamma_+(i);A_{>i}})\cap T^{a,b}(\Gamma(i)\backslash\TT_A,g_{\Gamma;A;\{i\}})$. (This follows from Definition \ref{dfn:SBAP}.)
\end{itemize}
The rest of the proof of (c) in \textbf{Case 2} follows from the same argument at the end of \textbf{Case 1}.

Now we conclude that (c) holds when $|A|=k$. Therefore (1) holds when $|A|=k$. This finishes the proof of Theorem \ref{thm:main.true}.
\end{proof}

\appendix
\section{A branched covering of $S^n$ by $\TT^n$}\label{appA}

\subsection{A $2^{n-1}$-fold branched covering of $S^n$ by $\TT^n$}
Due to technical reason, our treatment of the branched covering of $S^n$ by $\TT^n$ uses a slightly different parametrization compared to the classical approach. Normally, people heavily use $\sin(\cdot)$ and $\cos(\cdot)$ to construct an analytic branched covering of $S^n$ by $\TT^n$. However, this makes it significantly harder for us to insert generalized caps into $\TT^n$ in a neighborhood of the ramification locus while leaving the product metric unchanged elsewhere. To resolve this issue, we introduce the following functions:
\begin{dfn}\label{dfn:fake sin/cos}
We fix $\sn:\RR\to\RR$ and $\cn:\RR\to\RR$ to be $2\pi$-periodic, smooth functions satisfying the following properties:
\begin{enumerate}
\item[(1).] $\sn^2(t)+\cn^2(t)=1$. 
\item[(2).] $\sn(t)=\sn(\pi-t)=-\sn(t+\pi)$ and $\cn(t)=-\cn(\pi-t)=-\cn(t+\pi)$.
\item[(3).] $\sn'(t)>0$ on $(0,\pi/2)$ and $\cn'(t)<0$ on $(0,\pi/2)$.
\item[(4).] $\sn(t)=t$ for any $t\in[-\pi/8,\pi/8]$ and $\cn(t)=-t+(\pi/2)$ for any $t\in[3\pi/8,5\pi/8]$.
\end{enumerate}
\end{dfn}
\begin{rmk}
\begin{enumerate}
\item[(1).] One can easily check that $\sn,\cn$ always exist. With the above properties, one can easily verify that $\sn(m\pi/2)=\sin(m\pi/2)$ and $\cn(m\pi/2)=\cos(m\pi/2)$ for any $m\in\ZZ$. In particular, $\sn(m\pi)=\cn'(m\pi)=\cn(m\pi+\pi/2)=\sn'(m\pi+\pi/2)=0$ for any $m\in\ZZ$. Moreover, $\sn'(t)\sin'(t)>0$ whenever $\sin'(t)\neq 0$ and $\cn'(t)\cos'(t)>0$ whenever $\cos'(t)\neq 0$.
\item[(2).] The last property is the key technical assumption for verifying the second bullet point in Proposition \ref{prop:quotient.is.sphere}. See \eqref{eqn:sncn.near.sing}. 
\end{enumerate}
\end{rmk}

In this appendix, we always view $S^n$ as the ellipsoid in $\RR^{n+1}$ defined by
\begin{align}\label{eqn:standard.ellip}
\sum_{j=1}^{n+1}\frac{x_j^2}{j}=1.
\end{align}
We would like construct the branched covering $\Phi:\TT^n\to S^n$ such that for any $(\theta_1,\dots,\theta_n)\in\TT^n$, if we write $(x_1,\dots, x_{n+1}):=\Phi(\theta_1,\dots,\theta_n)\in S^n\subset \RR^{n+1}$, then we have 
\begin{align}\label{eqn:hyperboloids}
\sum_{j=1}^{n+1}\frac{x_j^2}{j-(i+\sn^2(\theta_i))}=1,~\text{ for any }i\in\{1,\dots,n\}
\end{align}
whenever the left-hand-side is well-defined. The classical treatment uses $\sin(\cdot)$ instead of $\sn(\cdot)$. See \cite{Kli95} or \cite{Jac1842}. The following lemma is the main reason for why we can find such a choice of $(x_1,\dots, x_{n+1})$.

\begin{lem}\label{ellip coor}
For any real numbers $\xi_1,\xi_2,\dots ,\xi_{n+1}$, we let 
\begin{align}\label{eqn:ellip.coor.eqn}
\cA:=\cA(\xi_1,\dots ,\xi_{n+1})=\mativ{\frac{1}{1-\xi_1}&\frac{1}{2-\xi_1}&\cdots&\frac{1}{{n+1}-\xi_1}}{\frac{1}{1-\xi_2}&\frac{1}{2-\xi_2}&\cdots&\frac{1}{{n+1}-\xi_2}}{\vdots&\vdots&\ddots&\vdots}{\frac{1}{1-\xi_{n+1}}&\frac{1}{2-\xi_{n+1}}&\cdots&\frac{1}{{n+1}-\xi_{n+1}}}.
\end{align}
When $\cA$ is well-defined, the $(n+1)$-dimensional vector $y=(y_1,\dots ,y_{n+1})^T$ defined by
\begin{align}\label{ellip coor formula}
y_j:=\frac{\displaystyle\prod_{l=1}^{n+1}(j-\xi_l)}{\displaystyle\prod_{\substack{1\leq l\leq n+1\\ l\neq j}}(j-l)}=\frac{\displaystyle\prod_{l=1}^{j}(j-\xi_l)}{(j-1)!}\frac{\displaystyle\prod_{l=1}^{n+1-j}(\xi_{j+l}-j)}{(n+1-j)!}, \quad j=1,\dots ,n+1
\end{align}
is a solution to $\cA x=(1,1,\dots ,1)^T$, where $(\cdot)^T$ denotes the transpose of matrices.

Moreover, if we assume in addition that $\xi_1<1<\xi_2<2<\cdots <\xi_{n+1}<n+1$, then the above solution is unique. 
\end{lem}
\begin{proof}
See \cite[pages 221-222]{Jac1842}.
\end{proof}
%
We are now ready to define the desired branched covering of $S^n$ by $\TT^n$: To simplify the formula for $\Phi$, in the rest of the appendix, we fix 
\begin{align}\label{eqn:covering.formula.simp}
\theta_0=0\text{ and }\theta_{n+1}=\pi/2. ~(\text{In particular, }\cn(\theta_0)=\sn(\theta_{n+1})=1.)
\end{align}
For any $(\theta_1,\dots,\theta_n)\in\TT^n$, the $j$-th component of $\Phi(\theta_1,\dots ,\theta_n)$ is given by
\begin{align}\label{covering formula}
\left(\Phi(\theta_1,\dots ,\theta_n)\right)_j
=\displaystyle\frac{\cn(\theta_{j-1})\sn(\theta_j)}{\sqrt{(j-1)!(n+1-j)!}}\sqrt{j\cdot\prod_{l=1}^{j-2}(j-l-\sn^2(\theta_l))\prod_{l=j+1}^n(l+\sn^2(\theta_l)-j)}.
\end{align}
By \eqref{covering formula} and the fact that $\sn^2(t)+\cn^2(t)=1$, we have 
\begin{align}\label{eqn:xj2}
\left(\left(\Phi(\theta_1,\dots ,\theta_n)\right)_j\right)^2=\frac{\displaystyle j\cdot\prod_{l=1}^{n}(j-l-\sn^2(\theta_l))}{\displaystyle\prod_{\substack{1\leq l\leq n+1\\ l\neq j}}(j-l)},\text{ for any }j\in\{1,\dots, n+1\}.
\end{align}
\begin{lem}\label{verify branched cover}
The map $\Phi$ defined in \eqref{covering formula} is a $\Gamma_n$-invariant surjective smooth map from $\TT^n$ to $S^n$ (identified with the ellipsoid in \eqref{eqn:standard.ellip}). Moreover, for any $(\theta_1,\dots,\theta_n)\in\TT^n$ such that $\theta_1,\dots,\theta_n\not\in(\pi/2)\ZZ$, $\Phi(\theta_1,\dots,\theta_n)$ lies on $\eqref{eqn:hyperboloids}$.
\end{lem}
\begin{proof}
We first show that the image of $\Phi$ is in $S^n$ and that $\Phi(\theta_1,\dots,\theta_n)$ lies on \eqref{eqn:hyperboloids} whenever $\theta_1,\dots,\theta_n\not\in(\pi/2)\ZZ$. For any $(\theta_1,\dots,\theta_n)\in\TT^2$, let $(x_1,\dots, x_{n+1})=\Phi(\theta_1,\dots, \theta_n)\in\RR^{n+1}$ and 
\begin{align}\label{eqn:theta.to.xi}
(\xi_1,\dots,\xi_{n+1})=(0,1+\sn^2(\theta_1),2+\sn^2(\theta_2),\dots, n+\sn^2(\theta_n)).
\end{align} 
Then by \eqref{eqn:xj2} and Lemma \ref{ellip coor}, we have 
\begin{align}\label{eqn:rel.to.ellip}
\cA(\xi_1,\dots,\xi_{n+1})\cdot (x_1^2,x_2^2,\dots ,x_{n+1}^2)^T=(1,1,\dots, 1)^T,
\end{align}
where $\cA(\xi_1,\dots,\xi_{n+1})$ is defined in \eqref{eqn:ellip.coor.eqn}. Equivalently, $(x_1,\dots, x_{n+1})$ lies on \eqref{eqn:standard.ellip} and \eqref{eqn:hyperboloids} whenever $\theta_1,\dots,\theta_n\not\in(\pi/2)\ZZ$. This shows that there exists a dense subset $U\subset \TT^n$ such that $\Phi(U)\subset S^n$ (identified with the ellipsoid in \eqref{eqn:standard.ellip}). Since $\Phi$ is a smooth map from $\TT^n$ to $\RR^{n+1}$ due to \eqref{covering formula}, we conclude that $\Phi$ is a smooth map from $\TT^{n}$ to $S^n$.

Notice that for any $i\in\{1,\dots,n-1\}$ and for any $j\in\{1,\dots,n\}$, by \eqref{eqn:covering.formula.simp} and Definition \ref{dfn:fake sin/cos}, $\cn(\theta_{j-1})\sn(\theta_{j})$ and $\sn^2(\theta_j)$ are invariant under $\gamma_i$ when viewed as functions of $(\theta_1,\dots,\theta_n)\in\TT^n$. (See \eqref{eqn:generator}.) Therefore, by \eqref{covering formula}, $\Phi$ is invariant under $\Gamma_n$-action.

It remains to show that $\Phi$ is surjective. By the compactness of $\TT^n$ and the smoothness of $\Phi$, it suffices to show that $\Phi(\TT^n)$ contains $\{(x_1,\dots,x_{n+1})\in S^n:x_1,\dots, x_{n+1}\in\RR\setminus \{0\}\}$. Fix an arbitrary $(x_1,\dots, x_{n+1})\in S^n$ with $x_1,\dots, x_{n+1}\in\RR\setminus \{0\}$. For any $j\in\{1,\dots,n\}$, the function $f(t):=\sum_{k=1}^{n+1}\frac{x_k^2}{k-t}$ is smooth on $(j,j+1)$ with $\lim_{t\to j_+}f(t)=-\infty$ and $\lim_{t\to (j+1)_-}f(t)=\infty$. Hence, there exists $\xi_{j+1}\in(j,j+1)$ such that $f(t)=1$. This and \eqref{eqn:standard.ellip} imply that \eqref{eqn:rel.to.ellip} holds for $\xi_1=0$, $x_1,\dots, x_{n+1}$ and the above choice of $\xi_2,\dots, \xi_{n+1}$. Choose $\theta_j$ such that 
\begin{align}\label{eqn:verify branched cover-1}
j+\sn^2(\theta_j)=\xi_{j+1}\text{ for any }j\in\{1,\dots,n\},~(\theta_j\text{ exists due to Definition \ref{dfn:fake sin/cos}.})
\end{align}
and that 
\begin{align}\label{eqn:verify branched cover-2}
x_j\cn(\theta_{j-1})\sn(\theta_j)>0\text{ for any }j\in\{1,\dots,n+1\}.
\end{align} 
By Lemma \ref{ellip coor}, \eqref{eqn:xj2}, \eqref{eqn:verify branched cover-1} and \eqref{eqn:verify branched cover-2}, we have $\Phi(\theta_1,\dots,\theta_n)=(x_1,\dots, x_{n+1})$. This finishes the proof.
\end{proof}

\begin{lem}\label{lem:preimage.count}
For any $(\theta_1,\dots,\theta_n)\in\TT^n$ such that $\theta_1,\dots,\theta_n\not\in (\pi/2)\ZZ$, we have $\Phi^{-1}(\Phi(\theta_1,\dots,\theta_n))=\Gamma_n\cdot(\theta_1,\dots,\theta_n)$.
\end{lem}
\begin{proof}
Let $(\xi_1,\dots,\xi_{n+1})=(0,1+\sn^2(\theta_1),2+\sn^2(\theta_2),\dots,n+\sn^2(\theta_{n+1}))$ and $(x_1,\dots,x_{n+1})=\Phi(\theta_1,\dots,\theta_n)$. If $\theta_1,\dots,\theta_n\not\in (\pi/2)\ZZ$, then by Definition \ref{dfn:fake sin/cos} and \eqref{covering formula}, we have $x_1,\dots, x_{n+1}\in\RR\setminus\{0\}$ and $\xi_1<1<\xi_2<2<\dots<\xi_{n+1}<n+1$.

For any $(\theta'_1,\dots,\theta'_n)\in \Phi^{-1}(x_1,\dots,x_{n+1})$, we let $(\xi'_1,\dots,\xi'_{n+1})=(0,1+\sn^2(\theta'_1),2+\sn^2(\theta'_2),\dots,n+\sn^2(\theta'_{n+1}))$. By Definition \ref{dfn:fake sin/cos}, \eqref{covering formula} and the fact that $x_1,\dots, x_{n+1}\in\RR\setminus\{0\}$, we also have $\xi'_1<1<\xi'_2<2<\dots<\xi'_{n+1}<n+1$ and $\theta_1',\dots,\theta_n'\not\in(\pi/2)\ZZ$.

We first show that $\sn^2(\theta'_j)=\sn^2(\theta_j)$ for any $j\in\{1,\dots,n\}$, which is equivalent to $(\xi_1,\dots,\xi_{n+1})=(\xi_1',\dots,\xi'_{n+1})$. Indeed, by Lemma \ref{verify branched cover} and \eqref{eqn:hyperboloids}, for any $j\in\{1,\dots,n\}$, we have $f(\xi_{j+1})=f(\xi'_{j+1})=1$,
where $f(t):=\sum_{k=1}^{n+1}\frac{x_k^2}{k-t}$. Since $f(t)$ is strictly increasing on $(j,j+1)$, we have $\xi_{j+1}=\xi'_{j+1}$. Therefore, $(\xi_1,\dots,\xi_{n+1})=(\xi_1',\dots,\xi'_{n+1})$. ($\xi_1=\xi'_1=0$ by assumption.)

By Lemma \ref{verify branched cover}, the assumption that $\theta_1,\dots,\theta_n\not\in (\pi/2)\ZZ$ and the remark after Definition \ref{sing index}, we have 
$$|\Phi^{-1}(x_1,\dots,x_{n+1})|\geq |\Gamma_n\cdot (\theta_1,\dots,\theta_n)|=|\Gamma_n|=2^{n-1}.$$ 
It remains to show that $|\Phi^{-1}(x_1,\dots,x_{n+1})|\leq 2^{n-1}$. Consider the map $F:\Phi^{-1}(x_1,\dots,x_{n+1})\to\{-1,1\}^{2n}$, defined by  
$$F(\theta'_1,\dots,\theta'_n)=\left(\frac{\sn(\theta'_1)}{\sn(\theta_1)},\dots,\frac{\sn(\theta'_n)}{\sn(\theta_n)},\frac{\cn(\theta'_1)}{\cn(\theta_1)},\dots,\frac{\cn(\theta'_n)}{\cn(\theta_n)}\right).$$
(The right-hand-side of the above is well-defined due to Definition \ref{dfn:fake sin/cos} and the assumption that $\theta_1,\dots,\theta_n\not\in (\pi/2)\ZZ$. The image of $F$ is in $\{-1,1\}^{2n}$ due to the fact that $\sn^2(\theta'_j)=\sn^2(\theta_j)$ for any $j\in\{1,\dots,n\}$.) Then \eqref{covering formula} implies that $F$ is injective and that for any $(\theta'_1,\dots,\theta'_n)\in \Phi^{-1}(x_1,\dots,x_{n+1})$, we have 
$$\frac{\sn(\theta'_1)}{\sn(\theta_1)}=\frac{\cn(\theta'_1)}{\cn(\theta_1)}\frac{\sn(\theta'_2)}{\sn(\theta_2)}=\frac{\cn(\theta'_2)}{\cn(\theta_2)}\frac{\sn(\theta'_3)}{\sn(\theta_3)}=\dots=\frac{\cn(\theta'_{n-1})}{\cn(\theta_{n-1})}\frac{\sn(\theta'_n)}{\sn(\theta_n)}=\frac{\cn(\theta'_n)}{\cn(\theta_n)}=1.$$
Hence
\begin{align*}
|\Phi^{-1}(x_1,\dots,x_{n+1})|\leq&|F(\Phi^{-1}(x_1,\dots,x_{n+1}))|\\
\leq &|\{(\sigma_1,\dots,\sigma_{2n})\in\{-1,1\}^{2n}:\sigma_1=\sigma_2\sigma_{n+1}=\sigma_3\sigma_{n+2}=\dots=\sigma_n\sigma_{2n-1}=\sigma_{2n}=1\}|\\
=&2^{n-1}.
\end{align*}
This finishes the proof.
\end{proof}
\begin{lem}\label{lem:deriv.covering.map}
For any $p=(\theta_1,\dots,\theta_n)\in\TT^n$, the vectors $D\Phi(p)(\del/\del\theta_1),\dots,D\Phi(p)(\del/\del\theta_n)$ are orthogonal to each other. Moreover, for any $k\in\{1,\dots,n+1\}$, we have 
\begin{align*}
\left\|D\Phi(p)\left(\frac{\del}{\del\theta_k}\right)\right\|^2\geq \max\left\{\frac{(\sn'(\theta_k))^2}{2(n+2)!}(\cn^2(\theta_{k-1})+\sn^2(\theta_k)),\frac{(\cn'(\theta_k))^2}{2(n+2)!}(\cn^2(\theta_{k})+\sn^2(\theta_{k+1}))\right\},
\end{align*}
where $\theta_0,\theta_{n+1}$ is defined in \eqref{eqn:covering.formula.simp}.
\end{lem}
\begin{proof}
By smoothness of $\Phi$, it suffices to proof this lemma under the additional assumption that $\theta_1,\dots,\theta_n\not\in(\pi/2)\ZZ$. Let $(x_1,\dots, x_{n+1})=\Phi(\theta_1,\dots,\theta_n)$ and $(\xi_1,\dots,\xi_{n+1})=(0,1+\sn^2(\theta_1),2+\sn^2(\theta_2),\dots,n+\sn^2(\theta_n))$.

By \eqref{eqn:xj2} and Definition \ref{dfn:fake sin/cos}, for any $k\in\{1,\dots, n\}$, we have 
\begin{align}\label{eqn:deriv.Phi}
(\xi_{k+1}-j)\frac{\del x_j}{\del\theta_k}(\theta_1,\dots,\theta_n)=\sn(\theta_k)\sn'(\theta_k)x_j=-\cn(\theta_k)\cn'(\theta_k)x_j.
\end{align}
Hence by \eqref{eqn:deriv.Phi} and \cite[Equation (26.4), page 224]{Jac1842}, for any $k,l\in\{1,\dots,n\}$ with $k\neq l$, we have 
\begin{align*}
\left\langle D\Phi(p)\left(\frac{\del}{\del\theta_k}\right),D\Phi(p)\left(\frac{\del}{\del\theta_l}\right)\right\rangle
=\sn(\theta_k)\sn'(\theta_k)\sn(\theta_l)\sn'(\theta_l)\sum_{j=1}^{n+1}\frac{x_j^2}{(\xi_{k+1}-j)(\xi_{l+1}-j)}=0.
\end{align*}
This proves that the vectors $D\Phi(p)(\del/\del\theta_1),\dots,D\Phi(p)(\del/\del\theta_n)$ are orthogonal to each other.

To prove the rest of this lemma, we first notice that for any $j\in\{1,\dots,n+1\}$, by Lemma \ref{verify branched cover}, \eqref{eqn:standard.ellip} and \eqref{covering formula}, we have 
\begin{align}\label{eqn:xj.est}
\frac{|\cn(\theta_{j-1})\sn(\theta_j)|}{\sqrt{n!}}\leq |x_j|\leq \sqrt{n+1}.
\end{align}
Hence for any $k\in\{1,\dots,n\}$, by \eqref{eqn:deriv.Phi} and \eqref{eqn:xj.est} we have  
\begin{align}\label{eqn:deriv.vector.norm.est-1}
\left\|D\Phi(p)\left(\frac{\del}{\del\theta_k}\right)\right\|^2=&(\sn(\theta_k)\sn'(\theta_k))^2\sum_{j=1}^{n+1}\frac{x^2_j}{(\xi_{k+1}-j)^2}\nonumber\\
\geq&(\sn(\theta_k)\sn'(\theta_k))^2\left(\sum_{j=1}^{n+1}{x^2_j}\right)^{-1}\left(\sum_{j=1}^{n+1}\frac{x^2_j}{j-\xi_{k+1}}\right)^2\nonumber\\
\geq&\frac{(\sn(\theta_k)\sn'(\theta_k))^2}{(n+1)^2}\left(\sum_{j=1}^{n+1}\frac{x^2_j}{j-\xi_{k+1}}\right)^2=\frac{(\sn(\theta_k)\sn'(\theta_k))^2}{(n+1)^2}=\frac{(\cn(\theta_k)\cn'(\theta_k))^2}{(n+1)^2}.
\end{align}
where the second last equality follows from Lemma \ref{verify branched cover} and \eqref{eqn:hyperboloids}. On the other hand, by \eqref{covering formula} and \eqref{eqn:xj.est}, we have 
\begin{align}\label{eqn:deriv.vector.norm.est-2}
\left\|D\Phi(p)\left(\frac{\del}{\del\theta_k}\right)\right\|^2\geq&\left(\frac{\del x_k}{\del \theta_k}\right)^2+\left(\frac{\del x_{k+1}}{\del \theta_{k}}\right)^2\geq \frac{(\cn(\theta_{k-1})\sn'(\theta_k))^2+(\cn'(\theta_k)\sn(\theta_{k+1}))^2}{n!}.
\end{align}
The lemma then follows from \eqref{eqn:deriv.vector.norm.est-1} and \eqref{eqn:deriv.vector.norm.est-2}.
\end{proof}

\subsection{Proof of Proposition \ref{prop:quotient.is.sphere}}\label{subsect:quotient.is.sphere}
Choose $\Phi_{\Gamma_n;\bV_n}:=\Phi$. To avoid lengthy formulas, for any $j\in\{1,\dots,n+1\}$ and for any $l\in\{1,\dots,n\}$, we let 
\begin{align}\label{eqn:formula.simplification.parts}
\Phi_{j,l}(\theta_l)=
\begin{cases}
\cn(\theta_l),~&\text{if }l=j-1,\\
\sn(\theta_l),~&\text{if }l=j,\\
\sqrt{|j-l-\sn^2(\theta_l)|},~&\text{otherwise}.
\end{cases}
\end{align}
Then for any $(\theta_1,\dots,\theta_n)\in\TT^n$ and for any $j\in\{1,\dots,n+1\}$, by \eqref{covering formula} and \eqref{eqn:formula.simplification.parts}, we have
\begin{align}\label{eqn:simplified.covering.formula}
(\Phi(\theta_1,\dots,\theta_n))_j=\frac{\sqrt{j}}{\sqrt{(j-1)!(n+1-j)!}}\prod_{l=1}^{n}\Phi_{j,l}(\theta_l).
\end{align}

For any $p=(\theta_1,\dots,\theta_n)\in\TT^n$, let $\cU_p$ be any good neighborhood of $p$. (See Definition \ref{dfn:good.nbhd}.) Then by Definition \ref{dfn:fake sin/cos}, \eqref{jth singularity}, Definition \ref{sing index} and the third property in Definition \ref{dfn:good.nbhd}, for any $l\in\text{Sing}_{\bV_n}(p)$ and for any $(\theta_1+\delta_1,\dots,\theta_n+\delta_n)\in\cU_p$ with $\delta_1,\dots,\delta_n\in(-\pi/8,\pi/8)$, there exist $\sigma_l,\sigma_{l+1},\sigma_{l,l+1}\in\{1,-1\}$ independent of the choice of $\delta_1,\dots,\delta_n$ such that
\begin{align}\label{eqn:sncn.near.sing}
\frac{\sn(\theta_l+\delta_l)}{\sqrt{1-\delta_l^2}}=\sigma_l,~\frac{\cn(\theta_{l+1}+\delta_{l+1})}{\sqrt{1-\delta_{l+1}^2}}=\sigma_{l+1}\text{ and }\cn(\theta_l+\delta_l)\sn(\theta_{l+1}+\delta_{l+1})=\sigma_{l,l+1}\delta_l\delta_{l+1}.
\end{align}
For any $(t,s)\in\RR^2$ with $|t|,|s|<\pi^2/32$, for any $l\in\text{Sing}_{\bV_n}(p)$, for any $j\in\{1,\dots,n+1\}\setminus \{l+1\}$, we define
\begin{align}\label{eqn:diffeo.after.z.to.z2.parts}
\Psi_{j,l}(t,s):=
\begin{cases}
\displaystyle\sigma_{l+1}\sqrt{(j-l-1)^2+(j-l-1)t-\frac{s^2}{4}},~&\text{if }l=j-2,\\
\displaystyle\frac{\sigma_{l,l+1}s}{2}, ~&\text{if }l=j-1,\\
\displaystyle\sigma_{l}\sqrt{(j-l-1)^2+(j-l-1)t-\frac{s^2}{4}},~&\text{if }l=j,\\
\displaystyle\sqrt{(j-l-1)^2+(j-l-1)t-\frac{s^2}{4}},~&\text{otherwise}.
\end{cases}
\end{align}
Then for any $(\theta_1+\delta_1,\dots,\theta_n+\delta_n)\in \cU_p$, by Definition \ref{dfn:fake sin/cos}, \eqref{eqn:formula.simplification.parts}, \eqref{eqn:sncn.near.sing} and \eqref{eqn:diffeo.after.z.to.z2.parts}, we have 
\begin{align}\label{eqn:large.terms.covering.formula}
\Phi_{j,l}(\theta_l+\delta_l)\Phi_{j,l+1}(\theta_{l+1}+\delta_{l+1})=\Psi_{j,l}(\delta_l^2-\delta^2_{l+1},2\delta_l\delta_{l+1}).
\end{align}
Hence, we define $\Psi_{\Gamma_n;\bV_n;p}:\Xi_{\Gamma_n;\bV_n;p}(\cU_p)\to \RR^{n+1}$ such that for any $j\in\{1,\dots,n+1\}$, we have
\begin{align}\label{eqn:diffeo.after.z.to.z2}
(\Psi_{\Gamma_n;\bV_n;p}(\zeta_1,\dots,\zeta_n))_j=\frac{\sqrt{j}\displaystyle\left(\prod_{l\in\bV_n\setminus\text{Sing}^+_{\bV_n}(p)}\Phi_{j,l}(\theta_l+\zeta_l)\right)\cdot\left(\prod_{l\in\text{Sing}_{\bV_n}(p)}\Psi_{j,l}(\zeta_l,\zeta_{l+1})\right)}{\sqrt{(j-1)!(n+1-j)!}}.
\end{align}
The fact that $\Phi|_{U_p}=\Psi_{\Gamma_n;\bV_n;p}\circ \Xi_{\Gamma_n;\bV_n;p}$ follows directly from \eqref{eqn:z.to.z2}, \eqref{eqn:simplified.covering.formula}, \eqref{eqn:large.terms.covering.formula} and \eqref{eqn:diffeo.after.z.to.z2}. Hence $\Psi_{\Gamma_n;\bV_n;p}$ is a smooth map from $\Xi_{\Gamma_n;\bV_n;p}(\cU_p)$ to $S^n$ due to \eqref{eqn:formula.simplification.parts}, \eqref{eqn:diffeo.after.z.to.z2.parts} and \eqref{eqn:diffeo.after.z.to.z2}. 

Suppose that $\text{Jac}\Psi_{\Gamma_n;\bV_n;p}$ does not vanish, then by the fact that $\Phi|_{\cU_p}=\Psi_{\Gamma_n;\bV_n;p}\circ \Xi_{\Gamma_n;\bV_n;p}$ and the fact that $\Xi_{\Gamma_n;\bV_n;p}$ is an open map, $\Phi$ is also an open map. Therefore, by Lemma \ref{lem:preimage.count} and the fact that $\Phi$ is smooth and open, for any $p\in\TT^n$, we have $\Phi^{-1}(\Phi(p))=\Gamma_n\cdot p$. Then by Lemma \ref{verify branched cover} and the above, $\Phi=\overline{\Phi}\circ\Pi_{\Gamma_n}$ for some homeomorphism $\overline{\Phi}:\Gamma_n\backslash\TT^n\to S^n$. In particular, we have 
$$\Pi_{\Gamma_n}|_{\cU_p}:=(\overline{\Phi})^{-1}\circ\Phi|_{\cU_p}=\left((\overline{\Phi})^{-1}\circ\Psi_{\Gamma_n;\bV_n;p}\right)\circ\Xi_{\Gamma_n;\bV_n;p}.$$ 
Apply Lemma \ref{lem:local.decomp.proj} to the above, we have $\Psi_{\Gamma_n;\bV_n;p}$ is a homeomorphism onto its image. Hence, by the assumption that $\text{Jac}\Psi_{\Gamma_n;\bV_n;p}$ does not vanish, $\Psi_{\Gamma_n;\bV_n;p}$ is a diffeomorphism onto its image. Therefore, Proposition \ref{prop:quotient.is.sphere} is proved if we can show that $\text{Jac}\Psi_{\Gamma_n;\bV_n;p}$ does not vanish.

For any $q=(\theta_1+\delta_1,\dots,\theta_n+\delta_n)\in \cU_p$ such that $\delta_1,\dots,\delta_n\in(-\pi/8,\pi/8)$, we have 
\begin{align}\label{eqn:Jac.Xi}
\text{Jac}\Xi_{\Gamma_n;\bV_n;p}(q)=4^{|\text{Sing}_{\bV_n}(p)|}\prod_{l\in\text{Sing}_{\bV_n}(p)}(\delta^2_l+\delta^2_{l+1}).
\end{align}


On the other hand, by Definition \ref{dfn:fake sin/cos}, Lemma \ref{lem:deriv.covering.map} and \eqref{eqn:sncn.near.sing}, for any $l\in\text{Sing}_{\bV_n}(p)$, we have 
\begin{align}\label{eqn:Jac.Phi.Sing}
\left\|D\Phi(q)\left(\frac{\del}{\del\theta_l}\right)\right\|\left\|D\Phi(q)\left(\frac{\del}{\del\theta_{l+1}}\right)\right\|
\geq& \frac{|\cn'(\theta_l+\delta_l)\sn'(\theta_{l+1}+\delta_{l+1})|}{2(n+2)!}(\cn^2(\theta_l+\delta_l)+\sn^2(\theta_{l+1}+\delta_{l+1}))\nonumber\\
=&\frac{1}{2(n+2)!}(\delta_l^2+\delta^2_{l+1}).
\end{align}
By Definition \ref{dfn:fake sin/cos}, Lemma \ref{lem:deriv.covering.map} and Definition \ref{dfn:good.nbhd}, there exists a constant $C(n,\cU_p)>0$ such that for any $q\in\cU_p$ and for any $l\in\bV_n\setminus\text{Sing}^+_{\bV_n}(p)$, we have
\begin{align}\label{eqn:Jac.Phi.reg}
\left\|D\Phi(q)\left(\frac{\del}{\del\theta_l}\right)\right\|\geq& \frac{\displaystyle\min_{\theta\in[0,\pi/2]}(|\sn'(\theta)|+|\cn'(\theta)|)}{2\sqrt{2(n+2)!}}\min_{\substack{s\in\{0,\dots,n\}\setminus\text{Sing}_{\bV_n}(p)\\
(\zeta_1,\dots,\zeta_n)\in\overline{\cU_p}\\
\zeta_0=0,\zeta_{n+1}=\pi/2}}\sqrt{\cn^2(\zeta_s)+\sn^2(\zeta_{s+1})}\nonumber\\
=:&C(n,\cU_p)>0.
\end{align} 
(Here, the assumptions $\zeta_0=0$ and $\zeta_{n+1}=\pi/2$ follows from \eqref{eqn:covering.formula.simp}, which is used in the statement of Lemma \ref{lem:deriv.covering.map}). By Lemma \ref{lem:deriv.covering.map}, \eqref{eqn:Jac.Phi.Sing} and \eqref{eqn:Jac.Phi.reg}, we have
\begin{align}\label{eqn:Jac.Phi}
\text{Jac}\Phi(q)=&\prod_{l=1}^n\left\|D\Phi(q)\left(\frac{\del}{\del\theta_l}\right)\right\|\nonumber\\
=&\left(\prod_{l\in\bV_n\setminus\text{Sing}^+_{\bV_n}(p)}\left\|D\Phi(q)\left(\frac{\del}{\del\theta_l}\right)\right\|\right)\left(\prod_{l\in\text{Sing}_{\bV_n}(p)}\left\|D\Phi(q)\left(\frac{\del}{\del\theta_l}\right)\right\|\left\|D\Phi(q)\left(\frac{\del}{\del\theta_{l+1}}\right)\right\|\right)\nonumber\\
\geq& \frac{(C(n,\cU_p))^{n-|\text{Sing}_{\bV_n}^+(p)|}}{(2(n+2)!)^{|\text{Sing}_{\bV_n}(p)|}}\prod_{l\in\text{Sing}_{\bV_n}(p)}(\delta^2_l+\delta^2_{l+1}).
\end{align}
Since $\text{Jac}\Xi_{\Gamma_n;\bV_n;p}(q)=\text{Jac}\Phi(q)\cdot(\text{Jac}\Xi_{\Gamma_n;\bV_n;p}(q))^{-1}$ holds for any $q$ in a dense subset of $\cU_p$, by \eqref{eqn:Jac.Xi} and \eqref{eqn:Jac.Phi}, $\text{Jac}\Xi_{\Gamma_n;\bV_n;p}(q)>0$ for any $q\in\cU_p$. This completes the proof of Proposition \ref{prop:quotient.is.sphere}.

\section{Existence of special caps}\label{sect:special.caps.exist}
In this section, we prove the following proposition:
\begin{prop}\label{prop:special.caps.exist}
For any $k\in\ZZ_+$ and for any $R>0$, there exists a special $k$-cap with Euclidean radius $R$. (See Definition \ref{dfn:special caps}.)
\end{prop}
Before proving Proposition \ref{prop:special.caps.exist}, we first recall a few elementary facts without proofs.
\begin{fact}\label{fact.warped.prod.curv}
Let $U$ be an open subset of $\RR^2$ equipped with the warped-product metric $g=dx^2+(\rho(x))^2dy^2$. Then the sectional curvature of $(U,g)$ at $(x,y)$ is $-\rho''(x)/\rho(x)$.
\end{fact}
\begin{fact}\label{fact:stupid}
There exists a bounded smooth function $f_1:[0,\infty)\to\RR$ such that the following holds:
\begin{itemize}
\item $f_1(x)=-3\pi\sin(2\pi x)/2$ on $[0,1/4]$.
\item $f_1(x)$ is increasing on $(1/4,1/2]$ with $f_1(1/2)=0$ and $\int_0^{1/2}f_1(x)dx=-1$.
\item $f_1(x)=0$ for any $x\in[1,\infty)$, and $f_1(x)>0$ for any $x\in(1/2,1)$
\item $\int_0^1f_1(x)dx=-1+\frac{1}{k}$.
\end{itemize}
\end{fact}
As a direct corollary of Fact \ref{fact:stupid}, the following holds:
\begin{fact}\label{fact:stupid.2}
For any $R\in\RR_+$, let 
\begin{align}\label{eqn:f.R.rho.R}
f_R(x):=f_1(x/R)/R\text{ and }\rho_R(x):=x+\int_0^x\int_0^sf_R(t)dtds.
\end{align} 
Then the following holds:
\begin{enumerate}[ref=\ref{fact:stupid.2}(\arabic*)]
\item $\rho_R''(x)=f_R(x)$. \label{fact:stupid.2-1}
\item $\rho_R'(x)\geq 0$, with $\rho_R'(x)=0$ if and only if $x=R/2$.\label{fact:stupid.2-2}
\item $-\rho_R''(x)/\rho_R(x)$ is positive and decreasing on $(0,R/2)$; $-\rho_R''(x)/\rho_R(x)<0$ on $(R/2,R)$; $-\rho_R''(x)/\rho_R(x)=0$ on $[R,\infty)$.\label{fact:stupid.2-3}
\item $\rho_R(x)=\frac{x}{k}-RL$ on $[R,\infty)$, where $L= -\int_0^1\int_0^sf_1(t)dtds-1+\frac{1}{k}\in(-1+\frac{1}{k},\frac{1}{k})$. \label{fact:stupid.2-4}
\end{enumerate}
\end{fact}
Now, we prove Proposition \ref{prop:special.caps.exist}.
\begin{proof}[Proof of Proposition \ref{prop:special.caps.exist}]
For any $T>0$, we let $C_T:=\{z\in\CC:|z|<T\}$. For simplicity, we write $\widetilde{R}:=R+\frac{kRL}{1-kL}=\frac{R}{1-kL}$, where $L$ is introduced in \eqref{eqn:f.R.rho.R}.

We first show that $(C_{\widetilde{R}},dr^2+(\rho_{\widetilde{R}}(r))^2d\theta^2)$ is an extended cap in the sense of Definition \ref{dfn:caps}, where $z=re^{i\theta}$. (See \eqref{eqn:f.R.rho.R} for the definition of $\rho_{\widetilde{R}}(\cdot)$.) Observe the following:
\begin{itemize}
\item There exists a smooth function $J:[0,\infty)\to\RR$ such that $\rho_{\widetilde{R}}(r)=r(1+r^2J(r^2))$. In particular, $dr^2+(\rho_{\widetilde{R}}(r))^2d\theta^2$ is a smooth Riemannian metric. (This is due to \eqref{eqn:f.R.rho.R} and Fact \ref{fact:stupid.2-2}.)
\item $\{z:|z|=\widetilde{R}/2\}$ is a closed geodesic in $(C_{\widetilde{R}},dr^2+(\rho_{\widetilde{R}}(r))^2d\theta^2)$. (This is due to Fact \ref{fact:stupid.2-2}.)
\item The sectional curvature of $(C_{\widetilde{R}},dr^2+(\rho_{\widetilde{R}}(r))^2d\theta^2)$ at $z=re^{i\theta}$ is $-\rho_{\widetilde{R}}''(r)/\rho_{\widetilde{R}}(r)$. (This is due to Fact \ref{fact.warped.prod.curv}.)
\item $-\rho_{\widetilde{R}}''(r)/\rho_{\widetilde{R}}(r)$ is positive and decreasing on $(0,\widetilde{R}/2)$; $-\rho_{\widetilde{R}}''(r)/\rho_{\widetilde{R}}(r)$ is negative on $(\widetilde{R}/2,\widetilde{R})$. (This is due to Fact \ref{fact:stupid.2-3}.)
\end{itemize}
These verify that $(C_{\widetilde{R}},dr^2+(\rho_{\widetilde{R}}(r))^2d\theta^2)$ is an extended cap in the sense of Definition \ref{dfn:caps}.

Choose $0<\epsilon\ll 1$ and a smooth, increasing function $h:[0,\infty)\to[0,\infty)$ such that the following holds:
\begin{itemize}
\item $h(r)=r^k$ for any $r\in[0,\epsilon R]$.
\item $h(r)=r+\frac{kLR}{1-kL}$ for any $r\geq R(1-\epsilon)$.
\end{itemize}
Let $F:C_R\to \CC$ be the smooth map defined as $F(z)=h(|z|)|z|^{-k}z^k$. Then, $F(C_R)=C_{R/(1-kL)}=C_{\widetilde{R}}$. Moreover, by Fact \ref{fact:stupid.2-4}, $(C_R, F^*(dr^2+(\rho_{\widetilde{R}}(r))^2d\theta^2))$ is a special $k$-cap with Euclidean radius $R$ in the sense of Definition \ref{dfn:special caps}.
\end{proof}

\newpage
\section{Details of certain claims without proofs}\label{sect:garbage}
This appendix collects proofs and additional details for a number of claims that are stated without proof in the main text. Most of these arguments are elementary or straightforward in nature, but their details can be repetitive and lengthy, and would obscure the main line of the paper.

The arguments presented here are not intended to be the most efficient or elegant proofs of the corresponding statements. Rather, this appendix is meant to serve as a reference: when a reader wishes to verify a particular claim or clarify a detail omitted from the main text, additional details may be found here. Accordingly, this appendix is intended primarily for consultation rather than continuous reading; its length also reflects why many of these details have been omitted from the main text.
\subsection{Verifying a claim in the remark after Definition \ref{dfn:consecutive}} We  prove the following claim:
\begin{center}
\emph{
The map $\{A_1,\dots, A_l\}\to E_A(A_1,\dots, A_l)$ gives a bijection between the collection of consecutive partitions of $A$ and the collection of subsets of $\cE(A)$.}
\end{center}
For any two subsets $B_1,B_2$ of $\RR$, we say that $B_1<B_2$ ($B_1>B_2$ resp.) if  $b_1<b_2$ ($b_1>b_2$ resp.) for any $b_1\in B_1$ and $b_2\in B_2$. Then, for any consecutive partition $\{A_1,\dots, A_l\}$ of $A$ and for any distinct $i,j\in\{1,\dots,l\}$, either $A_i<A_j$ or $A_j<A_i$. Indeed, by Definition \ref{dfn:consecutive}, $A_1,\dots,A_l$ are all consecutive subsets. In particular, $A_i=\{n\in\ZZ:\min A_i\leq n\leq \max A_i\}$ and $A_j=\{n\in\ZZ:\min A_j\leq n\leq \max A_j\}$. Since $\{A_1,\dots,A_l\}$ forms a partition of $A$, $A_i\cap A_j=\emptyset$. Hence, either $\max A_i<\min A_j$ or $\min A_i>\max A_j$. Equivalently, either $A_i<A_j$ or $A_j<A_i$.

Following the discussion of the previous paragraph, we assume WLOG that $A_1<A_2<\dots< A_l$. Since $\{A_1,\dots, A_l\}$ forms a partition of $A$ and that $A$ is consecutive, we have $\min A_1=\min A$, $1+\max A_j=\min A_{j+1}$ for any $j\in\{1,\dots, l-1\}$, and $\max A_l=\max A$. We then define two maps $\Phi:\{\text{consecutive partitions of }A\}\to \{E: E\subset \cE(A)\}$ and $\Psi:\{E: E\subset \cE(A)\}\to \{\text{consecutive partitions of }A\}$ as follows: for any consecutive partition $\{A_1,\dots, A_l\}$ of $A$, we let
$$\Phi(\{A_1,\dots, A_l\})=\{\max A_j:1\leq j\leq l-1\}=E_A(A_1,\dots, A_l).$$
For any nonempty subset $E$ of $\cE(A)$ with $E=\{j_1,\dots, j_l\}$ and $j_1<\dots <j_l$, we let 
\begin{align*}
\Psi(E)=\Psi(\{j_1,\dots, j_l\})=\{&\{n\in\ZZ:\min A\leq n\leq j_1\},\{n\in\ZZ:j_1+1\leq n\leq j_2\},\\
&\dots,\{n\in\ZZ:j_{l-1}+1\leq n\leq j_l\} ,\{n\in\ZZ:j_l+1\leq n\leq \max A\}\}.
\end{align*}
We also let $\Psi(\emptyset)=\{A\}$, the consecutive partition of $A$ with single atom $A$. Recall from Definition \ref{dfn:consecutive} that $\cE(A)=A\setminus\max A=\{n\in\ZZ:\min A\leq n\leq \max A-1\}$. It is not hard to verify that both $\Phi$ and $\Psi$ are well-defined. Moreover, it is not hard to verify that $\Psi\circ\Phi$ is the identity map on the set of consecutive partitions of $A$ and that $\Phi\circ\Psi$ is the identity map on the set of all subsets of $\cE(A)$. This verifies the claim.
\subsection{Proof of Lemma \ref{lem:properties.good.subgp}}
\begin{enumerate}
\item Commutativity of $\Gamma_n$ follows directly from \eqref{eqn:commute}. Notice that by \eqref{eqn:commute}, we have $\gamma_i^2=\id$. Therefore the map $(i_1,\dots,i_{n-1})\to\gamma_1^{i_1}\circ\cdots\circ\gamma_{n-1}^{i_{n-1}}$ is a well-defined, surjective homomorphism from $(\ZZ/2\ZZ)^{n-1}$ to $\Gamma_n$. It remains to show that $(i_1,\dots,i_{n-1})\to\gamma_1^{i_1}\circ\cdots\circ\gamma_{n-1}^{i_{n-1}}$ is injective. For any $(i_1,\dots,i_{n-1})$, let $j\in\{1,\dots,n-1\}$ be the smallest element such that $i_j\equiv 1(\text{mod }2)$. Then, by \eqref{eqn:commute} and \eqref{eqn:generator}, $\pr^{\bV_n}_j\circ(\gamma_1^{i_1}\circ\cdots\circ\gamma_{n-1}^{i_{n-1}})=\kappa_\pi\circ\pr^{\bV_n}_j$. In particular, this implies that $\gamma_1^{i_1}\circ\cdots\circ\gamma_{n-1}^{i_{n-1}}\neq \id$. Hence the map $(i_1,\dots,i_{n-1})\to\gamma_1^{i_1}\circ\cdots\circ\gamma_{n-1}^{i_{n-1}}$ is injective. This finishes the proof of Lemma \ref{lem:properties.good.subgp-(1)}.
\item We define $\Phi_A:(\ZZ/2\ZZ)^{k-1}\to\Gamma_A$ as $\Phi_A(i_1,\dots, i_{k-1})=\gamma_{m+1}^{i_1}\circ\cdots\circ\gamma_{m+k-1}^{i_{k-1}}$. Then, following the same arguments in the proof of Lemma \ref{lem:properties.good.subgp-(1)}, $\Phi_A$ is an isomorphism. Hence $\Phi_{\bV_{k}}\circ\Phi_{A}^{-1}:\Gamma_A\to\Gamma_k$ is an isomorphism such that $\Phi_{\bV_{k}}\circ\Phi_{A}^{-1}(\gamma_{m+j})=\gamma_j$ for any $\gamma_j\in\cS_k$. Let $\Psi^A_{\bV_k}:\TT_A\to\TT_k$ such that $\pr^{\bV_k}_j(\Psi^A_{\bV_k}(\theta_{m+1},\dots,\theta_{m+k}))=\theta_{m+j}$ for any $(\theta_{m+1},\dots,\theta_{m+k})\in\TT_A$ and for any $j\in\{1,\dots,k\}$. (In other words, $\Psi^A_{\bV_k}(\theta_{m+1},\dots,\theta_{m+k})=(\theta_{m+1},\dots,\theta_{m+k})$, where the $\theta_{m+j}$ is the $j$-th entry of the right-hand-side.) Then for any $(\theta_{m+1},\dots,\theta_{m+k})\in\TT_A$  and for any $j\in\{1,\dots,k-1\}$, we have
\begin{align*}
\gamma_{j}(\Psi^A_{\bV_k}(\theta_{m+1},\dots,\theta_{m+k}))=&\gamma_j(\theta_{m+1},\dots,\theta_{m+k})\\
=&(\theta_{m+1},\dots,\theta_{m+j-1},\kappa_\pi(\theta_{m+j}),\kappa_0(\theta_{m+j+1}),\dots,\theta_{m+k})\\
=&\Psi^A_{\bV_k}(\theta_{m+1},\dots,\theta_{m+j-1},\kappa_\pi(\theta_{m+j}),\kappa_0(\theta_{m+j+1}),\dots,\theta_{m+k})\\
=&\Psi^A_{\bV_k}(\gamma_{m+j}(\theta_{m+1},\dots,\theta_{m+k})).
\end{align*}
This proves that $\gamma_j\circ\Psi^A_{\bV_k}=\Psi^A_{\bV_k}\circ\gamma_{m+j}$. Hence, by \eqref{eqn:Gamma.A}, we have $\Psi^A_{\bV_k}\circ\gamma=\Phi_{\bV_{k}}\circ\Phi_{A}^{-1}(\gamma)\circ\Psi^A_{\bV_k}$ for any $\gamma\in\Gamma_A$. This proves the equivalence of the $\Gamma_A$-action on $\TT_A$ and the $\Gamma_k$-action on $\TT^k$ under $\Phi_{\bV_{k}}\circ\Phi_{A}^{-1}$ and $\Psi^A_{\bV_k}$.
\item Recall that in the proof of Lemma \ref{lem:properties.good.subgp-(2)}, for any $A=\{m+1,\dots, m+k\}$, the map $\Phi_A:(\ZZ/2\ZZ)^{k-1}\to\Gamma_A$, defined as $\Phi_A(i_1,\dots, i_{k-1})=\gamma_{m+1}^{i_1}\circ\cdots\circ\gamma_{m+k-1}^{i_{k-1}}$, is an isomorphism. Then, by \eqref{eqn:good.subgroup.struct}, Definition \ref{dfn:good subgp} and the above, for any $E\subset A$, we have $\Phi^{-1}_A(\Gamma_A(E))=\{(i_1,\dots,i_{k-1})\in(\ZZ/2\ZZ)^{k-1}:i_{j}\equiv 0(\text{mod }2)\text{ for any }m+j\in E\}$. Similarly, by \eqref{eqn:Gamma.A} and the above, for any $A'\subset A$, we have $\Phi_A^{-1}(\Gamma_{A'})=\{(i_1,\dots,i_{k-1})\in(\ZZ/2\ZZ)^{k-1}:i_{j}\equiv 0(\text{mod }2)\text{ for any }m+j\in \cE(A)\setminus\cE(A')\}$. Hence
\begin{align*}
&\Phi_A^{-1}(\Gamma_A(E)\cap\Gamma_{A'})\\
=&\Phi_A^{-1}(\Gamma_A(E))\cap\Phi_A^{-1}(\Gamma_{A'})\\
=&\{(i_1,\dots,i_{k-1})\in(\ZZ/2\ZZ)^{k-1}:i_{j}\equiv 0(\text{mod }2)\text{ for any }m+j\in E\cup(\cE(A)\setminus\cE(A'))\}\\
=&\Phi_A^{-1}(\Gamma_A(E\cup(\cE(A)\setminus\cE(A')))).
\end{align*}
In particular, by \eqref{eqn:good.subgroup.struct}, $\Gamma_A(E)\cap\Gamma_{A'}$ is generated by $\{\gamma_i\in\cS_A:i\not\in E\cup(\cE(A)\setminus\cE(A'))\}=\{\gamma_i\in\cS_{A'}:i\not\in E\}=\{\gamma_i\in\cS_{A'}:i\not\in E\cap\cE(A')\}$. Hence, by Definition \ref{dfn:good subgp}, we have $\Gamma_A(E)\cap \Gamma_{A'}=\Gamma_{A'}(E\cap\cE(A'))=\Gamma_A(E\cup(\cE(A)\setminus\cE(A')))$.
\item Recall that in the proof of Lemma \ref{lem:properties.good.subgp-(2)}, for any $A=\{m+1,\dots, m+k\}$, the map $\Phi_A:(\ZZ/2\ZZ)^{k-1}\to\Gamma_A$, defined as $\Phi_A(i_1,\dots, i_{k-1})=\gamma_{m+1}^{i_1}\circ\cdots\circ\gamma_{m+k-1}^{i_{k-1}}$, is an isomorphism. Then, by \eqref{eqn:good.subgroup.struct}, Definition \ref{dfn:good subgp} and the above, for any $E\subset A$, we have $\Phi^{-1}_A(\Gamma_A(E))=\{(i_1,\dots,i_{k-1})\in(\ZZ/2\ZZ)^{k-1}:i_{j}\equiv 0(\text{mod }2)\text{ for any }m+j\in E\}$. Therefore,
\begin{align*}
&\Phi_A^{-1}(\Gamma_A(E_1)\cap\Gamma_{A}(E_2))\\
=&\Phi_A^{-1}(\Gamma_A(E_1))\cap\Phi_A^{-1}(\Gamma_{A}(E_2))\\
=&\{(i_1,\dots,i_{k-1})\in(\ZZ/2\ZZ)^{k-1}:i_{j}\equiv 0(\text{mod }2)\text{ for any }m+j\in E_1\cup E_2\}\\
=&\Phi_A^{-1}(\Gamma_A(E_1\cup E_2)).
\end{align*}
This proves that $\Gamma_A(E_1)\cap\Gamma_{A}(E_2)=\Gamma_A(E_1\cup E_2)$. Similarly, 
\begin{align*}
&\Phi_A^{-1}(\langle\Gamma_A(E_1),\Gamma_{A}(E_2)\rangle)\\
=&\langle\Phi_A^{-1}(\Gamma_A(E_1)),\Phi_A^{-1}(\Gamma_{A}(E_2))\rangle\\
=&\{(i_1,\dots,i_{k-1})\in(\ZZ/2\ZZ)^{k-1}:i_{j}\equiv 0(\text{mod }2)\text{ for any }m+j\in E_1\cap E_2\}\\
=&\Phi_A^{-1}(\Gamma_A(E_1\cap E_2)).
\end{align*}
This proves that $\langle\Gamma_A(E_1),\Gamma_{A}(E_2)\rangle=\Gamma_A(E_1\cap E_2)$. 

Assume in addition that $E_1\cup E_2=\cE(A)$, by \eqref{eqn:good.subgroup.struct} and the above discussions, we have $\Gamma_A(E_1)\cap\Gamma_A(E_2)=\{\id\}$. Since in addition that $\Gamma_A$ is commutative and that $\Gamma_A(E_1),\Gamma_A(E_2)$ are subgroups of $\Gamma_A$, the map $(h_1,h_2)\to h_1h_2$ from $\Gamma_A(E_1)\times\Gamma_A(E_2)$ to $\langle\Gamma_A(E_1),\Gamma_A(E_2)\rangle$ is an isomorphism. This finishes the proof of Lemma \ref{lem:properties.good.subgp-(4)}.
\item We first show that $\cE(A)$ is the disjoint union of $E$ and $\cE(A_1),\dots,\cE(A_l)$. Indeed, since $A_1,\dots, A_l$ are pairwise disjoint due to the assumption that they form a consecutive partition of $A$, by \eqref{eqn:cutpoint}, $\cE(A_1),\dots,\cE(A_l)$ are pairwise disjoint. By \eqref{eqn:cutpoint} and the remark after Definition \ref{dfn:consecutive}, $E\cap\cE(A_i)=\emptyset$ for any $i\in\{1,\dots,l\}$. It remains to show that $E\cup\cE(A_1)\cup\cdots\cup\cE(A_l)=\cE(A)$. For any $i\in\cE(A)$, if $\{i,i+1\}\not\subset A_j$ for any $j\in\{1,\dots,l\}$, then by the remark after Definition \ref{dfn:consecutive}, $i\in E$; if $\{i,i+1\}\subset A_j$ for some $j\in\{1,\dots,l\}$, then by \eqref{eqn:cutpoint}, $i\in\cE(A_j)$. Since $E,\cE(A_1),\dots,\cE(A_j)$ are subsets of $\cE(A)$, the above discussion shows that $E\cup\cE(A_1)\cup\cdots\cup\cE(A_l)=\cE(A)$.

Let $\Gamma=\Gamma_A(E')$ for some $E'\subset\cE(A)$. (See Definition \ref{dfn:good subgp}.) Then by Lemma \ref{lem:properties.good.subgp-(3)} and Lemma \ref{lem:properties.good.subgp-(4)}, $\Gamma\cap\Gamma'=\Gamma_A(E\cup E')$ and $\Gamma'\cap\Gamma_{A_j}=\Gamma_A(E'\cup(\cE(A)\setminus\cE(A_j)))$ for any $j\in\{1,\dots, l\}$. For simplicity, we let $\Gamma(j)=\langle\Gamma_{A_1}\cap\Gamma',\dots,\widehat{\Gamma_{A_j}\cap\Gamma'},\dots,\Gamma_{A_l}\cap\Gamma'\rangle$. By consecutive application of Lemma \ref{lem:properties.good.subgp-(4)} and the fact that $\cE(A)$ is the disjoint union of $E$ and $\cE(A_1),\dots,\cE(A_l)$, we have 
$$\Gamma(j)=\Gamma_A\left(\bigcap_{1\leq i\leq l, i\neq j}E'\cup(\cE(A)\setminus\cE(A_i))\right)=\Gamma_A(E'\cup E\cup\cE(A_j))\text{ for any }j\in\{1,\dots,l\}$$
and
$$\Gamma'\cap\Gamma=\langle  \Gamma_{A_1}\cap\Gamma',\dots,\Gamma_{A_l}\cap\Gamma' \rangle.$$
In particular, by Lemma \ref{lem:properties.good.subgp-(4)}, for any $j\in\{1,\dots,l\}$, we have
$$\Gamma(j)\cap(\Gamma_{A_j}\cap\Gamma')=\Gamma_A((E'\cup E\cup\cE(A_j))\cup (E'\cup(\cE(A)\setminus\cE(A_j))))=\Gamma_A(\cE(A))=\{\id\}.$$
Since $\Gamma_A$ is commutative due to Lemma \ref{lem:properties.good.subgp-(1)}, the above equations imply that $(\Gamma_{A_1}\cap\Gamma')\times \cdots\times (\Gamma_{A_l}\cap\Gamma')$ is isomorphic to $\Gamma\cap\Gamma'$ via the isomorphism $(h_1,\dots,h_l)\to h_1h_2\dots h_l$. Notice that for any $\theta=(\eta_1,\dots,\eta_l)\in\TT_A$ with $\eta_j\in\TT_{A_j}$ and for any $h=h_1h_2\dots h_l\in\Gamma\cap\Gamma'$ with $h_j\in\Gamma_{A_j}\cap\Gamma'$, following the convention after \eqref{eqn:Gamma.A}, we have $h(\theta)=(h_1(\eta_1),\dots,h_l(\eta_l))$. Therefore, the $\Gamma\cap\Gamma'$-action on $\TT_A$ is equivalent to the product of the $\Gamma_{A_j}\cap\Gamma'$-action on $\TT_{A_j}$ via the isomorphism $(h_1,\dots,h_l)\to h_1h_2\dots h_l$ from $(\Gamma_{A_1}\cap\Gamma')\times \cdots\times (\Gamma_{A_l}\cap\Gamma')$ to $\Gamma\cap\Gamma'$, and the identity map on $\TT_A$.
\end{enumerate}

\subsection{Verifying a claim in Lemma \ref{lem:fixed points of gamma}}
At the end of the proof of Lemma \ref{lem:fixed points of gamma}, we claimed that ``the rest of the proof follows directly from \eqref{eqn:commute} and \eqref{eqn:generator}''. Since the proof of Lemma \ref{lem:fixed points of gamma} has already addressed why any $\gamma\in\Gamma_A$ with fixed points must be of the form $\gamma=\gamma_{i_1}\circ\cdots\circ\gamma_{i_q}$ for some $i_1,\dots,i_q\in\cE(A)$ with $i_j<i_{j+1}-1$ for any $j\in\{1,\dots,q-1\}$. It remains to prove that for such $\gamma$, the set of fixed points of $\gamma$ acting on $\TT_A$ is $\bigcap_{j=1}^q\Omega_{A}(i_j)$.

Indeed, for any $x\in\TT_A$ and for any $i\in A$, by \eqref{eqn:commute} and \eqref{eqn:generator}, we have 
$$\pr^A_i(\gamma(x))=\begin{cases}
\pr^A_i(x),~&\text{if }i\not\in\{i_1,i_1+1,i_2,i_2+1,\dots,i_q,i_q+1\},\\
\kappa_\pi(\pr^A_i(x)),~&\text{if }i\in\{i_1,\dots,i_q\},\\
\kappa_0(\pr^A_i(x)),~&\text{if }i\in\{i_1+1,\dots,i_q+1\}. 
\end{cases}$$ 
Hence $\gamma(x)=x$ if and only if $(\kappa_\pi(\pr^A_{i_j}(x)),\kappa_0(\pr^A_{i_j+1}(x)))=(\pr^A_{i_j}(x),\pr^A_{i_j+1}(x))$ for any $j\in\{1,\dots,q\}$. By \eqref{jth singularity}, this is equivalent to $x\in\Omega_A(i_j)$ for any $j\in\{1,\dots,q\}$. Hence the set of fixed points of $\gamma$ on $\TT_A$ is exactly $\TT_A$ is $\bigcap_{j=1}^q\Omega_{A}(i_j)$.

\subsection{Proof of the formula for $\text{Stab}_\Gamma(p)$ in the remark after Definition \ref{sing index}}
For any $\gamma\in\Gamma$, if $\gamma$ have fixed points, then by Lemma \ref{lem:fixed points of gamma}, $\gamma=\gamma_{i_1}\circ\gamma_{i_2}\circ\cdots \circ\gamma_{i_q}$ for some $i_1,\dots,i_q\in\cE(A)$ such that $i_j+1<i_{j+1}$ for any $j\in\{1,\dots, q-1\}$. Moreover, by Definition \ref{dfn:good subgp}, Lemma \ref{lem:properties.good.subgp-(1)} and Lemma \ref{lem:properties.good.subgp-(2)}, $\gamma_{i_1},\dots,\gamma_{i_q}\in\Gamma$. Hence, by Lemma \ref{lem:fixed points of gamma} and \eqref{total singularity subgp}, if $p\not\in\Omega_{\Gamma;A}$, then $\text{Stab}_\Gamma(p)=\{\id\}$.

For any $p\in\Omega_{\Gamma;A}$, if $p$ is fixed by $\gamma=\gamma_{i_1}\circ\cdots\circ\gamma_{i_s}$ for some $\gamma_{i_1},\dots,\gamma_{i_s}\in\Gamma$ such that $i_j+1<i_{j+1}$ for any $j\in\{1,\dots, s-1\}$, then by \eqref{eqn:generator} and \eqref{jth singularity}, $p\in\Omega_A(i_j)$ for any $j\in\{1,\dots,s\}$. Hence, by Definition \ref{sing index}, $i_1,\dots, i_s\in\Sing_{\Gamma;A}(p)$. This proves that $\text{Stab}_\Gamma(p)\subset\langle \gamma_i:i\in\Sing_{\Gamma;A}(p)\rangle$. On the other hand, for any $i\in\Sing_{\Gamma;A}(p)$, by Definition \ref{sing index}, $\gamma_i\in\Gamma$ and $p\in\Omega_A(i)$. Therefore, by \eqref{eqn:generator} and \eqref{jth singularity}, $\gamma_i(p)=p$. Hence $\text{Stab}_\Gamma(p)=\langle \gamma_i:i\in\Sing_{\Gamma;A}(p)\rangle$. This completes the proof of the formula for $\text{Stab}_\Gamma(p)$.

\subsection{Verifying that $\cU_p$ in the remark of Definition \ref{dfn:good.nbhd} is a good neighborhood}
Let $A=\{m+1,\dots,m+k\}$ and $p=(\theta_{m+1},\dots, \theta_{m+k})$ Recall in \eqref{eqn:good.nbhd}, we have 
$$\cU_p=\left(\prod_{i\in A\setminus\text{Sing}_{A}^+(p)}B_i(\theta_i,\epsilon'))\right)\times\left(\prod_{i\in\text{Sing}_{A}(p)}B_{i,i+1}((\theta_i,\theta_{i+1}),\epsilon)\right).$$
We would like to show that for any $\epsilon\in(0,\pi/8)$ and for any $0<\epsilon'\ll 1$ sufficiently small (depending on $p$), $\cU_p$ is a good neighborhood of $p$ in the sense of Definition \ref{dfn:good.nbhd}.
\begin{enumerate}
\item We let $0<\epsilon'\ll \pi/8$ so that for any $i\in A\setminus\Sing_A(p)$ and for any $N_1,N_2\in\ZZ$, $|\theta_i-\kappa_\pi(\theta_i)+2N_1\pi|+|\theta_{i+1}-\kappa_0(\theta_{i+1})+2N_2\pi|>5\epsilon'$. (If such an $\epsilon'$ cannot be found, then $p\in\Omega_A(i)$ for some $i\in A\setminus\Sing_A(p)$ due to \eqref{jth singularity}. Hence $i\in\Sing_A(p)$ according to Definition \ref{sing index}. This leads to a contradiction.)
Then, for any $i\in A\setminus((\Sing_A(p)-1)\cup\Sing_A(p)\cup(\Sing_A(p)+1))$, for any $q=(\theta_{m+1}',\dots, \theta_{m+k}')\in\overline{\cU}_p$ and for any $N_1,N_2\in\ZZ$, we have 
\begin{align*}
&|\theta_i'-\kappa_\pi(\theta_i')+2N_1\pi|+|\theta_{i+1}'-\kappa_0(\theta_{i+1}')+2N_2\pi|\\
\geq &|\theta_i-\kappa_\pi(\theta_i)+2N_1\pi|+|\theta_{i+1}-\kappa_0(\theta_{i+1})+2N_2\pi|\\
&-|\theta_i'-\kappa_\pi(\theta_i')-(\theta_i-\kappa_\pi(\theta_i))|-|\theta_{i+1}'-\kappa_0(\theta_{i+1}')-(\theta_{i+1}-\kappa_0(\theta_{i+1}))|\\
> &5\epsilon'-2|\theta_i'-\theta_i|-2|\theta_{i+1}'-\theta_{i+1}|\geq \epsilon',
\end{align*}
where we used the definition of $\cU_p$ in the last inequality. Recall that $\Sing_A(q)\cap(\Sing_A(q)+1)=\emptyset$, we have $\Sing_A(q)\cap(\Sing_A(q)-1)=\emptyset$. For any $i\in A\cap(\Sing_A(p)+1)$, we have $\theta_i=0$ or $\pi$ mod $2\pi$. Hence, for any $q=(\theta_{m+1}',\dots, \theta_{m+k}')\in\overline{\cU}_p$ and for any $N_1,N_2\in\ZZ$, we have 
\begin{align*}
&|\theta_i'-\kappa_\pi(\theta_i')+2N_1\pi|+|\theta_{i+1}'-\kappa_0(\theta_{i+1}')+2N_2\pi|\\
\geq &|\theta_i-\kappa_\pi(\theta_i)+2N_1\pi|-|\theta_i'-\kappa_\pi(\theta_i')-(\theta_i-\kappa_\pi(\theta_i))|\\
\geq &\pi-2|\theta_i'-\theta_i|\geq \pi-\frac{\pi}{4}>\epsilon',
\end{align*}
where we used the definition of $\cU_p$ in the second last inequality. Similarly, for any $i\in A\cap(\Sing_A(p)-1)$, we have $\theta_{i+1}=\pi/2$ or $-\pi/2$ mod $2\pi$. Hence, for any $q=(\theta_{m+1}',\dots, \theta_{m+k}')\in\overline{\cU}_p$ and for any $N_1,N_2\in\ZZ$, we have 
\begin{align*}
&|\theta_i'-\kappa_\pi(\theta_i')+2N_1\pi|+|\theta_{i+1}'-\kappa_0(\theta_{i+1}')+2N_2\pi|\\
\geq &|\theta_{i+1}-\kappa_0(\theta_{i+1})+2N_2\pi|\\
&-|\theta_{i+1}'-\kappa_0(\theta_{i+1}')-(\theta_{i+1}-\kappa_0(\theta_{i+1}))|\\
\geq &\pi-2|\theta_{i+1}'-\theta_{i+1}| \geq \pi-\frac{\pi}{4}>\epsilon',
\end{align*}
where we used the definition of $\cU_p$ in the second last inequality. The above shows that for any $q\in\overline{\cU_p}$ and for any $i\in A\setminus \Sing_A(p)$, we have $q\not\in\Omega_A(i)$. Hence, $\Sing_A(q)\subset \Sing_A(p)$.

\item If $\gamma\in\Gamma_{A'}$ does not have fixed points, then there exists some $i\in A'$ such that $\pr^A_i\circ\gamma(p)=\pr^A_i(p)+\pi$ for any $p\in\TT_A$. (This is due to \eqref{eqn:commute}, \eqref{eqn:generator} and Lemma \ref{lem:fixed points of gamma}.) By the definition of $\cU_p$, we have 
$$\pr^A_{A'}(\overline{\gamma\cU_p})\cap\pr^A_{A'}(\overline{\cU_p})=\emptyset.$$ Therefore, we only consider $\gamma\in\Gamma_{A'}$ which admits fixed points. In particular, by Lemma \ref{lem:fixed points of gamma}, $\gamma=\gamma_{i_1}\circ\cdots\circ\gamma_{i_s}$ for some $i_1<\cdots<i_s$ such that $i_{j}+1<i_{j+1}$ for any $j\in\{1,\dots,s-1\}$.

When $\pr^A_{A'}(\gamma p)= \pr^A_{A'}(p)$, by the assumption that $\gamma\in\Gamma_{A'}$, we have $\gamma p= p$. Hence by Lemma \ref{lem:fixed points of gamma} and the remark after Definition \ref{sing index}, we have $i_1,\dots,i_s\in\Sing_{A'}(\pr^A_{A'}(p))=\Sing_A(p)\cap\cE(A')$. In this case, it suffices to prove $\gamma\cU_p=\cU_p$ when $s=1$, and this follows directly from the fact that $\gamma_i(B_{i,i+1}(\pr^A_{i,i+1}(p)),\epsilon)=(B_{i,i+1}(\pr^A_{i,i+1}(p)),\epsilon)$.


When $\pr^A_{A'}(\gamma p)\neq \pr^A_{A'}(p)$, by the previous case and the fact that $\Gamma_{A'}$ is commutative, we assume without loss of generality that $i_1\not\in\Sing_A(p)$. To avoid cumbersome notations, we let $i=i_1$. We discuss three subcases:

\textbf{Case 1}: When $i\not\in\Sing_A(p)\pm1$, we have $i,i+1\not\in\Sing_A^+(p)$. For any $q=(\theta_{m+1}',\dots, \theta_{m+k}')\in\overline{\cU_p}$, we have 
\begin{align*}
&|\pr^A_i(\gamma(q))-\pr^A_i(p)|+|\pr^A_{i+1}(\gamma(q))-\pr^A_{i+1}(p)|\\
=&\left(\min_{N\in\ZZ}|\kappa_\pi(\theta_i')-\theta_i+2N\pi|\right)+\left(\min_{N\in\ZZ}|\kappa_0(\theta_{i+1}')-\theta_{i+1}+2N\pi|\right)\\
\geq& -|\kappa_\pi(\theta_i')-\kappa_\pi(\theta_i)|+\min_{N\in\ZZ}|\kappa_\pi(\theta_i)-\theta_i+2N\pi|\\
&-|\kappa_0(\theta_{i+1}')-\kappa_0(\theta_{i+1})|+\min_{N\in\ZZ}|\kappa_0(\theta_{i+1})-\theta_{i+1}+2N\pi|\\
=& -|\theta_i'-\theta_i|+\min_{N\in\ZZ}|\kappa_\pi(\theta_i)-\theta_i+2N\pi|\\
&-|\theta_{i+1}'-\theta_{i+1}|+\min_{N\in\ZZ}|\kappa_0(\theta_{i+1})-\theta_{i+1}+2N\pi|\\
\geq&-2\epsilon'+5\epsilon'=3\epsilon'>2\epsilon'.
\end{align*}
where we used the definition of $\cU_p$ and the assumption in (1) regarding $\epsilon'$ in the second last inequality.
This implies that $\pr^A_{i,i+1}(\gamma(q))\not\in\pr^A_{i,i+1}(\overline{\cU_p})$. Hence,  $\pr^A_{A'}(\overline{\gamma\cU_p})\cap\pr^A_{A'}(\overline{\cU_p})=\emptyset.$

\textbf{Case 2}: When $i\in\Sing_A(p)+1$, we have $\theta_i=0$ or $\pi$ mod $2\pi$. Hence, for any $q=(\theta_{m+1}',\dots, \theta_{m+k}')\in\overline{\cU_p}$, we have 
\begin{align*}
&|\pr^A_i(\gamma(q))-\pr^A_i(p)|\\
=&\left(\min_{N\in\ZZ}|\kappa_\pi(\theta_i')-\theta_i+2N\pi|\right)\\
\geq& -|\kappa_\pi(\theta_i')-\kappa_\pi(\theta_i)|+\min_{N\in\ZZ}|\kappa_\pi(\theta_i)-\theta_i+2N\pi|\\
=& -|\theta_i'-\theta_i|+\pi
\geq-\epsilon+\pi\geq 3\pi/4>2\epsilon+2\epsilon'.
\end{align*}
This implies that $\pr^A_{i}(\gamma(q))\not\in\pr^A_{i}(\overline{\cU_p})$. Hence,  $\pr^A_{A'}(\overline{\gamma\cU_p})\cap\pr^A_{A'}(\overline{\cU_p})=\emptyset.$

\textbf{Case 3}: When $i\in\Sing_A(p)-1$, we have $\theta_{i+1}=\pm\pi/2$ mod $2\pi$. Hence, for any $q=(\theta_{m+1}',\dots, \theta_{m+k}')\in\overline{\cU_p}$, we have 
\begin{align*}
&|\pr^A_{i+1}(\gamma(q))-\pr^A_{i+1}(p)|\\
=&\left(\min_{N\in\ZZ}|\kappa_\pi(\theta_{i+1}')-\theta_{i+1}+2N\pi|\right)\\
\geq& -|\kappa_\pi(\theta_{i+1}')-\kappa_\pi(\theta_{i+1})|+\min_{N\in\ZZ}|\kappa_\pi(\theta_{i+1})-\theta_{i+1}+2N\pi|\\
=& -|\theta_{i+1}'-\theta_{i+1}|+\pi
\geq-\epsilon+\pi\geq 3\pi/4>2\epsilon+2\epsilon'.
\end{align*}
This implies that $\pr^A_{i+1}(\gamma(q))\not\in\pr^A_{i+1}(\overline{\cU_p})$. Hence,  $\pr^A_{A'}(\overline{\gamma\cU_p})\cap\pr^A_{A'}(\overline{\cU_p})=\emptyset.$

\item This is obvious from the definition of $\cU_p$ once $\epsilon'$ is assumed to be smaller than $\pi/8$.
\end{enumerate}

\subsection{Proof of \eqref{eqn:local.decomp.proj-1}}
Let $p=(\theta_{m+1},\dots,\theta_{m+k})\in\TT_A$ and $q=(\theta_{m+1}+\delta_{m+1},\dots,\theta_{m+k}+\delta_{m+k}), x=(\theta_{m+1}+\delta_{m+1}',\dots,\theta_{m+k}+\delta_{m+k}')\in\cU_p$ with $\delta_j,\delta_j'\in[-\pi/8,\pi/8]$ for any $j\in \{m+1,\dots,m+k\}=A$, where $\cU_p$ is a good neighborhood of $p$. Then by \eqref{eqn:z.to.z2} and the comments after \eqref{eqn:z.to.z2}, $\Xi_{\Gamma;A;p}(x)=\Xi_{\Gamma;A;p}(q)$ if and only if the following holds:
\begin{itemize}
\item $\delta_i=\delta_i'$ for any $i\in A\setminus \Sing^+_{\Gamma;A}(p)$.
\item $(\delta_i+\mathbf{i}\delta_{i+1})^2=(\delta_i'+\mathbf{i}\delta_{i+1}')^2$ for any $i\in\Sing_{\Gamma;A}(p)$, where $\mathbf{i}^2=-1$. This is equivalent to $(\delta_i,\delta_{i+1})=\pm(\delta_i',\delta_{i+1}')$ for any $i\in\Sing_{\Gamma;A}(p)$.
\end{itemize}
On the other hand, by the remark after Definition \ref{sing index}, $\text{Stab}_\Gamma(p)=\langle\gamma_i:i\in\Sing_{\Gamma;A}(p)\rangle$. Moreover, by \eqref{eqn:generator}, for any $i\in\Sing_{\Gamma;A}(p)$, the following holds:
\begin{itemize}
\item $\pr^A_j(\gamma_i(q))=\theta_j+\delta_j=\pr^A_j(q)$ for any $j\in A\setminus\{i,i+1\}$.
\item By \eqref{jth singularity} and Definition \ref{sing index}, $\theta_i\in\{\pi/2,3\pi/2\}$ and $\theta_{i+1}\in\{0,\pi\}$. In addition, $\pr^A_{i,i+1}(\gamma_i(q))=(\kappa_\pi(\theta_i+\delta_i),\kappa_0(\theta_{i+1}+\delta_{i+1}))=(\theta_i-\delta_i,\theta_{i+1}-\delta_{i+1})$.
\end{itemize}
Hence, \eqref{eqn:local.decomp.proj-1} follows from the second property of Definition \ref{dfn:good.nbhd} (with $A'$ chosen to be $A$) and the above four bullet points.

\subsection{Proof of Corollary \ref{cor:quotient is mfld}}
Let $\Gamma=\Gamma_A(E)$ and $\{A_1,\dots,A_q\}$ be the consecutive partition of $A$ induced by $E$ with $|A_j|=k_j>0$ for any $j\in\{1,\dots, q\}$. Following the remark after Proposition \ref{prop:quotient.is.sphere}, $\Gamma\backslash\TT_A$ is equipped with a differential structure such that $\overline{\Phi}_{\Gamma;A}$ is a diffeomorphism. Hence, it suffices to show that the map $g\to\Phi_{\Gamma;A}^*g$ defines a bijection from the set of Riemannian metrics on $S^{k_1}\times \cdots\times S^{k_q}$ to $\cM_\Gamma(\TT_A)$.

For any Riemannian metric $g$ on $S^{k_1}\times \cdots\times S^{k_q}$, $\Phi_{\Gamma;A}^*g$ is a positive semidefinite symmetric $2$-tensor field on $\TT_A$ due to the fact that Riemannian metrics are positive definite symmetric $2$-tensor fields. $\Gamma$-symmetry of $\Phi_{\Gamma;A}^*g$ follows directly from the fact that $\Phi_{\Gamma;A}$ is $\Gamma$-invariant. (See the first bullet point in Proposition \ref{prop:quotient.is.sphere}.) Compatibility of $\Phi_{\Gamma;A}^*g$ follows directly from the second bullet point in Proposition \ref{prop:quotient.is.sphere}. This proves that the map $g\to\Phi_{\Gamma;A}^*g$ is a well defined map from the set of Riemannian metrics on $S^{k_1}\times \cdots\times S^{k_q}$ to $\cM_\Gamma(\TT_A)$.

Assume that $g_1,g_2$ are Riemannian metrics on $S^{k_1}\times \cdots\times S^{k_q}$, with $g_1|_{\Phi_{\Gamma;A}(p)}\neq g_2|_{\Phi_{\Gamma;A}(p)}$ for some $p\in\TT_A$. (Such a $p$ always exists due to the surjectivity of $\Phi_{\Gamma;A}$. See the first bullet point in Proposition \ref{prop:quotient.is.sphere}.) Let $\cU_p$ be a good neighborhood of $p$ and $\cV$ be a neighborhood of $\Phi_{\Gamma;A}(p)$ such that $g_1$ and $g_2$ are different when restricted to the tangent space of every point of $\cV$. By \eqref{eqn:z.to.z2} and the second bullet point in Proposition \ref{prop:quotient.is.sphere}, the set of points in $\cU_p\cap\Phi_{\Gamma;A}^{-1}(\cV)$ at which $D\Phi_{\Gamma;A}$ is degenerate has zero Lebesgue measure. Hence, $\Phi_{\Gamma;A}^{*}g_1\neq \Phi_{\Gamma;A}^*g_2$. This proves that the map $g\to\Phi_{\Gamma;A}^*g$ from the set of Riemannian metrics on $S^{k_1}\times \cdots\times S^{k_q}$ to $\cM_\Gamma(\TT_A)$ is injective.

For any $\widetilde{g}\in\cM_\Gamma(\TT_A)$, for any $p\in\TT_A$ and for any good neighborhood $\cU_p$, there is a Riemannian metric $g'$ on $\Xi_{\Gamma;A;p}(\cU_p)$ such that $\widetilde{g}|_{\cU_p}=\Xi_{\Gamma;A;p}^*(g')$. Therefore, we define a $2$-tensor field $g$ on $S^{k_1}\times \cdots\times S^{k_q}$ such that
\begin{align}\label{eqn:appC.quotient is mfld-1}
g|_{\Phi_{\Gamma;A}(\cU_p)}=((\Psi_{\Gamma;A;p})^{-1})^*(g').
\end{align}
Assume that $g$ is a well-defined Riemannian metric on $S^{k_1}\times \cdots\times S^{k_q}$, then by the second bullet point in Proposition \ref{prop:quotient.is.sphere}, 
\begin{align*}
\Phi_{\Gamma;A}^*(g|_{\Phi_{\Gamma;A}(\cU_p)})=\Phi_{\Gamma;A}^*(((\Psi_{\Gamma;A;p})^{-1})^*(g'))=\Xi_{\Gamma;A;p}^*(g')=\widetilde{g}|_{\cU_p}
\end{align*}
Hence $\Phi_{\Gamma;A}^*g=\widetilde{g}$. This proves that the map $g\to\Phi_{\Gamma;A}^*g$ from the set of Riemannian metrics on $S^{k_1}\times \cdots\times S^{k_q}$ to $\cM_\Gamma(\TT_A)$ is surjective.

It remains for us to verify that $g$ defined by \eqref{eqn:appC.quotient is mfld-1} is a well-defined Riemannian metric on $S^{k_1}\times \cdots\times S^{k_q}$. Let $\cD$ be the set of points in $\TT_A$ at which $D\Phi_{\Gamma;A}$ is degenerate. By Sard's theorem, $\Phi_{\Gamma;A}(\cD)$ has zero Lebesgue measure. Hence, it suffices to show the following two bullet points:
\begin{itemize}
\item $g$ is a well-defined Riemannian metric when restricted $S^{k_1}\times\cdots\times S^{k_q}\setminus\Phi_{\Gamma;A}(\cD)$.
\item $g$ can be smoothly extended to a Riemannian metric when restrict to a neighborhood of any point in $\Phi_{\Gamma;A}(\cD)$.
\end{itemize} 
For any point $p\in\TT_A\setminus \cD$, there exists a small neighborhood $\cV_p\subset\cU_p$ of $p$ such that $\Phi_{\Gamma;A}|_{\cV_p}$ is a diffeomorphism onto its image. Hence, by \eqref{eqn:appC.quotient is mfld-1} and the second bullet point in Proposition \ref{prop:quotient.is.sphere}, we have
\begin{align*}
g|_{\Phi_{\Gamma;A}(\cV_p)}=&((\Phi_{\Gamma;A}|_{\cV_p})^{-1})^*((\Phi_{\Gamma;A}^*g)|_{\cV_p})\\
=&((\Phi_{\Gamma;A}|_{\cV_p})^{-1})^*(\Phi_{\Gamma;A}^*(((\Psi_{\Gamma;A;p})^{-1})^*(g'))|_{\cV_p})\\
=&((\Phi_{\Gamma;A}|_{\cV_p})^{-1})^*(\Xi_{\Gamma;A;p}^*(g')|_{\cV_p})=((\Phi_{\Gamma;A}|_{\cV_p})^{-1})^*(\widetilde{g}|_{\cV_p}).
\end{align*}
By the $\Gamma$-invariance of $\Phi_{\Gamma;A}$, the $\Gamma$-symmetry of $\widetilde{g}$ and the above, $g$ is a well-defined Riemannian metric when restricted $S^{k_1}\times\cdots\times S^{k_q}\setminus\Phi_{\Gamma;A}(\cD)$. Notice that for any $p\in\cD$ and for any good neighborhood $\cU_p$ of $p$, by \eqref{eqn:appC.quotient is mfld-1} and the second bullet point in Proposition \ref{prop:quotient.is.sphere}, $g|_{\Phi_{\Gamma;A}(\cU_p)\setminus\Phi_{\Gamma;A}(\cD)}$ can be smoothly extended to a smooth Riemannian metric on $\Phi_{\Gamma;A}(\cU_p)$. Therefore, by the fact that $\Phi_{\Gamma;A}(\cD)$ has zero Lebesgue measure, $g$ is a well-defined Riemannian metric on $S^{k_1}\times\cdots\times S^{k_q}$. This finishes the proof of Corollary \ref{cor:quotient is mfld}.

\subsection{Proof of Lemma \ref{lem:branching.locus}}
\begin{enumerate}
\item By the remark after Proposition \ref{prop:quotient.is.sphere}, $\Gamma\backslash\TT_A$ is equipped with a differential structure such that $\overline{\Phi}_{\Gamma;A}$ is a diffeomorphism. Hence, by the first bullet point in Proposition \ref{prop:quotient.is.sphere}, a point on $\TT_A$ is a ramification point of $\Pi_\Gamma$ if and only if it is a ramification point for $\Phi_{\Gamma;A}$. For any $p\in\TT_A$, let $\cU_p$ be a good neighborhood of $p$. Then, by the second bullet point in Proposition \ref{prop:quotient.is.sphere}, $D\Phi_{\Gamma;A}$ is degenerate at $p$ if and only if $\Xi_{\Gamma;A;p}$ is degenerate at $0$. Following \eqref{eqn:z.to.z2}, this is equivalent to $\Sing_{\Gamma;A}(p)\neq \emptyset$. In particular, the ramification locus of $\Phi_{\Gamma;A}$, which is the same as the ramification locus of $\Pi_\Gamma$, is $\Omega_{\Gamma;A}$ due to the remark after Definition \ref{sing index}.

Recall from Lemma \ref{lem:properties.good.subgp-(1)} and Lemma \ref{lem:properties.good.subgp-(2)} that $\Gamma_A$ is commutative. In particular, for any $p\in\TT_A$ and for any $\gamma\in\Gamma_A$, $\text{Stab}_\Gamma(p)=\text{Stab}_\Gamma(\gamma(p))$. Observe that by the remark after Definition \ref{sing index}, $p\in\Omega_{\Gamma;A}$ if and only if $\text{Stab}_\Gamma(p)\neq \{\id\}$. Therefore, $\gamma\Omega_{\Gamma;A}=\Omega_{\Gamma;A}$ for any $\gamma\in\Gamma_A$.

\item Let $A=\{m+1,\dots,m+k\}$. For any $p=(\theta_{m+1},\dots, \theta_{m+k})\in\TT_A$, let $\cU_p$ be a good neighborhood of $p$. Notice that $\Omega_A(i)$ is the union of finitely many codimension-$2$ submanifolds of $\TT_A$. (See \eqref{jth singularity}.) Therefore, by \eqref{total singularity subgp} and Lemma \ref{lem:branching.locus-1}, $D\Pi_{\Gamma'}$ is nondegenerate at Lebesgue almost every point in $\cU_p$.

Let $q$ be any point in $\cU_p$ such that $D\Pi_{\Gamma'}$ is nondegenerate at $q$. By Definition \ref{dfn:good.nbhd}, $q$ can be uniquely written as $q=(\theta_{m+1}+\delta_{m+1},\dots, \theta_{m+k}+\delta_{m+k})$ for some $\delta_{m+1},\dots,\delta_{m+k}\in[-\pi/8,\pi/8]$. It follows from \eqref{eqn:z.to.z2} and Proposition \ref{prop:quotient.is.sphere} that 
\begin{align}\label{eqn:appC.branching.locus-1}
&\text{Jac}(\Pi^\Gamma_{\Gamma'})(\Pi_\Gamma(q))\nonumber\\
=&\frac{\text{Jac}\Pi_{\Gamma'}(q)}{\text{Jac}\Pi_{\Gamma}(q)}\nonumber\\
=&\frac{\text{Jac}\overline{\Phi}_{\Gamma;A}(\Pi_\Gamma(q))}{\text{Jac}\overline{\Phi}_{\Gamma';A}(\Pi_{\Gamma'}(q))}\cdot\frac{\text{Jac}\Phi_{\Gamma';A}(q)}{\text{Jac}\Phi_{\Gamma;A}(q)}\nonumber\\
=&\frac{\text{Jac}\overline{\Phi}_{\Gamma;A}(\Pi_\Gamma(q))}{\text{Jac}\overline{\Phi}_{\Gamma';A}(\Pi_{\Gamma'}(q))}\cdot\frac{\text{Jac}\Psi_{\Gamma';A;p}(\Xi_{\Gamma';A;p}(q))}{\text{Jac}\Psi_{\Gamma;A;p}(\Xi_{\Gamma';A;p}(q))}\cdot\frac{\text{Jac}\Xi_{\Gamma';A;p}(q)}{\text{Jac}\Xi_{\Gamma;A;p}(q)}\nonumber\\
=&\frac{\text{Jac}\overline{\Phi}_{\Gamma;A}(\Pi_\Gamma(q))}{\text{Jac}\overline{\Phi}_{\Gamma';A}(\Pi_{\Gamma'}(q))}\cdot\frac{\text{Jac}\Psi_{\Gamma';A;p}(\Xi_{\Gamma';A;p}(q))}{\text{Jac}\Psi_{\Gamma;A;p}(\Xi_{\Gamma';A;p}(q))}\cdot\frac{\displaystyle\prod_{i\in\Sing_{\Gamma';A}(p)}(4(\delta_i^2+\delta_{i+1}^2))}{\displaystyle\prod_{i\in\Sing_{\Gamma;A}(p)}(4(\delta_i^2+\delta_{i+1}^2))}\nonumber\\
=&\frac{\text{Jac}\overline{\Phi}_{\Gamma;A}(\Pi_\Gamma(q))}{\text{Jac}\overline{\Phi}_{\Gamma';A}(\Pi_{\Gamma'}(q))}\cdot\frac{\text{Jac}\Psi_{\Gamma';A;p}(\Xi_{\Gamma';A;p}(q))}{\text{Jac}\Psi_{\Gamma;A;p}(\Xi_{\Gamma';A;p}(q))}\cdot{\displaystyle\prod_{i\in\Sing_{\Gamma';A}(p)\setminus\Sing_{\Gamma;A}(p)}(4(\delta_i^2+\delta_{i+1}^2))}.
\end{align}
It is easy to see that the right-hand-side of the above formula is a smooth function on $\cU_p$, hence the above formula holds for any $q\in\cU_p$. Recall from Proposition \ref{prop:quotient.is.sphere} that $\overline{\Phi}_{\Gamma;A}$, $\overline{\Phi}_{\Gamma';A}$, $\Psi_{\Gamma';A;p}$ and $\Psi_{\Gamma;A;p}$ are diffeomorphisms onto their images. Hence $\text{Jac}(\Pi^\Gamma_{\Gamma'})(\Pi_\Gamma(p))=0$ if and only if $\Sing_{\Gamma';A}(p)\setminus\Sing_{\Gamma;A}(p)\neq \emptyset$, which is equivalent to $\Sing_{\Gamma;A}(p)\subsetneq\Sing_{\Gamma';A}(p)$ due to Definition \ref{sing index}.

It remains to show that the ramification locus of $\Pi^{\Gamma}_{\Gamma'}:\Gamma\backslash\TT_A\to \Gamma'\backslash\TT_A$ is $\Pi_\Gamma(\overline{\Omega_{\Gamma';A}\setminus\Omega_{\Gamma;A}})$. By the previous paragraph, for any $p\in\TT_A$, $\Pi_\Gamma(p)$ is in the ramification locus of $\Pi^{\Gamma}_{\Gamma'}$ if and only if $\Sing_{\Gamma;A}(p)\subsetneq\Sing_{\Gamma';A}(p)$. Therefore, it suffices to show that $\Sing_{\Gamma;A}(p)\subsetneq\Sing_{\Gamma';A}(p)$ if and only if $p\in\overline{\Omega_{\Gamma';A}\setminus\Omega_{\Gamma;A}}$. Indeed, for any $p\in\TT_A$ such that $\Sing_{\Gamma;A}(p)\subsetneq\Sing_{\Gamma';A}(p)$, we first observe that $\Sing_{\Gamma';A}(p)\neq\emptyset$. It follows from the remark after Definition \ref{sing index} that $p\in\Omega_{\Gamma';A}$. Let $\cU_p$ be a good neighborhood of $p$ and $i\in\Sing_{\Gamma';A}(p)\setminus\Sing_{\Gamma;A}(p)$. 
By Definition \ref{sing index}, $p\in\Omega_A(i)$ and $\gamma_i\in\Gamma'\setminus\Gamma$. Hence, by \eqref{eqn:good.subgroup.struct}, \eqref{jth singularity} and \eqref{total singularity subgp}, we have that $\Omega_A(i)\subset\Omega_{\Gamma';A}$ and that $\overline{\Omega_A(i)\setminus\Omega_{\Gamma;A}}=\Omega_A(i)$. By \eqref{jth singularity}, \eqref{total singularity subgp} and Definition \ref{sing index}, $\cU_p\cap\Omega_{\Gamma';A}\supset\cU_p\cap\Omega_A(i)\supset\{q\in\cU_p:\pr^A_{i,i+1}(q)=\pr^A_{i,i+1}(p)\}$, which contains a neighborhood of $p$ in $\Omega_A(i)$. Hence, $p\in\overline{\Omega_{\Gamma';A}\setminus\Omega_{\Gamma;A}}$. 
On the other hand, if $\Sing_{\Gamma';A}(p)=\Sing_{\Gamma;A}(p)$, by the same computations in \eqref{eqn:appC.branching.locus-1}, we have $\text{Jac}(\Pi^{\Gamma}_{\Gamma'}(\Pi_\Gamma(q))>0$ for any $q\in\cU_p$. Hence, $\Sing_{\Gamma';A}(q)=\Sing_{\Gamma;A}(q)$ for any $q\in\cU_p$. By the remark after Definition \ref{sing index}, $\Sing_{\Gamma';A}(q)\neq\emptyset$ (resp. $\Sing_{\Gamma;A}(q)\neq\emptyset$) if and only if $q\in\Omega_{\Gamma';A}$ (resp. $q\in\Omega_{\Gamma;A}$). Therefore, $\cU_p\cap(\Omega_{\Gamma';A}\setminus\Omega_{\Gamma;A})=\emptyset$ and hence $p\not\in\overline{\Omega_{\Gamma';A}\setminus\Omega_{\Gamma;A}}$. This finishes the proof of Lemma \ref{lem:branching.locus-2}.
\end{enumerate}

\subsection{Verifying the claim after Definition \ref{dfn:geo rank}}
Let $c:I\to M$ be a geodesic in a nearly nonpositively curved manifold $M$, where $I$ is an interval. In the remark after Definition \ref{dfn:geo rank}, we made the claim that ``$f_{M}(Y,t)\equiv\text{constant}$'' is equivalent to ``$f_{M}(Y,t)\equiv 0$''. Before we present the proof, we first a technical lemma, which is a slightly strengthened version of \cite[Lemma 5.9]{BG94}.
\begin{lem}[Strengthened version of {\cite[Lemma 5.9]{BG94}}]\label{lem:appC.BG94.Lem5.9}
Let $(M,g)$ be a nearly nonpositively curved surface with $(C,g|_C)$ being a cap in the sense of Definition \ref{dfn:caps}. Let $c:I\to M$ be a unit-speed geodesic in $(M,g)$ which enters $\overline{C}$ at $t=t_1$, remains in ${C}$ on $(t_1,t_2)$ and exits $\overline{C}$ at $t=t_2$. Then, for any Jacobi field $Y$ along $c$, we have 
\begin{align}\label{eqn:appC.BG94.Lem5.9}
\langle Y(t_2), Y'(t_2)\rangle\geq \langle Y(t_1), Y'(t_1)\rangle.
\end{align}
Moreover, if $\langle Y(t_2), Y'(t_2)\rangle= \langle Y(t_1), Y'(t_1)\rangle$, then $\langle Y(t_2), Y'(t_2)\rangle= \langle Y(t_1), Y'(t_1)\rangle=0$.
\end{lem}
\begin{proof}[Proof of Lemma \ref{lem:appC.BG94.Lem5.9}]
Apply a shift in the time parameter, we assume that $t_1=-T$ and $t_2=T$. Write $Y_\parallel(t):=\langle Y(t),c'(t)\rangle c'(t)$ and $Y_\perp(t)=Y(t)-Y_\parallel(t)$. In particular, $Y_\parallel$ is a Jacobi field which is everywhere tangent to $c$, and $Y_\perp$ is an orthogonal Jacobi field along $c$. Based on standard knowledge, we have the following facts:
\begin{itemize}
\item There exist some $a,b\in\RR$ such that $Y_\parallel(t)=(\xi+\lambda t)\cdot c'(t)$. In particular, $Y_\parallel'(t)=\lambda \cdot c'(t)$
\item Let $N(t)$ be the parallel unit normal vector field along $c$. Then, $Y_\perp(t)=y(t)\cdot N(t)$. Here, $y(t)$ is the solution to the following differential equation
\begin{align}\label{eqn:appC.BG94.Lem5.9-1}
y''(t)+k(t)y(t)=0,
\end{align}
where $k(t)$ is the sectional curvature at $c(t)$. In particular, by Definition \ref{dfn:caps}, $k(t)$ is an even function (due to the fact that caps are radially symmetric) and is positive on $(-T,T)$. In particular, for any solution nontrivial solution $y$ to \eqref{eqn:appC.BG94.Lem5.9-1} and for any $t\in(-T,T)$ such that $y(t)\neq 0$, we have 
\begin{align}\label{eqn:appC.BG94.Lem5.9-riccati}
\frac{d}{dt}\left(\frac{y'(t)}{y(t)}\right)=-\left(\frac{y'(t)}{y(t)}\right)^2-k(t)<0.
\end{align}
\item $\langle Y,Y'\rangle=\langle Y_\parallel,Y_\parallel'\rangle+\langle Y_\perp,Y_\perp'\rangle$.
\end{itemize}
By the first bullet point, we have $\langle Y_\parallel(T),Y_\parallel'(T)\rangle\geq \langle Y_\parallel(-T),Y_\parallel'(-T)\rangle$, with equality achieved if and only if $\lambda=0$ (and hence $\langle Y_\parallel,Y_\parallel'\rangle\equiv 0$). By the second and the third bullet points, it suffices to show that 
\begin{align}\label{eqn:appC.BG94.Lem5.9-2}
y(T)y'(T)\geq y(-T)y'(-T), \text{ with equality achieved if and only if }y'(-T)=0.
\end{align}
We recall the following known results: 
\begin{itemize}
\item[\hypertarget{BG1}{(BG1)}] There exist solutions $S(t),C(t)$ to \eqref{eqn:appC.BG94.Lem5.9-1} such that $S(0)=C'(0)=0$ and $S'(0)=C(0)=1$. Moreover, $S(t)$ is odd and $C(t)$ is even. (See \cite[page 33, lines 10-12]{BG89a})
\item[\hypertarget{BG2}{(BG2)}] When $t\in[-T,T]$, $S(t)=0$ if and only if $t=0$. Moreover, $S'(T)=S'(-T)=0$. (See \cite[Lemma 2.5]{BG89a}.)
\item[\hypertarget{BG3}{(BG3)}] $S(t)C'(t)-S'(t)C(t)\equiv -1$. (This follows directly from \eqref{eqn:appC.BG94.Lem5.9-1}, which is also mentioned in \cite[page 17, line -8]{BG94}.)
\item[\hypertarget{BG4}{(BG4)}] There exists some $b\leq  0$ such that
$$\begin{pmatrix}
y(T)\\
y'(T)
\end{pmatrix}=\matii{-1&b}{0&-1}\matii{y(-T)}{y'(-T)}.
$$
(See the proof of \cite[Lemma 5.9]{BG94}.)
\item[\hypertarget{BG5}{(BG5)}] There exists some $\tau\in(0,T]$ such that $\lim_{t\to\tau_-}C'(t)/C(t)=-\infty.$ (See \cite[Lemma 2.6]{BG89a}.)
\end{itemize}
We first prove that $\tau<T$ in \hyperlink{BG5}{(BG5)}. Assume the contrary that $\tau=T$, we have $C(T)=0$ due to smoothness of $C(\cdot)$ on $[-T,T]$ and \hyperlink{BG5}{(BG5)}. By \hyperlink{BG1}{(BG1)}, \hyperlink{BG2}{(BG2)} and \eqref{eqn:appC.BG94.Lem5.9-riccati}, $S'(\cdot)/S(\cdot)$ is well-defined, positive and decreasing on $(0,T)$ and vanishes at $T$. In particular, $S(t)>0$ for any $t\in(0,T]$. By \hyperlink{BG3}{(BG3)},
$$\frac{d}{dt}\left(\frac{C(t)}{S(t)}\right)=\frac{S(t)C'(t)-S'(t)C(t)}{(S(t))^2}=-\frac{1}{(S(t))^2}<0,\text{ for any }t\in(0,T].$$
Since $C(T)/S(T)=0$, the above implies that $C(t)>0$ for any $t\in[0,T)$. Consider the function
$$H(t)=C'(t)S'(t)+k(t)C(t)S(t).$$
In particular, $H(0)=H(T)=0$ due to \hyperlink{BG1}{(BG1)} and the fact that $C(T)=0$. 
Since $k(\cdot)$ is nonincreasing on $[0,T]$ due to \cite[Lemma 2.3]{BG89a} and Definition \ref{dfn:caps}, by \eqref{eqn:appC.BG94.Lem5.9-1} and \hyperlink{BG1}{(BG1)}, we have 
$$H'(t)=k'(t)C(t)S(t)\leq 0\text{ for any }t\in(0,T).$$
Recall that $S(t),C(t)>0$ for any $t\in(0,T)$ and that $H(0)=H(T)=0$. Therefore, $k'(t)\equiv 0$ on $[0,T]$. Since $k(\pm T)=0$ due to Definition \ref{dfn:caps}, we have $k(t)\equiv 0$ on $[0,T]$. This contradicts the fact that caps are positively curved. (See Definition \ref{dfn:caps}.) Thus, we conclude that $\tau<T$.

Notice that by \hyperlink{BG1}{(BG1)}, \hyperlink{BG2}{(BG2)} and \eqref{eqn:appC.BG94.Lem5.9-riccati}, $S'(t)/S(t)$ is well-defined, positive and strictly decreasing on $(0,T)$, and $S'(T)/S(T)=0$. Since $C(\cdot)$ is well-defined on $[-T,T]$ and is a nontrivial solution to \eqref{eqn:appC.BG94.Lem5.9-1}, \hyperlink{BG5}{(BG5)} implies that $C(\tau)=0$ and $C'(\tau)<0$. Hence, by \hyperlink{BG2}{(BG2)} and \eqref{eqn:appC.BG94.Lem5.9-riccati}, there exists $\epsilon\in(0, T-\tau)$ such that for any $t\in(\tau,\tau+\epsilon)$, we have $C'(t)/C(t)$ is well-defined and that 
$C'(t)/C(t)>S'(t)/S(t)>0$.
Recall that by \hyperlink{BG3}{(BG3)}, $C'(t)/C(t)\neq S'(t)/S(t)$ for any $t\in[-T,T]$ such that $C(t),S(t)\neq 0$. Thus, by \eqref{eqn:appC.BG94.Lem5.9-riccati}, we conclude that $C'(\cdot)/C(\cdot)$ is well-defined and decreasing on $(\tau, T]$, and that $C'(t)/C(t)>S'(t)/S(t)$ for any $t\in(\tau, T]$. (See \cite[Figure 1, page 33]{BG89a} for a picture of $S'(\cdot)/S(\cdot)$ and $C'(\cdot)/C(\cdot)$.) In particular, by \hyperlink{BG1}{(BG1)}, $C(\pm T), C'(\pm T)\neq 0$. Hence, by \hyperlink{BG1}{(BG1)} and \hyperlink{BG4}{(BG4)}, we have 
$$C(T)=-C(-T)+bC'(-T)=-C(T)+bC'(-T)\implies bC'(-T)=2C(T)\neq 0.$$
In particular, $b<0$. Therefore, by \hyperlink{BG4}{(BG4)}, we have
$$y(T)y'(T)=(-y(-T)+by'(-T))(-y'(-T))=y(-T)y'(-T)-b(y'(-T))^2\geq y(-T)y'(-T).$$
Equality in the above equation is achieved if and only if $y'(-T)=0$. This verifies \eqref{eqn:appC.BG94.Lem5.9-2} and hence finishes the proof of Lemma \ref{lem:appC.BG94.Lem5.9}.
\end{proof}

\begin{proof}[Proof of the equivalence between ``$f_{M}(Y,t)\equiv\text{constant}$'' and ``$f_{M}(Y,t)\equiv 0$'']
Since the latter condition is stronger, it suffices to prove that ``$f_{M}(Y,t)\equiv\text{constant}$'' implies ``$f_{M}(Y,t)\equiv 0$''. 
We split the proof into the following three cases:

\textbf{Case 1}: Assume that there exists a subinterval $I_0\subset I$ with nonempty interior such that $M$ is nonpositively curved at $c(t)$ for any $t\in I_0$. By \eqref{eqn:noninc baby Lyap fcn2}, \eqref{eqn:noninc Lyap fcn2} and the comments after \eqref{eqn:noninc Lyap fcn2}, we have 
$f_{M}(Y,t)=\langle Y(t), Y'(t)\rangle$ for any $t\in I_0$. Since $f_M(Y,t)$ is constant in $t$, after differentiating $f_M(Y,t)$ in $t$, we obtain 
$$\frac{d}{dt}f_{M}(Y,t)=|Y'(t)|^2-R(Y(t),c'(t),Y(t),c'(t))\equiv 0,\text{ for any }t\in I_0.$$
Since $R(Y(t),c'(t),Y(t),c'(t))\leq 0$ due to the assumption that $M$ is nonpositively curved at $c(t)$, we have $Y'(t)\equiv 0$ on $I_0$. Hence, by the assumption that $f_{M}(Y,t)$ is constant on $I$ and the above, we have $f_{M}(Y,t)\equiv 0$.

\textbf{Case 2}: Assume that $\dim(M)=2$. By \textbf{Case 1}, after passing to a subinterval of $I$ with nonempty interior if necessary, we assume that for any $t\in I$, $M$ is not nonpositively curved at $c(t)$. Then, by Definition \ref{dfn:caps} and Definition \ref{dfn:near nonpos curv}, $c(I)$ is contained in the cap part $C'$ of an extended cap $(C,g)$. In particular, $c(I)$ is not contained $\del C'$. By Definition \ref{dfn:caps} and its remark, there exist $a<a'\leq \inf I<\sup I\leq b'<b$ such that the following holds:
\begin{itemize}
\item $c$ can be extended to a geodesic defined on $[a,b]$.
\item $c([a,b])\subset C$.
\item $(a',b')=c^{-1}(C')\cap[a,b]$.
\item $c((a,a')),c((b',b))\subset C\setminus\overline{C'}$. (Since $c(I)\not\subset \del C'$, by the fact that $\del C'$ is totally geodesic in $(C,g)$, we have $c^{-1}(\del C')$ does not have accumulation points. This bullet point then follows if we replace $I$ by a subinterval.)
\end{itemize} 
By \eqref{eqn:infi Lyap2}, \eqref{eqn:noninc Lyap fcn2} and the definition of $Q^M_v$ for nearly nonpositively curved surfaces, we have 
\begin{align}\label{eqn:appC.geo rank-1}
f_M(Y,t)=&\frac{b'-t}{b'-a'}f_M(Y,a')+\frac{t-a'}{b'-a}f_M(Y,b')\nonumber\\
=&\frac{b'-t}{b'-a'}\langle Y(a'),Y'(a')\rangle+\frac{t-a'}{b'-a}\langle Y(b'),Y'(b')\rangle\nonumber\\
=&\frac{b'\langle Y(a'),Y'(a')\rangle-a'\langle Y(b'),Y'(b')\rangle}{b'-a'}+\frac{\langle Y(b'),Y'(b')\rangle-\langle Y(a'),Y'(a')\rangle}{b'-a'}\cdot t
\end{align}
By Lemma \ref{lem:appC.BG94.Lem5.9} and the above bullet points, $\langle Y(b'),Y'(b')\rangle\geq \langle Y(a'),Y'(a')\rangle$, with equality achieved only when $\langle Y(b'),Y'(b')\rangle\geq \langle Y(a'),Y'(a')\rangle=0$. Therefore, by \eqref{eqn:appC.geo rank-1} and the assumption that $f_M(Y,t)$ is constant on $I$, we have $f_M(Y,t)\equiv 0$ on $I$.

\textbf{Case 3}: Assume that $\dim(M)\geq 3$, then for any point $t\in I$, by Definition \ref{dfn:near nonpos curv}, $c(t)$ is contained in a regular neighborhood $U=U_0\times \cdots\times U_k$ of $p$, with $U_0$ nonpositively curved and $U_1,\dots, U_k$ being surfaces with nearly nonpositive curvature which contains a cap containing the $U_i$-component of $p$. Let $c_i, Y_i$ be the $U_i$-component of $c|_{c^{-1}(U)\cap I}$ and $Y|_{c^{-1}(U)\cap I}$. By \eqref{eqn:infi Lyap2}, \eqref{eqn:noninc Lyap fcn2} and the previous two cases, we have $f_{U_i}(Y_i,t)\equiv 0$ on $c^{-1}(U)\cap I$ for any $i\in\{0,\dots, k\}$. Hence, by \eqref{eqn:infi Lyap2}, \eqref{eqn:noninc Lyap fcn2} and the assumption that $f_M(Y,t)$ is constant on $I$, we have $f_M(Y,t)\equiv 0$ on $I$. This finishes the proof of ``$f_{M}(Y,t)\equiv 0$'' under the assumption that ``$f_{M}(Y,t)\equiv\text{constant}$''.
\end{proof}

\subsection{Proof of Fact \ref{fact:SBAP-3}}
The proof of Fact \ref{fact:SBAP-3} is essentially the same as the proof of \cite[Property 4, p.321]{KSS89}. In \cite[Property 4, p.321]{KSS89}. For reader's convenience, we provide the proof.

It is clear that $N\cap U$ has empty interior. Assume that for any connected balls $B_1\subset\cF$ and $B_2\subset X_0$ such that $B_1\times B_2\subset U$, $B_1\times B_2\setminus N$ is connected. Then for any connected open subset $O$ of $U$ and for any $p,q\in O\setminus N$, choose a continuous curve $c:[0,1]\to O$ such that $c(0)=p$ and $c(1)=q$. Then one can find open subsets $U_1,\dots, U_m$ of $O$ such that the following holds:
\begin{itemize}
\item $c([0,1])\subset U_1\cup\cdots\cup U_m$ with $p\in U_1$ and $q\in U_m$.
\item $U_j\cap U_{j+1}\neq\emptyset$ for any $j\in\{1,\dots,m\}$. In particular, there exists $p_j\in U_j\cap U_{j+1}\setminus N$ for any $j\in\{1,\dots,m\}$. (This is because $N\cap U$ has empty interior.)
\item $U_j\setminus N$ is connected. (Choose $U_j$ to be of the form $B_1\times B_2$ with $B_1\subset \cF$ and $B_2\subset X_0$ being connected balls satisfying $B_1\times B_2\subset U$.)
\end{itemize}
For simplicity, we write $p_0=p$ and $p_m=q$. Then for any $j\in\{1,\dots,m\}$, $p_{j-1}$ and $p_j$ are connected by a path in $U_j\setminus N$. Hence, $p,q$ are connected by a path in $O\setminus N$. This proves that $O\setminus N$ is connected due to arbitrariness of $p,q\in O\setminus N$. Therefore, $N$ is non-separating in $O$ due to arbitrariness of $O$ and Definition \ref{dfn:nonsep sets}.

It suffices to show that for any connected balls $B_1\subset\cF$ and $B_2\subset X_0$ such that $B_1\times B_2\subset U$, $B_1\times B_2\setminus N$ is connected. If not, then there exist disjoint, nonempty open subsets $G_1,G_2$ of $B_1\times B_2\setminus N$ such that $G_1\cup G_2=B_1\times B_2\setminus N$. For any $x\in B_2$, by the fact that $N_x$ is non-separating in $U_x$, either $(B_1\times \{x\})\setminus N\subset G_1$ or $(B_1\times \{x\})\setminus N\subset G_2$. Hence $G_j=(B_1\times H_j)\setminus N$ for any $j\in\{1,2\}$ and for some nonempty open subset $H_j$ of $B_2$. Moreover, $H_1\cup H_2= B_2$ and $H_1\cap H_2=\emptyset$, which contradicts the assumption that $B_2$ is connected. This finishes the proof.

\subsection{Proof of the remark after Definition \ref{dfn:cap.on.torus}}
Let $\iota_x:\CC\to\CC$, $\iota_y:\CC\to\CC$ such that $\iota_x(x+\mathbf{i}y)=-x+\mathbf{i}y$ and $\iota_y(x+\mathbf{i}y)=x-\mathbf{i}y$. Then, 
$$\iota_x\circ\phi_{p_{i,j}}\circ(\kappa_{y_i},\id)(\kappa_{y_i}(\theta_i)+\delta_i,\theta_j+\delta_j)=\iota_x\circ\phi_{p_{i,j}}(\theta_i-\delta_i,\theta_j+\delta_j)=\iota_x(-\delta_i+\mathbf{i}\delta_j)=\delta_i+\mathbf{i}\delta_j.$$
Notice that $\iota_x^{-1}(C_R)=C_R$ and $\iota_x^*g=g$ due to the remark after Definition \ref{dfn:gen caps}. Therefore, $(\iota_x\circ\phi_{p_{i,j}}\circ(\kappa_{y_i},\id))^{-1}(C_R)=(\kappa_{y_i},\id)^{-1}(B_{i,j})$ and $(\iota_x\circ\phi_{p_{i,j}}\circ(\kappa_{y_i},\id))^*g=(\phi_{p_{i,j}}\circ(\kappa_{y_i},\id))^*g=(\kappa_{y_i},\id)^*g_{\text{cap}}$. Hence, by Definition \ref{dfn:cap.on.torus}, $((\kappa_{y_i},\id)^{-1}(B_{i,j}),(\kappa_{y_i},\id)^*g_{\text{cap}})$ is a special $k$-cap centered at $(\kappa_{y_i}(\theta_i),\theta_j)$ with the same Euclidean radius. The proof of $((\id,\kappa_{y_j})^{-1}(B_{i,j}),(\id,\kappa_{y_j})^*g_{\text{cap}})$ being a special $k$-cap centered at $(\theta_i,\kappa_{y_j}(\theta_j))$ with the same Euclidean radius follows from a similar argument with $\iota_x$ replaced by $\iota_y$.

\subsection{Proof of the remark after Definition \ref{dfn:insert/forget tubes}}
\hyperlink{T2}{(T2)} follows directly from Definition \ref{dfn:insert/forget tubes}. Therefore, it remains to prove \hyperlink{T1}{(T1)} and \hyperlink{T3}{(T3)}.
\begin{enumerate}
\item[\hyperlink{T1}{(T1)}] 
We first show that the order of forgetting two tubes does not matter under the assumption that the order of inserting two tubes does not matter. Let $\cT_1,\cT_2$ be two tubes. Throughout the proof, for any $j\in\{1,2\}$, we fix the symmetric $2$-tensor field on the base of any special tube on $\cT_j$. Assume that the following constructions are well-defined:
\begin{itemize}
\item $(\TT_A,g_1)$ is obtained by forgetting the special tube structure on $\cT_1$ from $(\TT_A,g)$.
\item $(\TT_A,g_2)$ is obtained by forgetting the special tube structure on $\cT_2$ from $(\TT_A,g_1)$.
\end{itemize}
Then, by \hyperlink{T2}{(T2)}, the following constructions are well-defined:
\begin{itemize}
\item $(\TT_A,g_3)$ is obtained by inserting the special tube structure on $\cT_2$ back into $(\TT_A,g_2)$.
\item $(\TT_A,g_4)$ is obtained by inserting the special tube structure on $\cT_1$ back into $(\TT_A,g_3)$.
\end{itemize}
Moreover, by \hyperlink{T2}{(T2)}, $g_4=g$. Since we have assumed that the order of inserting two tubes does not matter, the following constructions are well-defined:
\begin{itemize}
\item $(\TT_A,g_3')$ is obtained by inserting the special tube structure on $\cT_1$ back into $(\TT_A,g_2)$.
\item $(\TT_A,g_4')$ is obtained by inserting the special tube structure on $\cT_2$ back into $(\TT_A,g_3')$.
\end{itemize}
Moreover, $g_4'=g_4=g$. By \hyperlink{T2}{(T2)} again, the following constructions are well-defined:
\begin{itemize}
\item $(\TT_A,g_5)$ is obtained by forgetting the special tube structure on $\cT_2$ from $(\TT_A,g_4')=(\TT_A,g)$.
\item $(\TT_A,g_6)$ is obtained by forgetting the special tube structure on $\cT_1$ from $(\TT_A,g_5)$.
\end{itemize}
Moreover, by \hyperlink{T2}{(T2)}, $g_6=g_2$. This proves that the order of forgetting two tubes does not matter.

It remains to show that the order of inserting two tubes does not matter. Let $(\cT_1,g|_{\cT_1})$ be a $\{i_1,j_1\}$-standard tube with base $B_1$ and $(\TT_A,g_1')$ is obtained by inserting a $\{i_1,j_1\}$-special tube with base $B_1$ into $(\TT_A,g)$. Assume that $(\cT_2,g_1'|_{\cT_2})$ is a $\{i_2,j_2\}$-standard tube with base $B_2$, we would like to show the following
\begin{itemize}
\item Starting with $(\TT_A,g)$, one can first insert a $\{i_2,j_2\}$-special tube with base $B_2$, and then insert a $\{i_1,j_1\}$-special tube with base $B_1$.
\item If we instead insert a $\{i_2,j_2\}$-special tube with base $B_2$ into $(\TT_A,g_1')$ and assume that for any $j\in\{1,2\}$, the symmetric $2$-tensor fields on $B_j$ involved in the special tubes in the above constructions are fixed, then the outcome of the constructions in both bullet points are the same.
\end{itemize}
We split the rest of the proof into the following three cases:

\textbf{Case 1}: Assume that $\{i_1,j_1\}\cap \{i_2,j_2\}=\emptyset$. For any $l_1,l_2\in A$ and for any $p\in\TT_A$, let 
\begin{align}\label{eqn:appC.insert/forget tubes-1.1a}
g_{1;l_1,l_2}'(p)=g_1'|_p(\del/\del\theta_{l_1},\del/\del\theta_{l_2}).
\end{align}
Then, by Definition \ref{dfn:special tubes} and the fact that $(\cT_2,g_1'|_{\cT_2})$ is a $\{i_2,j_2\}$-standard tube, for any $p\in \cT_2$ and for any $l\in A\setminus\{i_2,j_2\}$, we have 
\begin{align}\label{eqn:appC.insert/forget tubes-1.1a}
g_{1;i_2,l}'(p)=g_{1;j_2,l}'(p)=g_{1;l,i_2}'(p)=g_{1;l,j_2}'(p)=g_{1;i_2,j_2}'(p)=g_{1;j_2,i_2}'(p)=0
\end{align}
and 
\begin{align}\label{eqn:appC.insert/forget tubes-1.1b}
g_{1;i_2,i_2}'(p)=g_{1;j_2,j_2}'(p)=1.
\end{align}
Moreover, for any $l_1,l_2\in A\setminus\{i_2,j_2\}$ and for any $p\in\cT_2$, $g_{1;l_1,l_2}'(p)$ is independent of $\pr^A_{i_2,j_2}(p)$. By Definition \ref{dfn:insert/forget tubes}, $(\cT_1,g_1'|_{\cT_1})$ is a $\{i_1,j_1\}$-special tube with base $B_1$. Hence, by Definition \ref{dfn:special tubes}, for any $p\in \cT_1=B_1\times \TT_{A\setminus\{i_1,j_1\}}$ and for any $l\in A\setminus\{i_1,j_1\}$, we have
\begin{align}\label{eqn:appC.insert/forget tubes-1.1c}
g_{1;i_1,l}'(p)=g_{1;j_1,l}'(p)=g_{1;l,i_1}'(p)=g_{1;l,j_1}'(p)=0
\end{align}
and 
\begin{align}\label{eqn:appC.insert/forget tubes-1.1d}
\begin{split}
&(\pr^A_{i_1,i_2})^*g_{1,\text{cap}}\\
=&g_{1;i_1,i_1}'(p)d\theta_{i_1}\otimes d\theta_{i_1}+g_{1;i_1,j_1}'(p)d\theta_{i_1}\otimes d\theta_{j_1}+g_{1;j_1,i_1}'(p)d\theta_{j_1}\otimes d\theta_{i_1}+g_{1;j_1,i_1}'(p)d\theta_{j_1}\otimes d\theta_{j_1},
\end{split}
\end{align}
where $g_{1,\text{cap}}$ is the symmetric $2$-tensor field on the base of the special tube $(\cT_1,g_1'|_{\cT_1})$. Moreover, for any $l_1,l_2\in A\setminus\{i_1,j_1\}$ and for any $p\in\cT_2$, $g_{1;l_1,l_2}'(p)$ is independent of $\pr^A_{i_1,j_1}(p)$. As a corollary of \eqref{eqn:appC.insert/forget tubes-1.1a}-\eqref{eqn:appC.insert/forget tubes-1.1d}, there exist positive semidefinite symmetric $2$-tensor fields $g_{h;1}$ on $\TT_{A}\setminus\{i_1,j_1\}$, $g_{h;2}$ on $\TT_{A\setminus\{i_2,j_2\}}$ and $g_{h;1,2}$ on $\TT_{A\setminus\{i_1,i_2,j_1,j_2\}}$ such that 
\begin{align}\label{eqn:appC.insert/forget tubes-1.1e}
g_1'|_{\cT_1\setminus\cT_2}=(\pr^A_{i_1,j_1})^*(g_{1,\text{cap}})+(\pr^A_{A\setminus\{i_1,j_1\}})^*(g_{h;1}),
\end{align}
\begin{align}\label{eqn:appC.insert/forget tubes-1.1f}
g_1'|_{\cT_2\setminus \cT_1}= d\theta_{i_2}\otimes d\theta_{i_2}+d\theta_{j_2}\otimes d\theta_{j_2}+(\pr^A_{A\setminus\{i_2,j_2\}})^*(g_{h;2}),
\end{align}
and
\begin{align}\label{eqn:appC.insert/forget tubes-1.1g}
g_1'|_{\cT_1\cap \cT_2}=(\pr^A_{i_1,j_1})^*(g_{1,\text{cap}})+d\theta_{i_2}\otimes d\theta_{i_2}+d\theta_{j_2}\otimes d\theta_{j_2}+(\pr^A_{A\setminus\{i_1,i_2,j_1j_2\}})^*(g_{h;1,2}).
\end{align}
Apply \hyperlink{T2}{(T2)} and Definition \ref{dfn:insert/forget tubes} to \eqref{eqn:appC.insert/forget tubes-1.1e}-\eqref{eqn:appC.insert/forget tubes-1.1g}, we have 
\begin{align}\label{eqn:appC.insert/forget tubes-1.1e'}
g|_{\cT_1\setminus\cT_2}=d\theta_{i_1}\otimes d\theta_{i_1}+d\theta_{j_1}\otimes d\theta_{j_1}+(\pr^A_{A\setminus\{i_1,j_1\}})^*(g_{h;1}),
\end{align}
\begin{align}\label{eqn:appC.insert/forget tubes-1.1f'}
g|_{\cT_2\setminus \cT_1}= d\theta_{i_2}\otimes d\theta_{i_2}+d\theta_{j_2}\otimes d\theta_{j_2}+(\pr^A_{A\setminus\{i_2,j_2\}})^*(g_{h;2}),
\end{align}
and
\begin{align}\label{eqn:appC.insert/forget tubes-1.1g'}
g|_{\cT_1\cap \cT_2}=d\theta_{i_1}\otimes d\theta_{i_1}+d\theta_{j_1}\otimes d\theta_{j_1}+d\theta_{i_2}\otimes d\theta_{i_2}+d\theta_{j_2}\otimes d\theta_{j_2}+(\pr^A_{A\setminus\{i_1,i_2,j_1j_2\}})^*(g_{h;1,2}).
\end{align}
In particular, $(\cT_2,g|_{\cT_2})$ is a standard tube in the sense of Definition \ref{dfn:special tubes}. Hence one can define $(\TT_A, g_2')$ so that it is obtained by inserting a $\{i_2,j_2\}$-special tube with base $B_2$ into $(\TT_A,g)$. Assume that $g_{2,\text{cap}}$ is the symmetric $2$-tensor field on the base of the special tube $(\cT_2,g_2'|_{\cT_1})$. Then by Definition \ref{dfn:insert/forget tubes} and \eqref{eqn:appC.insert/forget tubes-1.1e'}-\eqref{eqn:appC.insert/forget tubes-1.1g'}, we have $g|_{\TT_A\setminus(\cT_1\cup\cT_2)}=g_2'|_{\TT_A\setminus(\cT_1\cup\cT_2)}$,
\begin{align}\label{eqn:appC.insert/forget tubes-1.1e''}
g_2'|_{\cT_1\setminus\cT_2}=d\theta_{i_1}\otimes d\theta_{i_1}+d\theta_{j_1}\otimes d\theta_{j_1}+(\pr^A_{A\setminus\{i_1,j_1\}})^*(g_{h;1}),
\end{align}
\begin{align}\label{eqn:appC.insert/forget tubes-1.1f''}
g_2'|_{\cT_2\setminus \cT_1}= (\pr^A_{i_2,j_2})^*(g_{2,\text{cap}})+(\pr^A_{A\setminus\{i_2,j_2\}})^*(g_{h;2}),
\end{align}
and
\begin{align}\label{eqn:appC.insert/forget tubes-1.1g''}
g_2'|_{\cT_1\cap \cT_2}=d\theta_{i_1}\otimes d\theta_{i_1}+d\theta_{j_1}\otimes d\theta_{j_1}+(\pr^A_{i_2,j_2})^*(g_{2,\text{cap}})+(\pr^A_{A\setminus\{i_1,i_2,j_1j_2\}})^*(g_{h;1,2}).
\end{align}
In particular, $(\cT_1,g_2'|_{\cT_1})$ is a standard tube in the sense of Definition \ref{dfn:special tubes}. Hence one can define $(\TT_A, g_2'')$ so that it is obtained by inserting a $\{i_1,j_1\}$-special tube with base $B_1$ into $(\TT_A,g_2')$, where the symmetric $2$-tensor field on $B_1$ is $g_{1,\text{cap}}$. By Definition \ref{dfn:insert/forget tubes} and \eqref{eqn:appC.insert/forget tubes-1.1e''}-\eqref{eqn:appC.insert/forget tubes-1.1g''}, we have $g|_{\TT_A\setminus(\cT_1\cup\cT_2)}=g_2''|_{\TT_A\setminus(\cT_1\cup\cT_2)}$,
\begin{align}\label{eqn:appC.insert/forget tubes-1.1e'''}
g_2''|_{\cT_1\setminus\cT_2}=(\pr^A_{i_1,j_1})^*(g_{1,\text{cap}})+(\pr^A_{A\setminus\{i_1,j_1\}})^*(g_{h;1}),
\end{align}
\begin{align}\label{eqn:appC.insert/forget tubes-1.1f'''}
g_2''|_{\cT_2\setminus \cT_1}= (\pr^A_{i_2,j_2})^*(g_{2,\text{cap}})+(\pr^A_{A\setminus\{i_2,j_2\}})^*(g_{h;2}),
\end{align}
and
\begin{align}\label{eqn:appC.insert/forget tubes-1.1g'''}
g_2''|_{\cT_1\cap \cT_2}=(\pr^A_{i_1,j_1})^*(g_{1,\text{cap}})+(\pr^A_{i_2,j_2})^*(g_{2,\text{cap}})+(\pr^A_{A\setminus\{i_1,i_2,j_1j_2\}})^*(g_{h;1,2}).
\end{align}
On the other hand, recall that $(\cT_2,g_1'|_{\cT_2})$ is a $\{i_2,j_2\}$-standard tube with base $B_2$. One can define $(\TT_A, g_1'')$ so that it is obtained by inserting a $\{i_2,j_2\}$-special tube with base $B_2$ into $(\TT_A,g_1')$, where the symmetric $2$-tensor field on $B_2$ is $g_{2,\text{cap}}$. By Definition \ref{dfn:insert/forget tubes} and \eqref{eqn:appC.insert/forget tubes-1.1e}-\eqref{eqn:appC.insert/forget tubes-1.1g}, we have $g|_{\TT_A\setminus(\cT_1\cup\cT_2)}=g_1'|_{\TT_A\setminus(\cT_1\cup\cT_2)}=g_1''|_{\TT_A\setminus(\cT_1\cup\cT_2)}$, \begin{align}\label{eqn:appC.insert/forget tubes-1.1e''''}
g_1''|_{\cT_1\setminus\cT_2}=(\pr^A_{i_1,j_1})^*(g_{1,\text{cap}})+(\pr^A_{A\setminus\{i_1,j_1\}})^*(g_{h;1}),
\end{align}
\begin{align}\label{eqn:appC.insert/forget tubes-1.1f''''}
g_1''|_{\cT_2\setminus \cT_1}= (\pr^A_{i_2,j_2})^*(g_{2,\text{cap}})+(\pr^A_{A\setminus\{i_2,j_2\}})^*(g_{h;2}),
\end{align}
and
\begin{align}\label{eqn:appC.insert/forget tubes-1.1g''''}
g_1''|_{\cT_1\cap \cT_2}=(\pr^A_{i_1,j_1})^*(g_{1,\text{cap}})+(\pr^A_{i_2,j_2})^*(g_{2,\text{cap}})+(\pr^A_{A\setminus\{i_1,i_2,j_1j_2\}})^*(g_{h;1,2}).
\end{align}
\eqref{eqn:appC.insert/forget tubes-1.1e'''}-\eqref{eqn:appC.insert/forget tubes-1.1g''''} then imply that $g_1''=g_2''$. This finishes the proof of \textbf{Case 1}.

\textbf{Case 2}: Assume that $|\{i_1,j_1\}\cap \{i_2,j_2\}|=1$. Without loss of generality, we assume that $i:=i_1=i_2$ and $j_1\neq j_2$. We would like to show that $\cT_1\cap \cT_2=\emptyset$. If not, there exists an open subset $O=\prod_{l\in A}O_l$ of $\cT_1\cap\cT_2$, where $O_l$ is an open subset of $\TT_l$.
Then by Definition \ref{dfn:special tubes} and the fact that $(\cT_2,g_1'|_{\cT_2})$ is a $\{i,j_2\}$-standard tube, for any $p\in O$, we have 
\begin{align}\label{eqn:appC.insert/forget tubes-1.2a}
g_{1;i,j_1}'(p)=g_{1;j_1,i}'(p)=0 \text{ and }g_{1;i,i}'(p)=1,
\end{align}
where we have used the notations in \eqref{eqn:appC.insert/forget tubes-1.1a}. Moreover, $g_{1;j_1,j_1}'(p)$ is independent of $\pr^A_{i,j_2}(p)$.
By Definition \ref{dfn:insert/forget tubes}, $(\cT_1,g_1'|_{\cT_1})$ is a $\{i_1,j_1\}$-special tube with base $B_1$. In particular, $(\cT_1,g_1'|_{\cT_1})$ is the product of a $\{i,j_1\}$-special cap on $B_1$ and $(\TT_{A\setminus\{i,j_1\}},g_h)$ for some positive semidefinite symmetric $2$-tensor field $g_h$. This implies that for any $p\in O$. $g_{1;j_1,j_1}'(p)$ only depends on $\pr^A_{i,j_1}(p)$. Following the discussion after \eqref{eqn:appC.insert/forget tubes-1.2a}, $g_{1;j_1,j_1}'(p)$ only depends on $\pr^A_{j_1}(p)$. This and \eqref{eqn:appC.insert/forget tubes-1.2a} imply that the symmetric $2$-tensor field of the aforementioned $\{i,j_1\}$-special cap on $B_1$ is equal to $f(\theta_{j_1})d\theta_{j_1}\otimes d\theta_{j_1}+d\theta_i\otimes d\theta_i$ when restricted to $O_i\times O_{j_1}$, where $f$ is a nonpositive smooth function on $O_i$. In particular, this special cap has zero curvature on $O_i\times O_{j_1}$. However, by Definition \ref{dfn:caps}, extended caps do not have an open subset with zero curvature. This leads to a contradiction due to Definition \ref{dfn:special caps} and Definition \ref{dfn:cap.on.torus}. Therefore, we conclude that $\cT_1\cap\cT_2=\emptyset$. Since inserting a special tube structure only changes the symmetric $2$-tensor field restricted to the tube, the order of inserting special tube structures into $\cT_1$ and $\cT_2$ does not change the outcome.

\textbf{Case 3}: Assume that $\{i_1,j_1\}=\{i_2,j_2\}$. Similar to \textbf{Case 2}, we would like to show that $\cT_1\cap \cT_2=\emptyset$, which is equivalent to $B_1\cap B_2=\emptyset$ in this case. Indeed, by Definition \ref{dfn:insert/forget tubes}, $(\cT_1,g_1'|_{\cT_1})$ is a $\{i_1,j_1\}$-special tube with base $B_1$. In particular, $(\cT_1,g_1'|_{\cT_1})$ is the product of a $\{i_1,j_1\}$-special cap on $B_1$ and $(\TT_{A\setminus\{i_1,j_1\}},g_h)$ for some positive semidefinite symmetric $2$-tensor field $g_h$. Since $(\cT_2,g_1'|_{\cT_2})$ is a $\{i_2,j_2\}$-standard tube with base $B_2$, $(\cT_2,g_1'|_{\cT_2})$ is the product of a $(B_2,d\theta_{i_2}\otimes d\theta_{i_2}+d\theta_{j_2}\otimes d\theta_{j_2})$ and $(\TT_{A\setminus\{i_2,j_2\}},g_h')$ for some positive semidefinite symmetric $2$-tensor field $g_h'$. Recall from Definition \ref{dfn:caps} and Definition \ref{dfn:special caps}, special caps do not have flat open subsets. Therefore, $B_1\cap B_2=\emptyset$, and hence $\cT_1\cap\cT_2=\emptyset$. Since inserting a special tube structure only changes the symmetric $2$-tensor field restricted to the tube, the order of inserting special tube structures into $\cT_1$ and $\cT_2$ does not change the outcome.

\item[\hyperlink{T3}{(T3)}] We first prove the ``forgetting'' part of \hyperlink{T3}{(T3)} while assuming the ``inserting'' part of \hyperlink{T3}{(T3)} holds. Assume that $(\TT_A,g')$ is obtained by forgetting a $\{i,j\}$-special tube with base $B$ from $(\TT_A,g)$. Then, by \hyperlink{T2}{(T2)}, $(\TT_A,g)$ is obtained by inserting the same $\{i,j\}$-special tube with base $B$ into $(\TT_A,g')$. By the ``inserting'' part of \hyperlink{T3}{(T3)}, $(\TT_A,g)\times (\TT_{A'\setminus A},h)$ is obtained by inserting a $\{i,j\}$-special tube with the same base into $(\TT_A,g')\times (\TT_{A'\setminus A},h)$. Then by \hyperlink{T2}{(T2)}, $(\TT_A,g')\times (\TT_{A'\setminus A},h)$ is obtained by forgetting the $\{i,j\}$-special tube with the same base from $(\TT_A,g)\times (\TT_{A'\setminus A},h)$. This finishes the proof of the ``forgetting'' part of \hyperlink{T3}{(T3)}.

It remains to prove the ``inserting'' part of \hyperlink{T3}{(T3)}. Let $\cT$ be a $\{i,j\}$-tube in $\TT_A$ with base $B$ and then $\cT'=\cT\times \TT_{A'\setminus A}$ is a $\{i,j\}$-tube in $\TT_A'$ with base $B$. Assume that one can insert a $\{i,j\}$-special tube into $(\TT_A,g)$ with base $B$, then by Definition \ref{dfn:insert/forget tubes}, $g|_{\cT}=d\theta_i\otimes d\theta_i+d\theta_j\otimes d\theta_j+(\pr^A_{A\setminus\{i,j\}})^*(g_h)$, where $g_h$ is a symmetric $2$-tensor field on $\TT_{A\setminus\{i,j\}}$. This implies that 
$$(\pr^{A'}_{A})^*g|_{\cT'}+(\pr^{A'}_{A'\setminus A})^*h=d\theta_i\otimes d\theta_i+d\theta_j\otimes d\theta_j+(\pr^{A'}_{A\setminus\{i,j\}})^*(g_h)+(\pr^{A'}_{A'\setminus A})^*h.$$
This implies that one can also insert a $\{i,j\}$-special tube into $(\TT_A,g)\times (\TT_{A\setminus A'},h)$ with base $B$. Assume that $(\TT_A,g')$ is obtained by inserting a $\{i,j\}$-special tube into $(\TT_A,g)$ with base $B$. Let $g_{\text{cap}}$ be the corresponding symmetric $2$-tensor field on $B$ in $(\cT,g'|_{\cT})$. Then, by Definition \ref{dfn:insert/forget tubes}, we have $g'|_{\TT_A\setminus \cT}=g|_{\TT_A\setminus \cT}$ and 
\begin{align}\label{eqn:appC.insert/forget tubes-1.3a}
g'|_{\cT}=(\pr^A_{i,j})^*(g_{\text{cap}})+(\pr^A_{A\setminus\{i,j\}})^*(g_h).
\end{align}
On the other hand, assume that $(\TT_{A'},\widetilde{g})$ is obtained by inserting a $\{i,j\}$-special tube into $(\TT_A,g)\times (\TT_{A'\setminus A},h)$ with base $B$. Assume in addition that $g_{\text{cap}}$ be the corresponding symmetric $2$-tensor field on $B$ in $(\cT',\widetilde{g}|_{\cT'})$. Then, by \eqref{eqn:appC.insert/forget tubes-1.3a} and Definition \ref{dfn:insert/forget tubes} we have 
$$g'|_{\TT_{A'}\setminus \cT'}=(\pr^{A'}_A)^*(g|_{\TT_A\setminus \cT})+(\pr^{A'}_{A'\setminus A})^*h=(\pr^{A'}_A)^*(g'|_{\TT_A\setminus \cT})+(\pr^{A'}_{A'\setminus A})^*h,$$ 
and 
$$\widetilde{g}|_{\cT'}=(\pr^{A'}_{i,j})^*(g_{\text{cap}})+(\pr^{A'}_{A\setminus\{i,j\}})^*(g_h)+(\pr^{A'}_{A'\setminus A})^*h=(\pr^{A'}_A)^*(g'|_{\TT_A\setminus \cT})+(\pr^{A'}_{A'\setminus A})^*h.$$
This proves that $(\TT_{A'},\widetilde{g})=(\TT_A,g)\times (\TT_{A'\setminus A},h)$ and hence verifies the ``inserting'' part of \hyperlink{T3}{(T3)}. 
\end{enumerate}

\subsection{The omitted part of the proof of Lemma \ref{lem:tube.index.properties-1}}
\begin{enumerate}
\item We split the proof into two cases:

\textbf{Case 1}: Assume that $j\not\in\{i,i+1\}$. Observe that $O_{i,i+1}\times \TT_{A\setminus \{i,i+1\}}$ is invariant under $\kappa_{0;j}$ and $\kappa_{\pi;j}$ for any subset $O_{i,i+1}$ of $\TT_{i,i+1}$. By \eqref{eqn:j-tubes} and the above observation, $\cT_{i;A}(R)$ and $\TT_A\setminus\cT_{i;A}(R)$ are invariant under $\kappa_{0;j}$ and $\kappa_{\pi;j}$.

\textbf{Case 2}: Assume that $j=i$ or $i+1$. Then, for any $f\in\{\kappa_{0;i},\kappa_{\pi;i},\kappa_{0;i+1},\kappa_{\pi;i+1}\}$ and for any set of the form $O_{i,i+1}\times \TT_{A\setminus \{i,i+1\}}$ with $O_{i,i+1}$ being a subset of $\TT_{i,i+1}$, $f(O_{i,i+1}\times \TT_{A\setminus \{i,i+1\}})=O_{i,i+1}\times \TT_{A\setminus \{i,i+1\}}$ if and only if $f(O_{i,i+1})=O_{i,i+1}$. (Here, for any isometry of $\TT_{i,i+1}$, we use the same notation for its product with the identity map on $\TT_{A\setminus\{i,i+1\}}$.) By \eqref{jth singularity} and \eqref{eqn:j-tubes}, $\cT_{i;A}(R)$ is of the form $O_{i,i+1}\times \TT_{A\setminus \{i,i+1\}}$ with 
\begin{align}\label{eqn:appC.tube.index.properties-1.1a}
O_{i,i+1}=B_{i,i+1}((\pi/2,0),R)\sqcup B_{i,i+1}((3\pi/2,0),R)\sqcup B_{i,i+1}((\pi/2,\pi),R)\sqcup B_{i,i+1}((3\pi/2,\pi),R).
\end{align}
Since $f$ is an isometry, it suffices to show that $f(O_{i,i+1})=O_{i,i+1}$ for any $f\in\{\kappa_{0;i},\kappa_{\pi;i},\kappa_{0;i+1},\kappa_{\pi;i+1}\}$. This can be verified by \eqref{eqn:appC.tube.index.properties-1.1a} and the following:
\begin{align*}
&\kappa_{0;i}(B_{i,i+1}((\pi/2,0),R))=B_{i,i+1}((-\pi/2,0),R)=B_{i,i+1}((3\pi/2,0),R).\\
&\kappa_{0;i}(B_{i,i+1}((3\pi/2,0),R))=B_{i,i+1}((-3\pi/2,0),R)=B_{i,i+1}((\pi/2,0),R).\\
&\kappa_{0;i}(B_{i,i+1}((\pi/2,\pi),R))=B_{i,i+1}((-\pi/2,\pi),R)=B_{i,i+1}((3\pi/2,\pi),R).\\
&\kappa_{0;i}(B_{i,i+1}((3\pi/2,\pi),R))= B_{i,i+1}((-3\pi/2,\pi),R)= B_{i,i+1}((\pi/2,\pi),R).\\
&\kappa_{\pi;i}(B_{i,i+1}((\pi/2,0),R))=B_{i,i+1}((\pi/2,0),R).\\
&\kappa_{\pi;i}(B_{i,i+1}((3\pi/2,0),R))=B_{i,i+1}((-\pi/2,0),R)=B_{i,i+1}((3\pi/2,0),R).\\
&\kappa_{\pi;i}(B_{i,i+1}((\pi/2,\pi),R))=B_{i,i+1}((\pi/2,\pi),R).\\
&\kappa_{\pi;i}(B_{i,i+1}((3\pi/2,\pi),R))=B_{i,i+1}((-\pi/2,\pi),R)=B_{i,i+1}((3\pi/2,\pi),R).\\
&\kappa_{0;i+1}(B_{i,i+1}((\pi/2,0),R))=B_{i,i+1}((\pi/2,0),R).\\
&\kappa_{0;i+1}(B_{i,i+1}((3\pi/2,0),R))=B_{i,i+1}((3\pi/2,0),R).\\
&\kappa_{0;i+1}(B_{i,i+1}((\pi/2,\pi),R))=B_{i,i+1}((\pi/2,-\pi),R)=B_{i,i+1}((\pi/2,\pi),R).\\
&\kappa_{0;i+1}(B_{i,i+1}((3\pi/2,\pi),R))=B_{i,i+1}((3\pi/2,-\pi),R)=B_{i,i+1}((3\pi/2,\pi),R).\\
&\kappa_{\pi;i+1}(B_{i,i+1}((\pi/2,0),R))=B_{i,i+1}((\pi/2,\pi),R).\\
&\kappa_{\pi;i+1}(B_{i,i+1}((3\pi/2,0),R))=B_{i,i+1}((3\pi/2,\pi),R).\\
&\kappa_{\pi;i+1}(B_{i,i+1}((\pi/2,\pi),R))=B_{i,i+1}((\pi/2,0),R).\\
&\kappa_{\pi;i+1}(B_{i,i+1}((3\pi/2,\pi),R))=B_{i,i+1}((3\pi/2,0),R).
\end{align*}
\item Since $p\in\cT_{i;A}(R)$, by \eqref{jth singularity}, \eqref{eqn:j-tubes} and the fact that $R\in(0,\pi/8)$, we have 
\begin{align}\label{eqn:appC.tube.index.properties-1.1b}
\pr^A_i(p)\in(3\pi/8,5\pi/8)\sqcup(11\pi/8,13\pi/8) \text{ and }\pr^A_{i+1}(p)\in(-\pi/8,\pi/8)\sqcup(7\pi/8,9\pi/8).
\end{align}
By \eqref{eqn:appC.tube.index.properties-1.1b} and the assumption on $q$, we have
\begin{align}\label{eqn:appC.tube.index.properties-1.1c}
\pr^A_i(q)\in(1\pi/8,7\pi/8)\sqcup(9\pi/8,15\pi/8) \text{ and }\pr^A_{i+1}(q)\in(-3\pi/8,3\pi/8)\sqcup(5\pi/8,11\pi/8).
\end{align}
In particular, \eqref{eqn:appC.tube.index.properties-1.1c} imply that 
$$\min\{|\pr^A_i(q)-0|,|\pr^A_i(q)-\pi|,|\pr^A_{i+1}(q)-\pi/2|,|\pr^A_{i+1}(q)-3\pi/2|\}\geq \pi/8>R'.$$
By \eqref{jth singularity}, \eqref{eqn:j-tubes} and the above, we have $q\not\in\cT_{j;A}(R')$ for any $j\in\{i-1,i+1\}\cap\cE(A)$.
\item Already proved after Lemma \ref{lem:tube.index.properties-1}.
\item For any $\Gamma$-invariant subset $O$ of $\TT_A$, we have $\Pi^{-1}_\Gamma(\Pi_\Gamma(O))=O$. Moreover, if $O$ is $\Gamma$-invariant, then $\TT_A\setminus O$ and $\overline{O}$ are also $\Gamma$-invariant. The rest of the proof then follows from Lemma \ref{lem:tube.index.properties-1-1}.
\item Since $\cT_{i;A}(R)$ is $\Gamma$-invariant due to Lemma \ref{lem:tube.index.properties-1-1}, $\overline{\cT_{i;A}(R)}$ is a $\Gamma$-invariant, compact subset of $\TT_A$. Hence, $\Pi_\Gamma(\overline{\cT_{i;A}(R)})$ is a closed subset of $\Gamma\backslash\TT_A$ containing $\Pi_\Gamma({\cT_{i;A}(R)})$. This proves that $\overline{\Pi_\Gamma({\cT_{i;A}(R)})}\subset\Pi_\Gamma(\overline{\cT_{i;A}(R)})$. Choose $R'\in(R,\pi/8)$. Then, by \eqref{eqn:j-tubes}, we have 
\begin{align}\label{eqn:appC.insert/forget tubes-1.5a}\overline{\Pi_\Gamma({\cT_{i;A}(R)})}\subset\Pi_\Gamma(\overline{\cT_{i;A}(R)})\subset \Pi_\Gamma(\cT_{i;A}(R')).
\end{align}
Similarly, 
\begin{align}\label{eqn:appC.insert/forget tubes-1.5b}\overline{\Pi_\Gamma({\cT_{i+1;A}(R)})}\subset \Pi_\Gamma(\cT_{i+1;A}(R')).
\end{align}
Notice that by Lemma \ref{lem:tube.index.properties-1-2} and Lemma \ref{lem:tube.index.properties-1-4}, we have 
\begin{align*}
\Pi^{-1}_\Gamma(\Pi_\Gamma(\cT_{i;A}(R'))\cap\Pi_\Gamma(\cT_{i+1;A}(R')))=&\Pi^{-1}_\Gamma(\Pi_\Gamma(\cT_{i;A}(R')))\cap\Pi^{-1}_\Gamma(\Pi_\Gamma(\cT_{i+1;A}(R')))\\
=&\cT_{i;A}(R')\cap\cT_{i+1;A}(R')=\emptyset.
\end{align*}
Therefore, $\Pi_\Gamma(\cT_{i;A}(R'))\cap\Pi_\Gamma(\cT_{i+1;A}(R'))=\emptyset$. The rest of the proof then follows from \eqref{eqn:appC.insert/forget tubes-1.5a}, \eqref{eqn:appC.insert/forget tubes-1.5b} and the above.
\end{enumerate}

\subsection{Proof of the remark after Definition \ref{dfn:construction}}\label{subsect:appC.construction}
Let $A=\{m+1,\dots,m+k\}$ and 
\begin{align}\label{eqn:appC.construction-a}
\widetilde{g}_{m}=\sum_{i=m+1}^{m+k}d\theta_i\otimes d\theta_i.
\end{align}
For any $i\in \cE(A)=A\setminus \{m+k\}$, we define $(\TT_A,\widetilde{g}_{i})$ inductively as follows:
\begin{itemize}
\item If $i\in E'$, let $\widetilde{g}_{i}=\widetilde{g}_{i-1}$.
\item If $i\not\in E'$, let $(\TT_A,\widetilde{g}_{i})$ be obtained by inserting all the $\{i,i+1\}$-special tubes described in Definition \ref{dfn:construction} into $(\TT_A,\widetilde{g}_{i-1})$.
\end{itemize}
For any $i\in\cE(A)$, let $O_{i,i+1}(\epsilon)$ be the same as in \eqref{eqn:appC.tube.index.properties-1.1a} with $R=\epsilon\in(\pi/10,\pi/8)$. In particular, $\cT_{i;A}(\epsilon)=O_{i,i+1}(\epsilon)\times \TT_{A\setminus\{i,i+1\}}$. For any $s\in\{1,2\}$, let $g_{s;i;\text{caps}}$ be the symmetric $2$-tensor field on $O_{i,i+1}(\pi/10)$ such that for any connected component $B$ of $O_{i,i+1}(\pi/10)$, $(B,g_{s;i;\text{caps}}|_B)$ is a $\{i,i+1\}$-special $s$-cap.
 
We would like to prove the following claims for any $i\in A-1=\{m\}\cup\cE(A)$:
\begin{itemize}
\item[\hypertarget{I1}{(I1)}] $(\TT_A,\widetilde{g}_{i})$ is well-defined for any $i\in A-1$. Moreover, $\widetilde{g}_m|_{\TT_A\setminus\bigcup_{j\in\cE(A)}\cT_{j;A}(\pi/10)}=\widetilde{g}_i|_{\TT_A\setminus\bigcup_{j\in\cE(A)}\cT_{j;A}(\pi/10)}$.
\item[\hypertarget{I2}{(I2)}] For any $i\in A-1$ and for any $i_1<\dots<i_l$ with $i_1,\dots,i_l\in\cE(A)$, we have
\begin{align}\label{eqn:appC.construction-b}
&\widetilde{g}_i|_{\left(\bigcap_{j=1}^l \cT_{i_j;A}(\pi/10)\right)\setminus\left(\bigcup_{s\in\cE(A)\setminus\{i_1,\dots,i_l\}}\cT_{s;A}(\pi/10)\right)}\nonumber\\
=&\left(\sum_{\substack{j:j\in\{1,\dots,l\}\\i_j\in E\setminus E'\\i_j\leq i}}(\pr^A_{i_j,i_j+1})^*g_{1;i_j;\text{caps}}\right)+\left(\sum_{\substack{j:j\in\{1,\dots,l\}\\i_j\in \cE(A)\setminus (E\cup E')\\i_j\leq i}}(\pr^A_{i_j,i_j+1})^*g_{2;i_j;\text{caps}}\right)\nonumber\\
&+\left(\sum_{\substack{j:j\in\{1,\dots,l\}\\i_j>i\text{ or }i_j\in E'}}d\theta_{i_j}\otimes d\theta_{i_j}+d\theta_{i_j+1}\otimes d\theta_{i_j+1}\right)+\left(\sum_{\substack{j:j\in A\\ j\not\in\{i_1,i_1+1,i_2,i_2+1,\dots, i_l,i_l+1\}}} d\theta_j\otimes d\theta_j\right),
\end{align}
provided that $i_j+1<i_{j+1}$ for any $j\in\{1,\dots,l-1\}$. (If $i_j+1=i_{j+1}$ for some $j\in\{1,\dots,l-1\}$, $\bigcap_{j=1}^l \cT_{i_j;A}(\pi/10)=\emptyset$ due to Lemma \ref{lem:tube.index.properties-1-5}.)
\end{itemize}
Assume that \hyperlink{I1}{(I1)} and \hyperlink{I2}{(I2)} hold, then the constructions in Definition \ref{dfn:construction} are well-defined. By the assumption on $\cU_q$, \eqref{eqn:good.nbhd} and \eqref{eqn:j-tubes}, we have $\cU_q=(\prod_{i\in\Sing_A(q)}B_{i,i+1}(\epsilon))\times O$, where $B_{i,i+1}(\epsilon)$ is a connected component of $O_{i,i+1}(\epsilon)$, and $O$ is an open subset of $\TT_{A\setminus \Sing_A^+(q)}$. Moreover, by \eqref{eqn:construction-1} and \eqref{eqn:j-tubes}, we have 
\begin{align*}
\cU_q\setminus\cT_{i;A}(\pi/10)=
\begin{cases}
\cU_q,~&\text{if }i\not\in\Sing_A(q),\\
(B_{i,i+1}(\epsilon)\setminus B_{i,i+1}(\pi/10))\times (\prod_{j\in\Sing_A(q)\setminus\{i\}}B_{j,j+1}(\epsilon))\times O,~&\text{if }i\in\Sing_A(q),
\end{cases}
\end{align*}
where $B_{i,i+1}(\pi/10)=B_{i,i+1}(\epsilon)\cap O_{i,i+1}(\pi/10)$ is a connected component of $O_{i,i+1}(\pi/10)$. The three bullet points in the remark after Definition \ref{dfn:construction} then follows from \hyperlink{I1}{(I1)}, \hyperlink{I2}{(I2)}, \eqref{eqn:construction-1} and the above.

Now we show that $\TT_A$ can be covered by good neighborhoods of the form \eqref{eqn:good.nbhd} with $\epsilon\in(\pi/10,\pi/8)$ and satisfying \eqref{eqn:construction-1}.  By Lemma \ref{lem:tube.index.properties-1-3}, there exists some $q\in\TT_A$ and a good neighborhood $\cU_q$ with the above choice of $\epsilon$ and a sufficiently small $\epsilon'\in(0,(\epsilon-(\pi/10))/2)$, such that $\Sing_A(q)=\{j\in\cE(A):p\in\cT_{j;A}(\epsilon)\}$. By \eqref{eqn:j-tubes} and the above, we have
\begin{align}\label{eqn:appC.construction-c}
(B_{j,j+1}(\pr^A_{j,j+1}(p),2\epsilon')\times \TT_{A\setminus\{j,j+1\}})\cap\cT_{j;A}(\pi/10)=\emptyset
\end{align} 
for any $j\in\cE(A)\setminus\Sing_A(q)=\{l\in\cE(A):p\not\in\cT_{l;A}(\epsilon)\}$.
It suffices to show that $\cU_q$ satisfies \eqref{eqn:construction-1}. If $\cU_q\cap\cT_{i;A}(\pi/10)\neq \emptyset$ for some $i\in\cE(A)\setminus \Sing_A(q)$, by Lemma \ref{lem:tube.index.properties-1-5} and the fact that $\cU_q$ is a subset of $\bigcap_{j\in\Sing_A(q)}\cT_{j;A}(\epsilon)$ when $\Sing_A(q)\neq \emptyset$, we have $\{i,i+1\}\cap\Sing^+_A(q)=\emptyset$. By \eqref{eqn:good.nbhd}, the fact that $p\in\cU_q$ and the fact that $\{i,i+1\}\cap\Sing^+_A(q)=\emptyset$, we have $\cU_q\subset B_{i,i+1}(\pr^A_{i,i+1}(p),2\epsilon')\times \TT_{A\setminus\{i,i+1\}}$, which contradicts \eqref{eqn:appC.construction-c} and the assumption that $\cU_q\cap\cT_{i;A}(\pi/10)\neq \emptyset$. This proves that $\cU_q$ satisfies \eqref{eqn:construction-1}.

It remains for us to prove \hyperlink{I1}{(I1)} and \hyperlink{I2}{(I2)}. It is clear from \eqref{eqn:appC.construction-a} that both \hyperlink{I1}{(I1)} and \hyperlink{I2}{(I2)} holds for $i=m$. Assume that \hyperlink{I1}{(I1)} and \hyperlink{I2}{(I2)} holds for $i=i_0$, then the well-definedness of $\widetilde{g}_{i_0+1}$ follows directly from \hyperlink{I2}{(I2)} in the case of $i=i_0$. The fact that $\widetilde{g}_m|_{\TT_A\setminus\bigcup_{j\in\cE(A)}\cT_{j;A}(\pi/10)}=\widetilde{g}_{i_0+1}|_{\TT_A\setminus\bigcup_{j\in\cE(A)}\cT_{j;A}(\pi/10)}$ and \eqref{eqn:appC.construction-b} in the case of $i=i_0+1$ then follow directly from Definition \ref{dfn:insert/forget tubes}. This proves \hyperlink{I1}{(I1)} when $i=i_0+1$. \hyperlink{I2}{(I2)} in the case of $i=i_0+1$ follows directly from \hyperlink{I2}{(I2)} in the case of $i=i_0$ and Definition \ref{dfn:insert/forget tubes}. By the above induction on $i$, the proof of \hyperlink{I1}{(I1)} and \hyperlink{I2}{(I2)} is complete.

\subsection{Proof of Lemma \ref{lem:properties.of.tg}}
\begin{enumerate}
\item We would like to prove that $\tg_{\Gamma;A;E'}$ is invariant under $\kappa_{0;i}$ and $\kappa_{\pi;i}$ for any $i\in A$, which is a stronger statement than Lemma \ref{lem:properties.of.tg-1} due to \eqref{eqn:generator} and \eqref{eqn:larger.gp}. 

We adopt the same notations in Subsection \ref{subsect:appC.construction}. Then, it suffices to show that $\tg_i$ is invariant under $\kappa_{0;j}$ and $\kappa_{\pi;j}$ for any $i\in A-1$ and for any $j\in A$. We prove this using induction on $i$.

Recall that $A=\{m+1,\dots,m+k\}$. When $i=m$, by \eqref{eqn:generator} and \eqref{eqn:larger.gp}, $\tg_i$ is clearly invariant under  $\kappa_{0;j}$ and $\kappa_{\pi;j}$ for any $j\in A$. Assume that $\tg_i$ is invariant under $\kappa_{0;j}$ and $\kappa_{\pi;j}$ for any $i\leq i_0$ and for any $j\in A$. Since $\tg_{i_0+1}=\tg_{i_0}$ when $i_0+1\in E'$, we only need to consider the case when $i_0+1\not\in E'$. By Definition \ref{dfn:insert/forget tubes} and Definition \ref{dfn:construction}, we have
\begin{align}\label{eqn:appC.properties.of.tg-1a}
\tg_{i_0+1}|_{\TT_A\setminus\cT_{i_0+1;A}(\pi/10)}=\tg_{i_0}|_{\TT_A\setminus\cT_{i_0+1;A}(\pi/10)}.
\end{align}
and there exists some $s\in\{1,2\}$ such that 
\begin{align}\label{eqn:appC.properties.of.tg-1b}
\tg_{i_0+1}|_{\cT_{i_0+1;A}(\pi/10)}-\tg_{i_0}|_{\cT_{i_0+1;A}(\pi/10)}=(\pr^A_{i_0+1,i_0+2})^*g_{s;i;\text{caps}}-(d\theta_{i_0+1}\otimes d\theta_{i_0+1}+d\theta_{i_0+2}\otimes d\theta_{i_0+2}).
\end{align}
By Lemma \ref{lem:tube.index.properties-1-1}, $\cT_{i_0+1;A}(\pi/10)$ and $\TT_A\setminus \cT_{i_0+1;A}(\pi/10)$ is invariant under $\kappa_{0;j}$ and $\kappa_{\pi;j}$. Moreover, by the remark after Definition \ref{dfn:special caps} and the remark after Definition \ref{dfn:cap.on.torus}, the right-hand-side of \eqref{eqn:appC.properties.of.tg-1b} is invariant under $\kappa_{0;j}$ and $\kappa_{\pi;j}$. Therefore, the invariance of $\tg_{i_0+1}$ under $\kappa_{0;j}$ and $\kappa_{\pi;j}$ follows from \eqref{eqn:appC.properties.of.tg-1a}, \eqref{eqn:appC.properties.of.tg-1b} and the invariance of $\tg_{i_0}$ under $\kappa_{0;j}$ and $\kappa_{\pi;j}$. This finishes the proof.

\item $\Gamma\cap\Gamma'$-symmetry of $\tg_{\Gamma;A;E'}$ has already been verified in Lemma \ref{lem:properties.of.tg-1}. It remains to verify the compatibility condition with $\Omega_{\Gamma\cap\Gamma';A}$ in Corollary \ref{cor:quotient is mfld} for $\tg_{\Gamma;A;E'}$.

By \eqref{eqn:z.to.z2}, Definition \ref{dfn:cap.on.torus} and the remark after Definition \ref{dfn:construction}, there exists a family of good neighborhoods $\{\cU_q\}_{q\in O\subset \TT_A}$ of the form \eqref{eqn:good.nbhd} with $\epsilon=\pi/10$ such that the following holds:
\begin{itemize}
\item $\bigcup_{q\in O}\cU_q=\TT_A$.
\item For any $q\in O$, $\tg_{\Gamma;A;E'}|_{\cU_q}=\Xi^*_{\Gamma\cap\Gamma';A;q}(g')$ for some Riemannian metric $g'$ on $\Xi_{\Gamma\cap\Gamma';A;q}(\cU_q)$.
\end{itemize}
For any good neighborhood $\cU_q$ of $q$ in the form of \eqref{eqn:good.nbhd} with $\epsilon=\pi/10$, and for any $p\in\cU_q$, by \eqref{eqn:good.nbhd} again, there is a good neighborhood $\cU_p$ of $p$ such that $\cU_p\subset \cU_q$. It remains to prove the following claim: For the above choice of $p,q$ and $\cU_p,\cU_q$, there exists a diffeomorphism $\Psi_{p,q}:\Xi_{\Gamma\cap\Gamma';A;p}(\cU_p)\to\Xi_{\Gamma\cap\Gamma';A;q}(\cU_p)$ such that
\begin{align}\label{eqn:appC.properties.of.tg-2a}
\Xi_{\Gamma\cap\Gamma';A;q}|_{\cU_p}=\Psi_{p,q}\circ\Xi_{\Gamma\cap\Gamma';A;p}.
\end{align}
Indeed, by Proposition \ref{prop:quotient.is.sphere}, there exist two diffeomorphisms 
\begin{align}\label{eqn:appC.properties.of.tg-2b}
\Psi_{\Gamma\cap\Gamma';A;q}:\Xi_{\Gamma\cap\Gamma';A;q}(\cU_q)\to\Phi_{\Gamma\cap\Gamma';A}(\cU_q)\text{ and }\Psi_{\Gamma\cap\Gamma';A;p}:\Xi_{\Gamma\cap\Gamma';A;p}(\cU_p)\to\Phi_{\Gamma\cap\Gamma';A}(\cU_p).
\end{align} 
such that
\begin{align}\label{eqn:appC.properties.of.tg-2c}
\Phi_{\Gamma\cap\Gamma';A}|_{\cU_p}=\Psi_{\Gamma\cap\Gamma';A;p}\circ\Xi_{\Gamma\cap\Gamma';A;p}=\Psi_{\Gamma\cap\Gamma';A;q}\circ\Xi_{\Gamma\cap\Gamma';A;q}|_{\cU_p}.
\end{align}
The existence of $\Psi_{p,q}$ satisfying \eqref{eqn:appC.properties.of.tg-2a} is then established by \eqref{eqn:appC.properties.of.tg-2b}, \eqref{eqn:appC.properties.of.tg-2c} and choosing $\Psi_{p,q}:= (\Psi_{\Gamma\cap\Gamma';A;q})^{-1}\circ\Psi_{\Gamma\cap\Gamma';A;p}$.

%

\item We adopt the same notations in Definition \ref{dfn:construction} and Subsection \ref{subsect:appC.construction}.  Let $E'=\{m+i_1,\dots, m+i_{q-1}\}$ with $0=:i_0<i_1<\dots<i_{q-1}<i_q:=m+k$, and $A_j=\{m+i_{j-1}+1, m+i_{j-1}+2,\dots, m+i_{j}\}$. Notice that by Lemma \ref{lem:properties.good.subgp-(3)} and the assumption that $\Gamma=\Gamma_A(E)$, we have $\Gamma\cap\Gamma_{A_j}=\Gamma_{A_j}(E\cap\cE(A_j))$ for any $j\in\{1,\dots, q\}$. Then, by Definition \ref{dfn:construction}, \hyperlink{T3}{(T3)}, \eqref{eqn:appC.construction-a} and the above, one can show that for any $j\in\{0,\dots, q\}$,
$$\tg_{m+i_j-1}=\left(\sum_{s=1}^j (\pr^A_{A_s})^*\tg_{\Gamma\cap\Gamma_{A_j};A_j}\right)+\sum_{s=j+1}^q\sum_{l\in A_s} d\theta_l\otimes d\theta_l.$$
In particular, the case when $j=q$ proves Lemma \ref{lem:properties.of.tg-3}.
\item By \hyperlink{T1}{(T1)}, the order of inserting special tube structures does not matter. In particular, by Definition \ref{dfn:construction}, $(\TT_A,\widetilde{g}_{\Gamma;A})$ is obtained by inserting the following special tubes into $(\TT_A,\widetilde{g}_{\Gamma;A;E'})$:
\begin{itemize}
\item The $\{i,i+1\}$-special $1$-tube with base $B_{i,i+1}(p_{i,i+1},\pi/10)$, for any $p_{i,i+1}\in\Omega(i)$ and for any $i\in  E\cap E'$.
\item The $\{i,i+1\}$-special $2$-tube with base $B_{i,i+1}(p_{i,i+1},\pi/10)$, for any $p_{i,i+1}\in\Omega(i)$ and for any $i\in E'\setminus E$.
\end{itemize}
By Definition \ref{dfn:insert/forget tubes}, the above operation does not change the symmetric $2$-tensor field on $\TT_A\setminus\left(\bigcup_{i\in E'}\cT_{i;A}(\pi/10)\right)$. Equivalently, $\tg_{\Gamma;A;E'}$ and $\tg_{\Gamma;A}$ are the same when restricted to $\TT_A\setminus\left(\bigcup_{i\in E'}\cT_{i;A}(\pi/10)\right)$.
\end{enumerate}

\subsection{The omitted part of the proof of Lemma \ref{lem:properties.of.g.Gamma}}
\begin{enumerate}
\item For any $q\in\TT_A$, if there exists a good neighborhood $\cU_q$ of the form \eqref{eqn:good.nbhd} with $\epsilon\in(\pi/8,\pi/10)$ satisfying \eqref{eqn:construction-1}, by the remark after Definition \ref{dfn:construction}, $(\cU_q,\tg_{\Gamma;A;E'}|_{\cU_q})$ is the product of the following manifolds equipped with positive semidefinite symmetric 2-tensor fields:
\begin{itemize}
\item The union of $(B_{i,i+1}(\pr^A_{i,i+1}(q),\epsilon)\setminus B_{i,i+1}(\pr^A_{i,i+1}(q),\pi/10), d\theta_i\otimes d\theta_i+d\theta_{i+1}\otimes d\theta_{i+1})$ and a $\{i,i+1\}$-special $1$-caps centered at $\pr^A_{i,i+1}(q)$ with Euclidean radius $\pi/10$, where $i\in \{j\in E\setminus E':q\in\Omega_A(j)\}=\Sing_{\Gamma';A}(q)\setminus\Sing_{\Gamma'\cap\Gamma;A}(q)$. (See Definition \ref{sing index} and its remark.)
\item The union of $(B_{i,i+1}(\pr^A_{i,i+1}(q),\epsilon)\setminus B_{i,i+1}(\pr^A_{i,i+1}(q),\pi/10), d\theta_i\otimes d\theta_i+d\theta_{i+1}\otimes d\theta_{i+1})$ and a $\{i,i+1\}$-special $2$-caps centered at $\pr^A_{i,i+1}(q)$ with Euclidean radius $\pi/10$, where $i\in \{j\in\cE(A)\setminus\{E\cup E'\}:q\in\Omega_A(j)\}=\Sing_{\Gamma'\cap\Gamma;A}(q)$. (See Definition \ref{sing index} and its remark.)
\item $\pr^A_{A\setminus\Sing^+_{\Gamma';A}(q)}(\cU_q)$ equipped with the standard Euclidean metric $\displaystyle \sum_{i\in A\setminus\Sing^+_{\Gamma';A}(q)}d\theta_i\otimes d\theta_i$. (See Definition \ref{sing index}.)
\end{itemize}
If $\pr^A_{i}(q)=\theta_i$ for any $i\in A$, then by \eqref{eqn:z.to.z2}, $\Xi_{\Gamma\cap\Gamma';A;q}$ is the product of the following maps:
\begin{itemize}
\item $(\theta_{i}+\delta_{i},\theta_{i+1}+\delta_{i+1})\to \delta_i+\mathbf{i}\delta_{i+1}$ from $B_{i,i+1}(\pr^A_{i,i+1}(q),\epsilon)$ to the ball of radius $\epsilon$ in $\CC$ centered at $0$. Here, $i\in\Sing_{\Gamma';A}(q)\setminus\Sing_{\Gamma\cap\Gamma';A}(q)$, and we have identified $\CC$ with $\RR^2$ via $z\to(\text{Re}(z),\text{Im}(z))$.
\item $(\theta_{i}+\delta_{i},\theta_{i+1}+\delta_{i+1})\to (\delta_i+\mathbf{i}\delta_{i+1})^2$ from $B_{i,i+1}(\pr^A_{i,i+1}(q),\epsilon)$ to the ball of radius $\epsilon^2$ in $\CC$ centered at $0$. Here, $i\in\Sing_{\Gamma\cap\Gamma';A}(q)$, and we have identified $\CC$ with $\RR^2$ via $z\to(\text{Re}(z),\text{Im}(z))$.
\item $(\theta_i+\delta_i)_{i\in A\setminus \Sing^+_{\Gamma';A}}\to (\delta_i)_{i\in A\setminus \Sing^+_{\Gamma';A}}$ from $\pr^A_{A\setminus\Sing^+_{\Gamma';A}(q)}(\cU_q)$ to an open subset of $\RR^{|A\setminus\Sing^+_{\Gamma';A}(q)|}$. 
\end{itemize}
Observe the following fact:
\begin{itemize}
\item Let $k\in\ZZ_+$, $C_R:=\{z\in\CC:|z|<R\}$. Assume that $(C_R,g)$ is a generalized extended $k$-cap in the sense of Definition \ref{dfn:gen caps}. Then there exists a Riemannian metric $g'$ on $C_{R^k}$ such that $(C_{R^k},g')$ is an extended cap, and that $g$ is the pullback of $g'$ under the map $z\to z^k$. (To see this, we adopt the notations in Definition \ref{dfn:gen caps}. Consider the map $\widetilde{F}:C_{R^k}\to C$ defined as $\widetilde{F}(z)=z\cdot(h(|z|^{1/k})/|z|)$. The first bullet point in Definition \ref{dfn:gen caps} imply that $\widetilde{F}$ is a diffeomorphism onto its image. Choose $g'=\widetilde{F}^*(g_C)$. Then by the fact that $(C,g_C)$ is an extended cap, the fact that $h(z)$ is strictly increasing, and the fact that $\widetilde{F}$ is radially symmetric, $(C_{R^k},g')$ is also an extended cap in the sense of Definition \ref{dfn:caps}.)
\end{itemize}
Then, by the above seven bullet points, there exists a Riemannian metric $g'$ on $\Xi_{\Gamma\cap\Gamma';A;q}(\cU_q)$ such that $\Xi_{\Gamma\cap\Gamma';A;q}^*(g')=\tg_{\Gamma;A;E'}$ and that $(\Xi_{\Gamma\cap\Gamma';A;q}(\cU_q),g')$ is isometric to the product of the following types of manifolds:
\begin{itemize}
\item  Open surfaces of the form $(C,g)$, where $C$ is a disc in $\RR^2$ containing a smaller disc $C'$ with the same center. In addition, $(C',g|_{C'})$ is an extended cap, and $g|_{C\setminus C'}$ is flat. Such surfaces are nearly nonpositively curved in the sense of Definition \ref{dfn:near nonpos curv} due to the definition of extended caps. (See Definition \ref{dfn:caps}.)
\item  An open subset of $\RR^{|A\setminus\Sing^+_{\Gamma';A}(q)|}$ with the standard Euclidean metric. 
\end{itemize} 
Recall from the remark after Proposition \ref{prop:quotient.is.sphere} that $(\Gamma\cap\Gamma')\backslash\TT_A$ is equipped with the differential structure so that $\overline{\Phi}_{\Gamma\cap\Gamma'}$ is a diffeomorphism. Therefore, by Definition \ref{dfn:construction.of.g} and Proposition \ref{prop:quotient.is.sphere}, $g_{\Gamma;A;E'}=(\Psi_{\Gamma\cap\Gamma';A;q}^{-1}\circ\overline{\Phi}_{\Gamma\cap\Gamma'})^*(g')$, where $\Psi_{\Gamma\cap\Gamma';A;q}^{-1}\circ\overline{\Phi}_{\Gamma\cap\Gamma'}$ is a diffeomorphism from $\Pi_{\Gamma\cap\Gamma'}(\cU_q)$ to $\Xi_{\Gamma\cap\Gamma';A;q}(\cU_q)$. This proves that $(\Pi_{\Gamma\cap\Gamma'}(\cU_q),g_{\Gamma;A;E'}|_{\Pi_{\Gamma\cap\Gamma'}(\cU_q)})$ is isometric to the product of nearly nonpositively curved manifolds in the sense of Definition \ref{dfn:near nonpos curv}. Hence, it is nearly nonpositively curved due to Definition \ref{dfn:near nonpos curv}. Recall that open subsets of the form $\cU_q$ satisfying \eqref{eqn:construction-1} cover $\TT_A$ due to the remark after Definition \ref{dfn:construction}. Hence, $(\TT_A,g_{\Gamma;A;E'})$ is nearly nonpositively curved in the sense of Definition \ref{dfn:near nonpos curv}.
\item Already proved after Lemma \ref{lem:properties.of.g.Gamma}.
\item Observe that $\Pi_{\Gamma\cap\Gamma'}$ is an open map, by \eqref{eqn:j-tubes}, $\Pi_{\Gamma\cap\Gamma'}(\TT_A\setminus\bigcup_{i\in E'}\cT_{i,A}(\pi/10))$ is an open subset. Notice that by Sard's theorem, the branching locus of $\Pi_{\Gamma\cap\Gamma'}$ has zero Lebesgue measure. Therefore, it suffices to show that $(\Pi_{\Gamma\cap\Gamma'})^*g_{\Gamma;A;E'}=(\Pi_{\Gamma\cap\Gamma'})^*((\Pi^{\Gamma\cap\Gamma'}_\Gamma)^*g_{\Gamma;A})$ on $(\Pi_{\Gamma\cap\Gamma'})^{-1}(\Pi_{\Gamma\cap\Gamma'}(\TT_A\setminus\bigcup_{i\in E'}\cT_{i,A}(\pi/10)))$. The rest of the proof then follows directly from \eqref{eqn:quotient.maps}, Lemma \ref{lem:tube.index.properties-1-4}, Lemma \ref{lem:properties.of.tg-4} and Definition \ref{dfn:construction.of.g}.
\item Already proved after Lemma \ref{lem:properties.of.g.Gamma}.
\item By Lemma \ref{lem:tube.index.properties-1-3} and Lemma \ref{lem:properties.of.g.Gamma-1}, it suffices to show that for any $p\in\TT_A\setminus(\bigcup_{j\in E''}\overline{\cT_{j;A}(\pi/10)})$, there exist $q\in\TT_A$ and a good neighborhood $\cU_q$ of the form \eqref{eqn:good.nbhd} with $\epsilon\in(\pi/10,\pi/8)$ and satisfying \eqref{eqn:construction-1}, such that $\Sing_A(q)=\{j\in\cE(A):p\in\cT_{j;A}(\epsilon)\}$ and $p\in\cU_q\subset \TT_A\setminus(\bigcup_{j\in E''}\overline{\cT_{j;A}(\pi/10)})$. 

Since $p\not\in \bigcup_{j\in E''}\overline{\cT_{j;A}(\pi/10)}$, there exists $\epsilon\in(\pi/10,\pi/8)$ such that 
\begin{align}\label{eqn:appC.properties.of.g.Gamma-0}
p\not\in \bigcup_{j\in E''}{\cT_{j;A}(\epsilon)}.
\end{align}
By the remark after Definition \ref{dfn:construction}, there exist some $q\in\TT_A$ and a good neighborhood $\cU_q$ with the above choice of $\epsilon$ and a sufficiently small $\epsilon'\in(0,(\epsilon-(\pi/10))/2)$, such that $\Sing_A(q)=\{j\in\cE(A):p\in\cT_{j;A}(\epsilon)\}$. Moreover, $\cU_q$ satisfies \eqref{eqn:construction-1}. It remains to prove that $\cU_q\subset \TT_A\setminus(\bigcup_{j\in E''}\overline{\cT_{j;A}(\pi/10)})$. We claim the following stronger statement: For any $i\in E''$, we have
\begin{align}\label{eqn:appC.properties.of.g.Gamma-1}
\cU_q\cap {\cT_{i;A}((\pi/10+\epsilon)/2)}=\emptyset.
\end{align}
We split the proof \eqref{eqn:appC.properties.of.g.Gamma-1} into the following three cases:

\textbf{Case 1}: Assume that $i-1,i+1\not\in\Sing_A(q)$. By \eqref{eqn:appC.properties.of.g.Gamma-1} and \eqref{eqn:j-tubes}, we have $i\not\in\Sing_A(q)$. Therefore, by \eqref{eqn:good.nbhd}, $|\pr^A_j(q')-\pr^A_j(q)|<\epsilon'<(\epsilon-(\pi/10))/2$ for any $j\in\{i,i+1\}$ and for any $q'\in\cU_q$. Hence, 
\begin{align}\label{eqn:appC.properties.of.g.Gamma-2}
d_{\TT_{i,i+1}}(\pr^A_{i,i+1}(q'),\pr^A_{i,i+1}(q))<\epsilon-(\pi/10),\text{ for any }q'\in\cU_q,
\end{align} 
where $d_{\TT_{i,i+1}}(\cdot,\cdot)$ is the distance function on $\TT_{i,i+1}$ under the standard product metric. Recall from \eqref{jth singularity} and \eqref{eqn:j-tubes} that for any $R\in(\pi/8,\pi/10)$, we have
\begin{align}\label{eqn:appC.properties.of.g.Gamma-3}
\cT_{i;A}(R)=\{q'\in\TT_A:d_{\TT_{i,i+1}}(\Omega(i),\pr^A_{i,i+1}(q'))<R\}.
\end{align}
Therefore, by \eqref{eqn:appC.properties.of.g.Gamma-0}, \eqref{eqn:appC.properties.of.g.Gamma-2} and \eqref{eqn:appC.properties.of.g.Gamma-3}, $\cU_q\cap\overline{\cT_{i;A}(\pi/10)}=\emptyset$. This proves \eqref{eqn:appC.properties.of.g.Gamma-1} in this case.

\textbf{Case 2}: Assume that  $i-1\in\Sing_A(q)$. Then, by Definition \ref{sing index} and \eqref{jth singularity}, we have $\pr^A_i(q)\in\{0,\pi\}$. Observe the following:
\begin{itemize}
\item $\pr^{i,i+1}_i(\Omega(i))=\{\pi/2,3\pi/2\}$. (This is due to \eqref{jth singularity}.)
\item $|\pr^A_i(q')-\pr^A_i(q)|\leq \pi/8$ for any $q'\in\cU_q$. (This is due to \eqref{eqn:good.nbhd}.)
\end{itemize}
Therefore, by \eqref{eqn:appC.properties.of.g.Gamma-3} and the above two bullet points, for any $q'\in\cU_q$ and for any $q''\in \overline{\cT_{i;A}(\pi/10)}$, we have 
$$|\pr^A_i(q')-\pr^A_i(q'')|\geq\min_{y\in\{\pi/2,3\pi/2\}}\{|\pr^A_i(q)-y|-|\pr^A_i(q')-\pr^A_i(q)|-|y-\pr^A_i(q'')|\}>\pi/4>0.$$
In particular, this proves \eqref{eqn:appC.properties.of.g.Gamma-1} in this case.

\textbf{Case 3}: Assume that  $i+1\in\Sing_A(q)$. Then, by Definition \ref{sing index} and \eqref{jth singularity}, we have $\pr^A_{i+1}(q)\in\{\pi/2,3\pi/2\}$. Observe the following:
\begin{itemize}
\item $\pr^{i,i+1}_{i+1}(\Omega(i))=\{0,\pi\}$. (This is due to \eqref{jth singularity}.)
\item $|\pr^A_{i+1}(q')-\pr^A_{i+1}(q)|\leq \pi/8$ for any $q'\in\cU_q$. (This is due to \eqref{eqn:good.nbhd}.)
\end{itemize}
Therefore, by \eqref{eqn:appC.properties.of.g.Gamma-3} and the above two bullet points, for any $q'\in\cU_q$ and for any $q''\in \overline{\cT_{i;A}(\pi/10)}$, we have 
$$|\pr^A_{i+1}(q')-\pr^A_{i+1}(q'')|\geq\min_{y\in\{0,\pi\}}\{|\pr^A_{i+1}(q)-y|-|\pr^A_{i+1}(q')-\pr^A_{i+1}(q)|-|y-\pr^A_{i+1}(q'')|\}>\pi/4>0.$$
In particular, this proves \eqref{eqn:appC.properties.of.g.Gamma-1} in this case.

%
\end{enumerate}

\subsection{Proof of the remark after Definition \ref{dfn:geo lift}}
We first show the existence of a lift of a geodesic $\tc$ of $c$ if $c^{-1}(\Pi_\Gamma(\Omega_{\Gamma;A}))$ does not have accumulation points. To prove the existence, it suffices to show that for any connected component $I'$ of $I\setminus c^{-1}(\Pi_\Gamma(\Omega_{\Gamma;A}))$, one can verify the following two properties:
\begin{itemize}
\item Any lift $\tc_{I'}:I' \to\TT_A$ of $c|_{I'}$ can be extended to a continuous function on $\overline{I'}$. 
\item For any $q\in\Pi_\Gamma^{-1}(c(\inf I'))$, there exists a lift $\tc_{I',q}:I'\to \TT_A$, such that $\lim_{t\to(\inf I')^+}\tc_{I',q}(t)=q$.
\end{itemize}

Since $\Omega_{\Gamma;A}$ is $\Gamma$-invariant and is the ramification locus of $\Pi_\Gamma$ (due to Lemma \ref{lem:branching.locus-1}), for any connected component $I'$ of $I\setminus c^{-1}(\Pi_\Gamma(\Omega_{\Gamma;A}))$, there exists a $\Gamma$-invariant open neighborhood $\cU$ of $c(I')$ such that $\Pi_\Gamma|_{\Pi_\Gamma^{-1}(\cU)}$ is a covering map. In particular, by Definition \ref{dfn:construction.of.g}, for any $t\in I'$ and any $p\in\Pi_\Gamma^{-1}(c(t))$, there exists a unique lift $\tc_{t,p}:I'\to \TT_A$ of $c|_{I'}$ such that $\tc_{t,p}(t)=p$. A direct consequence of this and the finiteness of $|\Gamma|$ is that $c|_{I'}$ has finitely many lifts. (See Definition \ref{dfn:good subgp} and Lemma \ref{lem:properties.good.subgp-(1)}.) 

Write $I'=(a,b)$, By the discussion in the previous paragraph, for any $q\in\Pi_\Gamma^{-1}(c(a))$, there exists a lift $\tc_{I',q}:I'\to\TT_A$ of $c|_{I'}$ and a sequence $\{t_n\}_{n\in\ZZ_+}\subset I'$, such that $\lim_{n\to \infty}t_n= a$ and $\lim_{n\to \infty}\tc_{I',q}(t_n)=q$.
If we assume that the first bullet point holds, then the second bullet point also holds. Therefore, it remains to verify the first bullet point. 

Let $\tc_{I'}:I'\to\TT_A$ be an arbitrary lift of $c|_{T'}$ and $q\in\Pi^{-1}_\Gamma(c(a))$, where $a=\inf I'$. For any $\epsilon>0$, let $\cV_q(\epsilon)$ be the ball of radius $\epsilon$ centered at $q$ with respect to the standard product metric. Then, there exists some $\tau=\tau(\epsilon)>0$ such that $c([a,a+\tau))\subset \Pi_\Gamma(\cV_q(\epsilon))$. By the openness of $\Pi_\Gamma$ and the above, we have 
\begin{align*}
\tc_{I'}((a,a+\tau))\subset\Pi_\Gamma^{-1}( \Pi_\Gamma(\cV_q(\epsilon)))=\bigcup_{\gamma\in\Gamma}\gamma\cV_q(\epsilon).
\end{align*}
Hence, there exists some $\gamma_{\epsilon}\in\Gamma$ such that 
\begin{align}\label{eqn:appC.geo lift-a}
\tc_{I'}((a,a+\tau))\cap\gamma_\epsilon\cV_q(\epsilon)\neq\emptyset.
\end{align}
Notice that for any $\gamma\in\Gamma$, if we assume that $\epsilon$ is sufficiently small, then $\gamma\cV_q(\epsilon)=\cV_q(\epsilon)$ if and only if $\gamma q=q$; $\gamma\cV_q(\epsilon)\cap\cV_q(\epsilon)=\emptyset$ if and only if $\gamma q\neq q$. (This is because $\Gamma$ acts on $\TT_A$ isometrically with respect to the standard product metric.) 
Therefore, by the connectedness of $\tc_{I'}$, \eqref{eqn:appC.geo lift-a} and the above observation, there exists a unique $\gamma \text{Stab}_\Gamma(q)\in\Gamma/\text{Stab}_\Gamma(q)$ independent of $\epsilon$, such that $\tc_{I'}((a,a+\tau))\subset \gamma\cV_q(\epsilon)$. As $\epsilon\to 0^+$, we have $\lim_{t\to a^+}\tc_{I'}(t)=\gamma q$. This extends $\tc$ continuously to $[\inf I',\sup I')$. The continuous extension to $\overline{I'}=[\inf I',\sup I']$ is similar. This finishes the proof of the first bullet point, and hence verifies the existence of a lift $\tc$ of $c$.

In the remark after Definition \ref{dfn:geo lift}, we also claimed that for any $\Gamma$-invariant subset $U$ of $\TT_A$, we have $c^{-1}(\Pi_\Gamma(U))=\tc^{-1}(U)$. To prove this, we first observe that $U=\Pi_{\Gamma}^{-1}(\Pi_{\Gamma}(U))$ due to the $\Gamma$-invariance of $U$. Then, by Definition \ref{dfn:geo lift},
$$c^{-1}(\Pi_\Gamma(U))=\tc^{-1}(\Pi_\Gamma^{-1}((\Pi_\Gamma(U)))=\tc^{-1}(U).$$

\subsection{Verifying the two bullet points in the proof of Lemma \ref{lem:pass.thru.not.crazy}}
%
%
\begin{itemize}
\item Let $t_1,t_2\in I$ such that $\tc(t_1)\in\cT_{i;A}(\pi/10)$ and $\tc(t_2)\in\TT_A\setminus\cT_{i;A}(\pi/9)$. Let 
\begin{align*}
t_+:=\inf\{t\in[t_1,t_2]:\tc(t)\in \TT_A\setminus\cT_{i;A}(\pi/9)\}\text{ and }t_-:=\sup\{t\in[t_1,t_+]:\tc(t)\in \overline{\cT_{i;A}(\pi/10)}\}.
\end{align*}
A direct corollary of the above is that 
\begin{align}\label{eqn:appC.pass.thru.not.crazy-1}
\tc((t_-,t_+))\subset \cT_{i;A}(\pi/9)\setminus \overline{\cT_{i;A}(\pi/10)}.
\end{align}
To prove the first bullet point, it suffices to show that $t_+-t_-\geq \pi/90$. Since $\overline{\cT_{i;A}(\pi/9)}\cap\overline{\cT_{j;A}(\pi/9)}=\emptyset$ for any $j\in\{i-1,i+1\}\cap\cE(A)$, by \eqref{eqn:appC.pass.thru.not.crazy-1}, $\tc((t_-,t_+))\cap\overline{\cT_{j;A}(\pi/10)}=\emptyset$ for any $j\in\{i-1,i,i+1\}\cap\cE(A)$. Hence, by the remark after Definition \ref{dfn:construction}, for any $t\in (t_-,t_+)$, there exists some  neighborhood $\cV_{\tc(t)}$ of $\tc(t)$ such that 
\begin{align}\label{eqn:appC.pass.thru.not.crazy-2}
\tg_{\Gamma;A}|_{\cV_{\tc(t)}}=d\theta_i\otimes d\theta_i+d\theta_{i+1}\otimes d\theta_{i+1}+(\pr^A_{A\setminus\{i,i+1\}})^*g'
\end{align}
where $g'$ is a symmetric $2$-tensor field on $\pr^A_{A\setminus\{i,i+1\}}(\cV_{\tc(t)})$. Since $\tc$ is a lift of a geodesic $c$ in $(\Gamma\backslash\TT_A,g_{\Gamma;A})$, by Definition \ref{dfn:construction.of.g} and Lemma \ref{lem:branching.locus-1}, the restriction of $\tc$ to any connected component of ${(t_-,t_+)\setminus c^{-1}(\Pi_\Gamma(\Omega_{\Gamma;A}))}$ is a geodesic in $(\TT_A\setminus \Omega_{\Gamma;A},\tg_{\Gamma;A})$ with the same speed as $c$. In view of \eqref{eqn:appC.pass.thru.not.crazy-2} and the fact that $c$ has speed at most one, the restriction of $\pr^A_{i,i+1}\circ\tc$ to any connected component of ${(t_-,t_+)\setminus c^{-1}(\Pi_\Gamma(\Omega_{\Gamma;A}))}$ is a geodesic in $\TT_{i,i+1}$ equipped with the standard product measure. Moreover, the speed of $\pr^A_{i,i+1}\circ\tc|_{(t_-,t_+)}$ is no faster than the speed of $\tc|_{(t_-,t_+)}$, which is no larger than one. Recall the assumption that $c^{-1}(\Pi_\Gamma(\Omega_{\Gamma;A}))$ does not have accumulation points. Since $\tc|_{(t_-,t_+)}$ is a continuous curve segment from a point in $\del\cT_{i;A}(\pi/10)$ to a point in $\del\cT_{i;A}(\pi/9)$, by \eqref{jth singularity}, \eqref{eqn:j-tubes} and the above discussion regarding $\pr^A_{i,i+1}\circ\tc$ and its speed, $t_+-t_-$ is larger than or equal to the distance between $\del\cT_{i;A}(\pi/10)$ and $\del\cT_{i;A}(\pi/9)$ under the standard product metric. Equivalently, $t_+-t_-\geq (\pi/9)-(\pi/10)=\pi/90$.
\item Assume that the second bullet point in Lemma \ref{lem:pass.thru.not.crazy} holds when $|A|=2$, we consider the case when $|A|\geq 3$. In particular, $E':=\{i-1,i+1\}\cap\cE(A)\neq \emptyset$. By the assumption that $0<t_2-t_1<2\delta<\pi/90$ and the first bullet point in Lemma \ref{lem:pass.thru.not.crazy}, we have $\tc([t_1,t_2])\subset \cT_{i;A}(\pi/9)$. In particular, by Lemma \ref{lem:tube.index.properties-1-4} and Lemma \ref{lem:tube.index.properties-1-5}, we have 
$$c([t_1,t_2])\subset\Pi_\Gamma\left(\TT_A\setminus\left(\bigcup_{j\in E'}\cT_{j;A}(\pi/10)\right)\right).$$ 
Apply Lemma \ref{lem:properties.of.g.Gamma-4} to $c|_{[t_1,t_2]}$ and $E'$, $\Pi_{\Gamma\cap\Gamma'}\circ\tc|_{[t_1,t_2]}$ is a geodesic in $((\Gamma\cap\Gamma')\backslash\TT_A,g_{\Gamma;A;E'})$, where $\Gamma'=\Gamma_A(E')$. Moreover, by Lemma \ref{lem:properties.of.g.Gamma-3}, $\Pi_{\Gamma\cap\Gamma'}\circ\tc|_{[t_1,t_2]}$ has the same speed as $c$. Since $\{i,i+1\}$ is an atom in the consecutive partition induced by $E'$, by Lemma \ref{lem:properties.of.g.Gamma-2}, $\Pi_{\Gamma\cap\Gamma_{\{i,i+1\}}}\circ\pr^A_{i,i+1}\circ\tc|_{[t_1,t_2]}$, which is equal to the $((\Gamma\cap\Gamma_{\{i,i+1\}})\backslash\TT_{i,i+1},g_{\Gamma\cap\Gamma_{\{i,i+1\};\{i,i+1\}}})$-component of $\Pi_{\Gamma\cap\Gamma'}\circ\tc|_{[t_1,t_2]}$, is a geodesic in $((\Gamma\cap\Gamma_{\{i,i+1\}})\backslash\TT_{i,i+1},g_{\Gamma\cap\Gamma_{\{i,i+1\};\{i,i+1\}}})$ with speed no larger than the speed of $c$. Recall that $c$ has speed no larger than one. By the $|A|=2$ case, we have 
$$\Pi_{\Gamma\cap\Gamma_{\{i,i+1\}}}\circ\pr^A_{i,i+1}\circ\tc([t_1,t_2])\subset \Pi_{\Gamma\cap\Gamma_{\{i,i+1\}}}(\cT_{i,\{i,i+1\}}(\pi/10)).$$
Apply Lemma \ref{lem:tube.index.properties-1-4} and \eqref{eqn:j-tubes} to the above, we have
$$\pr^A_{i,i+1}\circ\tc([t_1,t_2])\subset \cT_{i,\{i,i+1\}}(\pi/10)=\pr^A_{i,i+1}(\cT_{i;A}(\pi/10)).$$
By \eqref{eqn:j-tubes} again, we have $\tc([t_1,t_2])\subset\cT_{i;A}(\pi/10)$, and hence $c([t_1,t_2])\subset \Pi_\Gamma(\cT_{i;A}(\pi/10))$.
\end{itemize}

\subsection{Verifying the part in the proof of Lemma \ref{lem:rank-1.not.in.sing-2} which is similar to Lemma \ref{lem:rank-1.not.in.sing-1}}
We provide details for the part of the proof of Lemma \ref{lem:rank-1.not.in.sing-2} before \eqref{eqn:rank-1.not.in.sing-2}. This is the part in the proof of Lemma \ref{lem:rank-1.not.in.sing-2} which is similar to Lemma \ref{lem:rank-1.not.in.sing-1}.

If $c(\RR)\subset \Pi_\Gamma(\overline{\cT_{i;A}(\pi/10)})$, by Lemma \ref{lem:tube.index.properties-1-5}, $c(\RR)\cap\Pi_\Gamma(\cT_{j;A}(\pi/10)))=\emptyset$ for any $j\in E'$. Then, by Lemma \ref{lem:missing.tube=higher.rank}, $c$ is either constant, or has rank at least two. 

It remains to show that $c(\RR)\subset \Pi_\Gamma(\overline{\cT_{i;A}(\pi/10)})$. By Lemma \ref{lem:tube.index.properties-1-5}, $c(I')\cap\Pi_\Gamma(\cT_{j;A}(\pi/10)))=\emptyset$ for any $j\in E'$. Then, by Lemma \ref{lem:tube.index.properties-1-4}, we have 
$$c(I')\subset\Pi_\Gamma\left(\TT_A\setminus\left(\bigcup_{j\in E'}\cT_{j;A}(\pi/10)\right)\right).$$ 
Let $\Gamma'=\Gamma_A(E')$. Apply Lemma \ref{lem:properties.of.g.Gamma-4} to $c|_{I'}$ and $E'$, there is a smooth geodesic $\tc_{\Gamma\cap\Gamma'}:I_0\to\Pi_{\Gamma\cap\Gamma'}(\TT_A\setminus\bigcup_{j\in E'}\cT_{j;A}(\pi/10))$ in $((\Gamma\cap\Gamma')\backslash\TT_A,g_{\Gamma;A;E'})$ such that $\Pi_{\Gamma}^{\Gamma\cap\Gamma'}\circ\tc_{\Gamma\cap\Gamma'}=c|_{I_0}$.

\subsection{Proof of Lemma \ref{sublem:nomadic=rank-1-2}, similar to Lemma \ref{sublem:nomadic=rank-1-1}} Let $E'=\{i-1,i\}\cap\cE(A)$ and $\Gamma=\Gamma_A(E)$ for some $E\subset \cE(A)$. Then, by Lemma \ref{lem:properties.good.subgp-(4)}, we have $\Gamma\cap\Gamma'=\Gamma_A(E'\cup E)$. 
%
%

By Lemma \ref{lem:tube.index.properties-1-4}, Lemma \ref{lem:tube.index.properties-1-5} and the assumption on $I$, we have 
\begin{align}\label{eqn:appC.non.ints.bad}
c(I)\subset\Pi_\Gamma\left(\TT_A\setminus\left( \bigcup_{j\in E'}\cT_{j;A}(\pi/10)\right)\right).
\end{align}
Hence, by Definition \ref{dfn:geo lift} and Lemma \ref{lem:properties.of.g.Gamma-4}, the following holds:
\begin{itemize}
\item $\Pi_{\Gamma\cap\Gamma'}\circ\tc|_{I}$ is a geodesic in $(\Gamma\cap\Gamma'\backslash\TT_A,g_{\Gamma;A;E'})$. 
\item $(\Pi_{\Gamma\cap\Gamma'})_*\tY|_{I\setminus\Sigma}$ extends to a smooth Jacobi field along $\Pi_{\Gamma\cap\Gamma'}\circ\tc|_{I}$ in $(\Gamma\cap\Gamma'\backslash\TT_A,g_{\Gamma;A;E'})$.  (We use the same notation for the extended Jacobi field.)
\end{itemize}
Let 
\begin{align}\label{eqn:appC.non.ints}
I_0:=\left\{t\in I: c(t)\in \Pi_\Gamma\left(\TT_A\setminus\left( \bigcup_{j\in E'}\overline{\cT_{j;A}(\pi/10)}\right)\right)\right\}.
\end{align}
We claim that $I\setminus I_0$ does not have accumulation points. Let $I'$ be an arbitrary connected component of $I\setminus c^{-1}(\Pi_\Gamma(\Omega_{\Gamma;A}))$. Since $ c^{-1}(\Pi_\Gamma(\Omega_{\Gamma;A}))$ does not have accumulation points, it suffices to show that $\overline{I'}\setminus I_0$ does not have accumulation points. 

By Lemma \ref{lem:branching.locus}, Definition \ref{dfn:construction.of.g} and the fact that $\Pi_{\Gamma\cap\Gamma'}\circ\tc|_{I}$ is a geodesic in $(\Gamma\cap\Gamma'\backslash\TT_A,g_{\Gamma;A;E'})$, we have that $\tc|_{I'}$ is a geodesic in $(\TT_A\setminus\Omega_{\Gamma\cap\Gamma';A},\tg_{\Gamma;A;E'}|_{\TT_A\setminus\Omega_{\Gamma\cap\Gamma';A}})$. By the remark after Definition \ref{sing index}, we have
\begin{align}\label{eqn:appC.nomadic=rank-1-2-a}
\Sing_{\Gamma';A}(p)\cap E'=\emptyset\text{ for any }p\in \TT_A.
\end{align}
Assume that there is an accumulation point $t$ of $\overline{I'}\setminus I_0$. Apply Definition \ref{dfn:geo lift} and Lemma \ref{lem:tube.index.properties-1-4} to \eqref{eqn:appC.non.ints}, we have $\tc(t)\in \bigcup_{j\in E'}\overline{\cT_{j;A}(\pi/10)}$. Since $E'=\{i-1,i\}\cap\cE(A)$, by Lemma \ref{lem:tube.index.properties-1-5}, there exists a unique $i_0\in E'$ such that $\tc(t)\in\overline{\cT_{i_0;A}(\pi/10)}$. Moreover, $\{i_0-1,i_0+1\}\cap\cE(A)\supset E'\setminus\{i_0\}$. By Lemma \ref{lem:tube.index.properties-1-5}, $\tc(t)\in \TT_A\setminus\left(\bigcup_{j\in \{i_0-1,i_0+1\}\cap\cE(A)}\overline{\cT_{j;A}(\pi/10)}\right)$. Hence, By Lemma \ref{lem:properties.of.g.Gamma-5}, one can find $q\in\TT_A$ and a good neighborhood $\cU_q$ of the form \eqref{eqn:good.nbhd} containing $\tc(t)$, such that $\cU_q$ satisfies \eqref{eqn:construction-1} and that 
$$\cU_q\subset \TT_A\setminus\left(\bigcup_{j\in \{i_0-1,i_0+1\}\cap\cE(A)}\overline{\cT_{j;A}(\pi/10)}\right).$$
In particular, for any $p\in\cU_q$, by \eqref{eqn:j-tubes} and Definition \ref{sing index}, we have 
\begin{align}\label{eqn:appC.nomadic=rank-1-2-b}
\Sing_{\Gamma';A}(p)\cap\{i_0-1,i_0+1\}=\emptyset.
\end{align}
Hence, by \eqref{eqn:appC.nomadic=rank-1-2-a}, \eqref{eqn:appC.nomadic=rank-1-2-b} and Definition \ref{sing index}, we have $\Sing_{\Gamma';A}^+(p)\cap\{i_0,i_0+1\}=\emptyset$ for any $p\in\cU_q$. Apply the remark after Definition \ref{dfn:construction} to $\cU_q$, there exists a symmetric $2$-tensor field $g'$ on $\pr^A_{A\setminus \{i_0,i_0+1\}}(\cU_q)$ such that
\begin{align}\label{eqn:appC.nomadic=rank-1-2-c}
\tg_{\Gamma;A;E'}|_{\cU_q}=d\theta_{i_0}\otimes d\theta_{i_0}+d\theta_{i_0+1}\otimes d\theta_{i_0+1}+(\pr^A_{A\setminus \{i_0,i_0+1\}})*(g').
\end{align}
Recall that $\tc|_{I'}$ is a geodesic in $(\TT_A\setminus\Omega_{\Gamma\cap\Gamma';A},\tg_{\Gamma;A;E'}|_{\TT_A\setminus\Omega_{\Gamma\cap\Gamma';A}})$. \eqref{eqn:appC.nomadic=rank-1-2-c} implies that $\pr^A_{i_0,i_0+1}\circ\tc|_{I'\cap \tc^{-1}(\cU_q)}$ is a geodesic under the standard product metric $d\theta_{i_0}\otimes d\theta_{i_0}+d\theta_{i_0+1}\otimes d\theta_{i_0+1}$ on $\TT_{i_0,i_0+1}$. Observe from \eqref{eqn:appC.non.ints} that $\overline{I'}\setminus I_0\subset \tc^{-1}\left(\bigcup_{j\in E'} \del\cT_{j;A}(\pi/10)\right)$. Hence, by \eqref{eqn:j-tubes} and the assumptions on $t$ and $i_0$, we have that $\pr^A_{i_0,i_0+1}\circ\tc|_{I'\cap\tc^{-1}(\cU_q)}$ is constant. In particular, there exists an open interval $I''\subset I'\cap\tc^{-1}(\cU_q)$ and a point $p_0\in\overline{\cT_{i_0;A}(\pi/10)}$, such that $c(I'')\subset \Pi_\Gamma(\pr^A_{i_0,i_0+1}(\{p_0\})\times \TT_{A\setminus\{i_0,i_0+1\}}$. Apply Lemma \ref{lem:rank-1.not.in.sing-1} to $c$, we have $c(\RR)\subset \Pi_\Gamma(\overline{\cT_{i_0;A}(\pi/10)})$. Recall that $|A|\geq 3$ and hence $\{i_0-1,i_0+1\}\cap\cE(A)\neq\emptyset$. Therefore, by Lemma \ref{lem:tube.index.properties-1-5}, $c(\RR)\cap\cT_{j;A}(\pi/10)=\emptyset$ for any $j\in\{i_0-1,i_0+1\}\cap\cE(A)$. This contradicts the assumption that $c$ is nomadic. Therefore, we conclude that $I'\setminus I_0$ does not have accumulation points. Hence, $I\setminus I_0$ does not have accumulation points.

We then claim that for any $t\in I$,
\begin{align}\label{eqn:appC.middle.cover.lyap.const}
f_{(\Gamma\backslash\TT_A,g_{\Gamma;A})}(Y,t)=f_{(\Gamma\cap\Gamma')\backslash\TT_A, g_{\Gamma;A;E'})}((\Pi_{\Gamma\cap\Gamma'})_*\tY,t)\equiv 0.
\end{align}
Since $I\setminus I_0$ does not have accumulation points, by the continuity of  the functions $f_{(\Gamma\backslash\TT_A,g_{\Gamma;A})}(Y,\cdot)$ and $f_{(\Gamma\cap\Gamma')\backslash\TT_A, g_{\Gamma;A;E'})}((\Pi_{\Gamma\cap\Gamma'})_*\tY,\cdot)$, it suffices to verify \eqref{eqn:appC.middle.cover.lyap.const} for any connected component $I'$ of $I_0$ and for any $t\in I'$.

For any $t\in I'$, by Lemma \ref{lem:tube.index.properties-1-4}, Definition \ref{dfn:geo lift} and \eqref{eqn:appC.non.ints}, we have $\tc(t)\in\TT_A\setminus\left(\bigcup_{j\in E'}{\cT_{j;A}(\pi/10)}\right).$
By Lemma \ref{lem:properties.of.g.Gamma-1} and Lemma \ref{lem:properties.of.g.Gamma-5}, there exists a point $q\in\TT_A$ and a good neighborhood $\cU_q$ of $q$ such that $\tc(t)\in\cU_q\subset \TT_A\setminus\left(\bigcup_{j\in E'}\overline{\cT_{j;A}(\pi/10)}\right)$ and that $(\Pi_\Gamma(\cU_q),g_{\Gamma;A}|_{\Pi_\Gamma(\cU_q)})$ is nearly nonpositively curved.
Observe the following:
\begin{itemize}
\item $g_{\Gamma;A;E'}=(\Pi^{\Gamma\cap\Gamma'}_\Gamma)^*g_{\Gamma;A}$ when restricted to $\Pi_{\Gamma\cap\Gamma'}(\TT_A\setminus\bigcup_{j\in E'}\overline{\cT_{j,A}(\pi/10)})$. (This is due to Lemma \ref{lem:properties.of.g.Gamma-3}.)
\item For any nearly nonpositively curved manifolds $(M,g)$, for any open subset $U$ of $M$, for any geodesic $c_M:I_M\to (U,g|_U)$, for any Jacobi field $Y_M$ along $c_M$ and for any $t\in I_M\in c^{-1}(U)$, if $(U,g|_U)$ is nearly nonpositively curved, then $f_{(M,g)}(Y_M,t)=f_{(U,g|_U)}(Y_M,t)$. (This is due to \eqref{eqn:infi Lyap2} and \eqref{eqn:noninc Lyap fcn2}.)
\item For any $p\in\TT_A$ and any good neighborhood $\cU_p$ of $p$. If $\cU_p\subset \TT_A\setminus \left(\bigcup_{j\in E'}\overline{\cT_{j;A}(\pi/10)}\right)$, then $\Pi^{\Gamma\cap\Gamma'}_{\Gamma}|_{\Pi_{\Gamma\cap\Gamma'}(\cU_p)}$ is a diffeomorphism onto its image. (To see this, for any $q_1,q_2\in\cU_p$, if $\Pi_\Gamma(q_1)=\Pi_\Gamma(q_2)$, by Definition \ref{dfn:good.nbhd}, there exists $\gamma\in\text{Stab}_\Gamma(p)$ such that $\gamma(q_1)=q_2$. Observe that $E'\cap\Sing_A(p)=\emptyset$ due to Definition \ref{sing index} and \eqref{eqn:j-tubes}. Write $\Gamma=\Gamma_A(E)$. Recall that $\Gamma'=\Gamma_A(E')$. By Lemma \ref{lem:properties.good.subgp-(4)} and the remark after Definition \ref{sing index}, we have 
$$\Sing_{\Gamma\cap\Gamma';A}(p)=\Sing_A(p)\setminus(E\cup E')=\Sing_A(p)\setminus E=\Sing_{\Gamma;A}(p).$$ 
Hence, by the remark after Definition \ref{sing index} again, we have $\text{Stab}_\Gamma(p)=\text{Stab}_{\Gamma\cap\Gamma'}(p)$. This implies that $\Pi_{\Gamma\cap\Gamma'}(q_1)=\Pi_{\Gamma\cap\Gamma'}(q_2)$. Notice that by Lemma \ref{lem:properties.of.g.Gamma-3}, $\Pi^{\Gamma\cap\Gamma'}_{\Gamma}$ is everywhere nondegenerate when restricted to $\Pi_{\Gamma\cap\Gamma'}(\TT_A\setminus\bigcup_{j\in E'}\cT_{j,A}(\pi/10))$. Therefore, $\Pi^{\Gamma\cap\Gamma'}_{\Gamma}|_{\Pi_{\Gamma\cap\Gamma'}(\cU_p)}$ is a diffeomorphism onto its image.)
\end{itemize}
Hence, by the three bullet points in the above, we have 
\begin{align*}
f_{(\Gamma\backslash\TT_A,g_{\Gamma;A})}(Y,t)=&f_{(\Pi_\Gamma(\cU_q),g_{\Gamma;A}|_{\Pi_\Gamma(\cU_q)})}(Y,t)\\
=&f_{(\Pi_{\Gamma\cap\Gamma'}(\cU_q), g_{\Gamma;A;E'}|_{\Pi_{\Gamma\cap\Gamma'}(\cU_q)})}((\Pi_{\Gamma\cap\Gamma'})_*\tY,t)=f_{((\Gamma\cap\Gamma')\backslash\TT_A, g_{\Gamma;A;E'})}((\Pi_{\Gamma\cap\Gamma'})_*\tY,t)
\end{align*}
\eqref{eqn:appC.middle.cover.lyap.const} then holds for any $t\in I'$ due to the assumption that $f_{(\Gamma\backslash\TT_A,g_{\Gamma;A})}(Y,t)\equiv 0$ and the above. By the discussion after \eqref{eqn:appC.middle.cover.lyap.const}, we have \eqref{eqn:appC.middle.cover.lyap.const} holds for any $t\in I$.

Notice that $\{i\}$ is an atom in the consecutive partition induced by $E'$. (See the remark after Definition \ref{dfn:consecutive}.) Therefore, by Lemma \ref{lem:properties.of.g.Gamma-2}, the following holds:
\begin{itemize}
\item $ \pr^A_{i}\circ\tc|_{I}=\tc_i|_{I}$ is a geodesic in $\TT_{i}$. 
\item $(\pr^A_{i})_*\tY|_{I\setminus\Sigma}=\tY_i|_{I\setminus \Sigma}$ extends to a smooth Jacobi field along $\tc_i|_{I}$ in $\TT_{i}$. (We use the same notation for the extended Jacobi field.)
\end{itemize}
Moreover, by Lemma \ref{lem:rank-1 lem 3}, Lemma \ref{lem:properties.of.g.Gamma-2} and \eqref{eqn:appC.middle.cover.lyap.const}, we have
\begin{align}\label{eqn:appC.factor.middle.cover.lyap.const}
f_{\TT_i}(\tY_i,t)\equiv 0,~\forall t\in I.
\end{align}
$\tc_i'|_{I}(\cdot)$ is then a constant function due to the fact that it is a geodesic in the circle $\TT_i$. Since $\TT_i$ has constant zero curvature tensor, by \eqref{eqn:noninc baby Lyap fcn2}, \eqref{eqn:noninc Lyap fcn2} and the comments after \eqref{eqn:noninc Lyap fcn2}, we have $f_{\TT_i}(\tY_i,t)=\langle\tY_i(t),\tY_i'(t)\rangle$. Hence, \eqref{eqn:appC.factor.middle.cover.lyap.const} implies that $\tY_i(\cdot)$ is constant on $I$. This finishes the proof of Lemma \ref{sublem:nomadic=rank-1-2}.

\subsection{Proof of (C1)}
By Definition \ref{dfn:j.angle} and the assumption that $(\pr^A_i)_*(v)=(\pr^A_i)_*(w)$ for any $i\in A$, we have $\beta_i(v)=\beta_i(w)$ for any $i\in \cE(A)$ such that $\beta_i(v)$ is well-defined. Hence, by Definition \ref{dfn:exceptional value} and the assumption that $v\in T^{\text{NE}}\TT_A$, we have $w\in T^{\text{NE}}\TT_A$. For any $v\in  T^{\text{NE}}\TT_A$, we let $X_v$ be the smooth vector field on $\TT_A$ such that $(\pr^A_i)_*(v)=(\pr^A_i)_*(X_v(p))$ for any $i\in A$ and for any $p\in\TT_A$. Then, $\{X_v(p):p\in\TT_A\setminus \Omega_{\Gamma_A;A}\}$ is a subset of $T^{\text{NE}}(\TT_A\setminus \Omega_{\Gamma_A;A})$. Observe the following:
\begin{itemize}
\item $w\in \{X_v(p):p\in\TT_A\setminus \Omega_{\Gamma_A;A}\}\subset\widetilde{\cN}_A$. (This is due to the assumption that $w\in\widetilde{\cN}_A$, Definition \ref{dfn:normal.vectors} and the fact that $w\in T^{\text{NE}}\TT_A$.)
\item $\{X_v(p):p\in\TT_A\setminus \Omega_{\Gamma_A;A}\}$ is diffeomorphic to $\TT_A\setminus \Omega_{\Gamma_A;A}$, which is connected. (This is due to \eqref{jth singularity} and \eqref{total singularity subgp}, which shows that $\Omega_{\Gamma_A;A}$ is the union of finitely many codimension two submanifolds of $\TT_A$.)
\end{itemize}
\hyperlink{C1}{(C1)} then follows from the above two bullet points.

\subsection{Proof of (C2)}
Let $w\in T_q\TT_A$ such that $(\pr^A_{i,i+1})_*(w)\neq 0$ and that $(\pr^A_{j})_*(w)=(\pr^A_{j})_*(v)$ for any $j\in A\setminus\{i,i+1\}$. By the assumption on $q$, $w$ satisfies the first condition in Definition \ref{dfn:normal.vectors}. It remains to verify that $w$ satisfies the second condition in Definition \ref{dfn:normal.vectors}. We first prove that
\begin{align}\label{eqn:appC.C2-1}
(\pr^A_{j,j+1})_*(w)=(\pr^A_{j,j+1})_*(v)\text{ and }\beta_j(w)=\beta_j(v),\text{ for any }j\in\cE(A)\setminus\{i-1,i,i+1\}.
\end{align}
Indeed, for any $j\in\cE(A)\setminus\{i-1,i,i+1\}$, we have $\{j,j+1\}\cap\{i,i+1\}=\emptyset$. By the assumption on $w$, we have $(\pr^A_{j,j+1})_*(w)=(\pr^A_{j,j+1})_*(v)$. Apply Definition \ref{dfn:j.angle} to the above, we have $\beta_j(w)=\beta_j(v)$. This finishes the proof of \eqref{eqn:appC.C2-1}. 

Since $v\in\widetilde{\cN}_A$, for any $j\in \cE(A)\setminus\{i-1,i,i+1\}$, we have $\pr^A_{j,j+1}(v)\neq 0$ due to Definition \ref{dfn:normal.vectors}. In particular, by \eqref{eqn:appC.C2-1}, $(\pr^A_{j,j+1})_*(w)=(\pr^A_{j,j+1})_*(v)\neq 0$. Moreover, by Definition \ref{dfn:normal.vectors} again, one of the following cases holds:
\begin{itemize}
\item $\beta_j(v)$ is not a $j$-exceptional angle. By \eqref{eqn:appC.C2-1},  $\beta_j(w)=\beta_j(v)$ is also not a $j$-exceptional angle.
\item $j-1\in\cE(A)$, $q\in \cT_{j-1;A}(\pi/10)$ and $(\pr^A_{j+1})_*(v)\neq 0$. By the assumption on $w$ and $j$, we have $(\pr^A_{j+1})_*(w)=(\pr^A_{j+1})_*(v)\neq 0$.
\item $q\in \cT_{j;A}(\pi/10)$.
\item $j+1\in\cE(A)$, $q\in \cT_{j+1;A}(\pi/10)$ and $(\pr^A_{j})_*(v)\neq 0$. By the assumption on $w$ and $j$, we have $(\pr^A_{j})_*(w)=(\pr^A_{j})_*(v)\neq 0$.
\end{itemize}
Observe the following facts:
\begin{itemize}
\item If $i-1\in A$ (which is equivalent to $i-1\in\cE(A)$ due to the fact that $i\in\cE(A)$ and \eqref{eqn:cutpoint}), then $(\pr^A_{i-1})_*(w)\neq 0$ due to the assumptions on $w$ and the assumption that $(\pr^A_{i-1})_*(v)\neq 0$. In particular, $(\pr^A_{i-1,i})_*(w)\neq 0$. We emphasize that $q\in\cT_{i;A}(\pi/10)$ and that $\{i-1,i\}\setminus\{i,i+1\}=\{i-1\}$.
\item $q\in\cT_{i;A}(\pi/10)$ and that $(\pr^A_{i,i+1})_*(w)\neq 0$.
\item If $i+2\in A$ (which is equivalent to $i+1\in\cE(A)$ due to the fact that $i\in\cE(A)$ and \eqref{eqn:cutpoint}), then $(\pr^A_{i+2})_*(w)\neq 0$ due to the assumptions on $w$ and the assumption that $(\pr^A_{i+2})_*(v)\neq 0$. In particular, $(\pr^A_{i+1,i+2})_*(w)\neq 0$. We emphasize that $q\in\cT_{i;A}(\pi/10)$ and that $\{i+1,i+2\}\setminus\{i,i+1\}=\{i+2\}$.
\end{itemize}
Then the six bullet points in the above prove that $w$ satisfies the second condition in Definition \ref{dfn:normal.vectors}. Hence, $w\in\widetilde{\cN}_A$. This finishes the proof of \hyperlink{C2}{(C2)}.

\subsection{Verifying the last sentence in the proof of Lemma \ref{lem:basic.properties.normal.vector-3}}
For simplicity, we write $A_-:=\{j\in A:j\leq i\}$ and $A_+:=\{j\in A:j> i\}$. In the proof of Lemma \ref{lem:basic.properties.normal.vector-3}, we claimed that 
\begin{align}\label{eqn:appC.basic.properties.normal.vector-3-a}
\widetilde{\cN}_{\Gamma;A}(i)\subset \widetilde{\cN}_{A_-}\times \widetilde{\cN}_{A_+}
\end{align} immediately after mentioning 
\begin{align}\label{eqn:appC.basic.properties.normal.vector-3-b}
\cT_{j;A'}(\pi/10)\times \TT_{A\setminus A'}=\cT_{j;A}(\pi/10)\text{ for any }A'\subset A.
\end{align}
In this subsection, we prove \eqref{eqn:appC.basic.properties.normal.vector-3-a}, assuming \eqref{eqn:appC.basic.properties.normal.vector-3-b}.

We first notice that $\cE(A)\setminus \{i\}=\cE(A_-)\sqcup\cE(A_+)$ due to \eqref{eqn:cutpoint}. For any $v\in\widetilde{\cN}_{\Gamma;A}(i)$ with base point $p$ and for any $j\in\cE(A)\setminus\{i\}$, by Definition \ref{dfn:normal.vectors} and \eqref{eqn:j.higher.rk.normal.vector}, the following holds:
\begin{itemize}
\item The base point $p$ of $v$ is in $\TT_A\setminus\Omega_{\Gamma_A;A}$. By \eqref{jth singularity} and \eqref{total singularity subgp}, this is equivalent to $\pr^A_{l,l+1}(p)\not\in\Omega(l)$ for any $l\in\cE(A)$. Since $\cE(A)\setminus \{i\}=\cE(A_-)\sqcup\cE(A_+)$, by \eqref{eqn:tori.proj}, we have $\pr^{A_-}_{l,l+1}(\pr^A_{A_-}(p))\not\in\Omega(l)$ for any $l\in\cE(A_-)$ and that $\pr^{A_+}_{l,l+1}(\pr^A_{A_+}(p))\not\in\Omega(l)$ for any $l\in\cE(A_+)$. Then, by \eqref{jth singularity} and \eqref{total singularity subgp} again, we have that the base point $\pr^A_{A_-}(p)$ of $(\pr^A_{A_-})_*(v)$ is in $\TT_{A_-}\setminus\Omega_{\Gamma_{A_-};A_-}$, and that the base point $\pr^A_{A_+}(p)$ of $(\pr^A_{A_+})_*(v)$ is in $\TT_{A_+}\setminus\Omega_{\Gamma_{A_+};A_+}$.
\item $(\pr^A_{j,j+1})_*(v)\neq 0$. When $j<i$, this is equivalent to $(\pr^{A_-}_{j,j+1})_*((\pr^A_{A_-})_*(v))\neq 0$. When $j>i$, this is equivalent to $(\pr^{A_+}_{j,j+1})_*((\pr^A_{A_+})_*(v))\neq 0$. (See \eqref{eqn:tori.proj}.)
\end{itemize} 
When $j\in\cE(A)\setminus\{i-1,i,i+1\}$, by the second condition in Definition \ref{dfn:normal.vectors} and \eqref{eqn:j.higher.rk.normal.vector}, one of the following holds:
\begin{itemize}
\item $\beta_{j}(v)$ is not $j$-exceptional. When $j<i$, by Definition \ref{dfn:j.angle}, this is equivalent to $\beta_j((\pr^A_{A_-})_*(v))=\beta_{j}(v)$ being not $j$-exceptional. When $j>i$, by Definition \ref{dfn:j.angle}, this is equivalent to $\beta_j((\pr^A_{A_+})_*(v))=\beta_{j}(v)$ being not $j$-exceptional. 
\item $p\in\cT_{j-1;A}(\pi/10)$ and $(\pr^A_{j+1})_*(v)\neq 0$, under the additional assumption that $j>\min\cE(A)$. Recall that $\cE(A)\setminus \{i\}=\cE(A_-)\sqcup\cE(A_+)$. When $j<i-1$, we have $|A_-|\geq 2$ and $\cE(A_-)\supset \{j\}\neq\emptyset$. Moreover, $j>\min\cE(A_-)=\min\cE(A)$. By \eqref{eqn:appC.basic.properties.normal.vector-3-b} and the above, we have that $\pr^A_{A_-}(p)\in\cT_{j-1;A_-}(\pi/10)$ and that $(\pr^{A_-}_{j+1})_*((\pr^A_{A_-})_*(v))=(\pr^A_{j+1})_*(v)\neq 0$. When $j>i+1=\min\cE(A_+)$, by \eqref{eqn:appC.basic.properties.normal.vector-3-b} and the above, we have that $\pr^A_{A_+}(p)\in\cT_{j-1;A_+}(\pi/10)$ and that $(\pr^{A_+}_{j+1})_*((\pr^A_{A_+})_*(v))=(\pr^A_{j+1})_*(v)\neq 0$.
\item $p\in\cT_{j;A}(\pi/10)$. Recall that $\cE(A)\setminus \{i\}=\cE(A_-)\sqcup\cE(A_+)$. When $j<i-1$, by \eqref{eqn:appC.basic.properties.normal.vector-3-b} and the above, we have $\pr^A_{A_-}(p)\in\cT_{j;A_-}(\pi/10)$. When $j>i+1$, by \eqref{eqn:appC.basic.properties.normal.vector-3-b} and the above, we have $\pr^A_{A_+}(p)\in\cT_{j;A_+}(\pi/10)$.
\item $p\in\cT_{j+1;A}(\pi/10)$ and $(\pr^A_{j})_*(v)\neq 0$, under the additional assumption that $j<\max\cE(A)$. Recall that $\cE(A)\setminus \{i\}=\cE(A_-)\sqcup\cE(A_+)$. When $j<i-1=\max\cE(A_-)$, by \eqref{eqn:appC.basic.properties.normal.vector-3-b} and the above, we have $\pr^A_{A_-}(p)\in\cT_{j+1;A_-}(\pi/10)$ and that $(\pr^{A_-}_{j})_*((\pr^A_{A_-})_*(v))=(\pr^A_{j})_*(v)\neq 0$. When $j>i+1$, we have $\cE(A_+)\supset \{j\}\neq\emptyset$. Moreover, $j<\max\cE(A_+)=\max\cE(A)$. By \eqref{eqn:appC.basic.properties.normal.vector-3-b} and the above, we have that $\pr^A_{A_+}(p)\in\cT_{j+1;A_+}(\pi/10)$ and that $(\pr^{A_+}_{j})_*((\pr^A_{A_+})_*(v))=(\pr^A_{j})_*(v)\neq 0$.
\end{itemize}
By \eqref{eqn:vectors.avoiding.j.tube}, we have $\Pi_\Gamma(p)\not\in\Pi_{\Gamma}(\cT_{i;A}(\pi/10))$. Hence, by Lemma \ref{lem:tube.index.properties-1-4}, we have 
\begin{align}\label{eqn:appC.basic.properties.normal.vector-3-c}
p\not\in\cT_{i;A}(\pi/10).
\end{align}
When $j=i-1$, by the definition of $A_\pm$ and the fact that $\cE(A)\setminus \{i\}=\cE(A_-)\sqcup\cE(A_+)$, we have $|A_-|\geq 2$ and $\cE(A_-)\neq\emptyset$. By the second condition in Definition \ref{dfn:normal.vectors}, \eqref{eqn:j.higher.rk.normal.vector} and \eqref{eqn:appC.basic.properties.normal.vector-3-c}, one of the following holds:
\begin{itemize}
\item $\beta_{j}(v)$ is not $j$-exceptional. By Definition \ref{dfn:j.angle}, this is equivalent to $\beta_j((\pr^A_{A_-})_*(v))=\beta_{j}(v)$ being not $j$-exceptional.
\item $p\in\cT_{j-1;A}(\pi/10)$ and $(\pr^A_{j+1})_*(v)\neq 0$, under the additional assumption that $j>\min\cE(A)=\min\cE(A_-)$. By \eqref{eqn:appC.basic.properties.normal.vector-3-b} and the above, we have that $\pr^A_{A_-}(p)\in\cT_{j-1;A_-}(\pi/10)$ and that $(\pr^{A_-}_{j+1})_*((\pr^A_{A_-})_*(v))=(\pr^A_{j+1})_*(v)\neq 0$.
\item $p\in\cT_{j;A}(\pi/10)$. By \eqref{eqn:appC.basic.properties.normal.vector-3-b} and the above, we have $\pr^A_{A_-}(p)\in\cT_{j;A_-}(\pi/10)$.
\end{itemize}
When $j=i+1$, by the definition of $A_\pm$ and the fact that $\cE(A)\setminus \{i\}=\cE(A_-)\sqcup\cE(A_+)$, we have $|A_+|\geq 2$ and $\cE(A_+)\neq\emptyset$. By the second condition in Definition \ref{dfn:normal.vectors}, \eqref{eqn:j.higher.rk.normal.vector} and \eqref{eqn:appC.basic.properties.normal.vector-3-c}, one of the following holds:
\begin{itemize}
\item $\beta_{j}(v)$ is not $j$-exceptional. By Definition \ref{dfn:j.angle}, this is equivalent to $\beta_j((\pr^A_{A_+})_*(v))=\beta_{j}(v)$ being not $j$-exceptional.
\item $p\in\cT_{j;A}(\pi/10)$. By \eqref{eqn:appC.basic.properties.normal.vector-3-b} and the above, we have $\pr^A_{A_+}(p)\in\cT_{j;A_+}(\pi/10)$.
\item $p\in\cT_{j+1;A}(\pi/10)$ and $(\pr^A_{j+1})_*(v)\neq 0$, under the additional assumption that $j<\max\cE(A)=\max\cE(A_+)$. By \eqref{eqn:appC.basic.properties.normal.vector-3-b} and the above, we have that $\pr^A_{A_+}(p)\in\cT_{j+1;A_+}(\pi/10)$ and that $(\pr^{A_+}_{j})_*((\pr^A_{A_+})_*(v))=(\pr^A_{j})_*(v)\neq 0$.
\end{itemize}
Notice that 
\begin{itemize}
\item If $|A_-|=1$, then $A_-=\{i\}$, $\cE(A_-)=\emptyset$ and $i=\min\cE(A)$. Since $p\not\in \cT_{i;A}(\pi/10)$, by Definition \ref{dfn:normal.vectors} and \eqref{eqn:j.higher.rk.normal.vector}, we have $(\pr^A_i)_*(v)=(\pr^{A_-}_i)_*((\pr^A_{A_-})_*(v))\neq 0$. (Indeed, we first notice that $i-1\not\in\cE(A)$ and hence $\cT_{i-1;A}$ is not well-defined. Since $v\in\widetilde{\cN}_A$ due to \eqref{eqn:j.higher.rk.normal.vector} and $p\not\in \cT_{i;A}(\pi/10)$, by Definition \ref{dfn:normal.vectors}, either (1) $\beta_i(v)$ is not $i$-exceptional, which implies that $(\pr^A_i)_*(v)\neq 0$ due to the remark after Definition \ref{dfn:exceptional value}, or (2) $p\in\cT_{i+1;A}(\pi/10)$ and $(\pr^A_i)_*(v)\neq 0$.)
\item  If $|A_+|=1$, then $A_+=\{i+1\}$, $\cE(A_+)=\emptyset$ and $i=\max\cE(A)$. Since $p\not\in \cT_{i;A}(\pi/10)$, by Definition \ref{dfn:normal.vectors} and \eqref{eqn:j.higher.rk.normal.vector}, we have $(\pr^A_{i+1})_*(v)=(\pr^{A_+}_{i+1})_*((\pr^A_{A_+})_*(v))\neq 0$. (Indeed, we first notice that $i+1\not\in\cE(A)$ and hence $\cT_{i+1;A}$ is not well-defined. Since $v\in\widetilde{\cN}_A$ due to \eqref{eqn:j.higher.rk.normal.vector} and $p\not\in \cT_{i;A}(\pi/10)$, by Definition \ref{dfn:normal.vectors}, either (1) $\beta_i(v)$ is not $i$-exceptional, which implies that $(\pr^A_{i+1})_*(v)\neq 0$ due to the remark after Definition \ref{dfn:exceptional value}, or (2) $p\in\cT_{i-1;A}(\pi/10)$ and $(\pr^A_{i+1})_*(v)\neq 0$.)
\end{itemize} 
Hence, by Definition \ref{dfn:normal.vectors} and all bullet points in the above, we have that $(\pr^A_{A_-})_*(v)\in\widetilde{\cN}_{A_-}$ and that $(\pr^A_{A_+})_*(v)\in\widetilde{\cN}_{A_+}$. This proves \eqref{eqn:appC.basic.properties.normal.vector-3-a} due to the arbitrariness of $v\in\widetilde{\cN}_A(i)$.

\subsection{Proof of (M1)}

If $v\in\widetilde{\cQ}_{\Gamma;A}(i)$, by \eqref{eqn:middle.lift.vectors.avoid.j.tube}, $(\Pi^{\Gamma(i)}_{\Gamma})_*(v)\in\cQ_{\Gamma;A}(i)$. By \eqref{eqn:vectors.avoiding.j.tube}, there exists a geodesic $c:\RR\to \Gamma\backslash\TT_A$ in $(\Gamma\backslash\TT_A,g_{\Gamma;A})$ such that $c'(0)= (\Pi^{\Gamma(i)}_{\Gamma})_*(v)$ and that
$c(\RR)\cap\Pi_\Gamma(\cT_{i;A}(\pi/10))=\emptyset$. 
By Lemma \ref{lem:tube.index.properties-1-4}, we have 
\begin{align}\label{eqn:appC.M1-0a}
c(\RR)\subset&\Pi_\Gamma(\TT_A\setminus \cT_{i;A}(\pi/10))=(\Gamma\backslash\TT_A)\setminus \Pi_\Gamma({\cT_{i;A}(\pi/10)}))\nonumber\\\subset&\Pi_\Gamma(\TT_A\setminus \overline{\cT_{i;A}(\pi/11)})=(\Gamma\backslash\TT_A)\setminus \Pi_\Gamma(\overline{\cT_{i;A}(\pi/11)})).
\end{align}
Hence, by Lemma \ref{lem:tube.index.properties-1-4} again, we have 
\begin{align}\label{eqn:appC.M1-0b}
(\Pi^{\Gamma(i)}_\Gamma)^{-1}(c(\RR))\subset &(\Pi^{\Gamma(i)}_\Gamma)^{-1}(\Pi_\Gamma(\TT_A\setminus \cT_{i;A}(\pi/10)))\nonumber\\
=&\Pi_{\Gamma(i)}(\Pi_{\Gamma(i)}^{-1}((\Pi^{\Gamma(i)}_\Gamma)^{-1}(\Pi_\Gamma(\TT_A\setminus \cT_{i;A}(\pi/10)))))\nonumber\\
=&\Pi_{\Gamma(i)}(\Pi_{\Gamma}^{-1}(\Pi_\Gamma(\TT_A\setminus \cT_{i;A}(\pi/10))))\nonumber\\
=&\Pi_{\Gamma(i)}(\TT_A\setminus \cT_{i;A}(\pi/10))=(\Gamma(i)\backslash\TT_A)\setminus(\Pi_{\Gamma(i)}( \cT_{i;A}(\pi/10))).
\end{align}
Similarly, 
\begin{align}\label{eqn:appC.M1-0c}
(\Pi^{\Gamma(i)}_\Gamma)^{-1}(\Pi_\Gamma(\TT_A\setminus \overline{\cT_{i;A}(\pi/11)}))=\Pi_{\Gamma(i)}(\TT_A\setminus \overline{\cT_{i;A}(\pi/11)}).
\end{align}
Recall from \eqref{eqn:j.subgroup.Gamma} that $\Gamma(i):=\Gamma_A(\{i\})\cap\Gamma$. If we write $\Gamma=\Gamma_A(E)$ for some $E\subset \cE(A)$, then by Lemma \ref{lem:properties.good.subgp-(4)}, $\Gamma(i)=\Gamma_A(E\cup\{i\})$. Hence, by \eqref{jth singularity}, \eqref{total singularity subgp} and \eqref{eqn:j-tubes}, for any $R\in(0,\pi/8)$, we have  
\begin{align}\label{eqn:appC.M1-0d}
\overline{\Omega_{\Gamma;A}\setminus\Omega_{\Gamma(i);A}}\subset \Omega_A(i)\subset \cT_{i;A}(R).
\end{align}
Therefore, by \eqref{eqn:appC.M1-0a}-\eqref{eqn:appC.M1-0d} and Lemma \ref{lem:branching.locus-2}, the following holds: 
\begin{itemize}
\item The map $\Pi^{\Gamma(i)}_\Gamma$ restricted to $\Pi_{\Gamma(i)}(\TT_A\setminus \overline{\cT_{i;A}(\pi/11)})$ gives a smooth, everywhere nondegenerate map from the manifold $\Pi_{\Gamma(i)}(\TT_A\setminus \overline{\cT_{i;A}(\pi/11)})$ to the manifold $\Pi_\Gamma(\TT_A\setminus \overline{\cT_{i;A}(\pi/11)})$. 
\item $\Pi_{\Gamma}(\TT_A\setminus \overline{\cT_{i;A}(\pi/11)})$ is an open subset of $\Gamma\backslash\TT_A$. Similarly, $\Pi_{\Gamma(i)}(\TT_A\setminus \overline{\cT_{i;A}(\pi/11)})$ is  an open subset of $\Gamma(i)\backslash\TT_A$.
\item $\Pi_\Gamma(\TT_A\setminus \cT_{i;A}(\pi/10))$ is a compact subset of $\Gamma\backslash\TT_A$, with $\Pi_{\Gamma}(\TT_A\setminus \overline{\cT_{i;A}(\pi/11)})$ being an open neighborhood of it. Similarly, $\Pi_{\Gamma(i)}(\TT_A\setminus \cT_{i;A}(\pi/10))$ is a compact subset of $\Gamma(i)\backslash\TT_A$, with $\Pi_{\Gamma(i)}(\TT_A\setminus \overline{\cT_{i;A}(\pi/11)})$ being an open neighborhood of it.
\end{itemize}
For any $R>0$, we let $\tc_R:[-R,R]\to\Gamma(i)\backslash\TT_A$ be a smooth curve such that $\Pi^{\Gamma(i)}_\Gamma\circ\tc_R=c|_{[-R,R]}$ and that $\tc_R'(0)=v$. By \eqref{eqn:appC.M1-0a} and the three bullet points in the above, $\tc_R$ is well-defined for some $R>0$. Moreover, if $\tc_{R_1},\tc_{R_2}$ are well-defined for some $R_1>R_2>0$, then $\{t\in[-R_2,R_2]:\tc_{R_1}(t)=\tc_{R_2}(t)\}$ is a nonempty, open and closed subset of $[-R_2,R_2]$. In particular, 
\begin{align}\label{eqn:appC.M1-0e}
\tc_{R_1}|_{[-R_2,R_2]}=\tc_{R_2}.
\end{align} 
Let $R_0:=\sup\{R>0:\tc_R\text{ is well-defined.}\}$. We then claim that $R_0=\infty$. Assume the contrary that $R_0<\infty$, then by \eqref{eqn:appC.M1-0e}, there exists a smooth curve $\tc:(-R_0,R_0)\to\Gamma(i)\backslash\TT_A$ such that $\Pi^{\Gamma(i)}_\Gamma\circ \tc=c|_{(-R_0,R_0)}$ and that $\tc'(0)=v$. Notice that$ |(\Pi^{\Gamma(i)}_\Gamma)^{-1}(q)|\leq |\Gamma|<\infty$ for any $q\in\Gamma\backslash\TT_A$. By the above three bullet points and the above fact, there exist an open neighborhood $\cV_+$ of $c(R_0)$ and an open neighborhood $\cV_-$ of $c(-R_0)$, such that the restriction of $\Pi^{\Gamma(i)}_{\Gamma}$ to any connected component of $(\Pi^{\Gamma(i)}_{\Gamma})^{-1}(\cV_+)$ ($(\Pi^{\Gamma(i)}_{\Gamma})^{-1}(\cV_-)$ resp.) is a diffeomorphism onto $\cV_+$ ($\cV_-$ resp.). Choose $\epsilon>0$ such that $c^{-1}(\cV_{\pm})\supset (\pm R-\epsilon,\pm R+\epsilon)$. Then, by the above property of $\cV_\pm$, one can extend $\tc$ smoothly to $(-R_0-\epsilon,R_0+\epsilon)$ such that $\Pi^{\Gamma(i)}_\Gamma\circ\tc=c|_{(-R_0-\epsilon,R_0+\epsilon)}$. This contradicts the definition of $R_0$. Therefore, we conclude that $R_0=\infty$.

By \eqref{eqn:appC.M1-0e} and the fact that $R_0=\infty$, there exists a smooth curve $\tc:\RR\to \Gamma(i)\backslash\TT_A$ such that $\Pi^{\Gamma(i)}_\Gamma\circ\tc=c$ and that $\tc'(0)=v$.
Apply Lemma \ref{lem:properties.of.g.Gamma-4} to $\tc$, we have that $\tc$ is a geodesic in $(\Gamma(i)\backslash\TT_A,g_{\Gamma;A;\{i\}})$ such that $\Pi^{\Gamma(i)}_\Gamma\circ\tc=c$ and that $\tc'(0)=v$. Moreover, by \eqref{eqn:appC.M1-0b}, $\tc(\RR)\subset (\Gamma(i)\backslash\TT_A)\setminus(\Pi_{\Gamma(i)}( \cT_{i;A}(\pi/10)))$. This proves that for any $v\in\widetilde{\cQ}_{\Gamma;A}(i)$, $v$ is  tangent to a bi-infinite geodesic in $(\Gamma(i)\backslash\TT_A, g_{\Gamma;A;\{i\}})$ which never intersects $\Pi_{\Gamma(i)}(\cT_{i;A}(\pi/10))$. 

Now we prove the reverse direction in \hyperlink{M1}{(M1)}. Assume that $v=\tc'(0)$ for a bi-infinite geodesic $\tc:\RR\to\Gamma(i)\backslash\TT_A$ in $(\Gamma(i)\backslash\TT_A, g_{\Gamma;A;\{i\}})$ which never intersects $\Pi_{\Gamma(i)}(\cT_{i;A}(\pi/10))$. Then by Lemma \ref{lem:tube.index.properties-1-4}, we have 
\begin{align}\label{eqn:appC.M1-0f}
\tc(\RR)\subset \Pi_{\Gamma(i)}(\TT_A\setminus \cT_{i;A}(\pi/10)).
\end{align}
By Lemma \ref{lem:properties.of.g.Gamma-3}, $g_{\Gamma;A;\{i\}}=(\Pi^{\Gamma(i)}_\Gamma)^*g_{\Gamma;A}$ when restricted to $\Pi_{\Gamma(i)}(\TT_A\setminus \cT_{i;A}(\pi/10))$. Therefore, by \eqref{eqn:appC.M1-0f}, $\Pi^{\Gamma(i)}_\Gamma\circ\tc$ is a geodesic in $(\Gamma\backslash\TT_A,g_{\Gamma;A})$ with 
\begin{align}\label{eqn:appC.M1-0g}
\Pi^{\Gamma(i)}_\Gamma\circ\tc(\RR)\subset\Pi_{\Gamma}(\TT_A\setminus \cT_{i;A}(\pi/10)).
\end{align}
By Lemma \ref{lem:tube.index.properties-1-4} and \eqref{eqn:appC.M1-0g}, $\Pi^{\Gamma(i)}_\Gamma\circ\tc$ never intersects $\Pi_{\Gamma}( \cT_{i;A}(\pi/10))$. Since $(\Pi^{\Gamma(i)}_\Gamma)_*(v)=(\Pi^{\Gamma(i)}_\Gamma)_*(\tc'(0))$ is tangent to $\Pi^{\Gamma(i)}_\Gamma\circ\tc$. Hence, by \eqref{eqn:vectors.avoiding.j.tube}, \eqref{eqn:middle.lift.vectors.avoid.j.tube} and the above, we have $v\in\widetilde{\cQ}_{\Gamma;A}(i)$. This proves \hyperlink{M1}{(M1)}.

\subsection{Proof of Lemma \ref{ellip coor}}
Before we prove Lemma \ref{ellip coor}, we first introduce two technical lemmas.
\begin{lem}\label{elementary lemma}
For any integer $m\geq 2$ and $m$ pairwise distinct real numbers $t_1,\dots ,t_m$, we have 
$$\sum_{i=1}^m\prod_{\substack{1\leq j\leq m\\j\neq i}}\frac{1}{t_i-t_j}=0.$$
\end{lem}
\begin{proof}[Proof of Lemma \ref{elementary lemma}]
We prove by induction on $m$. Clearly Lemma \ref{elementary lemma} holds when $m=2$. Suppose Lemma \ref{elementary lemma} is true for the case of $m-1$, then
\begin{align*}
0-0=&\left(\sum_{i=1}^{m-1}\displaystyle\prod_{\substack{1\leq j\leq m-1\\j\neq i}}\frac{1}{t_i-t_j}\right)-\left(\sum_{\substack{1\leq i\leq m\\i\neq m-1}}\displaystyle\prod_{\substack{1\leq j\leq m\\j\neq m-1,j\neq i}}\frac{1}{t_i-t_j}\right)\\
=&\sum_{i=1}^{m-2}\left(\displaystyle\prod_{\substack{1\leq j\leq m-1\\j\neq i}}\left(\frac{1}{t_i-t_j}\right)-\prod_{\substack{1\leq j\leq m\\j\neq m-1,j\neq i}}\left(\frac{1}{t_i-t_j}\right)\right)+\prod_{j=1}^{m-2}\left(\frac{1}{t_{m-1}-t_j}\right)-\prod_{j=1}^{m-2}\left(\frac{1}{t_{m}-t_j}\right)\\
=&(t_{m-1}-t_m)\sum_{i=1}^m\prod_{\substack{1\leq j\leq m\\j\neq i}}\frac{1}{t_i-t_j}.
\end{align*}
Case $m$ of Lemma \ref{elementary lemma} then follows immediately from the above and the fact that $t_{m-1}\neq t_m$.
\end{proof}

\begin{lem}\label{ez linear alg-1}
Let $\cA=(a_{ij})_{1\leq i,j\leq n}$ be an $n\times n$ matrix with real entries. Suppose that $\cA$ satisfies the following properties:
\begin{enumerate}
\item[(1).] $a_{ij}>0$ whenever $i\leq j$ and $a_{ij}<0$ whenever $i>j$;
\item[(2).] For any $i,j,k\in\{1,\dots ,n\}$ with $i<j$, if $a_{ki}$ and $a_{kj}$ has the same sign, then $a_{ki}>a_{kj}$.
\end{enumerate}
Then $\det(\cA)>0$.
\end{lem}
\begin{proof}[Proof of Lemma \ref{ez linear alg-1}]
We prove Lemma \ref{ez linear alg-1} using induction. The lemma obviously holds when $n=1$. Suppose Lemma \ref{ez linear alg-1} holds for the case $n-1$, then in the case of $n$, we define 
$$\cL=\matv{1&-\frac{a_{12}}{a_{11}}&-\frac{a_{13}}{a_{11}}&\cdots&-\frac{a_{1n}}{a_{11}}}{0&1&0&\cdots&0}{0&0&1&\cdots&0}{\vdots&\vdots&\vdots&\ddots&\vdots}{0&0&0&\cdots&1}\mathrm{~and~hence~}\cA\cL=\matii{a_{11}&0}{\cC&\cA'},$$
where $\cA'=(a'_{ij})_{1\leq i,j\leq n-1}$ satisfies $a'_{ij}=a_{(i+1)(j+1)}-\left(a_{1(j+1)}a_{(i+1)1}/a_{11}\right)$ for any $i,j\in\{1,\dots ,n-1\}$. If $A'$ satisfies the two properties in Lemma \ref{ez linear alg-1}, then by the case $n-1$, we have $\det(A)=a_{11}\det(A')>0$.

Hence, it suffices to show that $\cA'=(a'_{ij})_{1\leq i,j\leq n-1}$ satisfies the two properties in Lemma \ref{ez linear alg-1}.
\begin{enumerate}
\item When $i\leq j$, by the first property of $\cA$, $a_{(i+1)(j+1)}, a_{11}, a_{1(j+1)}>0 $ and $a_{(i+1)1}<0$. Hence 
$$a'_{ij}=a_{(i+1)(j+1)}-\left(a_{1(j+1)}a_{(i+1)1}/a_{11}\right)>a_{(i+1)(j+1)}>0.$$
When $i>j$, by the first property of $\cA$, $ a_{11}, a_{1(j+1)}>0 $ and $a_{(i+1)(j+1)},a_{(i+1)1}<0$, Moreover, by the second property of $\cA$, we have $a_{1(j+1)}/a_{11}\in(0,1)$ and that $a_{(i+1)(j+1)}<a_{(i+1)1}<0$. Therefore
$$a'_{ij}=a_{(i+1)(j+1)}-\left(a_{1(j+1)}a_{(i+1)1}/a_{11}\right)<a_{(i+1)(j+1)}-a_{(i+1)1}<0.$$
This verifies the first property for $\cA'$.
\item For any $i,j,k\in\{1,\dots, n-1\}$ with $i<j$, by the first property for $\cA'$, $a_{ki}'$ and $a_{kj}'$ have the same sign if and only if $k\leq i$ or $k>j$. By the first property for $\cA$, we have $a_{(k+1)1}<0$ and $a_{11},a_{1(i+1)},a_{1(j+1)}>0$. By the second property for $\cA$, we have $a_{1(i+1)}-a_{1(j+1)}>0$ and $a_{(k+1)(i+1)}-a_{(k+1)(j+1)}>0$. Therefore,
\begin{align*}
a_{ki}'-a_{kj}'=&\left(a_{(k+1)(i+1)}-\frac{a_{(k+1)1}a_{1(i+1)}}{a_{11}}\right)-\left(a_{(k+1)(j+1)}-\frac{a_{(k+1)1}a_{1(j+1)}}{a_{11}}\right)\\
=&(a_{(k+1)(i+1)}-a_{(k+1)(j+1)})+\left(-\frac{a_{(k+1)1}}{a_{11}}\right)\cdot(a_{1(i+1)}-a_{1(j+1)})>0.
\end{align*}
This verifies the second property for $\cA'$.
\end{enumerate}
\end{proof}

Now, we prove Lemma \ref{ellip coor}.
\begin{proof}[Proof of Lemma \ref{ellip coor}]
We first prove that \eqref{ellip coor formula} gives a solution to $\cA x=(1,1,\dots ,1)^T$. Let $f_j(\xi_1,\dots ,\xi_{n+1})$ be the $j$-th component of $\cA y$ for any $j\in\{1,\dots ,n+1\}$. We have 
\begin{align}\label{jth ellip coor}
f_j(\xi_1,\dots ,\xi_{n+1})=\sum_{k=1}^{n+1}\frac{y_k}{k-\xi_j}
=\sum_{k=1}^{n+1}\frac{1}{k-\xi_j}\cdot\frac{\displaystyle\prod_{l=1}^{n+1}(k-\xi_l)}{\displaystyle\prod_{\substack{1\leq l\leq n+1\\ l\neq k}}(k-l)}
=\sum_{k=1}^{n+1}\frac{\displaystyle\prod_{\substack{1\leq l\leq n+1\\ l\neq j}}(k-\xi_l)}{\displaystyle\prod_{\substack{1\leq l\leq n+1\\ l\neq k}}(k-l)}.
\end{align}
In particular, $f_j$ is a polynomial on $\xi_1,\dots ,\xi_{n+1}$ which is independent on $\xi_j$. Notice that for any $l\in\{1,\dots ,n+1\}$ and any $\xi_1,\dots ,\xi_{l-1},\xi_{l+1},\dots ,\xi_{n+1}$ fixed, $f_j(\xi_1,\dots ,\xi_{l-1},\cdot,\xi_{l+1},\dots ,\xi_{n+1})$ is a linear function. Since
$$f_j(1,\dots ,{n+1})=\sum_{k=1}^{n+1}\frac{\displaystyle\prod_{\substack{1\leq l\leq n+1\\ l\neq j}}(k-l)}{\displaystyle\prod_{\substack{1\leq l\leq n+1\\ l\neq k}}(k-l)}=\frac{\displaystyle\prod_{\substack{1\leq l\leq n+1\\ l\neq j}}(j-l)}{\displaystyle\prod_{\substack{1\leq l\leq n+1\\ l\neq j}}(j-l)}=1,$$
it suffices to show that $\del^\alpha f_j/\del \xi_{\alpha}(1,\dots ,{n+1})=0$ for any $j\in\{1,\dots ,n+1\}$ and for any multi-index $\alpha$ in $\{1,\dots ,n+1\}\setminus\{j\}$ with no repetitions. By \eqref{jth ellip coor} and Lemma \ref{elementary lemma}, we have 
\begin{align*}
(-1)^{|\alpha|}\frac{\del^\alpha f_j}{\del\xi_{\alpha}}(1,\dots ,{n+1})=&\sum_{ k \in\alpha\sqcup\{j\}}\frac{\displaystyle\prod_{\substack{1\leq l\leq n+1\\ l\not\in\alpha\sqcup\{j\}}}(k-l)}{\displaystyle\prod_{\substack{1\leq l\leq n+1\\l\neq k}}(k-l)} 
=\sum_{ k \in\alpha\sqcup\{j\}}\frac{1}{\displaystyle\prod_{\substack{l\in\alpha\sqcup\{j\}\\l\neq k}}(k-l)}=0
\end{align*}

Hence we can conclude that $f_j(\xi_1,\dots ,\xi_{n+1})\equiv f_j(1,\dots ,{n+1})=1$ and therefore $y$ defined in \eqref{ellip coor formula} solves $\cA x=(1,1,\dots ,1)^T$.

It remains for us to prove that $y$ is the unique solution to $\cA x=(1,1,\dots ,1)^T$ when $\xi_1<1<\xi_2<2<\cdots <\xi_{n+1}<n+1$ and it suffices to show that $\det(\cA)\neq 0$ in this case. This follows directly from Lemma \ref{ez linear alg-1}.
\end{proof}

\bibliographystyle{alpha}
\bibliography{BGR-bib}

\end{document}